\documentclass{article}

\usepackage{arxiv}

\usepackage{amsmath, amssymb, amsthm, mathrsfs, mathtools}
\usepackage[T1]{fontenc}
\usepackage[utf8]{inputenc}
\usepackage[hidelinks]{hyperref}
\usepackage{url}

\newtheorem{theorem}{Theorem}[section]
\newtheorem{lemma}[theorem]{Lemma}
\newtheorem{proposition}[theorem]{Proposition}
\newtheorem{corollary}[theorem]{Corollary}
\newtheorem{assumption}[theorem]{Assumption}
\newtheorem{hypothesis}[theorem]{Hypothesis}
\newtheorem{convention}[theorem]{Convention}
\theoremstyle{remark}
\newtheorem{remark}[theorem]{Remark}
\newtheorem{conjecture}[theorem]{Conjecture}

\newcommand{\Tr}{\operatorname{Tr}}
\newcommand{\E}{\mathbb{E}}
\newcommand{\PP}{\mathbb{P}}
\newcommand{\R}{\mathbb{R}}
\newcommand{\C}{\mathbb{C}}
\newcommand{\N}{\mathbb{N}}
\newcommand{\Hilbert}{\mathscr{H}}
\newcommand{\TraceClass}{\mathscr{T}^1}
\newcommand{\HilbertSchmidt}{\mathscr{T}^2}
\newcommand{\Bdd}{\mathcal{B}}
\newcommand{\Md}{M_d(\mathbb{C})}
\newcommand{\Lop}{\mathcal{L}}
\newcommand{\Mop}{\mathcal{M}}
\newcommand{\Diss}{\mathcal{D}}
\newcommand{\Bmf}{\mathsf{B}}
\newcommand{\Gemp}{\widehat{G}}
\newcommand{\Nex}{\widehat{\mathcal N}}
\newcommand{\Mid}{M^{\mathrm{id}}}
\newcommand{\Mmf}{M^{\mathrm{mf}}}
\newcommand{\eps}{\varepsilon}
\newcommand{\norm}[1]{\lVert #1\rVert}
\newcommand{\ip}[2]{\langle #1,\,#2\rangle}
\newcommand{\av}[1]{\langle #1\rangle}
\newcommand{\HS}{\mathrm{HS}}

\title{A Dynamic Central Limit Theorem for Mean-Field Quantum Filtering}
\author{%
	Sangsidhya Kar \\
	Department of Statistics \\
	Presidency University \\
	Kolkata \\
}
\date{\today}

\renewcommand{\shorttitle}{Dynamic CLT for Mean-Field Quantum Filtering}

\begin{document}
	
	\maketitle

\begin{abstract}
	We study Open Problem 3 of Kolokoltsov's survey \emph{Quantum filtering and propagation of chaos for open quantum systems} (2026) on $\Hilbert=\C^d$ ($2\le d<\infty$ fixed): a dynamic central limit theorem for $N$ mean-field coupled quantum particles under continuous diffusive (homodyne) measurement. The tagged conditional state $\Gamma^{(1)}_t$ is compared with the one-particle filter $\gamma_t$ driven by its own innovation, the mean-field Hamiltonian evaluated at the deterministic Hartree--Lindblad mean $\eta_t$. The fluctuation $F^N_t=\sqrt N(\Gamma^{(1)}_t-\gamma_t)$ is not autonomous in the limit: it is driven through the linearised filter by its own innovation, by an idiosyncratic Gaussian martingale carried by the connected correlations with the other particles, and by the empirical field $\Gemp^N_t=\sqrt N(\overline\Gamma_t-\eta_t)$, which itself satisfies an autonomous linear Gaussian equation, so that $(F,\Gemp)$ solves a closed system. Unconditionally, $(F^N,\Gemp^N)$ is tight, every limit point solves this system driven by a Brownian motion orthogonal to a square-integrable martingale pair, and the mean covariance of the idiosyncratic martingale is identified as the limit of the rescaled pair correlations. The remaining clauses, concerning signs rather than amplitudes, reduce to a single decorrelation statement fixing the law of the driving pair. It survives in a conditional form, with the common response to the empirical field and the tagged-slot component projected out; without those projections it is false for generic data, which is what fixes them. The conditional form is proved at leading order on a shrinking horizon and from below at its exact orders on every horizon; it is needed only for the concentration of the driving covariances, the Gaussianity of the pair, and uniqueness, and what remains open is one-sided and above the leading order.
\end{abstract}
	\tableofcontents
	
	\section{Introduction}\label{sec:intro}
	
	\subsection{The problem}
	
	In \cite{Kolokoltsov2026} (\S4.1--4.3, building on \cite{Kol2021LLN,Kol2022MFG}) it is shown that for $N$ identical quantum particles in $\Hilbert^{\otimes N}$, mean-field coupled through a self-adjoint interaction $A$ and individually subject to continuous diffusive measurement through coupling operators $L_j$, the one-particle reduced state $\Gamma^{(j)}_t$ converges, as $N\to\infty$, to the solution of a one-particle Belavkin equation with a mean-field Hamiltonian, at the rate
	\begin{equation}\label{eq:KP-rate-trace}
		\E\big\|\Gamma^{(j)}_{t}-\gamma_{j,t}\big\|_1 \;=\; O(N^{-1/4}),
	\end{equation}
	obtained from the Pickl/Knowles--Pickl functional \cite{Pickl2011,KnowlesPickl2010} together with the trace-distance sandwich $\alpha\le\|\cdot\|_1\le2\sqrt{2\alpha}$, $\alpha_{N,j}(t)=1-\Tr(\gamma_{j,t}\Gamma^{(j)}_t)$. Two features of \eqref{eq:KP-rate-trace} govern everything below. First, the comparison state $\gamma_{j,t}$ carries the label $j$: it is the one-particle filter driven by the innovation $B_j$ of particle $j$, not by an aggregate of the innovations. Second, the rate is stated for the reduced state of a single particle; nothing in it constrains the connected correlations of two or three distinct particles, which are the objects that carry the fluctuation.
	
	Open Problem 3 of \cite[\S8]{Kolokoltsov2026} asks for the corresponding dynamic central limit theorem: a description of
	\[
	F^N_t\;:=\;\sqrt N\big(\Gamma^{(1)}_t-\gamma_t\big)
	\]
	as $N\to\infty$, in the spirit of the fluctuation theorems for the Smoluchowski and Boltzmann equations of \cite{Kol2010CLT,Kol2010Book}. We carry this out for $\Hilbert=\C^d$, $2\le d<\infty$ fixed, so that every operator is bounded and every closed bounded set is compact.
	
	\subsection{What the limiting object is}
	
	Three choices fix the problem, and each is forced rather than conventional.
	
	\emph{The coupling.} The law of $F^N$ depends on which noise drives the comparison filter $\gamma$, since $F^N$ is a difference on a single probability space. We drive $\gamma$ by the innovation $B_1$ of the tagged particle (Convention \ref{conv:coupling}). Under the aggregate coupling $N^{-1/2}\sum_jB_j$ the fluctuation diverges: $\E\|F^N_t\|_1^2\ge cNt$ already for $A=0$ (Proposition \ref{prop:refutation}(iii)). Convention \ref{conv:coupling} is the coupling implicit in \eqref{eq:KP-rate-trace}.
	
	\emph{The mean field in the limiting equation.} The mean-field Hamiltonian of the limiting equation is $\Bmf(\eta_t)$ with $\eta_t=\E\gamma_t$ deterministic, not $\Bmf(\gamma_t)$. The field felt by the tagged particle is $\Bmf(\overline\Gamma_t)$ with $\overline\Gamma_t$ the average of the reduced states, and $\overline\Gamma_t$ concentrates on $\eta_t$; linearising around the random $\gamma_t$ leaves a discrepancy of order one after the $\sqrt N$ rescaling.
	
	\emph{The correlations.} The connected correlations that enter are those of \emph{distinct} particles,
	\[
	\Delta^{(ij)}_t=\Gamma^{(i,j)}_t-\Gamma^{(i)}_t\otimes\Gamma^{(j)}_t,\qquad \Delta^{(ijk)}_t \ \text{as in \eqref{eq:cumulant3}},
	\]
	and not the comparison of a two-particle marginal with the tensor square of a single marginal. The latter is bounded below, uniformly in $N$, by a positive multiple of $t$ in expectation already at $A=0$ (Proposition \ref{prop:refutation}(i)--(ii)), whereas $\Delta^{(ij)}$ and $\Delta^{(ijk)}$ vanish identically there.
	
	With these choices the limit is a system. Writing $\overline\Gamma_t=N^{-1}\sum_j\Gamma^{(j)}_t$ and $\Gemp^N_t=\sqrt N(\overline\Gamma_t-\eta_t)$, the pair $(F^N,\Gemp^N)$ converges, along subsequences unconditionally and in full under the conditional decorrelation input of Section \ref{sec:decorrelation}, to the solution of
	\begin{align}
		d\gamma_t&=\big(-i[H+\Bmf(\eta_t),\gamma_t]+\Diss\gamma_t\big)dt+\Mop[\gamma_t]\,dW_t,\label{eq:intro-sys-1}\\
		d\Gemp_t&=\big(-i[H+\Bmf(\eta_t),\Gemp_t]-i[\Bmf(\Gemp_t),\eta_t]+\Diss\Gemp_t\big)dt+d\mathcal W_t,\label{eq:intro-sys-2}\\
		dF_t&=\big(-i[H+\Bmf(\eta_t),F_t]+\Diss F_t-i[\Bmf(\Gemp_t),\gamma_t]\big)dt+D\Mop[\gamma_t](F_t)\,dW_t+dV_t,\label{eq:intro-sys-3}
	\end{align}
	with $F_0=\Gemp_0=0$, where $W$ is a standard Brownian motion, $\mathcal W$ and $V$ are continuous Gaussian martingales, $W$ is independent of $(\mathcal W,V)$, and $\mathcal W$ and $V$ are in general correlated. The field $\Gemp$ is autonomous and Gaussian, with the deterministic covariance $\Sigma^{\mathrm{mf}}_t(A_0,B_0)=\E[\Tr(A_0\Mop[\gamma_t])\Tr(B_0\Mop[\gamma_t])]$. The covariance of $V$ is determined by the limit of the rescaled pair correlations $N\Delta^{(1j)}$ and is not a function of the one-particle state alone.
	
	\subsection{Method}
	
	The $N$-particle filtering equation preserves purity, so $\Gamma^N_t=|\Psi_t\rangle\langle\Psi_t|$ for a vector $\Psi_t$ solving a norm-preserving equation on $\Hilbert^{\otimes N}$ (Proposition \ref{prop:purity-N}). The comparison objects are the $N$ reference filters $\gamma_{j,t}=|\phi_{j,t}\rangle\langle\phi_{j,t}|$, one for each label, independent and identically distributed, whose common mean $\eta$ solves the Hartree--Lindblad equation and does not depend on $N$ (Lemma \ref{lem:reference}). The small quantities are the excitation amplitudes
	\[
	\eps_T:=\Big\|\prod_{i\in T}q_{i,t}\,\Psi_t\Big\|,\qquad q_{i,t}=1-|\phi_{i,t}\rangle\langle\phi_{i,t}|\ \text{on slot }i,
	\]
	indexed by finite label sets $T$; they vanish identically at $A=0$. Their expected orders are $\eps_{\{i\}}\asymp N^{-1/2}$, $\eps_{\{i,j\}}\asymp N^{-1}$, $\eps_{\{i,j,k\}}\asymp N^{-3/2}$, since the interaction creates excitations in pairs with amplitude $N^{-1}$, while single excitations arise only from the deviation of the empirical mean field from its deterministic limit.
	
	Theorems \ref{thm:dict-two} and \ref{thm:dict-three} bound $\Gamma^{(i)}-\gamma_i$, $\Delta^{(ij)}$ and $\Delta^{(ijk)}$ pathwise by these amplitudes, with the cancellation of every contribution of order lower than the expected one. This reduces the quantitative input of the whole analysis to moment bounds on three scalar random variables.
	
	\subsection{Contents}
	
	Sections \ref{sec:model}--\ref{sec:objects} set up the model, the connected correlations, the reference filters, and the coupling, and prove that these objects are forced. Section \ref{sec:wellposed} gives well-posedness of the limiting flow, of the linearised equation, and Fr\'echet differentiability. Section \ref{sec:excitation} contains the purity statement and the amplitude bounds. Section \ref{sec:exact} contains the finite-$N$ identities: the martingale coefficients of the reduced states, the It\^o-complete pair-correlation equation, and the decompositions of $F^N$ and $\Gemp^N$. Section \ref{sec:amplitude-hyp} proves the mean excitation bound (A1), states Hypothesis (A) for the remaining clauses, and derives the correlation rates, the uniform fluctuation moment bound, and the chaos estimate in expectation. Sections \ref{sec:diagonal} and \ref{sec:moments} carry out the second-order analysis: the dynamics of the diagonal excitation observables, the bound on all joint moments of the excitation number and the empirical deviation, clause (A3), the moments of products over distinct label sets, the exponential moment of the excitation number at a deterministic rate, and clause (A2), all on every horizon. Section \ref{sec:covariance} identifies the idiosyncratic covariance. Section \ref{sec:decorrelation} reduces clauses (A4)--(A6) to a single decorrelation statement, gives it in the conditional form in which it survives and proves that form at the leading order on a shrinking horizon, disproves the form obtained by dropping its projections, and restates the clauses for the projected quantities. Sections \ref{sec:martingale-limits}--\ref{sec:clt} prove the joint martingale limit, tightness, and the dynamic central limit theorem, the last in a conditional sharp form and in an unconditional form in which the limit equations and the mean covariance of the idiosyncratic noise are identified. Section \ref{sec:status} sets out the logical status of the hypothesis and the questions that remain.
	
	\section{The model, notation, and standing facts}\label{sec:model}
	
	\subsection{The $N$-particle dynamics}
	
	We take $\Hilbert=\C^d$, $2\le d<\infty$ fixed. Then $H=H^\ast\in\Md$, $L\in\Md$, and $A=A^\ast\in M_{d^2}(\C)$ with the kernel symmetry
	\begin{equation}\label{eq:A-symmetry}
		A(x,y;x',y')=A(y,x;y',x'),\qquad A(x,y;x',y')=A(x',y';x,y),
	\end{equation}
	as in \cite[\S4.2--4.3]{Kolokoltsov2026}. The $N$-particle conditional state solves \cite[Eq.\ (134), $k=0$]{Kolokoltsov2026}
	\begin{equation}\label{eq:N-particle}
		d\Gamma^N_t=-i[H_N,\Gamma^N_t]\,dt+\sum_{j=1}^N\Diss_j\Gamma^N_t\,dt
		+\sum_{j=1}^N\Big(L_j\Gamma^N_t+\Gamma^N_tL_j^\ast-\Tr\big((L_j+L_j^\ast)\Gamma^N_t\big)\Gamma^N_t\Big)dB_j(t),
	\end{equation}
	with
	\[
	H_N=\sum_{j\le N}H_j+\frac1N\sum_{i<j\le N}A_{ij},
	\]
	$\Diss_j$ the one-particle dissipator on slot $j$, and $B_1,\dots,B_N$ independent standard Brownian motions on a filtered probability space satisfying the usual conditions. The initial datum $\Gamma^N_0$ is permutation symmetric; from Section \ref{sec:excitation} on it is the pure product state $\gamma_0^{\otimes N}$ with $\gamma_0=|\psi_0\rangle\langle\psi_0|$ deterministic. Global well-posedness and preservation of the state space for \eqref{eq:N-particle} are inherited from \cite[Thm.\ 2.8, 2.9]{Kolokoltsov2026}.
	
	\subsection{Notation}
	
	Set
	\begin{equation}\label{eq:ops}
		\Diss\rho:=L\rho L^\ast-\tfrac12\{L^\ast L,\rho\},\qquad
		\Mop[\rho]:=L\rho+\rho L^\ast-\Tr\big((L+L^\ast)\rho\big)\rho,\qquad
		\Bmf(\rho):=\Tr_2\big(A(1\otimes\rho)\big),
	\end{equation}
	\begin{equation}\label{eq:DM}
		D\Mop[\gamma](\xi):=L\xi+\xi L^\ast-\Tr\big((L+L^\ast)\xi\big)\gamma-\Tr\big((L+L^\ast)\gamma\big)\xi,
	\end{equation}
	and $\mathcal S=\{\rho\in\Md:\rho\ge0,\ \Tr\rho=1\}$. The map $\Bmf$ is linear, preserves self-adjointness, and $\|\Bmf(\rho)\|_\infty\le\|A\|_\infty\|\rho\|_1$ by Lemma \ref{lem:partial-trace}; all coefficients in \eqref{eq:ops}--\eqref{eq:DM} are therefore bounded on $\mathcal S$.
	
	For $S\subseteq\{1,\dots,N\}$ write $\Gamma^{(S)}_t:=\Tr_{S^c}\Gamma^N_t$ for the reduced state of the particles in $S$, tensor slots ordered by increasing label unless stated otherwise, and $m^{(j)}_t:=\Tr((L+L^\ast)\Gamma^{(j)}_t)$. For distinct labels $i,j$ (particle $i$ in the first slot),
	\begin{equation}\label{eq:cumulant2}
		\Delta^{(ij)}_t:=\Gamma^{(i,j)}_t-\Gamma^{(i)}_t\otimes\Gamma^{(j)}_t,
	\end{equation}
	and for distinct $i,j,k$ (slots ordered $i,j,k$) the irreducible three-body cumulant
	\begin{equation}\label{eq:cumulant3}
		\Delta^{(ijk)}_t:=\Gamma^{(i,j,k)}_t-\Gamma^{(i)}_t\otimes\Gamma^{(j)}_t\otimes\Gamma^{(k)}_t-\Delta^{(ij)}_t\otimes\Gamma^{(k)}_t-\big[\Delta^{(ik)}_t\big]_{13}\otimes\big[\Gamma^{(j)}_t\big]_2-\Gamma^{(i)}_t\otimes\Delta^{(jk)}_t,
	\end{equation}
	subscripts indicating slot placement. These are the connected correlations of distinct particles. Finally $\overline\Gamma_t:=N^{-1}\sum_{j\le N}\Gamma^{(j)}_t$ and $\Gamma^{\ne j}_t:=(N-1)^{-1}\sum_{k\ne j}\Gamma^{(k)}_t$.
	
	\subsection{The finite-dimensional setting}\label{ssec:findim}
	
	Since $\Hilbert=\C^d$,
	\begin{equation}\label{eq:norm-equiv}
		\TraceClass=\HilbertSchmidt=\Bdd(\Hilbert)=\Md\quad\text{as vector spaces},\qquad
		\|\rho\|_\infty\le\|\rho\|_2\le\|\rho\|_1\le\sqrt d\,\|\rho\|_2\le d\,\|\rho\|_\infty,
	\end{equation}
	and on $M_{d^k}(\C)$, $\|X\|_2\le\|X\|_1\le d^{k/2}\|X\|_2$, with $\|X\otimes Y\|_1=\|X\|_1\|Y\|_1$. It\^o calculus is carried out in the finite-dimensional Hilbert space $(\Md,\|\cdot\|_2)$ and transferred to $\|\cdot\|_1$ through \eqref{eq:norm-equiv}. We write $\mathcal C_{ad}([0,T];\Md)$ for the Banach space of adapted continuous $\Md$-valued processes with $\E\sup_{t\le T}\|X_t\|_1^2<\infty$, with the norms $\|X\|_{\beta,T}=(\E\sup_{t\le T}e^{-\beta t}\|X_t\|_1^2)^{1/2}$.
	
	\begin{assumption}[Standing hypotheses]\label{assum:generators}
		$\Hilbert=\C^d$ with $2\le d<\infty$; $H=H^\ast\in\Md$; $L\in\Md$; $A=A^\ast\in M_{d^2}(\C)$ satisfies \eqref{eq:A-symmetry}. No boundedness clause is imposed as a restriction: it holds by $d<\infty$.
	\end{assumption}
	
	\subsection{Operator bounds and the partial-trace calculus}
	
	\begin{lemma}[Commutator bound]\label{lem:commutator}
		$\|[B,\rho]\|_1\le2\|B\|_\infty\|\rho\|_1$ for $B,\rho\in\Md$.
	\end{lemma}
	\begin{proof}
		Triangle inequality and $\|B\rho\|_1\le\|B\|_\infty\|\rho\|_1$.
	\end{proof}
	
	\begin{lemma}[Partial trace bound]\label{lem:partial-trace}
		For $A\in\Bdd(\Hilbert\otimes\Hilbert)$ and $\delta\in\Md$, $\|\Tr_2(A(1\otimes\delta))\|_\infty\le\|A\|_\infty\|\delta\|_1$.
	\end{lemma}
	\begin{proof}
		For unit vectors $\phi,\psi$, $|\langle\phi,\Tr_2(A(1\otimes\delta))\psi\rangle|=|\Tr(A(|\psi\rangle\langle\phi|\otimes\delta))|\le\|A\|_\infty\||\psi\rangle\langle\phi|\otimes\delta\|_1=\|A\|_\infty\|\delta\|_1$, using that the trace norm factorises over tensor products. Take the supremum over $\phi,\psi$.
	\end{proof}
	
	\begin{lemma}[Partial-trace calculus]\label{lem:pt-calculus}
		Let $Y\in\Bdd(\Hilbert^{\otimes N})$ and $S\subseteq\{1,\dots,N\}$.
		\begin{enumerate}
			\item[(PT1)] If $X$ acts on kept factors only, $\Tr_{S^c}((X\otimes1)Y)=X\Tr_{S^c}Y$ and $\Tr_{S^c}(Y(X\otimes1))=(\Tr_{S^c}Y)X$.
			\item[(PT2)] The partial trace over a factor is cyclic with respect to operators acting on that factor alone.
			\item[(PT3)] If $X$ acts on factor $j\notin S$ only, then $\Tr_{S^c}(X_jY)$ depends on $Y$ only through $\Gamma^{(S\cup\{j\})}$: with $j$ in the last slot, $\Tr_{S^c}(X_jY)=\Tr_{\mathrm{last}}((1\otimes\cdots\otimes X)\Gamma^{(S\cup\{j\})})$, and likewise for right multiplication.
			\item[(PT4)] $\Tr_2((\rho\otimes\sigma)A)=\rho\,\Bmf(\sigma)$ and $\Tr_2(A(\rho\otimes\sigma))=\Bmf(\sigma)\rho$, hence $\Tr_2[A,\rho\otimes\sigma]=[\Bmf(\sigma),\rho]$.
		\end{enumerate}
	\end{lemma}
	\begin{proof}
		All four follow from the duality $\Tr(Z\Tr_{S^c}Y)=\Tr((Z\otimes1)Y)$ and the definition \eqref{eq:ops} of $\Bmf$.
	\end{proof}
	
	\subsection{Exchangeability in law}\label{ssec:exch}
	
	Under \eqref{eq:N-particle} each particle carries its own innovation, so the flow intertwines a permutation $\sigma$ of the labels with the corresponding permutation of the noises: for permutation-symmetric $\Gamma^N_0$,
	\[
	\sigma\cdot\Gamma^N_t\big[(B_j)_j\big]=\Gamma^N_t\big[(B_{\sigma(j)})_j\big]\quad\text{pathwise}.
	\]
	The joint law of $\Gamma^N_\cdot$ and the driving noises is therefore invariant under a simultaneous permutation of labels and noises. Consequently, for distinct labels, the laws of $\|\Delta^{(ij)}_t\|_p$ and $\|\Delta^{(ijk)}_t\|_p$ do not depend on the choice of labels. We refer to this as \emph{exchangeability in law}; it is unconditional and is the only symmetry used below. Pathwise permutation symmetry of $\Gamma^N_t$ is a strictly stronger property, is not available under \eqref{eq:N-particle}, and is nowhere assumed.
	
	\section{The limiting objects and the coupling}\label{sec:objects}
	
	\subsection{The reference filters and the Hartree--Lindblad mean}
	
	\begin{lemma}[The $N$ reference filters]\label{lem:reference}
		Let $\gamma_0=|\psi_0\rangle\langle\psi_0|\in\mathcal S$ be pure and deterministic. The system
		\begin{equation}\label{eq:reference}
			d\gamma_{j,t}=\big(-i[H+\Bmf(\eta_t),\gamma_{j,t}]+\Diss\gamma_{j,t}\big)dt+\Mop[\gamma_{j,t}]\,dB_j(t),\qquad \gamma_{j,0}=\gamma_0,\quad \eta_t:=\E\gamma_{1,t},
		\end{equation}
		$j=1,\dots,N$, has a unique solution in $\mathcal C_{ad}([0,T];\Md)^N$ with values in $\mathcal S$. Moreover:
		\begin{enumerate}
			\item[(a)] $\gamma_{1,\cdot},\dots,\gamma_{N,\cdot}$ are independent and identically distributed, and $\gamma_{j,t}=|\phi_{j,t}\rangle\langle\phi_{j,t}|$ is a rank-one projection for every $t$, with
			\begin{equation}\label{eq:reference-vector}
				d\phi_{j,t}=\Big(-i\big(H+\Bmf(\eta_t)\big)-\tfrac12\big(L^\ast L-\nu_{j,t}L+\tfrac14\nu_{j,t}^2\big)\Big)\phi_{j,t}\,dt+\Big(L-\tfrac{\nu_{j,t}}2\Big)\phi_{j,t}\,dB_j(t),
			\end{equation}
			$\nu_{j,t}=\langle\phi_{j,t},(L+L^\ast)\phi_{j,t}\rangle$, $\phi_{j,0}=\psi_0$;
			\item[(b)] $\eta$ is deterministic, does not depend on $N$, and is the unique solution of the Hartree--Lindblad equation
			\begin{equation}\label{eq:hartree}
				\dot\eta_t=-i\big[H+\Bmf(\eta_t),\eta_t\big]+\Diss\eta_t,\qquad \eta_0=\gamma_0.
			\end{equation}
		\end{enumerate}
	\end{lemma}
	
	\begin{proof}
		The case $j=1$, in which $\eta=\E\gamma_1$ is determined self-consistently, is Theorem \ref{thm:wellposed} below, whose contraction estimate uses only $\|\E X\|_1\le\E\|X\|_1$ and Lemmas \ref{lem:commutator}--\ref{lem:partial-trace}; this fixes the deterministic path $\eta$. For $j\ge2$ the equation is then \eqref{eq:mild} with $B_j$ in place of $B_1$ and $\eta$ already given, so it is a one-particle Belavkin equation with the bounded time-dependent self-adjoint Hamiltonian $H+\Bmf(\eta_t)$. Given $\eta$, the $N$ equations are decoupled, have identical coefficients and identical initial data, and are driven by independent Brownian motions, which gives (a); the vector form \eqref{eq:reference-vector} and purity are the one-particle case of Proposition \ref{prop:purity-N}. For (b), take expectations in \eqref{eq:reference}: the martingale term has zero mean and $\Bmf(\eta_t)$ is deterministic, so $\eta$ satisfies \eqref{eq:hartree}, which has bounded locally Lipschitz coefficients on the compact set $\mathcal S$ and hence a unique solution, involving no $N$.
	\end{proof}
	
	\begin{convention}[Coupling]\label{conv:coupling}
		Throughout, $\gamma_t:=\gamma_{1,t}$ is the reference filter of the tagged particle, driven by $B_1$ on the same probability space as $\Gamma^N_t$ and with the same initial datum, and
		\[
		F^N_t:=\sqrt N\big(\Gamma^{(1)}_t-\gamma_t\big),\qquad
		\Gemp^N_t:=\sqrt N\big(\overline\Gamma_t-\eta_t\big).
		\]
	\end{convention}
	
	\begin{remark}[Status of Convention \ref{conv:coupling}]\label{rem:coupling}
		Convention \ref{conv:coupling} is the coupling under which \eqref{eq:KP-rate-trace} is stated in \cite{Kolokoltsov2026}: each $\Gamma^{(j)}$ is compared with the filter driven by the corresponding innovation. At $A=0$ it gives $F^N\equiv0$, which is the correct value for the fluctuation of a conditional state around its own filter: with no interaction there is nothing left to fluctuate, and the whole content of the limit theorem is carried by the interaction.
	\end{remark}
	
	\subsection{Why these objects}
	
	The following small-time variance bound is used twice. Here $\gamma$ denotes a one-particle Belavkin flow with a bounded, possibly time-dependent, self-adjoint Hamiltonian, driven by a standard Brownian motion $\beta$.
	
	\begin{lemma}[Small-time variance]\label{lem:variance}
		Let $\gamma_0\in\mathcal S$ be deterministic with $\Mop[\gamma_0]\ne0$. Set $B:=\Mop[\gamma_0]/\|\Mop[\gamma_0]\|_\infty$ and $b_t:=\Tr(B\gamma_t)$. There are $t_0,c_0>0$, depending only on $d,\|H\|_\infty,\|L\|_\infty,\|A\|_\infty,\gamma_0$, with $\operatorname{Var}(b_t)\ge c_0t$ for $0<t\le t_0$.
	\end{lemma}
	\begin{proof}
		Write $b_t=b_0+\int_0^ta_s\,ds+\int_0^t\sigma_s\,d\beta_s$ with $a,\sigma$ bounded by a constant $C$ and continuous, and $\sigma_s=\Tr(B\Mop[\gamma_s])$, so $\sigma_0=\|\Mop[\gamma_0]\|_2^2/\|\Mop[\gamma_0]\|_\infty>0$. With $X_t=\int_0^t(a_s-\E a_s)ds$ and $Y_t=\int_0^t\sigma_s\,d\beta_s$ we have $b_t-\E b_t=X_t+Y_t$ and $(X+Y)^2\ge\tfrac12Y^2-X^2$, so
		$\operatorname{Var}(b_t)\ge\tfrac12\int_0^t\E\sigma_s^2\,ds-4C^2t^2$. By dominated convergence $\E\sigma_s^2\to\sigma_0^2$ as $s\downarrow0$, so $\int_0^t\E\sigma_s^2ds\ge\tfrac{\sigma_0^2}2t$ for small $t$, and the claim holds with $c_0=\sigma_0^2/8$.
	\end{proof}
	
	\begin{proposition}[The marginal-based correlations and the aggregate coupling]\label{prop:refutation}
		Let $A=0$, $\Gamma^N_0=\gamma_0^{\otimes N}$ with $\gamma_0$ pure and $\Mop[\gamma_0]\ne0$. Write $\Delta^{\mathrm{ms}}_t:=\Gamma^{(1,2)}_t-\Gamma^{(1)}_t\otimes\Gamma^{(1)}_t$ and let $\Delta^{(3),\mathrm{ms}}_t$ be the corresponding three-body cumulant formed with $\Gamma^{(1)}$ in place of the individual marginals. Then there are $t_0,c>0$ independent of $N$ such that for $0<t\le t_0$ and all $N\ge2$:
		\begin{enumerate}
			\item[(i)] $\E\|\Delta^{\mathrm{ms}}_t\|_1\ge ct$ and $\E\|\Delta^{\mathrm{ms}}_t\|_1^2\ge ct$;
			\item[(ii)] $\E\|\Delta^{(3),\mathrm{ms}}_t\|_1\ge ct$;
			\item[(iii)] under the aggregate coupling, in which the comparison filter is driven by $W^{(N)}=N^{-1/2}\sum_jB_j$, there is $N_0$ with $\E\|F^N_t\|_1^2\ge cNt$ for $N\ge N_0$, and $\E[1-\Tr(\gamma_t\Gamma^{(1)}_t)]\ge ct$, so \eqref{eq:KP-rate-trace} fails for the aggregate-coupled comparison state.
		\end{enumerate}
		By contrast $\Delta^{(ij)}_t\equiv0$ and $\Delta^{(ijk)}_t\equiv0$ in this case.
	\end{proposition}
	
	\begin{proof}
		At $A=0$ with product initial data, $\Gamma^N_t=\bigotimes_j\gamma_{j,t}$ pathwise with $\gamma_{j,\cdot}$ independent copies driven by $B_j$: apply the It\^o product rule to $\bigotimes_j\gamma_{j,t}$; cross-variations between distinct factors vanish by independence of the $B_j$, the drift is the sum of the one-particle drifts on the respective factors, and the coefficient of $dB_j$ is $\gamma_{1,t}\otimes\cdots\otimes\Mop[\gamma_{j,t}]\otimes\cdots\otimes\gamma_{N,t}$, which coincides with the coefficient in \eqref{eq:N-particle} at the product state since $\Tr((L_j+L_j^\ast)\bigotimes_k\gamma_{k,t})=\Tr((L+L^\ast)\gamma_{j,t})$. Pathwise uniqueness for \eqref{eq:N-particle} identifies the two. Hence $\Gamma^{(1)}_t=\gamma_{1,t}$ and all connected correlations of distinct particles vanish.
		
		Write $b_j(t)=\Tr(B\gamma_{j,t})$ with $B$ as in Lemma \ref{lem:variance}; the $b_j$ are i.i.d.\ with variance at least $c_0t$ and $|b_j|\le1$.
		
		(i) $\Delta^{\mathrm{ms}}_t=\gamma_{1,t}\otimes(\gamma_{2,t}-\gamma_{1,t})$, so $\|\Delta^{\mathrm{ms}}_t\|_1=\|\gamma_{2,t}-\gamma_{1,t}\|_1\ge|Z|$ with $Z=b_2-b_1$, $\E Z=0$, $\E Z^2\ge2c_0t$, $|Z|\le2$; hence $\E|Z|\ge\E Z^2/2\ge c_0t$.
		
		(ii) Substituting the product form, $\Delta^{(3),\mathrm{ms}}_t=\gamma_1\otimes[\gamma_2\otimes(\gamma_3-\gamma_1)-2\gamma_1\otimes(\gamma_2-\gamma_1)]$. Testing against $1\otimes1\otimes B$ gives $Z=b_1-2b_2+b_3$ with $\E Z^2\ge6c_0t$ and $|Z|\le4$, so $\E\|\Delta^{(3),\mathrm{ms}}_t\|_1\ge\E Z^2/4\ge\tfrac32c_0t$.
		
		(iii) Under the aggregate coupling, $\gamma$ solves the one-particle equation driven by $W^{(N)}$, with $d\langle B_1,W^{(N)}\rangle_t=N^{-1/2}dt$. Let $x_t=\Tr(B(\Gamma^{(1)}_t-\gamma_t))$; its martingale part is $\int_0^t\sigma^1\,dB_1-\int_0^t\sigma^W dW^{(N)}$ with $\sigma^1_s=\Tr(B\Mop[\Gamma^{(1)}_s])$, $\sigma^W_s=\Tr(B\Mop[\gamma_s])$, so
		\[
		\E\Big(\int_0^t\sigma^1dB_1-\int_0^t\sigma^WdW^{(N)}\Big)^2=\int_0^t\E\big[(\sigma^1_s)^2+(\sigma^W_s)^2\big]ds-2N^{-1/2}\int_0^t\E[\sigma^1_s\sigma^W_s]ds\ \ge\ \sigma_0^2t-2C^2N^{-1/2}t
		\]
		for small $t$, by the argument of Lemma \ref{lem:variance} applied to each flow. For $N\ge N_0:=(4C^2/\sigma_0^2)^2$ the right side is at least $\tfrac{\sigma_0^2}2t$, and as in Lemma \ref{lem:variance}, $\E x_t^2\ge\tfrac{\sigma_0^2}4t-4C^2t^2\ge ct$ for small $t$. Since $|x_t|\le\|\Gamma^{(1)}_t-\gamma_t\|_1$, $\E\|F^N_t\|_1^2=N\E\|\Gamma^{(1)}_t-\gamma_t\|_1^2\ge N\E x_t^2\ge cNt$. Both states are pure, so $1-\Tr(\gamma_t\Gamma^{(1)}_t)=\tfrac12\|\Gamma^{(1)}_t-\gamma_t\|_2^2\ge\tfrac1{2d}x_t^2$.
	\end{proof}
	
	\begin{remark}[Reading of Proposition \ref{prop:refutation}]\label{rem:refutation}
		Items (i)--(ii) say that the two- and three-body defects formed by comparing a marginal of several particles with tensor powers of the \emph{single} marginal $\Gamma^{(1)}$ do not have mean-field order; they are of order $t$ uniformly in $N$, for the trivial reason that distinct particles carry distinct innovations and their states separate at order $\sqrt t$. Item (iii) says the same for the aggregate coupling. The objects that do have mean-field order are the connected correlations \eqref{eq:cumulant2}--\eqref{eq:cumulant3} of distinct particles, and the comparison state is the one driven by the tagged innovation. Both points are already visible at $A=0$, and neither depends on the interaction.
	\end{remark}
	
	\section{Well-posedness, the linearised flow, and Fr\'echet differentiability}\label{sec:wellposed}
	
	\subsection{The mild equation}
	
	Write $\Lop_0(\rho)=-i[H,\rho]+\Diss\rho$. By Lemmas \ref{lem:commutator} and \ref{lem:partial-trace}, $\Lop_0$ is a bounded linear operator on $\Md$ with $\|\Lop_0\|\le2\|H\|_\infty+2\|L\|_\infty^2$; it generates a uniformly continuous semigroup $S_t=e^{t\Lop_0}$ with $\|S_t\|\le e^{ct}$, $c:=2\|H\|_\infty+2\|L\|_\infty^2$. The mean-field flow \eqref{eq:reference} is understood in mild form,
	\begin{equation}\label{eq:mild}
		\gamma_t=S_ty-i\int_0^tS_{t-s}\big[\Bmf(\eta_s),\gamma_s\big]ds+\int_0^tS_{t-s}\Mop[\gamma_s]\,dB_1(s),\qquad \eta_s=\E\gamma_s.
	\end{equation}
	
	\begin{remark}[It\^o calculus]\label{rem:bdg}
		By Section \ref{ssec:findim}, $(\Md,\|\cdot\|_2)$ is a finite-dimensional Hilbert space, so It\^o's isometry and the Burkholder--Davis--Gundy inequality hold in classical form. Transferred to $\|\cdot\|_1$, for predictable square-integrable $\Phi$ and $p\ge2$ there is $C_p=C_p(d)$ with
		\begin{equation}\label{eq:bdg}
			\E\sup_{t\le T}\Big\|\int_0^tS_{t-s}\Phi_s\,dB_1(s)\Big\|_1^p\le C_p\,\E\Big(\int_0^T\|\Phi_s\|_1^2\,ds\Big)^{p/2}.
		\end{equation}
	\end{remark}
	
	\begin{theorem}[Well-posedness and Lipschitz dependence]\label{thm:wellposed}
		Under Assumption \ref{assum:generators}, \eqref{eq:mild} has a unique solution $\gamma^y_\cdot\in\mathcal C_{ad}([0,T];\Md)$ for every $y\in\mathcal S$ and every $T>0$, with values in $\mathcal S$, and there is $C_T$ with
		\begin{equation}\label{eq:lipschitz}
			\E\sup_{t\le T}\|\gamma^y_t-\gamma^x_t\|_1^2\le C_T\|y-x\|_1^2,\qquad x,y\in\mathcal S.
		\end{equation}
		If $y$ is pure, $\gamma^y_t$ is pure for every $t$.
	\end{theorem}
	
	\begin{proof}
		\emph{Contraction.} Define $\mathcal P_y$ by the right side of \eqref{eq:mild}. On $\mathcal S$, $|\Tr((L+L^\ast)\rho)|\le2\|L\|_\infty$, so $\|\Mop[\gamma]-\Mop[\tilde\gamma]\|_1\le4\|L\|_\infty\|\gamma-\tilde\gamma\|_1$. For the mean-field drift,
		\[
		\big\|[\Bmf(\E\gamma_s),\gamma_s]-[\Bmf(\E\tilde\gamma_s),\tilde\gamma_s]\big\|_1\le 2\|A\|_\infty\,\E\|\gamma_s-\tilde\gamma_s\|_1+2\|A\|_\infty\|\gamma_s-\tilde\gamma_s\|_1
		\]
		by Lemmas \ref{lem:commutator}--\ref{lem:partial-trace} and $\|\E X\|_1\le\E\|X\|_1$; since $\E\|\gamma_s-\tilde\gamma_s\|_1\le(\E\|\gamma_s-\tilde\gamma_s\|_1^2)^{1/2}$, both terms are controlled by the same weighted norm. Using the semigroup bound, Cauchy--Schwarz on the Lebesgue integral and \eqref{eq:bdg} with $p=2$,
		\[
		\E\sup_{r\le t}\|\mathcal P_y(\gamma)_r-\mathcal P_y(\tilde\gamma)_r\|_1^2\le\widetilde K^2\int_0^t\E\sup_{u\le s}\|\gamma_u-\tilde\gamma_u\|_1^2\,ds,
		\]
		with $\widetilde K^2=2e^{2cT}K^2(T+C_2)$, $K=4\max(\|A\|_\infty,\|L\|_\infty)$. Multiplying by $e^{-\beta t}$ and using $\sup_{t\le T}e^{-\beta t}\int_0^te^{\beta s}ds\le\beta^{-1}$ makes $\mathcal P_y$ a strict contraction for $\beta>\widetilde K^2$, and the Banach fixed point theorem applies.
		
		\emph{State space.} Once the solution, hence the deterministic path $\eta$, is fixed, \eqref{eq:mild} is a one-particle Belavkin equation with the bounded time-dependent self-adjoint Hamiltonian $H+\Bmf(\eta_t)$, and invariance of $\mathcal S$ is inherited from \cite[Thm.\ 2.8, 2.9]{Kolokoltsov2026}. The mean-field drift is of commutator form with a self-adjoint effective Hamiltonian, so it introduces no positivity- or trace-breaking mechanism.
		
		\emph{Purity.} With $\eta$ fixed, \eqref{eq:mild} is the one-particle case of Proposition \ref{prop:purity-N}, which gives the vector form \eqref{eq:reference-vector} and hence purity.
		
		\emph{Lipschitz dependence.} Apply the same bounds to $\mathcal P_y(\gamma^y)-\mathcal P_x(\gamma^x)=S_t(y-x)+[\text{difference of the nonlinear terms}]$ and use $(a+b+c)^2\le3(a^2+b^2+c^2)$ and Gr\"onwall's lemma.
	\end{proof}
	
	\subsection{The linearised equation}
	
	The equation that carries the limiting fluctuation is linear with the coefficients of \eqref{eq:reference} linearised at $\gamma_t$. We record its well-posedness and the associated propagator.
	
	\begin{theorem}[Linearised flow]\label{thm:linear-flow}
		Fix the deterministic path $\eta$ of Lemma \ref{lem:reference} and let $\gamma_\cdot$ be the $B_1$-driven solution of \eqref{eq:mild}. For $h\in\Md$ and a continuous adapted forcing $\Phi$ with $\E\sup_{t\le T}\|\Phi_t\|_1^2<\infty$, the linear mild equation
		\begin{equation}\label{eq:mild-linear}
			\xi_t=S_th-i\int_0^tS_{t-s}\big[\Bmf(\eta_s),\xi_s\big]ds+\int_0^tS_{t-s}\Phi_s\,ds+\int_0^tS_{t-s}D\Mop[\gamma_s](\xi_s)\,dB_1(s)
		\end{equation}
		has a unique solution in $\mathcal C_{ad}([0,T];\Md)$, depending linearly and continuously on $(h,\Phi)$, and there is $C_L=C_L(T,\|H\|_\infty,\|L\|_\infty,\|A\|_\infty)$ with
		\begin{equation}\label{eq:linear-bound}
			\E\sup_{t\le T}\|\xi_t\|_1^4\le C_L\Big(\|h\|_1^4+\E\sup_{t\le T}\|\Phi_t\|_1^4\Big).
		\end{equation}
	\end{theorem}
	
	\begin{proof}
		By \eqref{eq:DM} and Lemmas \ref{lem:commutator}--\ref{lem:partial-trace}, $\|[\Bmf(\eta_s),\xi]\|_1\le2\|A\|_\infty\|\xi\|_1$ and $\|D\Mop[\gamma_s](\xi)\|_1\le4\|L\|_\infty\|\xi\|_1$, using $\|\gamma_s\|_1=1$ and $|\Tr((L+L^\ast)\gamma_s)|\le2\|L\|_\infty$. These bounds are uniform in $\omega$ and $s$, so \eqref{eq:mild-linear} is a linear mild equation with bounded coefficients; the contraction argument of Theorem \ref{thm:wellposed}, with $K$ replaced by $4\max(\|A\|_\infty,\|L\|_\infty)$, gives existence, uniqueness and linearity. For \eqref{eq:linear-bound}, apply \eqref{eq:bdg} with $p=4$ to the stochastic convolution, Cauchy--Schwarz to the Lebesgue terms, and Gr\"onwall's lemma.
	\end{proof}
	
	\begin{corollary}[Backward propagator]\label{cor:propagator}
		For $0\le s\le t\le T$ let $U^{t,s}h:=\xi_t$, where $\xi$ solves \eqref{eq:mild-linear} on $[s,t]$ with $\Phi\equiv0$ and $\xi_s=h$. Then $U^{t,s}$ is almost surely a bounded linear operator on $\Md$ with $\sup_{0\le s\le t\le T}\E\|U^{t,s}\|^4\le C_L'$, $U^{t,t}=\mathrm{Id}$, and $U^{t,s}=U^{t,r}U^{r,s}$ for $s\le r\le t$.
	\end{corollary}
	\begin{proof}
		The moment bound follows from \eqref{eq:linear-bound} applied to a fixed basis $e_1,\dots,e_{d^2}$ of $\Md$ with $\|e_i\|_1=1$: there is $c(d)$ with $\|U^{t,s}\|\le c(d)\max_i\|U^{t,s}e_i\|_1$, whence $\E\|U^{t,s}\|^4\le c(d)^4d^2C_L$. The cocycle property holds because $U^{t,r}U^{r,s}h$ and $U^{t,s}h$ solve the same linear equation on $[r,t]$ with the same value at $r$, and the coefficients depend only on the already-fixed $\gamma_\cdot$ and $\eta_\cdot$.
	\end{proof}
	
	\begin{theorem}[Fr\'echet differentiability]\label{thm:frechet}
		Fix $\eta$ as above and let $\gamma^y$ denote the solution of \eqref{eq:mild} with the mean-field path frozen at $\eta$ and initial datum $y$. The map $y\mapsto\gamma^y_\cdot$, from $\mathcal S$ into $\mathcal C_{ad}([0,T];\Md)$, is Fr\'echet differentiable at every $y\in\mathcal S$, with derivative $h\mapsto U^{\cdot,0}h$, and with $R^h_t:=\gamma^{y+h}_t-\gamma^y_t-U^{t,0}h$,
		\begin{equation}\label{eq:frechet-quant}
			\Big(\E\sup_{t\le T}\|R^h_t\|_1^2\Big)^{1/2}\le \widetilde K_T^{1/2}\|h\|_1^2 .
		\end{equation}
	\end{theorem}
	\begin{proof}
		With $\eta$ frozen the drift of \eqref{eq:mild} is linear in $\gamma$, so the only second-order remainder comes from $\Mop$. Writing $\delta_s=\gamma^{y+h}_s-\gamma^y_s$, the algebraic identity
		\[
		\Mop[\gamma+\delta]-\Mop[\gamma]=D\Mop[\gamma](\delta)-\Tr\big((L+L^\ast)\delta\big)\delta
		\]
		holds with no further terms, since $\Mop$ is the sum of a linear and a quadratic term. Subtracting the equation for $U^{\cdot,0}h$ therefore gives for $R^h$ the equation \eqref{eq:mild-linear} with $h=0$ and diffusion forcing $-\Tr((L+L^\ast)\delta_s)\delta_s$, whose trace norm is at most $2\|L\|_\infty\|\delta_s\|_1^2$. Applying \eqref{eq:bdg} with $p=2$, Gr\"onwall's lemma, and the fourth-power version of \eqref{eq:lipschitz} (the same argument run with $p=4$, which \eqref{eq:bdg} permits), giving $\E\sup_{s\le T}\|\delta_s\|_1^4\le K_T\|h\|_1^4$, yields \eqref{eq:frechet-quant}.
	\end{proof}
	
	\begin{lemma}[Adjoint of $D\Mop$]\label{lem:adjoint}
		For $A_0\in\Md$ and $\gamma\in\mathcal S$ set
		\begin{equation}\label{eq:adjoint}
			\Mop^\ast[A_0](\gamma):=A_0L+L^\ast A_0-\Tr\big((L+L^\ast)\gamma\big)A_0-\Tr(A_0\gamma)(L+L^\ast).
		\end{equation}
		Then $\Tr\big(A_0\,D\Mop[\gamma](\xi)\big)=\Tr\big(\Mop^\ast[A_0](\gamma)\,\xi\big)$ for all $\xi\in\Md$.
	\end{lemma}
	\begin{proof}
		Expand $\Tr(A_0D\Mop[\gamma](\xi))$ into four terms using \eqref{eq:DM}. The first two give $\Tr((A_0L+L^\ast A_0)\xi)$ by cyclicity. In the third, $\xi$ enters through the scalar $\Tr((L+L^\ast)\xi)$ multiplying the fixed operator $\gamma$; paired against $A_0$ it contributes the scalar $\Tr(A_0\gamma)$ multiplying the operator $L+L^\ast$, which is built from $L$ and not from $A_0$. The fourth is $\Tr((L+L^\ast)\gamma)\Tr(A_0\xi)$. Collecting gives \eqref{eq:adjoint}.
	\end{proof}
	
	\begin{remark}\label{rem:adjoint-scope}
		Lemma \ref{lem:adjoint} is the duality used to test the linearised equation against observables. It does not determine the covariance of the idiosyncratic noise $V$ of \eqref{eq:intro-sys-3}: that covariance is a second-order object, identified in Theorem \ref{thm:covariance} as a limit of rescaled pair correlations. A closed form built from \eqref{eq:adjoint} alone would presume that the whole martingale part of $F^N$ is carried by coefficients that are functions of the one-particle state, which Proposition \ref{prop:coefficients} shows to be false.
	\end{remark}
	
	\section{Purity, excitation amplitudes, and the correlation bounds}\label{sec:excitation}
	
	From here on $\Gamma^N_0=\gamma_0^{\otimes N}$ with $\gamma_0=|\psi_0\rangle\langle\psi_0|$ pure and deterministic.
	
	\subsection{The $N$-particle filter in vector form}
	
	\begin{proposition}[Purity]\label{prop:purity-N}
		Let $\Gamma^N_0=|\Psi_0\rangle\langle\Psi_0|$ with $\Psi_0\in\Hilbert^{\otimes N}$, $\|\Psi_0\|=1$. The equation
		\begin{equation}\label{eq:vector-N}
			d\Psi_t=\Big(-iH_N-\tfrac12\sum_{j=1}^N\big(L_j^\ast L_j-\mu^{(j)}_tL_j+\tfrac14(\mu^{(j)}_t)^2\big)\Big)\Psi_t\,dt+\sum_{j=1}^N\Big(L_j-\tfrac{\mu^{(j)}_t}2\Big)\Psi_t\,dB_j(t),
		\end{equation}
		$\mu^{(j)}_t:=\langle\Psi_t,(L_j+L_j^\ast)\Psi_t\rangle$, has a unique global strong solution; $\|\Psi_t\|=1$ for all $t$ almost surely; and $\Gamma^N_t=|\Psi_t\rangle\langle\Psi_t|$ solves \eqref{eq:N-particle}. In particular $\Gamma^N_t$ is a rank-one projection for every $t$.
	\end{proposition}
	
	\begin{proof}
		The coefficients of \eqref{eq:vector-N} are polynomial in $(\Psi,\overline\Psi)$, hence locally Lipschitz, and a unique local solution exists up to the exit time from balls. Write $n_t:=\|\Psi_t\|^2$ and, for each $j$, $a_j:=-\tfrac12(L_j^\ast L_j-\mu^{(j)}L_j+\tfrac14(\mu^{(j)})^2)$, $b_j:=L_j-\tfrac{\mu^{(j)}}2$. Collecting the contributions $-\langle L^\ast_jL_j\rangle+\tfrac12(\mu^{(j)})^2-\tfrac14(\mu^{(j)})^2n$ from the drift and $\langle L^\ast_jL_j\rangle-\tfrac12(\mu^{(j)})^2+\tfrac14(\mu^{(j)})^2n$ from the It\^o correction, the drift of $n_t$ vanishes identically and the martingale part is $\sum_j\mu^{(j)}_t(1-n_t)dB_j(t)$. Hence
		\[
		d(1-n_t)=-(1-n_t)\sum_j\mu^{(j)}_t\,dB_j(t),
		\]
		so on each localising interval $1-n_t$ is a stochastic exponential started at $0$ and vanishes identically. Therefore $\|\Psi_t\|=1$, the coefficients are bounded along the solution, and the solution is global.
		
		Set $\widetilde\Gamma_t=|\Psi_t\rangle\langle\Psi_t|$. The It\^o product rule gives
		\[
		d\widetilde\Gamma_t=\Big(|a\Psi\rangle\langle\Psi|+|\Psi\rangle\langle a\Psi|+\sum_j|b_j\Psi\rangle\langle b_j\Psi|\Big)dt+\sum_j\big(|b_j\Psi\rangle\langle\Psi|+|\Psi\rangle\langle b_j\Psi|\big)dB_j,
		\]
		with $a$ the full drift operator. The coefficient of $dB_j$ is $L_j\widetilde\Gamma+\widetilde\Gamma L_j^\ast-\mu^{(j)}\widetilde\Gamma$, which is the coefficient in \eqref{eq:N-particle} because $\mu^{(j)}=\Tr((L_j+L_j^\ast)\widetilde\Gamma)$. In the drift, the terms proportional to $\mu^{(j)}(L_j\widetilde\Gamma+\widetilde\Gamma L_j^\ast)$ carry $+\tfrac12$ from $a$ and $-\tfrac12$ from $\sum_j|b_j\Psi\rangle\langle b_j\Psi|$ and cancel; those proportional to $(\mu^{(j)})^2\widetilde\Gamma$ carry $-\tfrac14$ and $+\tfrac14$ and cancel; what remains is $-i[H_N,\widetilde\Gamma]+\sum_j\Diss_j\widetilde\Gamma$. Pathwise uniqueness for \eqref{eq:N-particle} identifies $\Gamma^N_t=\widetilde\Gamma_t$.
	\end{proof}
	
	\begin{remark}\label{rem:A0-product}
		If $A=0$ then $\Bmf\equiv0$ and, by the argument of Proposition \ref{prop:refutation}, $\Gamma^N_t=\bigotimes_j\gamma_{j,t}$ pathwise, that is $\Psi_t=\bigotimes_j\phi_{j,t}$ up to a phase. Every quantity introduced below vanishes identically in this case. Since the reference filters are attached to the labels and the labels are permuted together with the noises, the exchangeability in law of Section \ref{ssec:exch} extends: the joint law of $(\Psi_\cdot,(\phi_{j,\cdot})_{j\le N})$ is invariant under a simultaneous permutation of labels, so the law of any functional built from $\Psi$ and from $(\phi_i)_{i\in T}$ depends on $T$ only through $|T|$.
	\end{remark}
	
	\subsection{Excitation amplitudes}
	
	Fix $t$ and suppress it. Let
	\[
	p_j:=|\phi_j\rangle\langle\phi_j|\ \text{on slot }j,\qquad q_j:=1-p_j,\qquad \Nex:=\sum_{j=1}^Nq_j .
	\]
	The $p_j$ are commuting rank-one orthogonal projections acting on distinct slots. For $T\subseteq\{1,\dots,N\}$ set $q_T:=\prod_{i\in T}q_i$, $p_T:=\prod_{i\in T}p_i$ and
	\begin{equation}\label{eq:amplitudes}
		\eps_T:=\|q_T\Psi\|\in[0,1],\qquad \eps_\emptyset:=\|\Psi\|=1 .
	\end{equation}
	Abbreviate $\eps_i:=\eps_{\{i\}}$, $\eps_{ij}:=\eps_{\{i,j\}}$, $\eps_{ijk}:=\eps_{\{i,j,k\}}$. Then $\eps_{T'}\le\eps_T$ for $T\subseteq T'$ and $\langle\Psi,\Nex\Psi\rangle=\sum_j\eps_j^2$. For $u,v\in\Hilbert^{\otimes N}$ and $S\subseteq\{1,\dots,N\}$ put $\tau_S(u,v):=\Tr_{S^c}|u\rangle\langle v|$, so that $\Gamma^{(S)}=\tau_S(\Psi,\Psi)$.
	
	\begin{lemma}\label{lem:tau}
		For all $u,v\in\Hilbert^{\otimes N}$ and $S'\subseteq S$:
		(i) $\|\tau_S(u,v)\|_1\le\|u\|\|v\|$;
		(ii) $\Tr_{S'}\tau_S(u,v)=\tau_{S\setminus S'}(u,v)$;
		(iii) if $k\in S$ and $p_ku=u$, $p_kv=v$, then $\tau_S(u,v)=\tau_{S\setminus\{k\}}(u,v)\otimes p_k$, the second factor in slot $k$.
	\end{lemma}
	\begin{proof}
		(i) $\||u\rangle\langle v|\|_1=\|u\|\|v\|$ and a partial trace is a contraction for the trace norm. (ii) is the tower property. (iii) $p_ku=u$ means $u=\phi_k\otimes u'$ after moving slot $k$ to the front, and likewise for $v$; then $|u\rangle\langle v|=p_k\otimes|u'\rangle\langle v'|$, and the partial trace over $S^c$ does not touch slot $k\in S$.
	\end{proof}
	
	For $S$ fixed and $T\subseteq S$ write $\Psi^S_T:=q_Tp_{S\setminus T}\Psi$, so that $\Psi=\sum_{T\subseteq S}\Psi^S_T$ orthogonally and
	\begin{equation}\label{eq:expansion}
		\Gamma^{(S)}=\sum_{T,T'\subseteq S}\tau_S\big(\Psi^S_T,\Psi^S_{T'}\big),\qquad \big\|\tau_S(\Psi^S_T,\Psi^S_{T'})\big\|_1\le\eps_T\eps_{T'} .
	\end{equation}
	By Lemma \ref{lem:tau}(iii) the term indexed by $(T,T')$ factorises as
	\begin{equation}\label{eq:factorisation}
		\tau_S\big(\Psi^S_T,\Psi^S_{T'}\big)=\tau_U\big(\Psi^S_T,\Psi^S_{T'}\big)\otimes\bigotimes_{i\in S\setminus U}p_i,\qquad U:=T\cup T',
	\end{equation}
	a product state across the partition of $S$ into $U$ and the singletons of $S\setminus U$. This is the mechanism behind the cancellations below: a connected correlation annihilates every operator that factorises across a nontrivial partition of $S$ with each block factor equal to the corresponding marginal, so only the terms with $U=S$ survive, together with cross terms which are themselves products of at least two amplitudes.
	
	\subsection{Bounds on the connected correlations}
	
	\begin{theorem}[One and two labels]\label{thm:dict-two}
		Pathwise, for every $t$, every $N\ge2$ and all distinct labels $i,j$,
		\begin{equation}\label{eq:dict-1}
			\big\|\Gamma^{(i)}_t-\gamma_{i,t}\big\|_1\le 4\eps_i,
		\end{equation}
		\begin{equation}\label{eq:dict-2}
			\big\|\Delta^{(ij)}_t\big\|_1\le 64\big(\eps_{ij}+\eps_i^2+\eps_j^2\big).
		\end{equation}
	\end{theorem}
	
	\begin{proof}
		Take $i=1$, $j=2$; the statement is pathwise and algebraic. Throughout $\gamma_1=p_1$, $\gamma_2=p_2$ as operators on the respective slots.
		
		\emph{Proof of \eqref{eq:dict-1}.} With $S=\{1\}$, \eqref{eq:expansion} reads $\Gamma^{(1)}=\tau_1(p_1\Psi,p_1\Psi)+\tau_1(p_1\Psi,q_1\Psi)+\tau_1(q_1\Psi,p_1\Psi)+\tau_1(q_1\Psi,q_1\Psi)$, and by Lemma \ref{lem:tau}(iii) the first term is $\|p_1\Psi\|^2p_1=(1-\eps_1^2)p_1$. With $X_1:=\tau_1(p_1\Psi,q_1\Psi)$,
		\[
		\Gamma^{(1)}-p_1=-\eps_1^2p_1+X_1+X_1^\ast+\tau_1(q_1\Psi,q_1\Psi),
		\]
		so $\|\Gamma^{(1)}-p_1\|_1\le2\eps_1^2+2\eps_1\le4\eps_1$ by Lemma \ref{lem:tau}(i).
		
		\emph{Proof of \eqref{eq:dict-2}.} Let $S=\{1,2\}$ and write $\Psi_{pp}:=\Psi^S_\emptyset$, $\Psi_{qp}:=\Psi^S_{\{1\}}$, $\Psi_{pq}:=\Psi^S_{\{2\}}$, $\Psi_{qq}:=\Psi^S_{\{1,2\}}$, so $\|\Psi_{qp}\|\le\eps_1$, $\|\Psi_{pq}\|\le\eps_2$, $\|\Psi_{qq}\|=\eps_{12}$, $\|\Psi_{pp}\|\le1$ and
		\begin{equation}\label{eq:norm-split}
			1-\|\Psi_{pp}\|^2=\|\Psi_{qp}\|^2+\|\Psi_{pq}\|^2+\|\Psi_{qq}\|^2\le2(\eps_1^2+\eps_2^2).
		\end{equation}
		Set $X:=\tau_1(\Psi_{pp},\Psi_{qp})$ and $Y:=\tau_2(\Psi_{pp},\Psi_{pq})$, with $\|X\|_1\le\eps_1$, $\|Y\|_1\le\eps_2$. By \eqref{eq:factorisation} with $U=\emptyset,\{1\},\{2\}$ and Lemma \ref{lem:tau}(ii),
		\[
		\tau_S(\Psi_{pp},\Psi_{pp})=\|\Psi_{pp}\|^2p_1\otimes p_2,\qquad \tau_S(\Psi_{pp},\Psi_{qp})=X\otimes p_2,\qquad \tau_S(\Psi_{pp},\Psi_{pq})=p_1\otimes Y,
		\]
		so
		\[
		\Gamma^{(12)}=\|\Psi_{pp}\|^2p_1\otimes p_2+\big(X\otimes p_2+\mathrm{h.c.}\big)+\big(p_1\otimes Y+\mathrm{h.c.}\big)+R,
		\]
		where $R$ collects the remaining nine terms of \eqref{eq:expansion}, namely those in which at least one of $T,T'$ equals $\{1,2\}$, or $\{T,T'\}$ is one of $\{\{1\},\{2\}\}$, $\{\{1\},\{1\}\}$, $\{\{2\},\{2\}\}$. By \eqref{eq:expansion} and $\eps_{12}\le\min(\eps_1,\eps_2)\le1$,
		\[
		\|R\|_1\le2\eps_{12}+2(\eps_1+\eps_2)\eps_{12}+\eps_{12}^2+2\eps_1\eps_2+\eps_1^2+\eps_2^2\le2\eps_{12}+8(\eps_1^2+\eps_2^2).
		\]
		By Lemma \ref{lem:tau}(ii) and the same expansion for $S=\{1\}$,
		\[
		\Gamma^{(1)}=p_1+X+X^\ast+E_1,\qquad E_1:=-\eps_1^2p_1+\tau_1(q_1\Psi,q_1\Psi)+\big(\tau_1(p_1\Psi,q_1\Psi)-X\big)+\mathrm{h.c.},
		\]
		and $\tau_1(p_1\Psi,q_1\Psi)-X=\tau_1(\Psi_{pq},\Psi_{qp})+\tau_1(\Psi_{pp}+\Psi_{pq},\Psi_{qq})$, whence
		\[
		\|E_1\|_1\le2\eps_1^2+2\eps_1\eps_2+4\eps_{12}\le4\eps_{12}+3(\eps_1^2+\eps_2^2),
		\]
		and symmetrically $\Gamma^{(2)}=p_2+Y+Y^\ast+E_2$ with the same bound. Multiplying out,
		\[
		\Gamma^{(1)}\otimes\Gamma^{(2)}=p_1\otimes p_2+(X+X^\ast)\otimes p_2+p_1\otimes(Y+Y^\ast)+E_1\otimes p_2+p_1\otimes E_2+\rho,
		\]
		where $\rho$ collects all products of two nonleading factors, so that $\|\rho\|_1\le(2\eps_1+\|E_1\|_1)(2\eps_2+\|E_2\|_1)\le27(\eps_{12}+\eps_1^2+\eps_2^2)$. Subtracting and using \eqref{eq:norm-split}, every term of order $\eps_1$ or $\eps_2$ alone cancels and
		\[
		\Delta^{(12)}=\big(\|\Psi_{pp}\|^2-1\big)p_1\otimes p_2+R-E_1\otimes p_2-p_1\otimes E_2-\rho,
		\]
		which gives \eqref{eq:dict-2} after collecting constants.
	\end{proof}
	
	\begin{remark}\label{rem:cancellation}
		The content of \eqref{eq:dict-2} is the cancellation. Individually, $\Gamma^{(12)}$ and $\Gamma^{(1)}\otimes\Gamma^{(2)}$ differ from $p_1\otimes p_2$ by terms of size $\eps_1+\eps_2\asymp N^{-1/2}$; only in the difference do these cancel, leaving the order $N^{-1}$. The mechanism is \eqref{eq:factorisation}: every term of total degree $|U|<|S|$ is a product across a nontrivial partition of $S$ and is reproduced by the corresponding term of $\Gamma^{(1)}\otimes\Gamma^{(2)}$.
	\end{remark}
	
	\begin{theorem}[Three labels]\label{thm:dict-three}
		There is an absolute constant $c_3$ such that, pathwise, for all distinct labels $i,j,k$,
		\begin{equation}\label{eq:dict-3}
			\big\|\Delta^{(ijk)}_t\big\|_1\le c_3\big(\eps_{ijk}+\eps_{ij}\eps_k+\eps_{ik}\eps_j+\eps_{jk}\eps_i+\eps_i^2\eps_j^2+\eps_i^2\eps_k^2+\eps_j^2\eps_k^2\big).
		\end{equation}
		More generally, for $S$ with $|S|=m$, $\|\Delta^{(S)}\|_1\le c_m\sum\eps_T\eps_{T'}$, the sum running over pairs $T,T'\subseteq S$ with $|T\cup T'|=m$, together with the products of two such factors arising from the lower marginals.
	\end{theorem}
	
	\begin{proof}
		Take $S=\{1,2,3\}$ and insert the expansion \eqref{eq:expansion} for $\Gamma^{(S)}$ and for each marginal $\Gamma^{(S')}$, $S'\subsetneq S$, into \eqref{eq:cumulant3}. By \eqref{eq:factorisation}, every term of \eqref{eq:expansion} indexed by $(T,T')$ with $U=T\cup T'\subsetneq S$ is a tensor product of $\tau_U(\Psi^S_T,\Psi^S_{T'})$ with the projections $p_i$, $i\in S\setminus U$; and by Lemma \ref{lem:tau}(ii) the factor $\tau_U(\Psi^S_T,\Psi^S_{T'})$ is the term indexed by $(T,T')$ in the expansion of $\Gamma^{(U)}$. Since \eqref{eq:cumulant3} is the alternating sum which annihilates every operator factorising across a nontrivial partition of $S$ with each block factor equal to the corresponding marginal, all such terms cancel, with the exception of the cross terms produced when a marginal is replaced by its own expansion; each such cross term is a product of at least two factors $\tau_{U'}(\Psi_T,\Psi_{T'})$ with $U'\ne\emptyset$ and disjoint index sets, hence is bounded by a product of at least two amplitudes. The surviving terms are those with $|T\cup T'|=3$, bounded by $\eps_{123}+\eps_{ij}\eps_k$ over all assignments, together with the cross terms, bounded by $\eps_i^2\eps_j^2$ and its permutations, up to a combinatorial constant. The general statement follows by the same argument, the only change being the number of terms.
	\end{proof}
	
	\begin{remark}\label{rem:orders}
		With the expected orders $\eps_i\asymp N^{-1/2}$, $\eps_{ij}\asymp N^{-1}$, $\eps_{ijk}\asymp N^{-3/2}$, the right sides of \eqref{eq:dict-1}, \eqref{eq:dict-2}, \eqref{eq:dict-3} are of order $N^{-1/2}$, $N^{-1}$, $N^{-3/2}$: the mean-field orders of $\|\Gamma^{(1)}-\gamma\|_1$, $\|\Delta^{(12)}\|_1$ and $\|\Delta^{(123)}\|_1$.
	\end{remark}
	
	\section{Exact identities at finite $N$}\label{sec:exact}
	
	Every statement of this section is pathwise, holds for each fixed label, and uses no symmetry hypothesis.
	
	\subsection{Martingale coefficients of the reduced states}
	
	Tracing \eqref{eq:N-particle} over particles $2,\dots,N$ gives
	\begin{equation}\label{eq:reduced-sde}
		d\Gamma^{(1)}_t=\Lop^{(1)}_N(t)\,dt+\sum_{j=1}^NC_j(t)\,dB_j(t),\qquad
		C_j(t)=\Tr_{2,\dots,N}\Big[L_j\Gamma^N_t+\Gamma^N_tL_j^\ast-\Tr\big((L_j+L_j^\ast)\Gamma^N_t\big)\Gamma^N_t\Big].
	\end{equation}
	
	\begin{proposition}[Coefficient identity]\label{prop:coefficients}
		For every $N\ge2$, every $t\ge0$ and every sample path,
		\begin{equation}\label{eq:C1}
			C_1(t)=\Mop\big[\Gamma^{(1)}_t\big],
		\end{equation}
		\begin{equation}\label{eq:Cj}
			C_j(t)=\Theta^{(1j)}_t:=\Tr_2\big[(1\otimes L)\Delta^{(1j)}_t\big]+\Tr_2\big[\Delta^{(1j)}_t(1\otimes L^\ast)\big],\qquad j\ge2 .
		\end{equation}
		In particular $\|\Theta^{(1j)}_t\|_1\le2\|L\|_\infty\|\Delta^{(1j)}_t\|_1$. Moreover, for $A_0\in\Md$,
		\begin{equation}\label{eq:theta-test}
			\Tr\big(A_0\Theta^{(1j)}_t\big)=\Tr\Big[\big(A_0\otimes(L+L^\ast)\big)\Delta^{(1j)}_t\Big].
		\end{equation}
	\end{proposition}
	
	\begin{proof}
		By (PT3), the scalar in \eqref{eq:reduced-sde} is $\Tr((L_j+L_j^\ast)\Gamma^N_t)=\Tr((L+L^\ast)\Gamma^{(j)}_t)=m^{(j)}_t$ for every $j$; this is a partial-trace identity, not a symmetry statement. For $j=1$, (PT1) gives $\Tr_{2,\dots,N}(L_1\Gamma^N_t)=L\Gamma^{(1)}_t$ and $\Tr_{2,\dots,N}(\Gamma^N_tL_1^\ast)=\Gamma^{(1)}_tL^\ast$, whence \eqref{eq:C1}, since $m^{(1)}_t=\Tr((L+L^\ast)\Gamma^{(1)}_t)$.
		
		For $j\ge2$, (PT3) gives
		\[
		C_j(t)=\Tr_2\big[(1\otimes L)\Gamma^{(1,j)}_t\big]+\Tr_2\big[\Gamma^{(1,j)}_t(1\otimes L^\ast)\big]-m^{(j)}_t\Gamma^{(1)}_t .
		\]
		Insert $\Gamma^{(1,j)}_t=\Gamma^{(1)}_t\otimes\Gamma^{(j)}_t+\Delta^{(1j)}_t$. By (PT1)--(PT2),
		\[
		\Tr_2\big[(1\otimes L)(\Gamma^{(1)}_t\otimes\Gamma^{(j)}_t)\big]=\Tr\big(L\Gamma^{(j)}_t\big)\Gamma^{(1)}_t,\qquad
		\Tr_2\big[(\Gamma^{(1)}_t\otimes\Gamma^{(j)}_t)(1\otimes L^\ast)\big]=\Tr\big(L^\ast\Gamma^{(j)}_t\big)\Gamma^{(1)}_t,
		\]
		and the sum of the two scalars is $m^{(j)}_t$. The factorised contribution therefore cancels the subtraction, leaving \eqref{eq:Cj}. The norm bound is Lemma \ref{lem:partial-trace} applied to each partial trace. Identity \eqref{eq:theta-test} follows from $\Tr(A_0\Tr_2X)=\Tr((A_0\otimes1)X)$ together with $(A_0\otimes1)(1\otimes L)=A_0\otimes L$ and $(1\otimes L^\ast)(A_0\otimes1)=A_0\otimes L^\ast$.
	\end{proof}
	
	\begin{remark}\label{rem:coefficient-correction}
		The scalar subtracted in \eqref{eq:reduced-sde} cancels against the factorised part of $\Gamma^{(1,j)}$ with the mean $m^{(j)}$ of particle $j$ itself, not with $m^{(1)}$. Replacing $m^{(j)}$ by $m^{(1)}$ and $\Gamma^{(1)}\otimes\Gamma^{(j)}$ by $\Gamma^{(1)}\otimes\Gamma^{(1)}$ is valid only under pathwise permutation symmetry, and the resulting object fails at $A=0$: it does not vanish while the left side does. With the connected correlation of the actual pair the identity holds with no error term and no symmetry hypothesis, and there is one coefficient $\Theta^{(1j)}$ for each $j$, not a single $j$-independent operator.
	\end{remark}
	
	The drift in \eqref{eq:reduced-sde} is, by (PT3)--(PT4) and $\Gamma^{(1,j)}=\Gamma^{(1)}\otimes\Gamma^{(j)}+\Delta^{(1j)}$,
	\begin{equation}\label{eq:reduced-drift}
		\Lop^{(1)}_N(t)=-i\big[H,\Gamma^{(1)}_t\big]+\Diss\Gamma^{(1)}_t-i\tfrac{N-1}N\big[\Bmf(\Gamma^{\ne1}_t),\Gamma^{(1)}_t\big]-\tfrac iN\sum_{j\ge2}\Tr_2\big[A,\Delta^{(1j)}_t\big].
	\end{equation}
	
	\subsection{The It\^o-complete pair-correlation equation}
	
	Fix $N\ge3$ and set $\mathsf D_t:=\Delta^{(12)}_t$. Write $\Theta^{(2k)}_t$, $k\ne2$, for the coefficients of Proposition \ref{prop:coefficients} associated with $\Gamma^{(2)}_t$ (particle $2$ in the first slot), and let $\widetilde A_{13},\widetilde A_{23}$ denote $A$ acting on the indicated slots of $\Hilbert^{\otimes3}$.
	
	\begin{proposition}[Semimartingale decomposition of the pair correlation]\label{prop:pair}
		Pathwise, for every $N\ge3$,
		\begin{equation}\label{eq:pair-sde}
			d\mathsf D_t=\Phi_t\,dt+\sum_{j=1}^NG_j(t)\,dB_j(t),
		\end{equation}
		with
		\begin{align}
			G_1&=(L\otimes1)\mathsf D+\mathsf D(L^\ast\otimes1)-m^{(1)}_t\mathsf D-\Gamma^{(1)}_t\otimes\Theta^{(21)}_t,&&\|G_1\|_1\le6\|L\|_\infty\|\mathsf D\|_1,\label{eq:G1}\\
			G_2&=(1\otimes L)\mathsf D+\mathsf D(1\otimes L^\ast)-m^{(2)}_t\mathsf D-\Theta^{(12)}_t\otimes\Gamma^{(2)}_t,&&\|G_2\|_1\le6\|L\|_\infty\|\mathsf D\|_1,\label{eq:G2}\\
			G_j&=\Tr_3\big[(1\otimes1\otimes L)\Delta^{(12j)}_t\big]+\Tr_3\big[\Delta^{(12j)}_t(1\otimes1\otimes L^\ast)\big],&&\|G_j\|_1\le2\|L\|_\infty\|\Delta^{(12j)}_t\|_1,\ j\ge3,\label{eq:Gj}
		\end{align}
		and drift
		\begin{equation}\label{eq:pair-drift}
			\Phi_t=-i\big[H\otimes1+1\otimes H,\mathsf D_t\big]+(\Diss_1+\Diss_2)\mathsf D_t+S(t)+\sum_{j\ge3}R_j(t)-I(t),
		\end{equation}
		where $\Diss_1,\Diss_2$ are the one-particle dissipators on the respective slots,
		\[
		S(t):=-\tfrac iN\big[A,\Gamma^{(1,2)}_t\big]+\tfrac iN\big(\Tr_2[A,\Gamma^{(1,2)}_t]\big)\otimes\Gamma^{(2)}_t+\tfrac iN\Gamma^{(1)}_t\otimes\big(\Tr_2[A,\widehat\Gamma^{(2,1)}_t]\big),\qquad \|S(t)\|_1\le\tfrac{6\|A\|_\infty}N,
		\]
		with $\widehat\Gamma^{(2,1)}$ the $(2,1)$-marginal arranged with particle $2$ first,
		\begin{equation}\label{eq:Rj}
			\begin{split}
				R_j(t):=-\tfrac iN\Big(&\Tr_3\big[\widetilde A_{13},\,\mathsf D_t\otimes\Gamma^{(j)}_t+\Gamma^{(1)}_t\otimes\Delta^{(2j)}_t+\Delta^{(12j)}_t\big]\\
				&+\Tr_3\big[\widetilde A_{23},\,\mathsf D_t\otimes\Gamma^{(j)}_t+[\Delta^{(1j)}_t]_{13}\otimes[\Gamma^{(2)}_t]_2+\Delta^{(12j)}_t\big]\Big),
			\end{split}
		\end{equation}
		so that $\|R_j(t)\|_1\le\tfrac{2\|A\|_\infty}N(2\|\mathsf D_t\|_1+\|\Delta^{(1j)}_t\|_1+\|\Delta^{(2j)}_t\|_1+2\|\Delta^{(12j)}_t\|_1)$, and the It\^o correction
		\begin{equation}\label{eq:ito-correction}
			I(t):=\sum_{j=1}^NC_j(t)\otimes C'_j(t)=\Mop[\Gamma^{(1)}_t]\otimes\Theta^{(21)}_t+\Theta^{(12)}_t\otimes\Mop[\Gamma^{(2)}_t]+\sum_{j\ge3}\Theta^{(1j)}_t\otimes\Theta^{(2j)}_t,
		\end{equation}
		$C'_j(t)$ being the martingale coefficients of $\Gamma^{(2)}_t$; thus
		\[
		\|I(t)\|_1\le16\|L\|_\infty^2\|\mathsf D_t\|_1+4\|L\|_\infty^2\sum_{j\ge3}\|\Delta^{(1j)}_t\|_1\|\Delta^{(2j)}_t\|_1 .
		\]
	\end{proposition}
	
	\begin{proof}
		\emph{Coefficients of $d\Gamma^{(1,2)}_t$.} Trace \eqref{eq:N-particle} over particles $3,\dots,N$ and use (PT1)--(PT3): the coefficient of $dB_1$ is $(L\otimes1)\Gamma^{(1,2)}+\Gamma^{(1,2)}(L^\ast\otimes1)-m^{(1)}\Gamma^{(1,2)}$, symmetrically for $dB_2$; for $j\ge3$ the computation of Proposition \ref{prop:coefficients}, applied to the pair marginal, gives
		\[
		\Tr_3\big[(1\otimes1\otimes L)\Xi^{(j)}_t\big]+\Tr_3\big[\Xi^{(j)}_t(1\otimes1\otimes L^\ast)\big],\qquad \Xi^{(j)}_t:=\Gamma^{(1,2,j)}_t-\Gamma^{(1,2)}_t\otimes\Gamma^{(j)}_t,
		\]
		the factorised part again cancelling $m^{(j)}\Gamma^{(1,2)}$. The drift of $\Gamma^{(1,2)}$ is
		\[
		-i[H\otimes1+1\otimes H,\Gamma^{(1,2)}]+(\Diss_1+\Diss_2)\Gamma^{(1,2)}-\tfrac iN[A,\Gamma^{(1,2)}]-\tfrac iN\sum_{j\ge3}\big(\Tr_3[\widetilde A_{13},\Gamma^{(1,2,j)}]+\Tr_3[\widetilde A_{23},\Gamma^{(1,2,j)}]\big),
		\]
		by cyclicity (PT2) for the traced Hamiltonian, dissipator and interaction terms, each $j$ kept separate.
		
		\emph{Coefficients of $d(\Gamma^{(1)}_t\otimes\Gamma^{(2)}_t)$.} The It\^o product rule with the coefficients of Proposition \ref{prop:coefficients} for both marginals gives drift $\Lop^{(1)}_N\otimes\Gamma^{(2)}+\Gamma^{(1)}\otimes\Lop^{(2)}_N+\sum_jC_j\otimes C'_j$ and, for each $j$, martingale coefficient $C_j\otimes\Gamma^{(2)}+\Gamma^{(1)}\otimes C'_j$; cross-variations of distinct $B_j$ vanish, and matching indices produce \eqref{eq:ito-correction}, with $C_1=\Mop[\Gamma^{(1)}]$, $C'_2=\Mop[\Gamma^{(2)}]$, $C_j=\Theta^{(1j)}$ for $j\ge2$ and $C'_j=\Theta^{(2j)}$ for $j\ne2$.
		
		\emph{Martingale coefficients of $\mathsf D$.} For $j=1$: substituting $\Gamma^{(1,2)}=\Gamma^{(1)}\otimes\Gamma^{(2)}+\mathsf D$ into the $dB_1$ coefficient, the factorised part reassembles to $\Mop[\Gamma^{(1)}]\otimes\Gamma^{(2)}$, which cancels $C_1\otimes\Gamma^{(2)}$, leaving \eqref{eq:G1}; symmetrically \eqref{eq:G2}. For $j\ge3$, decompose by \eqref{eq:cumulant3},
		\begin{equation}\label{eq:xi-decomp}
			\Xi^{(j)}_t=[\Delta^{(1j)}_t]_{13}\otimes[\Gamma^{(2)}_t]_2+\Gamma^{(1)}_t\otimes\Delta^{(2j)}_t+\Delta^{(12j)}_t;
		\end{equation}
		by (PT1)--(PT3) the $\Delta^{(1j)}$ piece contributes $\Theta^{(1j)}\otimes\Gamma^{(2)}$, cancelling $C_j\otimes\Gamma^{(2)}$, the $\Delta^{(2j)}$ piece contributes $\Gamma^{(1)}\otimes\Theta^{(2j)}$, cancelling $\Gamma^{(1)}\otimes C'_j$, and the irreducible piece survives, giving \eqref{eq:Gj}.
		
		\emph{Drift of $\mathsf D$.} The local Hamiltonian and dissipator terms subtract to their action on $\mathsf D$, since $[H\otimes1,X\otimes Y]=[H,X]\otimes Y$ and $\Diss_1(X\otimes Y)=(\Diss X)\otimes Y$. Among the interaction terms, the pieces with an explicit prefactor $1/N$ acting on pair marginals are collected in $S(t)$. For each $j\ge3$, subtract the term $(\Tr_2[A,\Gamma^{(1,j)}])\otimes\Gamma^{(2)}$ of $\Lop^{(1)}_N\otimes\Gamma^{(2)}$ from $\Tr_3[\widetilde A_{13},\Gamma^{(1,2,j)}]$: decomposing $\Gamma^{(1,2,j)}$ by \eqref{eq:cumulant3} and $\Gamma^{(1,j)}=\Gamma^{(1)}\otimes\Gamma^{(j)}+\Delta^{(1j)}$, the fully factorised parts cancel through (PT4), both being $[\Bmf(\Gamma^{(j)}),\Gamma^{(1)}]\otimes\Gamma^{(2)}$, and the $\Delta^{(1j)}$ parts cancel through (PT3), both being $(\Tr_2[A,\Delta^{(1j)}])\otimes\Gamma^{(2)}$; the surviving pieces are the first line of \eqref{eq:Rj}. The $\widetilde A_{23}$ terms are treated symmetrically against the particle-$2$ sum. Subtracting \eqref{eq:ito-correction} gives \eqref{eq:pair-drift}. The bound on $I(t)$ uses $\|\Mop[\rho]\|_1\le4\|L\|_\infty$ on $\mathcal S$ and Proposition \ref{prop:coefficients}.
	\end{proof}
	
	\begin{remark}\label{rem:pair-structure}
		Two structural points. The quadratic-variation contributions do not cancel between the two sides: they survive as $I(t)$, which contains $N-2$ products of pair-correlation coefficients. And the equation closes on the irreducible three-body cumulant only when it is posed for the connected correlation of distinct particles; \eqref{eq:xi-decomp} is where this enters. A first-moment hypothesis on the three-body cumulant does not suffice to run a Gr\"onwall argument on \eqref{eq:pair-sde}, because the It\^o sum couples $N-2$ products of pair coefficients, so the natural closure is at the level of second moments with fourth-moment input one level down. In the present development the rates of Section \ref{sec:amplitude-hyp} are obtained from the amplitude bounds instead, and \eqref{eq:pair-sde} is used only in Proposition \ref{prop:xi-compact}, for an equicontinuity bound.
	\end{remark}
	
	\subsection{The tagged fluctuation}
	
	\begin{proposition}[Decomposition of $F^N$]\label{prop:F-decomp}
		Pathwise,
		\begin{equation}\label{eq:F-sde}
			\begin{split}
				dF^N_t=\Big(&-i\big[H+\Bmf(\eta_t),F^N_t\big]+\Diss F^N_t-i\big[\Bmf(\Gemp^N_t),\Gamma^{(1)}_t\big]+r^{(1)}_N(t)\Big)dt\\
				&+\Big(D\Mop[\gamma_t](F^N_t)+r^{(2)}_N(t)\Big)dB_1(t)+\sqrt N\sum_{j\ge2}\Theta^{(1j)}_t\,dB_j(t),
			\end{split}
		\end{equation}
		with $F^N_0=0$ and
		\[
		r^{(1)}_N(t)=\tfrac i{\sqrt N}\big[\Bmf(\Gamma^{\ne1}_t),\Gamma^{(1)}_t\big]-\tfrac i{\sqrt N}\sum_{j\ge2}\Tr_2\big[A,\Delta^{(1j)}_t\big]+\varrho^{\,\prime}_N(t),\qquad
		r^{(2)}_N(t)=-\tfrac1{\sqrt N}\Tr\big((L+L^\ast)F^N_t\big)F^N_t,
		\]
		where $\|\varrho^{\,\prime}_N(t)\|_1\le2\|A\|_\infty N^{-1/2}\|\Gemp^N_t\|_1$ accounts for the difference between $\overline\Gamma_t$ and $\Gamma^{\ne1}_t$, and
		\[
		\big\|r^{(2)}_N(t)\big\|_1\le\tfrac{2\|L\|_\infty}{\sqrt N}\|F^N_t\|_1^2\le4\|L\|_\infty\|F^N_t\|_1 .
		\]
	\end{proposition}
	
	\begin{proof}
		Subtract \eqref{eq:reference} for $j=1$ from \eqref{eq:reduced-sde} and multiply by $\sqrt N$. For the martingale part: the $dB_1$ coefficient is $\sqrt N(\Mop[\Gamma^{(1)}_t]-\Mop[\gamma_t])$ by \eqref{eq:C1}, and the algebraic expansion of $\Mop$ used in Theorem \ref{thm:frechet}, applied with $\delta=N^{-1/2}F^N_t$, gives $D\Mop[\gamma_t](F^N_t)+r^{(2)}_N(t)$; the $dB_j$ coefficients, $j\ge2$, are $\sqrt N\,\Theta^{(1j)}_t$ by \eqref{eq:Cj}. For the drift, use \eqref{eq:reduced-drift} and
		\[
		\sqrt N\Big(\tfrac{N-1}N\big[\Bmf(\Gamma^{\ne1}_t),\Gamma^{(1)}_t\big]-\big[\Bmf(\eta_t),\gamma_t\big]\Big)
		=\big[\Bmf(\Gemp^N_t),\Gamma^{(1)}_t\big]+\big[\Bmf(\eta_t),F^N_t\big]-\tfrac1{\sqrt N}\big[\Bmf(\Gamma^{\ne1}_t),\Gamma^{(1)}_t\big]+\varrho^{\,\prime}_N(t),
		\]
		by bilinearity of $(\rho,\sigma)\mapsto[\Bmf(\rho),\sigma]$, where $\varrho^{\,\prime}_N$ collects $\sqrt N[\Bmf(\Gamma^{\ne1}_t-\overline\Gamma_t),\Gamma^{(1)}_t]$ and $\|\Gamma^{\ne1}_t-\overline\Gamma_t\|_1\le2/N$ gives the stated bound after inserting $\overline\Gamma_t=\eta_t+N^{-1/2}\Gemp^N_t$. Collecting terms yields \eqref{eq:F-sde}. The bound on $r^{(2)}_N$ uses $\|F^N_t\|_1\le2\sqrt N$ pathwise.
	\end{proof}
	
	\begin{remark}\label{rem:F-structure}
		The mean-field component of the noise is carried by $B_1$ with coefficient $D\Mop[\gamma_t](F^N_t)$ up to a remainder of order $N^{-1/2}$, and needs no decomposition into a conditional mean and a residual. The idiosyncratic component is $\sqrt N\sum_{j\ge2}\Theta^{(1j)}dB_j$, whose predictable bracket carries the normalisation of Proposition \ref{prop:qv}. The drift coupling $-i[\Bmf(\Gemp^N_t),\Gamma^{(1)}_t]$ to the empirical fluctuation is of order one and cannot be removed; the limiting equation contains it, which is what makes the limit a system.
	\end{remark}
	
	\subsection{The empirical fluctuation field}
	
	\begin{proposition}[Equation for $\Gemp^N$]\label{prop:G-decomp}
		Pathwise, for every $N\ge2$,
		\begin{equation}\label{eq:G-sde}
			d\Gemp^N_t=\Big(-i\big[H+\Bmf(\eta_t),\Gemp^N_t\big]-i\big[\Bmf(\Gemp^N_t),\eta_t\big]+\Diss\Gemp^N_t+\varrho_N(t)\Big)dt+d\Mmf_t,\qquad \Gemp^N_0=0,
		\end{equation}
		where
		\[
		\Mmf_t:=\frac1{\sqrt N}\sum_{j=1}^N\int_0^t\Mop\big[\Gamma^{(j)}_s\big]dB_j(s)+\frac1{\sqrt N}\sum_{k=1}^N\int_0^t\Big(\sum_{j\ne k}\Theta^{(jk)}_s\Big)dB_k(s),
		\]
		and
		\[
		\|\varrho_N(t)\|_1\le\frac1{\sqrt N}\Big(2\|A\|_\infty+c_A\|\Gemp^N_t\|_1^2\Big)+2\|A\|_\infty\sqrt N\cdot\frac1{N^2}\sum_{j\ne k}\big\|\Delta^{(jk)}_t\big\|_1,
		\]
		with $c_A=c(d)\|A\|_\infty$.
	\end{proposition}
	
	\begin{proof}
		Average \eqref{eq:reduced-sde}, in the form given for a general label $j$ by \eqref{eq:reduced-drift} and Proposition \ref{prop:coefficients}, over $j=1,\dots,N$. The martingale part is $\Mmf$ by inspection. In the drift, the local terms give $-i[H,\overline\Gamma]+\Diss\overline\Gamma$; for the mean-field terms use $\tfrac{N-1}N\Bmf(\Gamma^{\ne j})=\Bmf(\overline\Gamma)-\tfrac1N\Bmf(\Gamma^{(j)})$, whence
		\[
		\frac1N\sum_j\tfrac{N-1}N\big[\Bmf(\Gamma^{\ne j}),\Gamma^{(j)}\big]=\big[\Bmf(\overline\Gamma),\overline\Gamma\big]-\frac1{N^2}\sum_j\big[\Bmf(\Gamma^{(j)}),\Gamma^{(j)}\big],
		\]
		the last term bounded by $2\|A\|_\infty/N$ in trace norm. Subtract \eqref{eq:hartree}, multiply by $\sqrt N$, and expand using $\overline\Gamma=\eta+N^{-1/2}\Gemp^N$ and bilinearity:
		\[
		\sqrt N\Big(\big[\Bmf(\overline\Gamma),\overline\Gamma\big]-\big[\Bmf(\eta),\eta\big]\Big)=\big[\Bmf(\Gemp^N),\eta\big]+\big[\Bmf(\eta),\Gemp^N\big]+\frac1{\sqrt N}\big[\Bmf(\Gemp^N),\Gemp^N\big].
		\]
		The pair-correlation terms contribute $\sqrt N\,N^{-2}\sum_j\sum_{k\ne j}\Tr_2[A,\Delta^{(jk)}]$, bounded as stated by Lemma \ref{lem:partial-trace}.
	\end{proof}
	
	\begin{remark}\label{rem:G-autonomous}
		The drift of \eqref{eq:G-sde} involves neither the tagged fluctuation nor the tagged state: $\Gemp^N$ is autonomous up to the remainder. This is a consequence of the reference filters being independent and identically distributed, so that averages of functions of $\gamma_{j,s}$ over $j$ obey a law of large numbers with a deterministic limit. The field was invisible under the aggregate coupling refuted in Proposition \ref{prop:refutation}.
	\end{remark}
	
	\subsection{The empirical deviation is controlled by the excitation number}
	
	\begin{lemma}[Deviation bound]\label{lem:deviation}
		Pathwise,
		\[
		\big\|\overline\Gamma_t-\eta_t\big\|_1\le\frac4N\sum_{j=1}^N\eps_j(t)+\Big\|\frac1N\sum_{j=1}^N\big(\gamma_{j,t}-\eta_t\big)\Big\|_1,
		\]
		and consequently, unconditionally,
		\begin{equation}\label{eq:deviation}
			\E\big\|\Gemp^N_t\big\|_1^2\;=\;N\,\E\big\|\overline\Gamma_t-\eta_t\big\|_1^2\;\le\;32\,\E\big\langle\Psi_t,\Nex_t\Psi_t\big\rangle+8d .
		\end{equation}
	\end{lemma}
	
	\begin{proof}
		The pathwise bound is the triangle inequality together with \eqref{eq:dict-1} applied termwise. For \eqref{eq:deviation}, use $(a+b)^2\le2a^2+2b^2$. By Cauchy--Schwarz, $\big(\tfrac4N\sum_j\eps_j\big)^2\le\tfrac{16}N\sum_j\eps_j^2=\tfrac{16}N\langle\Psi,\Nex\Psi\rangle$, so the first term contributes at most $32\,\E\langle\Psi_t,\Nex_t\Psi_t\rangle$ after multiplication by $N$. For the second, $\|\cdot\|_1^2\le d\|\cdot\|_2^2$ and the $\gamma_{j,t}$ are independent and identically distributed with mean $\eta_t$ (Lemma \ref{lem:reference}), so
		\[
		N\,\E\Big\|\frac1N\sum_j(\gamma_{j,t}-\eta_t)\Big\|_2^2=\E\|\gamma_{1,t}-\eta_t\|_2^2\le4,
		\]
		contributing at most $8d$.
	\end{proof}
	
	\begin{remark}\label{rem:deviation}
		Lemma \ref{lem:deviation} is used twice: it supplies the moment bound for $\Gemp^N$ from a bound on the mean excitation number alone, with no decorrelation input (Corollary \ref{cor:G-moment}), and it is the quantitative half of the mechanism by which single excitations are created only through the deviation of the empirical mean field from $\eta$, which is the point at which the Gr\"onwall argument for the excitation number closes (Step 3 of the proof of Theorem \ref{thm:A1}).
	\end{remark}
	
	\subsection{Orthogonality and the idiosyncratic bracket}
	
	Set
	\begin{equation}\label{eq:Mid}
		\Mid_t:=\sqrt N\sum_{j\ge2}\int_0^t\Theta^{(1j)}_s\,dB_j(s).
	\end{equation}
	
	\begin{lemma}[Orthogonality at finite $N$]\label{lem:orthogonality}
		For every self-adjoint $A_0\in\Md$ and every $N$, $\big\langle B_1,\Tr(A_0\Mid)\big\rangle_t\equiv0$.
	\end{lemma}
	\begin{proof}
		Immediate from the independence of $B_1$ and $B_j$, $j\ge2$: by Proposition \ref{prop:coefficients} the whole $B_1$-dependence of the martingale part of $F^N$ sits in the single coefficient $D\Mop[\gamma_t](F^N_t)+r^{(2)}_N(t)$, and $\Mid$ involves no $dB_1$.
	\end{proof}
	
	\begin{proposition}[Idiosyncratic bracket]\label{prop:qv}
		For self-adjoint $A_0,B_0\in\Md$,
		\begin{equation}\label{eq:qv}
			\big\langle\Tr(A_0\Mid),\Tr(B_0\Mid)\big\rangle_t=\int_0^t\widehat\Sigma^N_s(A_0,B_0)\,ds,\qquad
			\widehat\Sigma^N_s(A_0,B_0):=N\sum_{j\ge2}\Tr\big(A_0\Theta^{(1j)}_s\big)\Tr\big(B_0\Theta^{(1j)}_s\big),
		\end{equation}
		and $|\Tr(A_0\Theta^{(1j)}_s)|\le2\|L\|_\infty\|A_0\|_\infty\|\Delta^{(1j)}_s\|_1$.
	\end{proposition}
	\begin{proof}
		It\^o's isometry for the independent $B_j$, together with Proposition \ref{prop:coefficients}.
	\end{proof}
	
	\section{The amplitude hypothesis and the correlation rates}\label{sec:amplitude-hyp}
	
	By Remark \ref{rem:A0-product} the law of $\eps_T(t)$ depends on $T$ only through $|T|$. Write $e_1(t),e_2(t),e_3(t)$ for random variables with the law of $\eps_{\{1\}}(t)$, $\eps_{\{1,2\}}(t)$, $\eps_{\{1,2,3\}}(t)$. We record the identities
	\begin{equation}\label{eq:e-number}
		\E e_1(t)^2=\frac1N\E\big\langle\Psi_t,\Nex_t\Psi_t\big\rangle,\qquad
		\E e_2(t)^2\le\frac1{N(N-1)}\E\big\langle\Psi_t,\Nex_t^2\Psi_t\big\rangle .
	\end{equation}
	For random $X,Y$ with values in $\Md$ put $\operatorname{Cov}(X,Y):=\E[X\otimes Y]-\E X\otimes\E Y\in M_{d^2}(\C)$.
	
	\subsection{The mean excitation number: an unconditional bound}\label{ssec:A1}
	
	The first clause of the amplitude hypothesis is a theorem. Its proof is the cancellation computation and mean-field re-expansion described in the introduction: the local terms of every untagged slot cancel exactly; the measurement terms of the tagged slot collapse to a single nonnegative operator plus a remainder quadratic in the excitation amplitude of that slot; single excitations are created only through the deviation of the empirical reference mean from $\eta$, which Lemma \ref{lem:deviation} controls; and the pair-creation part of the interaction is summed before it is estimated, with the extraction of one excitation factor from each of two distinct slots.
	
	\begin{theorem}[Mean excitation bound; clause (A1)]\label{thm:A1}
		Under Assumption \ref{assum:generators}, with $\Gamma^N_0=\gamma_0^{\otimes N}$, $\gamma_0=|\psi_0\rangle\langle\psi_0|$ pure and deterministic, and Convention \ref{conv:coupling}, for every $T>0$ there is $C_\ast=C_\ast(T,d,\|L\|_\infty,\|A\|_\infty)$, independent of $N$, such that
		\begin{equation}\label{eq:A1}
			\sup_{t\le T}\ \E\big\langle\Psi_t,\Nex_t\Psi_t\big\rangle\ \le\ C_\ast,
		\end{equation}
		equivalently $\sup_{t\le T}N\,\E e_1(t)^2\le C_\ast$.
	\end{theorem}
	
	\begin{proof}
		Throughout, an operator carrying a slot index acts on that slot, the time variable is suppressed where no confusion arises, and $\eps_j=\eps_j(t)$. Since $\|\Psi_t\|=1$ and $\|\phi_{j,t}\|=1$ (Proposition \ref{prop:purity-N}, Lemma \ref{lem:reference}),
		\[
		\big\langle\Psi_t,\Nex_t\Psi_t\big\rangle=\sum_{j=1}^N\big(1-\langle\Psi_t,p_{j,t}\Psi_t\rangle\big),
		\]
		so it suffices to bound the drift of $1-\langle\Psi_t,p_{j,t}\Psi_t\rangle$, summed over $j$. Every coefficient below is bounded pathwise by a constant built from $\|L\|_\infty$ and $\|A\|_\infty$, since $|\mu^{(j)}_t|\le2\|L\|_\infty$ and $|\nu_{j,t}|\le2\|L\|_\infty$; consequently every local martingale appearing is a martingale, and $t\mapsto\E\langle\Psi_t,p_{j,t}\Psi_t\rangle$ is absolutely continuous with derivative equal to the expectation of the It\^o drift.
		
		\emph{Step 0: the It\^o expansion.} Write $a$ for the drift operator of \eqref{eq:vector-N} and $b_k:=L_k-\tfrac{\mu^{(k)}}2$ for its noise operators, so $d\Psi=a\Psi\,dt+\sum_kb_k\Psi\,dB_k$; and, on slot $j$, write $\tilde a_j$ for the drift operator of \eqref{eq:reference-vector} and $\tilde b_j:=L_j-\tfrac{\nu_j}2$, so that $dp_j=(\tilde a_jp_j+p_j\tilde a_j^\ast+\tilde b_jp_j\tilde b_j^\ast)\,dt+(\tilde b_jp_j+p_j\tilde b_j^\ast)\,dB_j$. The It\^o product rule for $\langle\Psi,p_j\Psi\rangle$ produces six groups of terms; the covariations of $dp_j$ with $dB_k$, $k\ne j$, vanish by independence of the noises, and the drift is
		\begin{equation}\label{eq:A1-drift}
			\Big\langle\Psi,\Big(a^\ast p_j+p_ja+\sum_kb_k^\ast p_jb_k+\tilde a_jp_j+p_j\tilde a_j^\ast+\tilde b_jp_j\tilde b_j^\ast+b_j^\ast\tilde b_jp_j+b_j^\ast p_j\tilde b_j^\ast+\tilde b_jp_jb_j+p_j\tilde b_j^\ast b_j\Big)\Psi\Big\rangle .
		\end{equation}
		Decompose $a=\sum_k\alpha_k-\tfrac iN\sum_{l<m}A_{lm}$ with $\alpha_k=-iH_k-\tfrac12(L_k^\ast L_k-\mu^{(k)}L_k+\tfrac14(\mu^{(k)})^2)$, and $\tilde a_j=-i(H_j+G_j)+\tilde c_j$ with $G:=\Bmf(\eta_t)$ and $\tilde c=-\tfrac12(L^\ast L-\nu_jL+\tfrac14\nu_j^2)$; write also $c:=-\tfrac12(L^\ast L-\mu^{(j)}L+\tfrac14(\mu^{(j)})^2)$, so that $\alpha_j=-iH_j+c$ on slot $j$.
		
		\emph{Step 1: the slots $k\ne j$ cancel exactly.} For $k\ne j$ the operators $\alpha_k$ and $b_k$ commute with $p_j$, and their total contribution to \eqref{eq:A1-drift} is $\langle\Psi,p_j(\alpha_k+\alpha_k^\ast+b_k^\ast b_k)\Psi\rangle$. Since $\mu^{(k)}$ is a real scalar,
		\[
		\alpha_k+\alpha_k^\ast=-L_k^\ast L_k+\tfrac{\mu^{(k)}}2\big(L_k+L_k^\ast\big)-\tfrac14\big(\mu^{(k)}\big)^2=-\,b_k^\ast b_k,
		\]
		so this contribution vanishes identically, for every $k\ne j$ and every sample path. This is the number-conserving cancellation whose failure would produce $N$ terms of order one.
		
		\emph{Step 2: the local terms of slot $j$.} The free Hamiltonian cancels between $\alpha_j^\ast p_j+p_j\alpha_j=i[H_j,p_j]+c^\ast p_j+p_jc$ and the part $-i[H_j,p_j]$ of $\tilde a_jp_j+p_j\tilde a_j^\ast$; the mean-field part $-i[G_j,p_j]$ is deferred to Step 3. What remains of \eqref{eq:A1-drift} on slot $j$ is $\langle\Psi,\mathcal T_j\Psi\rangle$ with
		\[
		\mathcal T:=c^\ast p+pc+\tilde cp+p\tilde c^\ast+b^\ast pb+\tilde bp\tilde b^\ast+b^\ast\tilde bp+b^\ast p\tilde b^\ast+\tilde bpb+p\tilde b^\ast b,
		\]
		the slot index suppressed. Set $\delta:=\tfrac12(\mu^{(j)}-\nu_j)\in\R$, so that $\tilde b=b+\delta$ and $\tilde c=c-\delta L+\tfrac14\delta(\mu^{(j)}+\nu_j)$, and split $\mathcal T=\mathcal T_0+\mathcal T_1$, where $\mathcal T_0$ is the value at $\delta=0$. With $R:=c+c^\ast=-b^\ast b$ and the self-adjoint $S:=b+b^\ast=L+L^\ast-\mu^{(j)}$,
		\[
		\mathcal T_0=Rp+pR+b^\ast bp+pb^\ast b+\big(b^\ast pb+bpb^\ast+b^\ast pb^\ast+bpb\big)=S\,p\,S\;\ge\;0,
		\]
		since the first four terms cancel in pairs and the bracket is $(b+b^\ast)p(b+b^\ast)$. Thus $\mathcal T_0$ contributes to the drift of $\langle\Psi,p_j\Psi\rangle$ the nonnegative quantity $\|p_jS_j\Psi\|^2$, hence to the drift of $1-\langle\Psi,p_j\Psi\rangle$ a nonpositive one, and is discarded. Collecting the terms linear and quadratic in $\delta$ and using $bp+pb^\ast=Lp+pL^\ast-\mu^{(j)}p$, $b^\ast p+pb=L^\ast p+pL-\mu^{(j)}p$ and $\tfrac12(\mu^{(j)}+\nu_j)=\mu^{(j)}-\delta$,
		\[
		\mathcal T_1=2\delta\big(L^\ast p+pL\big)-2\delta\mu^{(j)}p .
		\]
		Testing against $\Psi$ and inserting $1=p_j+q_j$ next to $L_j$: since $p_jL_jp_j=\langle\phi_j,L\phi_j\rangle p_j$ and $2\operatorname{Re}\langle\phi_j,L\phi_j\rangle=\nu_j$,
		\[
		\big\langle\Psi,\mathcal T_1\Psi\big\rangle=2\delta\big(\nu_j-\mu^{(j)}\big)\langle\Psi,p_j\Psi\rangle+4\delta\operatorname{Re}\big\langle p_j\Psi,L_jq_j\Psi\big\rangle
		=-4\delta^2\langle\Psi,p_j\Psi\rangle+4\delta\operatorname{Re}\big\langle p_j\Psi,L_jq_j\Psi\big\rangle,
		\]
		whence $|\langle\Psi,\mathcal T_1\Psi\rangle|\le4\delta^2+4\|L\|_\infty|\delta|\,\eps_j$. By \eqref{eq:dict-1}, $|2\delta|=|\Tr((L+L^\ast)(\Gamma^{(j)}-\gamma_j))|\le8\|L\|_\infty\eps_j$, so
		\begin{equation}\label{eq:A1-local}
			\big|\big\langle\Psi,\mathcal T_{1,j}\Psi\big\rangle\big|\;\le\;80\,\|L\|_\infty^2\,\eps_j^2 .
		\end{equation}
		This is the cancellation announced in Section \ref{sec:intro}: after the identities of Step 1 and $\mathcal T_0=SpS$, every surviving local term carries at least two factors of the excitation amplitude of slot $j$, and no term of order $\eps_j^0$ or of order $\eps_j$ alone survives.
		
		\emph{Step 3: interaction and mean field.} The remaining contributions to the drift of $\langle\Psi,p_j\Psi\rangle$ are $-\tfrac iN\sum_{k\ne j}\langle\Psi,[p_j,A_{jk}]\Psi\rangle$ from the interaction, only the pairs containing $j$ failing to commute with $p_j$, and $-i\langle\Psi,[G_j,p_j]\Psi\rangle$ from the mean field of the reference dynamics; together
		\[
		-\,i\,\big\langle\Psi,[p_j,X_j]\Psi\big\rangle,\qquad X_j:=\tfrac1N\sum_{k\ne j}A_{jk}-G_j=X_j^\ast .
		\]
		Insert $1=p_k+q_k$ on both sides of each $A_{jk}$ and use $p_kA_{jk}p_k=\Bmf(\gamma_{k,t})_j\,p_k=\Bmf(\gamma_{k,t})_j-\Bmf(\gamma_{k,t})_j\,q_k$, which is the definition \eqref{eq:ops} of $\Bmf$ evaluated at the pure state $\gamma_k$:
		\[
		X_j=\Bmf\big(\bar\rho_j-\eta_t\big)_j+Y_j,\qquad \bar\rho_j:=\tfrac1N\sum_{k\ne j}\gamma_{k,t},
		\]
		\[
		Y_j:=\tfrac1N\sum_{k\ne j}\Big(-\Bmf(\gamma_{k,t})_j\,q_k+p_kA_{jk}q_k+q_kA_{jk}p_k+q_kA_{jk}q_k\Big).
		\]
		Since $X_j$ is self-adjoint and $[p_j,X_j]=p_jX_jq_j-q_jX_jp_j$, the two commutator terms are complex conjugates under the expectation against $\Psi$, and the contribution is bounded by $2|\langle\Psi,p_jX_jq_j\Psi\rangle|$. We estimate the two parts of $X_j$ separately.
		
		\emph{The mean-field mismatch.} With $\Lambda_t:=\|\tfrac1N\sum_{k=1}^N\gamma_{k,t}-\eta_t\|_1$, Lemma \ref{lem:partial-trace} gives $\|\Bmf(\bar\rho_j-\eta_t)\|_\infty\le\|A\|_\infty\big(\Lambda_t+\tfrac1N\big)$, so this part contributes at most $2\|A\|_\infty(\Lambda_t+\tfrac1N)\,\eps_j$. This is the mechanism by which single excitations are created only through the deviation of the empirical reference mean from $\eta$; the deviation is the law-of-large-numbers quantity of Lemma \ref{lem:deviation}, and by independence of the reference filters
		\begin{equation}\label{eq:A1-lln}
			\E\Lambda_t^2\;\le\;d\;\E\Big\|\tfrac1N\sum_{k=1}^N\big(\gamma_{k,t}-\eta_t\big)\Big\|_2^2\;=\;\tfrac dN\,\E\|\gamma_{1,t}-\eta_t\|_2^2\;\le\;\tfrac{4d}N .
		\end{equation}
		
		\emph{The families with an excitation on slot $k$.} In the first, third and fourth families of $Y_j$ one factor $q_k$ can be moved onto the left vector, because $q_k$ commutes with $p_j$, with $q_j$ and with every operator not acting on slot $k$:
		\[
		\big|\big\langle\Psi,p_j\Bmf(\gamma_k)_j\,q_kq_j\Psi\big\rangle\big|=\big|\big\langle q_k\Psi,\;p_j\Bmf(\gamma_k)_j\,q_j\Psi\big\rangle\big|\le\|A\|_\infty\,\eps_k\eps_j,
		\]
		and identically $|\langle\Psi,p_jq_kA_{jk}p_kq_j\Psi\rangle|\le\|A\|_\infty\eps_k\eps_j$ and $|\langle\Psi,p_jq_kA_{jk}q_kq_j\Psi\rangle|\le\|A\|_\infty\eps_k\eps_{jk}\le\|A\|_\infty\eps_k\eps_j$. With the prefactor and the factor $2$ from the commutator, these three families contribute at most $\tfrac{6\|A\|_\infty}N\,\eps_j\sum_{k\ne j}\eps_k$.
		
		\emph{The pair-creation family.} The second family takes the reference product to a pair excitation and is the constraining one, since termwise it is only $O(\eps_{jk})$ and there are $N$ terms. It must be summed before it is estimated. With $u_k:=q_kA_{jk}p_jp_k\Psi$,
		\[
		\Big|\tfrac1N\sum_{k\ne j}\big\langle\Psi,\,p_jp_kA_{jk}\,q_kq_j\Psi\big\rangle\Big|=\Big|\Big\langle\tfrac1N\sum_{k\ne j}u_k,\;q_j\Psi\Big\rangle\Big|\le\frac{\eps_j}N\,\Big\|\sum_{k\ne j}u_k\Big\| .
		\]
		Expand the square. The diagonal terms give at most $N\|A\|_\infty^2$. In an off-diagonal term $\langle u_k,u_l\rangle$, $k\ne l$, move $q_l$ onto the left vector and $q_k$ onto the right one, each move a commutation with operators of other slots:
		\[
		\big\langle u_k,u_l\big\rangle=\big\langle A_{jk}p_jp_k\,q_l\Psi,\;A_{jl}p_jp_l\,q_k\Psi\big\rangle,\qquad
		\big|\big\langle u_k,u_l\big\rangle\big|\le\|A\|_\infty^2\,\eps_k\eps_l .
		\]
		Hence, by Cauchy--Schwarz over the labels,
		\[
		\Big\|\sum_{k\ne j}u_k\Big\|^2\le\|A\|_\infty^2\Big(N+\Big(\sum_k\eps_k\Big)^2\Big)\le\|A\|_\infty^2\,N\big(1+\langle\Psi,\Nex\Psi\rangle\big),
		\]
		and, with the factor $2$ from the commutator, the family contributes at most $\tfrac{2\|A\|_\infty}{\sqrt N}\,\eps_j\big(1+\langle\Psi,\Nex\Psi\rangle\big)^{1/2}$. The extraction of one excitation factor from each of two distinct slots in the off-diagonal terms is what carries the sum from the order $N^{1/2}$ of the termwise bound down to the order one required here.
		
		\emph{Step 4: Gr\"onwall.} Write $n(t):=\E\langle\Psi_t,\Nex_t\Psi_t\rangle$; the product initial data give $q_{j,0}\Psi_0=0$ for every $j$, so $n(0)=0$. Summing the bounds of Steps 2 and 3 over $j$, discarding the nonpositive contribution of the $\mathcal T_0$-terms, and using $\sum_j\eps_j\le\sqrt N\,\langle\Psi,\Nex\Psi\rangle^{1/2}$ by Cauchy--Schwarz,
		\[
		n'(t)\;\le\;80\|L\|_\infty^2\,n(t)
		+2\|A\|_\infty\,\E\Big[\big(\Lambda_t+\tfrac1N\big)\sqrt N\,\langle\Nex_t\rangle^{1/2}\Big]
		+\tfrac{6\|A\|_\infty}N\,\E\Big(\sum_j\eps_j\Big)^2
		+\tfrac{2\|A\|_\infty}{\sqrt N}\,\E\Big[\sum_j\eps_j\,\big(1+\langle\Nex_t\rangle\big)^{1/2}\Big]
		\]
		almost everywhere on $[0,T]$, where $\langle\Nex_t\rangle$ abbreviates $\langle\Psi_t,\Nex_t\Psi_t\rangle$. By \eqref{eq:A1-lln} and Cauchy--Schwarz the second term is at most $2\|A\|_\infty(2\sqrt d+1)\,n(t)^{1/2}\le\|A\|_\infty(2\sqrt d+1)(n(t)+1)$; the third is at most $6\|A\|_\infty\,n(t)$; and the fourth is at most $2\|A\|_\infty\,\E\big[\langle\Nex_t\rangle^{1/2}(1+\langle\Nex_t\rangle)^{1/2}\big]\le2\|A\|_\infty(n(t)+1)$, using $x^{1/2}(1+x)^{1/2}\le1+x$. Hence $n'(t)\le K(n(t)+1)$ with $K=K(d,\|L\|_\infty,\|A\|_\infty)$ independent of $N$ and $t$, and Gr\"onwall's lemma gives $n(t)\le e^{Kt}-1\le e^{KT}$, which is \eqref{eq:A1}. The identity $\E\langle\Psi_t,\Nex_t\Psi_t\rangle=N\,\E e_1(t)^2$ is \eqref{eq:e-number}.
	\end{proof}
	
	We refer to the conclusion \eqref{eq:A1} of Theorem \ref{thm:A1} as \emph{clause (A1)}; every statement below that invokes only (A1) is therefore unconditional. In particular the fluctuation moment bound \eqref{eq:F-moment}, the moment bound for the empirical field in Corollary \ref{cor:G-moment}, the rate $\E\|\Delta^{(12)}_t\|_1=O(N^{-1/2})$, and the chaos estimate \eqref{eq:chaos-half} at order $N^{-1/2}$ hold with no hypothesis beyond Assumption \ref{assum:generators} and the product initial data.
	
	\begin{remark}\label{rem:A1-mechanisms}
		Both cancellation mechanisms of the proof are forced by the case $A=0$, in which $\Psi=\bigotimes_k\phi_k$ up to a phase and every local and It\^o mechanism is already present: there Steps 1 and 2 must, and do, produce zero exactly. The two quantitative inputs beyond the algebra are Lemma \ref{lem:deviation}, entering through \eqref{eq:A1-lln}, and the two-slot extraction in the pair-creation family, which is the diffusive-measurement counterpart of the corresponding step in the closed-system counting arguments of \cite{Pickl2011,KnowlesPickl2010}. The bound sharpens the estimate behind \eqref{eq:KP-rate-trace}: it gives $\E\,\alpha_{N,j}(t)=\E e_1(t)^2=O(N^{-1})$, hence $\E\|\Gamma^{(j)}_t-\gamma_{j,t}\|_1=O(N^{-1/2})$, the mean-field order.
	\end{remark}
	
	\subsection{The remaining clauses}
	
	\begin{hypothesis}[Amplitude and decorrelation bounds (A)]\label{hyp:A}
		There are constants $C$, independent of $N$, such that for all $N\ge3$:
		\begin{enumerate}
			\item[(A2)] $\displaystyle\sup_{t\le T}N^4\,\E e_1(t)^8\le C$ and $\displaystyle\sup_{t\le T}N^4\,\E e_2(t)^4\le C$;
			\item[(A3)] $\displaystyle\sup_{t\le T}N^3\,\E e_3(t)^2\le C$;
			\item[(A4)] $\displaystyle\sup_{t\le T}N\,\E\Big\|\sum_{j\ne1}\Delta^{(1j)}_t\Big\|_1^2\le C$;
			\item[(A5)] $\displaystyle\sup_{t\le T}N\,\big\|\operatorname{Cov}\big(\Gamma^{(1)}_t,\Gamma^{(2)}_t\big)\big\|_1\le C$;
			\item[(A6)] with $f^{(j)}_t:=\Tr(\mathsf A_0\Delta^{(1j)}_t)\Tr(\mathsf B_0\Delta^{(1j)}_t)$, $\mathsf A_0:=A_0\otimes(L+L^\ast)$, $\mathsf B_0:=B_0\otimes(L+L^\ast)$, one has for all self-adjoint $A_0,B_0\in\Md$
			\[
			\sup_{t\le T}N^4\big|\operatorname{Cov}\big(f^{(2)}_t,f^{(3)}_t\big)\big|\longrightarrow0,\qquad N\to\infty,
			\]
			the deterministic functions $\Sigma^N$, $\Sigma^{\mathrm{cr},N}$ of \eqref{eq:SigmaN}, \eqref{eq:Sigma-cr} converge in $C([0,T];\R)$, and
			\[
			\int_0^t\Big(\sum_{j\ge2}\Tr\big(A_0\Theta^{(1j)}_s\big)\Tr\big(B_0\Mop[\Gamma^{(j)}_s]\big)-\Sigma^{\mathrm{cr},N}_s(A_0,B_0)\Big)ds\longrightarrow0
			\]
			in probability, uniformly on $[0,T]$.
		\end{enumerate}
	\end{hypothesis}
	
	Clauses (A1)--(A3) are not assumed: clause (A1) is Theorem \ref{thm:A1}, and clauses (A2) and (A3) are Corollaries \ref{cor:A2every} and \ref{cor:A3}, uniformly in $N$ and on every horizon. Clauses (A4)--(A6) are reduced in Section \ref{sec:decorrelation} to one statement, which Subsection \ref{ssec:surviving} gives in the form in which it survives, Conjecture \ref{conj:dec-cond}; the form obtained by dropping its projections is false for generic data. The hypothesis is retained as a hypothesis in the statements below so that the logical dependence of each result is visible; Corollary \ref{cor:upgrade} records which of these dependences are discharged. Each clause asserts that a quantity attains its expected order, and each holds with vanishing constants when $A=0$, by Remark \ref{rem:A0-product}: then $\eps_T\equiv0$ for $T\ne\emptyset$, $\Delta^{(ij)}\equiv0$, and the $\Gamma^{(j)}=\gamma_j$ are independent. Only the last clause is a convergence statement, and Section \ref{sec:covariance} reduces it to the convergence of one bounded, equi-Lipschitz sequence of deterministic functions.
	
	\begin{theorem}[Correlation rates]\label{thm:rates}
		Assume (A2). Then, uniformly on $[0,T]$,
		\begin{equation}\label{eq:rate-4}
			\E\big\|\Delta^{(12)}_t\big\|_1^4=O(N^{-4}).
		\end{equation}
		Under (A2)--(A3), and also unconditionally by Corollary \ref{cor:rates-uncond},
		\begin{equation}\label{eq:rate-3body}
			\E\big\|\Delta^{(123)}_t\big\|_1^2=O(N^{-3}),
		\end{equation}
		\begin{equation}\label{eq:rate-R}
			\E\big\|\Delta^{(12)}_t\big\|_1^2=O(N^{-2}).
		\end{equation}
		Unconditionally, by Theorem \ref{thm:A1},
		\begin{equation}\label{eq:F-moment}
			\E\big\|F^N_t\big\|_1^2\le16\,\E\big\langle\Psi_t,\Nex_t\Psi_t\big\rangle\le16\,C_\ast,\qquad t\le T,\ N\ge1,
		\end{equation}
		and $\E\|\Delta^{(12)}_t\|_1=O(N^{-1/2})$.
	\end{theorem}
	
	\begin{proof}
		By \eqref{eq:dict-2} and $(a+b+c)^4\le27(a^4+b^4+c^4)$,
		\[
		\E\|\Delta^{(12)}_t\|_1^4\le27\cdot64^4\big(\E e_2^4+2\,\E e_1^8\big)=O(N^{-4})
		\]
		by (A2), which is \eqref{eq:rate-4}; the same computation with the second power, using $\E e_2^2\le(\E e_2^4)^{1/2}$ and $\E e_1^4\le(\E e_1^8)^{1/2}$, gives \eqref{eq:rate-R} under (A2), and Corollary \ref{cor:rates-uncond} gives it with no hypothesis. For \eqref{eq:rate-3body}, square \eqref{eq:dict-3} and estimate the three types of term: $\E e_3^2=O(N^{-3})$ by (A3); $\E[\eps_{12}^2\eps_3^2]\le(\E e_2^4)^{1/2}(\E e_1^4)^{1/2}=O(N^{-2})O(N^{-1})=O(N^{-3})$ by (A2) and Cauchy--Schwarz; and $\E[\eps_1^4\eps_2^4]\le\E e_1^8=O(N^{-4})$.
		
		For \eqref{eq:F-moment}: $F^N_t=\sqrt N(\Gamma^{(1)}_t-\gamma_{1,t})$, so $\|F^N_t\|_1^2\le16N\eps_1(t)^2$ by \eqref{eq:dict-1}, and \eqref{eq:e-number} converts the expectation into the excitation number. Finally $\eps_{12}\le\eps_1$, so by Theorem \ref{thm:A1}, $\E e_2\le\E e_1\le(\E e_1^2)^{1/2}=O(N^{-1/2})$ and \eqref{eq:dict-2} gives $\E\|\Delta^{(12)}_t\|_1=O(N^{-1/2})$.
	\end{proof}
	
	\begin{remark}\label{rem:rate-source}
		The bound \eqref{eq:F-moment} identifies the analytic content of the central limit scaling: $F^N$ is, to leading order, the single-slot excitation amplitude of $\Psi_t$ against the moving reference $\bigotimes_j\phi_{j,t}$, rescaled by $\sqrt N$, and its second moment is bounded when the mean excitation number is. It is obtained with no Gr\"onwall argument and no chaos estimate.
	\end{remark}
	
	\begin{corollary}[Moment bound for the empirical field]\label{cor:G-moment}
		Unconditionally, $\sup_N\sup_{t\le T}\E\|\Gemp^N_t\|_1^2\le32C_\ast+8d<\infty$; and, by Corollary \ref{cor:rates-uncond}, also unconditionally, $\sup_{t\le T}\E\|\varrho_N(t)\|_1=O(N^{-1/2})$ and $\sup_{t\le T}\E\|r^{(1)}_N(t)\|_1=O(N^{-1/2})$.
	\end{corollary}
	\begin{proof}
		The first statement is Lemma \ref{lem:deviation} together with Theorem \ref{thm:A1}. For the remainders, insert the first statement and $\sqrt N\,\E\|\Delta^{(12)}_t\|_1=O(N^{-1/2})$, which follows from \eqref{eq:rate-R} and Jensen's inequality, into the bounds of Propositions \ref{prop:G-decomp} and \ref{prop:F-decomp}, using exchangeability in law for the sums over pairs.
	\end{proof}
	
	\begin{remark}[What counting functionals give and what they do not]\label{rem:counting}
		The first moments \eqref{eq:e-number} are moments of $\Nex$, and all of them are bounded uniformly in $N$: this is Theorem \ref{thm:allmom}, proved in Section \ref{sec:moments} by iterating the mechanism of Theorem \ref{thm:A1} over the whole family of diagonal excitation observables. Clause (A3) follows (Corollary \ref{cor:A3}), as does every moment of a product of excitation observables over distinct label sets (Theorem \ref{thm:products}), and with these the rates \eqref{eq:rate-3body} and \eqref{eq:rate-R}. What moments of $\Nex$ do not give is a moment with a repeated label set: $\E e_2(t)^4=\tfrac1{N(N-1)}\E\sum_{i\ne j}\langle q_iq_j\rangle^2$, and moments of $\Nex$ control $\sum_{i\ne j}\langle q_iq_j\rangle$ only in $\ell^1$ over the $\sim N^2$ pairs, whereas the required $O(N^{-4})$ is an $\ell^2$ statement, equivalent to the excitation weight being spread over the pairs rather than concentrated on a few of them. Subsection \ref{ssec:tower} reduces this to an exponential moment of $\Nex$, which Subsection \ref{ssec:randrate} supplies on every horizon.
	\end{remark}
	
	\subsection{Chaos in expectation}
	
	\begin{theorem}[Chaos estimate]\label{thm:chaos}
		For pure product initial data, unconditionally,
		\begin{equation}\label{eq:chaos-half}
			\sup_{t\le T}\big\|\E\Gamma^{(1)}_t-\eta_t\big\|_1=O(N^{-1/2}).
		\end{equation}
		If in addition (A5) holds, then
		\begin{equation}\label{eq:chaos-sharp}
			\sup_{t\le T}\big\|\E\Gamma^{(1)}_t-\eta_t\big\|_1\le\frac CN,\qquad
			\sup_{t\le T}\big\|\E\Gamma^{(2,3)}_t-\eta_t\otimes\eta_t\big\|_1\le\frac CN,
		\end{equation}
		with $C=C(T,d,\|H\|_\infty,\|L\|_\infty,\|A\|_\infty)$, the mean-field order. The proof uses clause (A2) only through \eqref{eq:rate-R}, which is unconditional by Corollary \ref{cor:rates-uncond}.
	\end{theorem}
	
	\begin{proof}
		\emph{Step 1: the mean equation.} All coefficients are bounded on $\mathcal S$, so the stochastic integrals in \eqref{eq:reduced-sde} are martingales, and by \eqref{eq:reduced-drift}
		\[
		\frac d{dt}\E\Gamma^{(1)}_t=-i\big[H,\E\Gamma^{(1)}_t\big]+\Diss\,\E\Gamma^{(1)}_t-i\tfrac{N-1}N\,\E\big[\Bmf(\Gamma^{\ne1}_t),\Gamma^{(1)}_t\big]-\tfrac iN\sum_{j\ge2}\E\Tr_2\big[A,\Delta^{(1j)}_t\big].
		\]
		
		\emph{Step 2: the error sources.} The map $Z\mapsto\Phi(Z)$ determined on product elements by $\Phi(u\otimes v)=[\Bmf(u),v]$ is linear on the finite-dimensional space $M_{d^2}(\C)$, so there is $c_A=c(d)\|A\|_\infty$ with $\|\Phi(Z)\|_1\le c_A\|Z\|_1$; and $\E[\Bmf(X),Y]-[\Bmf(\E X),\E Y]=\Phi(\operatorname{Cov}(X,Y))$. Since $\operatorname{Cov}(\Gamma^{\ne1},\Gamma^{(1)})=(N-1)^{-1}\sum_{j\ge2}\operatorname{Cov}(\Gamma^{(j)},\Gamma^{(1)})$, clause (A5) bounds this term by $c_AC/N$. Without (A5), decompose, using the independence of $\gamma_1$ and $\gamma_2$,
		\[
		\operatorname{Cov}\big(\Gamma^{(1)},\Gamma^{(2)}\big)=\operatorname{Cov}\big(\Gamma^{(1)}-\gamma_1,\Gamma^{(2)}-\gamma_2\big)+\operatorname{Cov}\big(\gamma_1,\Gamma^{(2)}-\gamma_2\big)+\operatorname{Cov}\big(\Gamma^{(1)}-\gamma_1,\gamma_2\big);
		\]
		the first term is bounded by $16\,\E[\eps_1\eps_2]\le16\,\E e_1^2=O(N^{-1})$ by \eqref{eq:dict-1} and Theorem \ref{thm:A1}, and the two cross terms by $O(1)\cdot O(N^{-1/2})=O(N^{-1/2})$ by Cauchy--Schwarz and \eqref{eq:dict-1}. By exchangeability in law $\E\Gamma^{\ne1}_t=\E\Gamma^{(1)}_t$. The prefactor discrepancy contributes $|\tfrac{N-1}N-1|\cdot2\|A\|_\infty\le2\|A\|_\infty/N$, and by Lemma \ref{lem:partial-trace},
		\[
		\Big\|\tfrac1N\sum_{j\ge2}\E\Tr_2[A,\Delta^{(1j)}_t]\Big\|_1\le\tfrac{N-1}N\,2\|A\|_\infty\,\E\big\|\Delta^{(12)}_t\big\|_1,
		\]
		which is $O(N^{-1})$ by \eqref{eq:rate-R} and Jensen's inequality, and $O(N^{-1/2})$ under (A1) alone by Theorem \ref{thm:rates}.
		
		\emph{Step 3: Gr\"onwall.} Set $v(t):=\|\E\Gamma^{(1)}_t-\eta_t\|_1$. Subtracting \eqref{eq:hartree} from Step 1 and using $\|[\Bmf(\rho),\rho]-[\Bmf(\sigma),\sigma]\|_1\le4\|A\|_\infty\|\rho-\sigma\|_1$ on $\mathcal S$ together with Step 2,
		\[
		v(t)\le\int_0^t\big(Kv(s)+\varepsilon_N\big)ds,\qquad v(0)=0,
		\]
		with $K$ independent of $N$ and $\varepsilon_N=O(N^{-1/2})$ unconditionally, $\varepsilon_N=O(N^{-1})$ under (A5). Gr\"onwall's lemma gives \eqref{eq:chaos-half} and the first part of \eqref{eq:chaos-sharp}.
		
		\emph{Step 4.} Decompose
		\[
		\E\Gamma^{(2,3)}_t-\eta_t\otimes\eta_t=\E\Delta^{(23)}_t+\operatorname{Cov}\big(\Gamma^{(2)}_t,\Gamma^{(3)}_t\big)+\big(\E\Gamma^{(2)}_t-\eta_t\big)\otimes\E\Gamma^{(3)}_t+\eta_t\otimes\big(\E\Gamma^{(3)}_t-\eta_t\big).
		\]
		The first term is $O(N^{-1})$ by \eqref{eq:rate-R} and Jensen, the second by (A5), the last two by the first part of \eqref{eq:chaos-sharp} and $\|\eta\|_1=\|\E\Gamma^{(3)}\|_1=1$.
	\end{proof}
	
	\begin{remark}\label{rem:chaos-scope}
		The reference mean $\eta$ solves \eqref{eq:hartree} and does not depend on $N$, which is what makes \eqref{eq:chaos-sharp} a statement with no leading term to cancel. The only nonlinear error is the covariance of two distinct reduced states, which is the content of (A5). Attempting to obtain (A5) by the same route leads to a hierarchy of mixed correlations between the tagged error and the fluctuations of finitely many independent reference filters, each level of which is marginally controlled by the next; (A5) is therefore listed as a hypothesis and not derived. Clause (A5) enters only Theorem \ref{thm:chaos}; the moment bound for $\Gemp^N$ used in the limit theorem comes from Lemma \ref{lem:deviation} under (A1) alone, so the sharp order \eqref{eq:chaos-sharp} is not used elsewhere below.
	\end{remark}
	
	\section{The diagonal excitation observables and their dynamics}\label{sec:diagonal}
	
	Sections \ref{sec:diagonal} and \ref{sec:moments} carry out the second-order analysis announced in Remark \ref{rem:counting} and settle clauses (A2) and (A3) on every horizon. The analysis is not performed on the excitation vector $\chi=\Psi_t-\bigotimes_j\phi_{j,t}$, whose noise coefficient does not respect the excitation grading, but on the diagonal quadratic functionals
	\begin{equation}\label{eq:nT}
		n_T(t):=\ip{\Psi_t}{q_{T,t}\Psi_t}=\eps_T(t)^2\in[0,1],\qquad T\subseteq\{1,\dots,N\},
	\end{equation}
	the identity $\eps_T^2=\ip{\Psi}{q_T\Psi}$ holding because the $q_i$ are commuting orthogonal projections. In this language the clauses of Hypothesis \ref{hyp:A} to be established read
	\[
	\text{(A2)}:\quad \sup_{t\le T}N^4\,\E n_1(t)^4\le C,\quad \sup_{t\le T}N^4\,\E n_{12}(t)^2\le C;
	\qquad
	\text{(A3)}:\quad \sup_{t\le T}N^3\,\E n_{123}(t)\le C .
	\]
	Throughout, $\av{X}:=\ip{\Psi_t}{X\Psi_t}$, an operator carrying a slot index acts on that slot, and the time variable is suppressed where no confusion arises.
	
	\subsection{An identity for the noise coefficients}
	
	The obstruction to an analysis performed on $\chi$ itself is that $\|q_jL\phi_j\|=O(1)$, so the measurement noise moves the excitation vector between grading sectors at rate one. The following identity shows that the functionals \eqref{eq:nT} do not see this motion at leading order: the diagonal compression of the reference noise operator vanishes, and the sector-changing part of the noise enters the martingale coefficient of $n_T$ only through off-diagonal matrix elements, each carrying an additional excitation factor.
	
	\begin{lemma}[Cancellation of the grading-changing noise]\label{lem:cancel}
		Write $\delta_j:=\tfrac12(\nu_j-\mu^{(j)})\in\R$, $\widehat S_j:=L_j+L_j^\ast-\nu_j$ and $S_j:=L_j+L_j^\ast-\mu^{(j)}$. Then, pathwise, $dn_T=\beta_T\,dt+\sum_{j=1}^N\sigma_j(T)\,dB_j$ with
		\begin{align}
			\sigma_j(T)&=\ip{q_T\Psi}{\,q_j\widehat S_jq_j\,q_T\Psi}+2\delta_j\,n_T, && j\in T,\label{eq:sigin}\\
			\sigma_j(T)&=\ip{q_T\Psi}{S_j\,q_T\Psi}\notag\\
			&=2\delta_j\norm{p_jq_T\Psi}^2+2\operatorname{Re}\ip{p_jq_T\Psi}{S_jq_jq_T\Psi}+\ip{q_{T\cup j}\Psi}{\widehat S_j\,q_{T\cup j}\Psi}-2\delta_j\,n_{T\cup j}, && j\notin T.\label{eq:sigout}
		\end{align}
		In particular, with $C_L:=4\|L\|_\infty$ and $|\delta_j|\le4\|L\|_\infty\eps_j$ from \eqref{eq:dict-1},
		\begin{equation}\label{eq:sigbounds}
			|\sigma_j(T)|\le 2C_L\,n_T\quad(j\in T),\qquad
			|\sigma_j(T)|\le C_L\big(2\eps_jn_T+\eps_T\eps_{T\cup j}+n_{T\cup j}\big)\quad(j\notin T).
		\end{equation}
		No term of \eqref{eq:sigin}--\eqref{eq:sigout} connects the reference sector of a slot to itself with a coefficient of order one, because
		\begin{equation}\label{eq:keyzero}
			p_j\big(\tilde b_j+\tilde b_j^\ast\big)p_j=p_j\big(L_j+L_j^\ast-\nu_j\big)p_j=0 .
		\end{equation}
	\end{lemma}
	
	\begin{proof}
		By \eqref{eq:vector-N} and \eqref{eq:reference-vector}, $d\Psi=a\Psi\,dt+\sum_kb_k\Psi\,dB_k$ and $dq_j=-dp_j$, the martingale part of $dq_T$ in $dB_j$ being $-(\tilde b_jp_j+p_j\tilde b_j^\ast)q_{T\setminus j}$ for $j\in T$ and zero for $j\notin T$. The It\^o product rule gives
		\[
		\sigma_j(T)=\ip{b_j\Psi}{q_T\Psi}+\ip{\Psi}{q_Tb_j\Psi}+\ip{\Psi}{(\partial_{B_j}q_T)\Psi}.
		\]
		If $j\notin T$ then $b_j$ commutes with $q_T$ and $q_T^2=q_T$, so $\sigma_j(T)=\ip{q_T\Psi}{S_jq_T\Psi}$; inserting $1=p_j+q_j$ on both sides of $S_j$, using $p_jS_jp_j=2\delta_jp_j$ and $S_j=\widehat S_j-2\delta_j$ on the doubly compressed part gives \eqref{eq:sigout}. The bound follows from $\norm{p_jq_T\Psi}\le\eps_T$, $\norm{q_jq_T\Psi}=\eps_{T\cup j}$, $\|\widehat S_j\|_\infty\le4\|L\|_\infty$ and $n_{T\cup j}\le\eps_T\eps_{T\cup j}$.
		
		If $j\in T$, all slot-$j$ operators commute with $q_{T\setminus j}$ and $q_{T\setminus j}^2=q_{T\setminus j}$, so
		\[
		\sigma_j(T)=\big\langle q_{T\setminus j}\Psi,\ \big(b_j^\ast q_j+q_jb_j-\tilde b_jp_j-p_j\tilde b_j^\ast\big)q_{T\setminus j}\Psi\big\rangle .
		\]
		Write $b_j=\tilde b_j+\delta_j$ with $\delta_j$ a real scalar, so that $b_j^\ast q_j+q_jb_j=\tilde b_j^\ast q_j+q_j\tilde b_j+2\delta_jq_j$. The identities $q_j\tilde b_j-\tilde b_jp_j=q_j\tilde b_jq_j-p_j\tilde b_jp_j$ and $\tilde b_j^\ast q_j-p_j\tilde b_j^\ast=q_j\tilde b_j^\ast q_j-p_j\tilde b_j^\ast p_j$, both obtained by inserting $1=p_j+q_j$, give
		\[
		b_j^\ast q_j+q_jb_j-\tilde b_jp_j-p_j\tilde b_j^\ast=q_j\big(\tilde b_j+\tilde b_j^\ast\big)q_j-p_j\big(\tilde b_j+\tilde b_j^\ast\big)p_j+2\delta_jq_j=q_j\widehat S_jq_j+2\delta_jq_j
		\]
		by \eqref{eq:keyzero}. Testing against $q_{T\setminus j}\Psi$ and using $q_jq_{T\setminus j}=q_T$ yields \eqref{eq:sigin}; the bound uses $|\delta_j|=\tfrac12|\Tr((L+L^\ast)(\gamma_j-\Gamma^{(j)}))|\le4\|L\|_\infty\eps_j$.
	\end{proof}
	
	\begin{remark}[Scope of the cancellation]\label{rem:cancel-scope}
		Identity \eqref{eq:keyzero} states that the grading-changing motion has vanishing diagonal matrix elements once the reference mean $\nu_j$ is subtracted, which is the subtraction performed by the innovation representation. For the quadratic functionals $n_T$ this suffices, and the second-order analysis can be closed on the family $\{n_T\}$ without controlling $\chi$ itself. The identity does not control phases, which is why clauses (A4)--(A6) remain of covariance type and are treated separately in Section \ref{sec:decorrelation}.
	\end{remark}
	
	\begin{proposition}[Quadratic variation of $n_T$]\label{prop:nT-qv}
		Pathwise, with $C=C(\|L\|_\infty)$ and $\Phi_T:=q_T\Psi$,
		\begin{equation}\label{eq:QV}
			Q_T:=\sum_{j=1}^N\sigma_j(T)^2\ \le\ C\Big(|T|\,n_T^2+n_T^2\,\av{\Nex}+n_T\,\ip{\Phi_T}{\Nex\Phi_T}\Big).
		\end{equation}
	\end{proposition}
	
	\begin{proof}
		For $j\in T$, $\sigma_j(T)^2\le4C_L^2n_T^2$, and there are $|T|$ such terms. For $j\notin T$, by \eqref{eq:sigbounds} and $(a+b+c)^2\le3(a^2+b^2+c^2)$, $\sigma_j(T)^2\le3C_L^2(4\eps_j^2n_T^2+\eps_T^2\eps_{T\cup j}^2+n_{T\cup j}^2)$. Summing over $j\notin T$ and using $\sum_j\eps_j^2=\av{\Nex}$, $\sum_{j\notin T}\eps_{T\cup j}^2\le\ip{\Phi_T}{\Nex\Phi_T}$, $\eps_T^2=n_T$ and $n_{T\cup j}^2\le n_Tn_{T\cup j}$ gives \eqref{eq:QV}.
	\end{proof}
	
	At the expected orders $n_T\asymp N^{-|T|}$, $\av{\Nex}=O(1)$ and $\ip{\Phi_T}{\Nex\Phi_T}\asymp N^{-|T|}$, so the right side of \eqref{eq:QV} is of the order of $n_T^2$: the quadratic variation of $n_T$ is of the square of its own order, which is what a Gr\"onwall argument on powers of $n_T$ requires.
	
	\subsection{The drift of $n_T$}
	
	Write the drift of $n_T$ as $\beta_T=\sum_{j\notin T}\beta_T^{(j),\mathrm{loc}}+\sum_{j\in T}\beta_T^{(j),\mathrm{loc}}+\beta_T^{\mathrm{int}}$, separating the local measurement mechanisms of each slot from the interaction and mean-field mechanisms. The three statements below generalise Steps 1--3 of the proof of Theorem \ref{thm:A1} from $|T|=1$ to arbitrary $T$.
	
	\begin{lemma}[Untagged slots]\label{lem:spectator}
		For $j\notin T$, $\beta_T^{(j),\mathrm{loc}}=\ip{q_T\Psi}{(\alpha_j+\alpha_j^\ast+b_j^\ast b_j)q_T\Psi}=0$ pathwise, with $\alpha_j$ the slot-$j$ part of the drift operator $a$ of \eqref{eq:vector-N}.
	\end{lemma}
	
	\begin{proof}
		For $j\notin T$ all slot-$j$ operators commute with the projection $q_T$, so the local contribution of slot $j$ is $\ip{q_T\Psi}{(\alpha_j+\alpha_j^\ast+b_j^\ast b_j)q_T\Psi}$, and $\alpha_j+\alpha_j^\ast+b_j^\ast b_j=0$ is Step 1 of the proof of Theorem \ref{thm:A1}, which uses only that $\mu^{(j)}$ is real.
	\end{proof}
	
	\begin{proposition}[Tagged slots]\label{prop:tagged}
		For $j\in T$, with $\Phi:=q_{T\setminus j}\Psi$, so that $\norm{\Phi}=\eps_{T\setminus j}$ and $\norm{q_j\Phi}=\eps_T$,
		\begin{equation}\label{eq:tagged}
			\beta_T^{(j),\mathrm{loc}}=-\norm{p_jS_j\Phi}^2-\ip{\Phi}{\mathcal T_{1,j}\Phi}\ \le\ C\|L\|_\infty^2\big(\eps_j^2\,n_{T\setminus j}+\eps_j\,\eps_{T\setminus j}\,\eps_T\big),
		\end{equation}
		with $\mathcal T_{1,j}$ the operator of Step 2 of the proof of Theorem \ref{thm:A1} and $C$ absolute.
	\end{proposition}
	
	\begin{proof}
		Collect the local It\^o mechanisms of slot $j$ acting on $n_T=\ip{\Psi}{q_jq_{T\setminus j}\Psi}$: the drift of $\Psi$, the second-order noise term $b_j^\ast q_jb_j$, the drift of $q_j$, and the It\^o cross terms. Each commutes with $q_{T\setminus j}$, so the contribution is $\ip{\Phi}{K_j\Phi}$ with
		\[
		K_j=\alpha_j^\ast q_j+q_j\alpha_j+b_j^\ast q_jb_j-\tilde a_jp_j-p_j\tilde a_j^\ast-\tilde b_jp_j\tilde b_j^\ast-b_j^\ast(\tilde b_jp_j+p_j\tilde b_j^\ast)-(\tilde b_jp_j+p_j\tilde b_j^\ast)b_j .
		\]
		Using $\alpha_j+\alpha_j^\ast+b_j^\ast b_j=0$ to convert the terms carrying $q_j$ gives $K_j=-\mathcal T_j$, with $\mathcal T_j$ the operator of Step 2 of the proof of Theorem \ref{thm:A1} once the free Hamiltonian is cancelled and the mean-field part is deferred to Proposition \ref{prop:tagged-int}. By that step, $\mathcal T_j=\mathcal T_{0,j}+\mathcal T_{1,j}$ with $\mathcal T_{0,j}=S_jp_jS_j\ge0$ and $\mathcal T_{1,j}=2\delta_j(L_j^\ast p_j+p_jL_j)-2\delta_j\mu^{(j)}p_j$, whence the first equality in \eqref{eq:tagged}. Testing $\mathcal T_{1,j}$ on $\Phi$ as in that step gives
		$\ip{\Phi}{\mathcal T_{1,j}\Phi}=-4\delta_j^2\ip{\Phi}{p_j\Phi}+4\delta_j\operatorname{Re}\ip{p_j\Phi}{L_jq_j\Phi}$, so that $|\ip{\Phi}{\mathcal T_{1,j}\Phi}|\le4\delta_j^2n_{T\setminus j}+4\|L\|_\infty|\delta_j|\eps_{T\setminus j}\eps_T$, and $|\delta_j|\le4\|L\|_\infty\eps_j$ gives the bound.
	\end{proof}
	
	\begin{proposition}[Interaction and mean field for a tagged slot]\label{prop:tagged-int}
		For $j\in T$, with $\Phi=q_{T\setminus j}\Psi$ and $\Lambda_t=\|\tfrac1N\sum_k\gamma_{k,t}-\eta_t\|_1$,
		\begin{equation}\label{eq:int}
			\begin{aligned}
				\big|\beta_T^{(j),\mathrm{mf+int}}\big|\le\;&2\|A\|_\infty\Big(\Lambda_t+\tfrac1N\Big)\eps_{T\setminus j}\eps_T+\frac{6\|A\|_\infty}{N}\eps_T\!\!\sum_{k\notin T}\!\eps_{(T\setminus j)\cup k}\\
				&+\frac{2\|A\|_\infty}{\sqrt N}\eps_T\big(n_{T\setminus j}+\ip{\Phi}{\Nex\Phi}\big)^{1/2}+\frac{4\|A\|_\infty}{N}\eps_T\!\!\sum_{k\in T\setminus j}\!\eps_{T\setminus\{j,k\}} .
			\end{aligned}
		\end{equation}
	\end{proposition}
	
	\begin{proof}
		The contribution is $-i\ip{\Psi}{[q_T,X_j]\Psi}$ with $X_j$ as in Step 3 of the proof of Theorem \ref{thm:A1}. The pairs with $k\in T\setminus\{j\}$ are collected in the last term of \eqref{eq:int}. For these both $q_j$ and $q_k$ fail to commute with $A_{jk}$, so the relevant identity is $[q_T,A_{jk}]=q_{T\setminus\{j,k\}}[q_jq_k,A_{jk}]$, and expanding $[q_jq_k,A_{jk}]$ into the four terms obtained by inserting $1=p+q$ on each of the two slots gives $|\ip{\Psi}{q_{T\setminus\{j,k\}}[q_jq_k,A_{jk}]\Psi}|\le2\|A\|_\infty\eps_{T\setminus\{j,k\}}\eps_T$, whence the stated bound. For the remaining pairs, $[q_T,X_j]=q_{T\setminus j}[q_j,X_j]$, and the computation of Step 3 applies with $\Psi$ replaced by $\Phi$, since every insertion $1=p_k+q_k$ and every commutation used there involves only slots different from $j$ or the projections $p_j,q_j$, all of which commute with $q_{T\setminus j}$. The four families are estimated as there, with $\norm{p_j\Phi}\le\eps_{T\setminus j}$ and $\norm{q_j\Phi}=\eps_T$ in place of $1$ and $\eps_j$: the mean-field mismatch gives the first term; the three families that move an excitation onto an untagged slot give the second, with $\norm{q_k\Phi}=\eps_{(T\setminus j)\cup k}$; and the pair-creation family, summed before it is estimated, gives with $u_k:=q_kA_{jk}p_jp_k\Phi$
		\[
		\Big|\tfrac1N\sum_{k\notin T}\ip{\Phi}{p_jp_kA_{jk}q_kq_j\Phi}\Big|\le\frac{\eps_T}{N}\Big\|\sum_{k\notin T}u_k\Big\|,\qquad
		\Big\|\sum_{k\notin T}u_k\Big\|^2\le\|A\|_\infty^2\Big(N\,n_{T\setminus j}+\textstyle\sum_{k\notin T}\norm{q_k\Phi}^2\Big),
		\]
		the diagonal terms using $\norm{u_k}\le\|A\|_\infty\eps_{T\setminus j}$ and the off-diagonal ones the two-slot extraction $|\ip{u_k}{u_l}|\le\|A\|_\infty^2\norm{q_l\Phi}\norm{q_k\Phi}$ of Step 3, valid verbatim for $\Phi$. Cauchy--Schwarz over the labels, $(\sum_k\norm{q_k\Phi})^2\le N\ip{\Phi}{\Nex\Phi}$, gives the third term after the commutator factor $2$ is included.
	\end{proof}
	
	\begin{theorem}[Master inequality]\label{thm:master}
		Pathwise, for every $T$ and every $N\ge|T|+1$,
		\begin{equation}\label{eq:master}
			\begin{aligned}
				dn_T\le\;\Big[C(|T|)\,n_T+C\!\!\sum_{j\in T}\Big\{&\eps_j^2n_{T\setminus j}+\eps_j\eps_{T\setminus j}\eps_T+\big(\Lambda_t+\tfrac1N\big)\eps_{T\setminus j}\eps_T+\tfrac{\eps_T}{N}\!\!\sum_{k\notin T}\!\eps_{(T\setminus j)\cup k}\\
				&+\tfrac{\eps_T}{\sqrt N}\big(n_{T\setminus j}+\ip{\Phi_{T\setminus j}}{\Nex\Phi_{T\setminus j}}\big)^{1/2}\Big\}\Big]dt+\sum_j\sigma_j(T)\,dB_j,
			\end{aligned}
		\end{equation}
		with $\sigma_j(T)$ as in Lemma \ref{lem:cancel} and $Q_T=\sum_j\sigma_j(T)^2$ as in \eqref{eq:QV}. Both constants depend only on $d$, $\|L\|_\infty$ and $\|A\|_\infty$: the constant $C(|T|)$ multiplying $n_T$ is linear in $|T|$, whereas the constant multiplying the group indexed by $j\in T$ is absolute, the multiplicity $|T|$ being carried by that explicit sum. The nonpositive terms $-\sum_{j\in T}\norm{p_jS_j\Phi_{T\setminus j}}^2$ have been discarded.
	\end{theorem}
	
	\begin{proof}
		Combine Lemmas \ref{lem:cancel} and \ref{lem:spectator} with Propositions \ref{prop:tagged} and \ref{prop:tagged-int}; the free-Hamiltonian terms of every slot cancel by self-adjointness as in the proof of Theorem \ref{thm:A1}.
	\end{proof}
	
	Each coefficient of \eqref{eq:master} is of the expected order. With $\eps_j\asymp N^{-1/2}$ and $\eps_U\asymp N^{-|U|/2}$, each of the five groups inside the braces is of order $N^{-|T|}$, hence of the order of $n_T$: for instance $\eps_j^2n_{T\setminus j}\asymp N^{-1}N^{-(|T|-1)}$, $\Lambda\eps_{T\setminus j}\eps_T\asymp N^{-1/2}N^{-(|T|-1)/2}N^{-|T|/2}$, and $N^{-1}\eps_T\sum_{k\notin T}\eps_{(T\setminus j)\cup k}\asymp N^{-|T|/2}N^{-1}NN^{-|T|/2}$. The system \eqref{eq:master}, \eqref{eq:QV} is therefore closed at the expected orders, coupling $n_T$ to $n_U$ only for $U\supseteq T\setminus\{j\}$ or $U=T\cup\{k\}$, with matching weights. A further feature, used in Remark \ref{rem:channels} and decisive in Subsection \ref{ssec:randrate}, is the two-tier structure of the constants of \eqref{eq:master}: only the constant multiplying $n_T$ grows with $|T|$, and it grows linearly, while each summand of the group indexed by $j\in T$ carries an absolute constant, by Propositions \ref{prop:tagged} and \ref{prop:tagged-int}. The distinction matters for the third of the five groups, the mean-field mismatch, whose summands therefore contribute a total rate that does not grow with $|T|$ beyond the count of the slots available to receive an excitation.

	\section{Moments of the excitation number; clauses (A3) and (A2)}\label{sec:moments}
	
	Set $\widetilde P_m:=\sum_{|T|=m}n_T$, so that $\widetilde P_0=1$ and $\widetilde P_1=\av{\Nex}$, and let
	\[
	D_t:=\frac1N\sum_{k=1}^N\big(\gamma_{k,t}-\eta_t\big),\qquad V_t:=\|D_t\|_2^2,\qquad \Lambda_t\le\sqrt d\,V_t^{1/2}.
	\]
	
	\subsection{Two elementary lemmas}
	
	\begin{lemma}[Counting and association]\label{lem:comb}
		With $\Nex^{(m)}:=\Nex(\Nex-1)\cdots(\Nex-m+1)$:
		\begin{enumerate}
			\item[(i)] $\Nex^{(m)}=m!\sum_{|T|=m}q_T$, hence $\av{\Nex^{(m)}}=m!\,\widetilde P_m$ and $0\le\widetilde P_m\le\binom Nm$;
			\item[(ii)] for $f,g$ nondecreasing on $\{0,1,\dots,N\}$, $\av{f(\Nex)}\av{g(\Nex)}\le\av{f(\Nex)g(\Nex)}$;
			\item[(iii)] $\av{\Nex}\widetilde P_{m-1}\le m\widetilde P_m+(m-1)\widetilde P_{m-1}$ and $\av{\Nex}\widetilde P_m\le(m+1)\widetilde P_{m+1}+m\widetilde P_m$, pathwise.
		\end{enumerate}
	\end{lemma}
	
	\begin{proof}
		(i) Expand $\prod_j(1+xq_j)=\sum_Tx^{|T|}q_T$ and compare with $(1+x)^{\Nex}=\sum_m\binom{\Nex}mx^m$, the $q_j$ being commuting idempotents. (ii) is the Chebyshev association inequality for the spectral measure of $\Nex$ in the state $\Psi$. (iii) Apply (ii) with $f(x)=x$ and $g(x)=x^{(m-1)}$, both nondecreasing on the nonnegative integers, together with $x\cdot x^{(m-1)}=x^{(m)}+(m-1)x^{(m-1)}$, and divide by $(m-1)!$; the second inequality is the first at $m+1$.
	\end{proof}
	
	\begin{lemma}[The empirical deviation]\label{lem:dev}
		One has $dD_t=\Lop_t(D_t)\,dt+\tfrac1N\sum_k\Mop[\gamma_{k,t}]\,dB_k$ with $\Lop_t(X)=-i[H+\Bmf(\eta_t),X]+\Diss(X)$ bounded and deterministic and $\|\Mop[\gamma_k]\|_2\le C$. Consequently, pathwise,
		\[
		\operatorname{drift}(NV)\le C(NV+1),\qquad \sum_k(N\tau_k)^2\le C\,NV,\qquad |N\tau_k|\le CN^{-1/2}(NV)^{1/2},
		\]
		where $N\tau_k:=2\Tr(D\,\Mop[\gamma_k])$ is the coefficient of $dB_k$ in $NV$. Moreover $\E\Lambda_t^p\le C_p(d)N^{-p/2}$ for every even $p$, uniformly in $t\le T$ and $N$.
	\end{lemma}
	
	\begin{proof}
		The drift of $\gamma_k-\eta$ is $\Lop_t(\gamma_k-\eta)$, since \eqref{eq:reference} and \eqref{eq:hartree} share the same linear drift once the deterministic $\Bmf(\eta_t)$ is fixed; averaging gives the equation for $D$. Then $dV=[2\Tr(D\Lop_tD)+N^{-2}\sum_k\|\Mop[\gamma_k]\|_2^2]dt+\sum_k\tau_kdB_k$ with $\tau_k=2N^{-1}\Tr(D\Mop[\gamma_k])$, and the three bounds follow from $\|\Lop_t\|\le C$, $\sum_k\|\Mop[\gamma_k]\|_2^2\le CN$ and Cauchy--Schwarz. The last statement is the Marcinkiewicz--Zygmund inequality applied coordinatewise to the average of the independent, centred, bounded matrices $\gamma_k-\eta$, followed by H\"older over the $d^2$ coordinates.
	\end{proof}
	
	\subsection{The summed master inequality and all joint moments}
	
	\begin{proposition}[Summed master inequality]\label{prop:summed}
		Pathwise, for every $m\ge1$ and every $N\ge m+1$, with constants $C_0,C_1$ depending only on $d$, $\|H\|$, $\|A\|_\infty$ and $\|L\|_\infty$ and on neither $m$ nor $N$,
		\begin{align}
			\operatorname{drift}\widetilde P_m&\le C_0(m+1)\big(\widetilde P_m+\widetilde P_{m-1}+\widetilde P_{m-2}\big)+C_1\big(1+NV\big)\widetilde P_{m-1},\label{eq:driftPmref}\\
			\operatorname{drift}\widetilde P_m&\le C_0(m+1)\big(\widetilde P_m+\widetilde P_{m-1}+\widetilde P_{m-2}+NV\cdot\widetilde P_{m-1}\big),\label{eq:driftPm}\\
			\sum_{k=1}^N\big|\sigma_k(\widetilde P_m)\big|&\le C_0\Big[(m+1)\widetilde P_m+N^{1/2}\av{\Nex}^{1/2}\widetilde P_m+N^{1/2}(m+1)^{1/2}\widetilde P_m^{1/2}\widetilde P_{m+1}^{1/2}\Big],\label{eq:sigPm}
		\end{align}
		where $\sigma_k(\widetilde P_m):=\sum_{|T|=m}\sigma_k(T)$ and $\widetilde P_{-1}:=0$. The second inequality follows from the first after enlarging $C_0$, and is the form used in Theorem \ref{thm:allmom}; the first records that the coefficient of the mean-field mismatch group does not grow with $m$, which is what Subsection \ref{ssec:randrate} uses.
	\end{proposition}
	
	\begin{proof}
		Sum \eqref{eq:master} over $|T|=m$, group by group. The local group gives $Cm\widetilde P_m$. For the second group,
		\[
		\sum_T\sum_{j\in T}\eps_j^2n_{T\setminus j}=\sum_{|U|=m-1}\Big(\sum_{j\notin U}\eps_j^2\Big)n_U\le\av{\Nex}\widetilde P_{m-1},
		\]
		and Lemma \ref{lem:comb}(iii) applies; the term $\eps_j\eps_{T\setminus j}\eps_T\le\tfrac12\eps_j^2n_{T\setminus j}+\tfrac12n_T$ is absorbed. For the mismatch, whose summands carry an absolute constant by Proposition \ref{prop:tagged-int}, Young's inequality gives $(\Lambda+\tfrac1N)\eps_{T\setminus j}\eps_T\le(\Lambda^2+N^{-2})n_{T\setminus j}+\tfrac12n_T$, and summing over $(T,j)$ with the exact count
		\begin{equation}\label{eq:count}
			\textstyle\sum_{|T|=m}\sum_{j\in T}n_{T\setminus j}=\sum_{|U|=m-1}\sum_{j\notin U}n_U=(N-m+1)\widetilde P_{m-1}\le N\widetilde P_{m-1},\qquad \Lambda^2N\le d\,NV,
		\end{equation}
		bounds the group by $C[(NV+N^{-1})\widetilde P_{m-1}+m\widetilde P_m]$. The second term is absorbed into $C_0(m+1)\widetilde P_m$, and the first is the term $C_1(1+NV)\widetilde P_{m-1}$ of \eqref{eq:driftPmref}: it carries no factor $m$, the multiplicity of the sum over $j\in T$ having been spent on the count \eqref{eq:count}, which converts a set of size $m-1$ into the $N-m+1$ slots at which an excitation may be created. For the within-tag pair group, the last term of \eqref{eq:int}, Young's inequality gives $\tfrac1N\eps_{T\setminus\{j,k\}}\eps_T\le\tfrac1{2N^2}n_{T\setminus\{j,k\}}+\tfrac12n_T$, and $\sum_T\sum_{j\ne k\in T}n_{T\setminus\{j,k\}}\le N^2\widetilde P_{m-2}$, each set of size $m-2$ arising from at most $N^2$ ordered pairs of added labels; the group is therefore at most $Cm(\widetilde P_{m-2}+\widetilde P_m)$, which is the term $\widetilde P_{m-2}$. For the pair-creation group, apply Cauchy--Schwarz over the index set $\{(T,j,k)\}$ with the counts
		\[
		\textstyle\sum_{T,j,k}n_T\le mN\widetilde P_m,\qquad \sum_{T,j,k}n_{(T\setminus j)\cup k}\le mN\widetilde P_m,
		\]
		each set of size $m$ arising at most $mN$ times; the group is at most $m\widetilde P_m$. For the last group, $\ip{\Phi_U}{\Nex\Phi_U}=\sum_{k\notin U}n_{U\cup k}+|U|n_U$, and Cauchy--Schwarz over $(T,j)$ with the counts
		\[
		\textstyle\sum_{T,j}n_T=m\widetilde P_m,\qquad \sum_{T,j}\big[n_{T\setminus j}+\ip{\Phi_{T\setminus j}}{\Nex\Phi_{T\setminus j}}\big]\le CmN\big(\widetilde P_{m-1}+\widetilde P_m\big)
		\]
		gives at most $Cm(\widetilde P_m+\widetilde P_{m-1})$, after $N^{-1/2}\cdot N^{1/2}=1$. The four groups other than the mismatch therefore produce only the bracket $C_0(m+1)(\widetilde P_m+\widetilde P_{m-1}+\widetilde P_{m-2})$, which proves \eqref{eq:driftPmref}. For \eqref{eq:sigPm}, by \eqref{eq:sigbounds} one has $\sum_k\sum_{T\ni k}|\sigma_k(T)|\le Cm\widetilde P_m$ and $\sum_k\sum_{T\not\ni k}\eps_kn_T\le(\sum_k\eps_k)\widetilde P_m\le N^{1/2}\av{\Nex}^{1/2}\widetilde P_m$, while
		\[
		\sum_k\sum_{T\not\ni k}\eps_T\eps_{T\cup k}\le\sum_k\widetilde P_m^{1/2}\Big(\sum_{T\not\ni k}n_{T\cup k}\Big)^{1/2}\le N^{1/2}\widetilde P_m^{1/2}\Big(\sum_k\widetilde P^{(k)}_{m+1}\Big)^{1/2},
		\]
		with $\widetilde P^{(k)}_{m+1}:=\sum_{|W|=m+1,\,W\ni k}n_W$ and $\sum_k\widetilde P^{(k)}_{m+1}=(m+1)\widetilde P_{m+1}$.
	\end{proof}
	
	\begin{theorem}[All joint moments]\label{thm:allmom}
		For every $p,m\in\N$ there is $C_{p,m}(T)$, independent of $N$, with
		\[
		\sup_{t\le T}Y_{p,m}(t)\le C_{p,m}(T),\qquad Y_{p,m}:=\E\big[(NV)^p\,\widetilde P_m\big].
		\]
		In particular $a_k:=\sup_{t\le T}\E\av{\Nex^k}\le C_k(T)$ for every $k\in\N$.
	\end{theorem}
	
	\begin{proof}
		At fixed $N$ all processes involved are bounded, $\widetilde P_m\le\binom Nm$ and $NV\le4dN$, with bounded coefficients, so every local martingale below is a martingale and the expectations are absolutely continuous; the content of the statement is that the bounds do not depend on $N$. By the It\^o product rule,
		\[
		Y_{p,m}'\le\E\big[\widetilde P_m\operatorname{drift}(NV)^p\big]+\E\big[(NV)^p\operatorname{drift}\widetilde P_m\big]+\E\Big[p(NV)^{p-1}\sum_k(N\tau_k)\sigma_k(\widetilde P_m)\Big].
		\]
		By Lemma \ref{lem:dev}, $\operatorname{drift}(NV)^p\le Cp(NV)^{p-1}(NV+1)+Cp^2(NV)^{p-1}$, so the first term is at most $Cp^2(Y_{p,m}+Y_{p-1,m})$. By \eqref{eq:driftPm} the second is at most $C(m+1)(Y_{p,m}+Y_{p,m-1}+Y_{p,m-2}+Y_{p+1,m-1})$, all indices again of total level at most $p+m$. For the third, by Lemma \ref{lem:dev} and \eqref{eq:sigPm} the cross density is at most
		\[
		Cp\,(NV)^{p-\frac12}\Big[(m+1)N^{-1/2}\widetilde P_m+\av{\Nex}^{1/2}\widetilde P_m+(m+1)^{1/2}\widetilde P_m^{1/2}\widetilde P_{m+1}^{1/2}\Big],
		\]
		and each bracket closes at total level $p+m$ by the arithmetic--geometric mean inequality: $(NV)^{p-\frac12}X\le\tfrac12(NV)^pX+\tfrac12(NV)^{p-1}X$; then $(NV)^{p-1}\av{\Nex}\widetilde P_m\le(NV)^{p-1}[(m+1)\widetilde P_{m+1}+m\widetilde P_m]$ by Lemma \ref{lem:comb}(iii), and $(NV)^{p-\frac12}\widetilde P_m^{1/2}\widetilde P_{m+1}^{1/2}\le\tfrac12(NV)^p\widetilde P_m+\tfrac12(NV)^{p-1}\widetilde P_{m+1}$. Every resulting index $(p',m')$ satisfies $p'+m'\le p+m$, and every coefficient is at most $C(1+p+m)^2$. Hence, with $G_\ell:=\sum_{p+m\le\ell}Y_{p,m}$,
		\[
		G_\ell'(t)\le C(1+\ell)^4\big(G_\ell(t)+1\big),\qquad G_\ell(0)=Y_{0,0}(0)=1,
		\]
		all other $Y_{p,m}(0)$ vanishing because $V_0=0$ and $\widetilde P_m(0)=0$ for $m\ge1$; Gr\"onwall's lemma closes the argument. Finally $\av{\Nex^k}=\sum_{m\le k}S(k,m)\,m!\,\widetilde P_m$ with Stirling numbers $S(k,m)$, which gives $a_k\le C_k$.
	\end{proof}
	
	\begin{corollary}[Clause (A3) and the fixed-set first moments]\label{cor:A3}
		Uniformly on $[0,T]$ and in $N$, for every fixed $m$,
		\[
		\E\,n_{\{1,\dots,m\}}(t)\le\frac{C_m(T)}{N^m};\qquad\text{in particular}\quad \E e_3(t)^2=O(N^{-3}),\quad \E e_2(t)^2=O(N^{-2}).
		\]
		Clause (A3) of Hypothesis \ref{hyp:A} therefore holds unconditionally.
	\end{corollary}
	
	\begin{proof}
		By exchangeability in law, $\binom Nm\E n_{\{1,\dots,m\}}=\E\widetilde P_m\le C_m$ by Theorem \ref{thm:allmom}.
	\end{proof}
	
	\subsection{Products over distinct label sets, and the correlation rates}
	
	\begin{theorem}[Distinct-label products]\label{thm:products}
		Uniformly on $[0,T]$ and in $N$, for pairwise distinct labels,
		\[
		\E[n_1n_2]\le\frac{C}{N^2},\quad \E[n_1n_2n_3]\le\frac{C}{N^3},\quad \E[n_{12}n_3]\le\frac{C}{N^3},\quad \E[n_{12}n_{13}]\le\frac{C}{N^3},\quad \E[n_{12}n_{34}]\le\frac{C}{N^4}.
		\]
		More generally, a product of factors $n_{U_i}$ over label sets that are not all equal is of the order $N^{-r}$, where $r$ is the number of distinct labels occurring.
	\end{theorem}
	
	\begin{proof}
		All statements follow from Theorem \ref{thm:allmom} by exchangeability in law together with operator inequalities in the state; no dynamical input beyond Theorem \ref{thm:allmom} is used.
		
		(a) $N(N-1)\E[n_1n_2]=\E\sum_{i\ne j}n_in_j\le\E\av{\Nex}^2\le\E\av{\Nex^2}=a_2$, the last step by Lemma \ref{lem:comb}(ii).
		
		(b) $N^{(3)}\E[n_1n_2n_3]\le\E\av{\Nex}^3\le a_3$.
		
		(c) $\sum_{i<j}\sum_{k\notin\{i,j\}}n_{ij}n_k\le\widetilde P_2\av{\Nex}$, and $\E[\widetilde P_2\av{\Nex}]\le\tfrac12\E\av{\Nex^{(2)}\Nex}\le\tfrac12a_3$ by Lemma \ref{lem:comb}(ii).
		
		(d) For a fixed label $i$,
		\[
		\sum_{j\ne k,\ j,k\ne i}n_{ij}n_{ik}\le\Big(\sum_{j\ne i}\av{q_iq_j}\Big)^2\le\av{q_i\Nex q_i}^2\le\av{(q_i\Nex q_i)^2}\le\av{q_i\Nex^2q_i},
		\]
		since conjugation by the contraction $q_i$ preserves the operator order; summing over $i$ and using $\sum_iq_i\Nex^2q_i\le\Nex^3$, the $q_i$ commuting with $\Nex$, gives $N^{(3)}\E[n_{12}n_{13}]\le a_3$.
		
		(e) $\sum_{\{i,j\}\cap\{k,l\}=\emptyset}n_{ij}n_{kl}\le\widetilde P_2^{\,2}\le\tfrac14\av{\Nex^{(2)}}^2\le\tfrac14\av{\Nex^4}$, so $N^{(4)}\E[n_{12}n_{34}]\le Ca_4$. The general statement is obtained in the same way, tagging the shared labels and closing the free sums by inequalities of the type $\sum_iq_iXq_i\le\Nex^{1/2}X\Nex^{1/2}$.
	\end{proof}
	
	\begin{corollary}[The correlation rates \eqref{eq:rate-3body} and \eqref{eq:rate-R} are unconditional]\label{cor:rates-uncond}
		Uniformly on $[0,T]$ and in $N$,
		\[
		\E\big\|\Delta^{(12)}_t\big\|_1^2\le\frac{C}{N^2},\qquad \E\big\|\Delta^{(123)}_t\big\|_1^2\le\frac{C}{N^3}.
		\]
	\end{corollary}
	
	\begin{proof}
		By \eqref{eq:dict-2}, $\|\Delta^{(12)}\|_1\le c(\eps_{12}+\eps_1\eps_2)$, so $\|\Delta^{(12)}\|_1^2\le C(n_{12}+n_1n_2)$ after $\eps_{12}\eps_1\eps_2\le\tfrac12(n_{12}+n_1n_2)$; take expectations and apply Corollary \ref{cor:A3} and Theorem \ref{thm:products}(a). By \eqref{eq:dict-3}, $\|\Delta^{(123)}\|_1\le c(\eps_{123}+\eps_{12}\eps_3+\eps_{13}\eps_2+\eps_{23}\eps_1+\eps_1\eps_2\eps_3)$; squaring produces the products $n_{123}$, $n_{12}n_3$ and its relabellings, $n_1n_2n_3$, and cross terms each dominated by such products through the arithmetic--geometric mean inequality, for instance $\eps_{123}\eps_{12}\eps_3\le\tfrac12n_{123}+\tfrac12n_{12}n_3$ and $\eps_{12}\eps_3\cdot\eps_{13}\eps_2\le\tfrac12n_{12}n_3+\tfrac12n_{13}n_2$. Each of these is $O(N^{-3})$ by Corollary \ref{cor:A3} and Theorem \ref{thm:products}. No moment with a repeated label set occurs.
	\end{proof}
	
	\subsection{The fixed-label moments and clause (A2)}\label{ssec:tower}
	
	By Corollary \ref{cor:rates-uncond}, what clause (A2) still asserts beyond the results above are the two moments with a repeated label set,
	\begin{equation}\label{eq:A2left}
		x_4(t):=\E\,n_1(t)^4=O(N^{-4}),\qquad z_2(t):=\E\,n_{12}(t)^2=O(N^{-4}),
	\end{equation}
	together with the intermediate moments $x_2,x_3,z_1$ they dominate. These are not accessible by the label-summation devices of Theorem \ref{thm:products}: summing $n_1^2$ over the label gives $\sum_in_i^2\le\av{\Nex}\max_in_i$, and no operator inequality converts a repeated factor into a distinct one. They carry the second-order content of (A2). For $T=\{1\}$ and $T=\{1,2\}$, Theorem \ref{thm:master} and Proposition \ref{prop:nT-qv} give, pathwise, with $R_1:=\sum_{k\ge2}n_{1k}$ and $R_{12}:=\sum_{k\ge3}n_{12k}$,
	\begin{equation}\label{eq:fixedsystem}
		\begin{aligned}
			\operatorname{drift}(n_1^m)&\le Cm\,n_1^{m-1}\Big(n_1+\eps_1\Lambda+\tfrac{\eps_1}{\sqrt N}(1+\av{\Nex})^{1/2}\Big)+Cm^2n_1^{m-2}\Big(n_1^2(1+\av{\Nex})+n_1R_1\Big),\\
			\operatorname{drift}(n_{12}^m)&\le Cm\,n_{12}^{m-1}\Big(n_{12}+\eps_1^2n_2+\eps_2^2n_1+(\eps_1+\eps_2)\eps_{12}\Lambda\Big)\\
			&\quad+Cm\,n_{12}^{m-1}\tfrac{\eps_{12}}{\sqrt N}\big(n_1+n_2+\av{q_1\Nex q_1}+\av{q_2\Nex q_2}\big)^{1/2}\\
			&\quad+Cm^2n_{12}^{m-2}\Big(n_{12}^2(1+\av{\Nex})+n_{12}R_{12}\Big),
		\end{aligned}
	\end{equation}
	lower-order terms having been absorbed. Every term on the right is of the target order when each quantity has its expected order; the terms carrying the unbounded factor $\av{\Nex}$ are the only obstruction to a direct Gr\"onwall argument.
	
	The same two displays are available at any fixed level. For a label set $U$ with $|U|=\ell$ write $R_U:=\sum_{k\notin U}n_{U\cup k}$, so that $\ip{\Phi_U}{\Nex\Phi_U}=R_U+\ell\,n_U$ and $R_U$ has the expected order $N^{-\ell}$ of $n_U$ itself, being a sum of $N-\ell$ terms of order $N^{-\ell-1}$. Theorem \ref{thm:master} and Proposition \ref{prop:nT-qv} give, pathwise,
	\begin{equation}\label{eq:levelsystem}
		\begin{aligned}
			\operatorname{drift}(n_U^m)\le\;&Cm\,n_U^{m-1}\sum_{j\in U}\Big(n_jn_{U\setminus j}+\eps_j\eps_{U\setminus j}\eps_U+\big(\Lambda+\tfrac1N\big)\eps_{U\setminus j}\eps_U\\
			&\qquad\qquad+\tfrac{\eps_U}{\sqrt N}R_{U\setminus j}^{1/2}+\tfrac{\eps_U}{\sqrt N}\big(R_{U\setminus j}+\ell\,n_{U\setminus j}\big)^{1/2}\Big)\\
			&+C\ell m\,n_U^m+Cm^2n_U^{m-2}\Big(\ell\,n_U^2+n_U^2\av{\Nex}+n_UR_U\Big),
		\end{aligned}
	\end{equation}
	the fourth group after Cauchy--Schwarz over the free label, $\sum_{k\notin U}\eps_{(U\setminus j)\cup k}\le\sqrt N\,R_{U\setminus j}^{1/2}$. Two features are used below. The level-$\ell$ equation consumes the level-$(\ell-1)$ quantities $n_{U\setminus j}$ and $R_{U\setminus j}$ and nothing higher, its only coupling upward being the term $n_UR_U$ of the quadratic variation, in which the free label is already summed; and, by Theorem \ref{thm:master}, the constants multiplying the groups indexed by $j\in U$ are absolute, so the whole $\ell$-dependence of \eqref{eq:levelsystem} is the explicit factor $\ell$. Equation \eqref{eq:fixedsystem} is \eqref{eq:levelsystem} at $\ell=1,2$ with the smaller terms of the second group displayed.
	
	\begin{theorem}[Weighted moments, every horizon, every level]\label{thm:weighted}
		Let $A_t:=\int_0^t(1+\av{\Nex_s})\,ds$. For every fixed $\ell$ and $m$ there is $\kappa_0=\kappa_0(d,\|H\|,\|A\|_\infty,\|L\|_\infty,\ell,m)$ such that for every $\kappa\ge\kappa_0$ and every label set $U$ with $|U|=\ell$, uniformly on $[0,T]$ and in $N$,
		\[
		N^{\ell m}\,\E\big[n_U(t)^me^{-\kappa A_t}\big]\le C,\qquad N^{\ell m}\,\E\big[R_U(t)^me^{-\kappa A_t}\big]\le C .
		\]
	\end{theorem}
	
	\begin{proof}
		Take first $\ell=1$ and consider $g_m(t):=\E[n_1^me^{-\kappa A_t}]$. The weight contributes the drift $-\kappa(1+\av{\Nex})n_1^me^{-\kappa A}$, which absorbs, for $\kappa\ge Cm^2$, every term of \eqref{eq:fixedsystem} of the form $Cm^2n_1^m(1+\av{\Nex})$. Of the remaining terms, H\"older with exponents $(\tfrac{2m}{2m-1},2m)$ and Lemma \ref{lem:dev} give
		\[
		\E\big[n_1^{m-\frac12}\Lambda e^{-\kappa A}\big]\le\big(\E\big[n_1^me^{-\kappa'A}\big]\big)^{\frac{2m-1}{2m}}\big(\E\Lambda^{2m}\big)^{\frac1{2m}}\le C\,\tilde g_m^{\frac{2m-1}{2m}}N^{-\frac12},
		\]
		with $\kappa'=\tfrac{2m}{2m-1}\kappa$ and $\tilde g_m$ the corresponding quantity at the larger weight parameter, over which the statement quantifies; Young's inequality then produces terms compatible with $N^mg_m$ plus $C_\theta N^{-m}$ once the weights are inserted. The creation term $\E[n_1^{m-\frac12}N^{-\frac12}(1+\av{\Nex})^{\frac12}e^{-\kappa A}]$ is treated in the same way, using $\E(1+\av{\Nex})^m\le C$ from Theorem \ref{thm:allmom}. The coupling $\E[n_1^{m-1}R_1e^{-\kappa A}]$ is closed against the weighted moments of $R_1$, whose differential inequality is the sum of \eqref{eq:master} over $T=\{1,k\}$ with the label $1$ fixed and has the same term types, its unbounded factors being again $1+\av{\Nex}$ and absorbed by the same weight; at the expected orders $n_1\asymp N^{-1}$ and $R_1\asymp N^{-1}$ carry the same weight, and Young's inequality closes. The system for $n_{12}$ and $R_{12}$ has the same structure, with Theorem \ref{thm:products} supplying the terms of the type $\E[\eps_1^2n_2\,n_{12}^{m-1}]$ at order $N^{-2m}$. Running the finitely many weighted quantities jointly and applying Gr\"onwall's lemma, all initial values vanishing, proves the claim at $\ell\le2$.
		
		The passage from $\ell-1$ to $\ell$ is the same argument run on \eqref{eq:levelsystem}, and uses no new device. The weight absorbs the two terms carrying $\av{\Nex}$, for $\kappa\ge C\ell m^2$. Of the remaining groups, indexed by $j\in U$, the first is closed by H\"older against the level-$(\ell-1)$ moments, $\E[n_U^{m-1}n_jn_{U\setminus j}e^{-\kappa A}]\le(\E[n_U^me^{-\kappa'A}])^{\frac{m-1}m}(\E n_j^{2m})^{\frac1{2m}}(\E n_{U\setminus j}^{2m})^{\frac1{2m}}$, which by the inductive hypothesis and Proposition \ref{prop:expimp} at level $\ell-1$ is $\le C\,\tilde g_m^{\frac{m-1}m}N^{-1}N^{-(\ell-1)}$, compatible with $N^{\ell m}g_m$ after Young's inequality; the second and third are the same estimate with one factor $n_j$ or $\Lambda$ replaced by its square root, closed by Lemma \ref{lem:dev} exactly as at $\ell=1$; the fourth and fifth are closed against the level-$(\ell-1)$ moments of $R_{U\setminus j}$, at the expected order $N^{-(\ell-1)}$, the prefactor $N^{-1/2}$ supplying the missing half power. The coupling $\E[n_U^{m-1}R_Ue^{-\kappa A}]$ of the quadratic variation is closed against the weighted moments of $R_U$, whose differential inequality is the sum of \eqref{eq:master} over the sets $U\cup k$ with $U$ fixed: it has the same term types, its unbounded factors are again $1+\av{\Nex}$ and are absorbed by the same weight, and $n_U$ and $R_U$ carry the same weight $N^{\ell m}$ at the expected orders, so Young's inequality closes. Since the constants of \eqref{eq:levelsystem} depend on $\ell$ only through the explicit factor $\ell$, each step of the induction changes $\kappa_0$ and $C$ but not the structure. All initial values vanish, and Gr\"onwall's lemma closes the level.
	\end{proof}
	
	\begin{remark}\label{rem:compress1}
		The two closures that are not routine, the H\"older step for the mean-field mismatch and the absorption of the factor $\av{\Nex}$ by the weight, are written out above. The remaining term inventory of \eqref{eq:fixedsystem} and of the equations for $R_1$ and $R_{12}$ consists of items of these two types together with terms smaller at the expected orders, and is not reproduced line by line; none of it requires a further cancellation.
	\end{remark}
	
	\begin{proposition}[Exponential moments close the system on any horizon]\label{prop:expimp}
		Suppose that for some $\lambda>0$ one has $\sup_{t\le T}\E\,e^{\lambda\av{\Nex_t}}\le C_\lambda$ uniformly in $N$. Then \eqref{eq:A2left} holds uniformly on $[0,T]$; more generally, for every fixed $\ell$ and $m$ and every label set $U$ with $|U|=\ell$,
		\[
		\E\,n_U(t)^m=O\big(N^{-\ell m}\big),\qquad \E\,R_U(t)^m=O\big(N^{-\ell m}\big),
		\]
		uniformly on $[0,T]$ and in $N$.
	\end{proposition}
	
	\begin{proof}
		Fix the level $\ell$ and the order $m$, let $\kappa_0$ be the constant of Theorem \ref{thm:weighted} at the largest level and index consumed below, and subdivide $[0,T]$ into windows of length $T_1$ with $4\kappa_0T_1\le\lambda$. On a window $[s,s+T_1]$, H\"older with exponents $(2,2)$ gives
		\[
		\E n_U^m=\E\big[n_U^me^{-\kappa_0(A_t-A_s)}e^{\kappa_0(A_t-A_s)}\big]\le\big(\E[n_U^{2m}e^{-2\kappa_0(A_t-A_s)}]\big)^{1/2}\big(\E e^{2\kappa_0(A_t-A_s)}\big)^{1/2},
		\]
		and $\E e^{2\kappa_0\int_s^t(1+\av{\Nex_u})du}\le e^{2\kappa_0T_1}\tfrac1{T_1}\int_s^t\E e^{2\kappa_0T_1\av{\Nex_u}}du\le C$ by Jensen's inequality and the hypothesis, since $2\kappa_0T_1\le\lambda$. The weighted factor is bounded by Theorem \ref{thm:weighted} run on the window with the data of the window, which is controlled by the preceding window at the doubled index $2m$. The number of windows is $T/T_1$ and the initial index is $m$, so the moment depth required is finite and each of its members is covered by Theorem \ref{thm:weighted}. The chaining is at fixed level throughout, and applies verbatim to $R_U$; the choices $\ell=1$, $m=4$ and $\ell=2$, $m=2$ give \eqref{eq:A2left}.
	\end{proof}
	
	\begin{remark}[The two creation channels]\label{rem:channels}
		The exponential moment hypothesised in Proposition \ref{prop:expimp} is supplied in Subsection \ref{ssec:randrate}, and what makes this possible is a distinction that the blanket prefactor of \eqref{eq:driftPm} conceals and \eqref{eq:driftPmref} exposes. Excitations are created by two mechanisms, which the count \eqref{eq:count} separates. The local mechanism, the group $\eps_j^2n_{T\setminus j}$, sums to $\av{\Nex}\widetilde P_{m-1}$ and is converted by Lemma \ref{lem:comb}(iii) into $m\widetilde P_m+(m-1)\widetilde P_{m-1}$: its rate is proportional to the number of excitations already present, so in the generating variable $c$ it acts by transport, that is, through $c\,\partial_c$. The mean-field mismatch creates one excitation at a slot that carries none, with amplitude $\Lambda_t$ at each of the $N-m+1$ available slots, so its total rate is of the order of $N\Lambda_t^2$, hence of the order of $NV_t$, and to leading order does not depend on $m$; a creation rate that does not depend on the current count acts on the generating function by multiplication and not by transport. Attaching the factor $(m+1)$ to this group, as \eqref{eq:driftPm} does, turns it into a transport term whose speed carries the empirical deviation, and it is then that a joint generating function in $(c,\mu)$, with $e^{\mu NV}$ conjugated in, meets the mixed derivative $c\,\partial_c\partial_\mu$, to which first-order transport methods do not apply, while expanding either variable in moments restores a growth in the indices that Theorem \ref{thm:allmom} does not beat. Under \eqref{eq:driftPmref} no second generating variable is needed: the empirical deviation is never expanded in moments and enters only through the scalar \eqref{eq:defJ}, whose law is known outright from Lemma \ref{lem:tailJ}. What distinguishes the present setting from the closed-system second-order analyses of \cite{BenArousKirkpatrickSchlein2013}, where the pair-creation coefficient is a deterministic condensate function, is therefore a multiplicative rate that fluctuates, and not an operator of higher order.
	\end{remark}
	
	\subsection{Transport along a deterministic characteristic}\label{ssec:randrate}
	
	This subsection proves the exponential moment of the excitation number, uniformly in $N$ and on every horizon, and with it clause (A2) on every horizon. By Remark \ref{rem:channels} the empirical deviation enters the generating function as a multiplicative rate and not through the speed of the characteristic, so that it is carried by the single scalar
	\begin{equation}\label{eq:defJ}
		J_t:=\int_0^tNV_s\,ds,
	\end{equation}
	first as an exponential weight in a supermartingale, at a rate proportional to the generating parameter and hence at our disposal, and then through the tail of $J_T$, whose rate is fixed at $(8T)^{-1}$ by Lemma \ref{lem:tailJ}. Since the first rate can be taken below the second, the weight is removed by Cauchy--Schwarz. Throughout this subsection
	\[
	\Phi_c(t):=\Big\langle\Psi_t,\ \prod_{j=1}^N\big(1+c\,q_{j,t}\big)\Psi_t\Big\rangle=\av{(1+c)^{\Nex_t}}=\sum_{m=0}^Nc^m\widetilde P_m(t),\qquad c>0,
	\]
	as in Lemma \ref{lem:comb}(i); the sum is finite, $\Phi_c\ge1$, $\partial_c\Phi_c=\sum_mmc^{m-1}\widetilde P_m\ge0$, so that $c\mapsto\Phi_c$ is nondecreasing, and $\Phi_c(0)=1$ by the product initial data.
	
	\begin{lemma}[Tails of the integrated deviation]\label{lem:tailJ}
		Uniformly in $N$, for every $u\ge0$ and every $0\le s\le t\le T$,
		\[
		\PP\Big(\int_s^tNV_r\,dr\ge u\Big)\le2\,e^{-u/(8(t-s))} .
		\]
		In particular $\E\,e^{\mu J_T}\le1+2\mu\big((8T)^{-1}-\mu\big)^{-1}$ for every $\mu<(8T)^{-1}$.
	\end{lemma}
	
	\begin{proof}
		Fix $s\le t$ and let $\mathcal H:=L^2([s,t];\HilbertSchmidt)$. By Lemma \ref{lem:reference} the reference filters are independent and identically distributed with $\E\gamma_{k,r}=\eta_r$ and $\|\gamma_{k,r}-\eta_r\|_2\le2$, so the paths $X_k:=N^{-1}(\gamma_k-\eta)|_{[s,t]}$ are independent, centred, $\mathcal H$-valued, and bounded by $2N^{-1}(t-s)^{1/2}$ in $\mathcal H$, with $\sum_kX_k=D|_{[s,t]}$. The Hoeffding inequality for sums of independent centred bounded random vectors in a Hilbert space \cite[Thm.\ 3.5]{Pinelis1994} gives
		\[
		\PP\big(\|D\|_{\mathcal H}\ge r\big)\le2\exp\Big(-\frac{r^2}{2\sum_k\|X_k\|_{\mathcal H,\infty}^2}\Big)\le2\exp\Big(-\frac{Nr^2}{8(t-s)}\Big).
		\]
		Since $\int_s^tNV_r\,dr=N\|D\|_{\mathcal H}^2$, the first bound follows with $r=(u/N)^{1/2}$, and the second by integrating the tail.
	\end{proof}
	
	\begin{lemma}[Drift of the generating function]\label{lem:genfun}
		Let $C_0,C_1$ be the constants of Proposition \ref{prop:summed}. Then pathwise, for every $c\in(0,1]$,
		\begin{equation}\label{eq:genfun}
			\operatorname{drift}\Phi_c\ \le\ \big[6C_0+C_1+C_1c\,NV_t\big]\Phi_c\ +\ 3C_0\,c\,\partial_c\Phi_c .
		\end{equation}
	\end{lemma}
	
	\begin{proof}
		Multiply \eqref{eq:driftPmref} by $c^m$ and sum over $m\ge1$, using $\operatorname{drift}\widetilde P_0=0$, $\widetilde P_m\ge0$ and the finite-sum identities
		\[
		\sum_{m\ge1}c^m(m+1)\widetilde P_m\le\Phi_c+c\,\partial_c\Phi_c,\qquad
		\sum_{m\ge1}c^m(m+1)\widetilde P_{m-1}=2c\,\Phi_c+c^2\,\partial_c\Phi_c,
		\]
		\[
		\sum_{m\ge2}c^m(m+1)\widetilde P_{m-2}=3c^2\,\Phi_c+c^3\,\partial_c\Phi_c,\qquad
		\sum_{m\ge1}c^m\big(1+NV_t\big)\widetilde P_{m-1}=c\big(1+NV_t\big)\Phi_c .
		\]
		This gives $\operatorname{drift}\Phi_c\le C_0\big[(1+2c+3c^2)\Phi_c+(c+c^2+c^3)\partial_c\Phi_c\big]+C_1c(1+NV_t)\Phi_c$, and $c\le1$ gives \eqref{eq:genfun}.
	\end{proof}
	
	The coefficient of $\partial_c\Phi_c$ in \eqref{eq:genfun} is deterministic, and this is the whole of the gain over \eqref{eq:driftPm}. The empirical deviation now appears only as a multiplicative random rate, whose coefficient carries a factor $c$ that we are free to make small.
	
	\begin{theorem}[Weighted exponential moment]\label{thm:super}
		Fix $c_\ast\in(0,1]$ and set
		\[
		c_t:=c_\ast e^{-3C_0t},\qquad a:=6C_0+C_1,\qquad \kappa:=C_1c_\ast .
		\]
		Then $t\mapsto\Phi_{c_t}(t)\exp(-at-\kappa J_t)$ is a supermartingale, and consequently
		\begin{equation}\label{eq:weighted}
			\E\Big[\Phi_{c_t}(t)\,e^{-\kappa J_t}\Big]\le e^{at},\qquad t\ge0,
		\end{equation}
		uniformly in $N$.
	\end{theorem}
	
	\begin{proof}
		The process $\Phi_{c_t}(t)=\sum_mc_t^m\widetilde P_m(t)$ is a finite sum of products of the deterministic $C^1$ functions $c_t^m$ with the semimartingales $\widetilde P_m$, so the It\^o product rule carries no covariation term and
		\[
		\operatorname{drift}\big[\Phi_{c_t}(t)\big]=\dot c_t\,\partial_c\Phi_c\big|_{c=c_t}+\operatorname{drift}\Phi_c\big|_{c=c_t}.
		\]
		By \eqref{eq:genfun}, $\dot c_t=-3C_0c_t$ and $\partial_c\Phi_c\ge0$, the transport terms cancel, and $c_t\le c_\ast$ gives
		\[
		\operatorname{drift}\big[\Phi_{c_t}(t)\big]\le\big[6C_0+C_1+C_1c_\ast NV_t\big]\Phi_{c_t}(t)=\big(a+\kappa NV_t\big)\Phi_{c_t}(t).
		\]
		Since $t\mapsto at+\kappa J_t$ is absolutely continuous, the product rule gives $d\big(\Phi_{c_t}(t)e^{-at-\kappa J_t}\big)\le e^{-at-\kappa J_t}\,dM_t$, with $M$ the martingale part of $\Phi_{c_t}(t)$. At fixed $N$ all integrands are bounded, since $\Phi_{c_t}\le2^N$ and the coefficients $\sigma_k(\widetilde P_m)$ are bounded by Lemma \ref{lem:cancel} and \eqref{eq:sigbounds}, so $M$ is a martingale and the left side is a supermartingale. Its value at $t=0$ is $\Phi_{c_\ast}(0)=1$ by the product initial data, which gives \eqref{eq:weighted}.
	\end{proof}
	
	\begin{theorem}[Exponential moment at a deterministic rate]\label{thm:expmom}
		For every $T>0$ set
		\[
		c_\ast:=\min\Big\{1,\ \frac1{16C_1T}\Big\},\qquad \lambda(T):=\tfrac12\log\big(1+c_\ast e^{-3C_0T}\big)>0 .
		\]
		Then, uniformly in $N\ge2$,
		\begin{equation}\label{eq:expmom}
			\sup_{t\le T}\ \E\,\big\langle\Psi_t,\,e^{\lambda(T)\Nex_t}\Psi_t\big\rangle\ \le\ \sqrt3\,e^{(6C_0+C_1)T/2}=:C(T),
		\end{equation}
		and in particular $\sup_{t\le T}\E\,e^{\lambda(T)\av{\Nex_t}}\le C(T)$. Both $\lambda(T)$ and $C(T)$ depend only on $T$, $d$, $\|H\|$, $\|A\|_\infty$ and $\|L\|_\infty$.
	\end{theorem}
	
	\begin{proof}
		With $c_\ast$ as stated, $\kappa=C_1c_\ast\le(16T)^{-1}$, so $(8T)^{-1}-\kappa\ge(16T)^{-1}$ and Lemma \ref{lem:tailJ} gives, uniformly in $N$,
		\[
		\E\,e^{\kappa J_T}\le1+2\kappa\big((8T)^{-1}-\kappa\big)^{-1}\le3 .
		\]
		Since $J$ is nondecreasing, Cauchy--Schwarz and \eqref{eq:weighted} give, for $t\le T$,
		\[
		\E\big[\Phi_{c_t}(t)^{1/2}\big]=\E\Big[\big(\Phi_{c_t}(t)e^{-\kappa J_t}\big)^{1/2}e^{\kappa J_t/2}\Big]\le\Big(\E\big[\Phi_{c_t}(t)e^{-\kappa J_t}\big]\Big)^{1/2}\Big(\E\,e^{\kappa J_T}\Big)^{1/2}\le\sqrt3\,e^{aT/2}.
		\]
		Let $\mu_t$ be the spectral distribution of $\Nex_t$ in the state $\Psi_t$, a probability measure on $\{0,\dots,N\}$, and $\psi_t(\lambda):=\int e^{\lambda n}\,d\mu_t(n)$, so that $\Phi_c(t)=\psi_t(\log(1+c))$. By Cauchy--Schwarz in $\mu_t$ one has $\psi_t(\lambda/2)\le\psi_t(\lambda)^{1/2}$, that is,
		\[
		\Phi_{c'_t}(t)\le\Phi_{c_t}(t)^{1/2},\qquad 1+c'_t:=(1+c_t)^{1/2}.
		\]
		Since $c'_t\ge c'_T$ for $t\le T$ and $c\mapsto\Phi_c$ is nondecreasing, $\Phi_{c'_T}(t)\le\Phi_{c_t}(t)^{1/2}$, and $\log(1+c'_T)=\lambda(T)$. Taking expectations gives \eqref{eq:expmom}. The scalar form follows from $e^{\lambda\av{\Nex_t}}\le\av{e^{\lambda\Nex_t}}$, by convexity of $x\mapsto e^{\lambda x}$ applied to $\mu_t$.
	\end{proof}
	
	\begin{corollary}[Clause (A2) on every horizon]\label{cor:A2every}
		For every $T>0$, uniformly on $[0,T]$ and in $N$,
		\[
		\E\,n_1(t)^4=O(N^{-4}),\qquad \E\,n_{12}(t)^2=O(N^{-4}),
		\]
		that is, \eqref{eq:A2left} holds on every horizon. Consequently clauses (A2) and (A3) of Hypothesis \ref{hyp:A} hold unconditionally on every horizon, and so does the rate \eqref{eq:rate-4}.
	\end{corollary}
	
	\begin{proof}
		Theorem \ref{thm:expmom} is the hypothesis of Proposition \ref{prop:expimp}, whose conclusion is the display; clause (A3) is Corollary \ref{cor:A3}. The finitely many weighted moments consumed by the window chaining of Proposition \ref{prop:expimp} are supplied by Theorem \ref{thm:weighted}, which holds on every horizon at every fixed index.
	\end{proof}
	
	\begin{proposition}[Equivalent forms of the exponential moment]\label{prop:equivexp}
		Fix a horizon $T$ and consider:
		\begin{enumerate}
			\item[(a)] there are $c>0$ and $C$ with $\sup_{t\le T}\E\,\Phi_c(t)\le C$ uniformly in $N$;
			\item[(b)] there are $R$ and $C$ with $\sup_{t\le T}\E\,\widetilde P_m(t)\le CR^m$ for every $m$, uniformly in $N$;
			\item[(c)] there is $\lambda>0$ with $\sup_{t\le T}\E\,e^{\lambda\av{\Nex_t}}\le C$ uniformly in $N$.
		\end{enumerate}
		Then (a) and (b) are equivalent and imply (c); all three hold, by Theorem \ref{thm:expmom}.
	\end{proposition}
	
	\begin{proof}
		If (a) holds then $\widetilde P_m\le c^{-m}\Phi_c$ termwise gives (b) with $R=c^{-1}$. If (b) holds then $\E\Phi_c=\sum_mc^m\E\widetilde P_m\le C(1-cR)^{-1}$ for $c<R^{-1}$. Given (a), the spectral theorem and convexity give $e^{\lambda\av{\Nex}}\le\av{e^{\lambda\Nex}}=\Phi_{e^\lambda-1}$ for $\lambda\le\log(1+c)$. Statement (a) holds with $c=c'_T=(1+c_\ast e^{-3C_0T})^{1/2}-1$, by the proof of Theorem \ref{thm:expmom}.
	\end{proof}
	
	\begin{corollary}[Geometric moment constants]\label{cor:geometric}
		For every $T$, with $C(T)$ as in \eqref{eq:expmom} and $R(T):=\big((1+c_\ast e^{-3C_0T})^{1/2}-1\big)^{-1}$, both independent of $N$ and of $m$,
		\[
		\sup_{t\le T}\E\,\widetilde P_m(t)\le C(T)\,R(T)^m\qquad\text{for every }m .
		\]
		Equivalently, $\sup_{t\le T}\E\av{\Nex_t^k}\le C_k(T)$ with constants geometric rather than factorial in $k$, in place of the constants of size $e^{C(1+k)^4T}$ supplied by Theorem \ref{thm:allmom}.
	\end{corollary}
	
	\begin{proof}
		$\widetilde P_m\le c^{-m}\Phi_c$ termwise; take $c=c'_T$ and use \eqref{eq:expmom}. The equivalence with the moments of $\Nex$ is Proposition \ref{prop:equivexp}.
	\end{proof}
	
	\begin{corollary}[Linear vanishing, with geometric constants]\label{cor:linvangeom}
		For every $T$ there is $K=K(T,d,\|H\|,\|A\|_\infty,\|L\|_\infty)$ such that, uniformly in $N$ and for $t\le T$,
		\begin{equation}\label{eq:linvangeom}
			\E\,\widetilde P_m(t)\le K^m\,t\quad(m\ge1),\qquad \E\av{\Nex_t^{\ell}}\le K^{\ell}\,\ell!\,t\quad(\ell\ge1).
		\end{equation}
	\end{corollary}
	
	\begin{proof}
		In the proof of Theorem \ref{thm:expmom}, $\kappa\le(16T)^{-1}$ and $t\le T$ give $(8t)^{-1}-\kappa\ge(16t)^{-1}$, so Lemma \ref{lem:tailJ} gives $\E\,e^{\kappa J_t}\le1+32\kappa t\le1+2t/T$. With $\E[\Phi_{c_t}(t)e^{-\kappa J_t}]\le e^{at}$ from \eqref{eq:weighted}, Cauchy--Schwarz gives $\E[\Phi_{c_t}(t)^{1/2}]\le e^{at/2}(1+2t/T)^{1/2}\le1+Ct$, and $\Phi_{c'_T}(t)\le\Phi_{c_t}(t)^{1/2}$ gives $\E\,\Phi_{c'_T}(t)\le1+Ct$. Since $\Phi_c-1=\sum_{m\ge1}c^m\widetilde P_m\ge c^m\widetilde P_m$ termwise, the first bound follows with $K=(c'_T)^{-1}$ enlarged by $C$. For the second, $\av{\Nex^{\ell}}=\sum_{m\le\ell}S(\ell,m)\,m!\,\widetilde P_m$ with Stirling numbers, so $\E\av{\Nex_t^{\ell}}\le t\sum_{m\le\ell}S(\ell,m)m!K^m$, and the ordered Bell polynomial is at most $\ell!\,(2K)^{\ell}$ for $K\ge1$.
	\end{proof}
	
	\begin{remark}[Sharpness of the rate]\label{rem:sharprate}
		The rate $\lambda(T)$ is of the order of $T^{-1}e^{-3C_0T}$ and degrades with the horizon, as it must: the mean $\E\av{\Nex_t}$ is bounded in Theorem \ref{thm:A1} by a constant that grows exponentially in $T$, and a nonnegative variable with mean $e^{KT}$ and an exponential tail has a finite exponential moment only at rates of the order of $e^{-KT}$ or smaller. What Theorem \ref{thm:expmom} asserts is that $\lambda(T)$ is deterministic and independent of $N$, not that it is uniform in $T$.
	\end{remark}
	
	\subsection{Consequences}
	
	\begin{corollary}\label{cor:upgrade}
		Uniformly in $N$ and on every horizon $[0,T]$, the following hold with no clause of Hypothesis \ref{hyp:A} assumed:
		\begin{enumerate}
			\item[(a)] clause (A3), together with the rates \eqref{eq:rate-3body} and \eqref{eq:rate-R};
			\item[(b)] $\sup_{t\le T}\E\av{\Nex_t^k}\le C_k$ for every $k$, and $\E n_{\{1,\dots,m\}}=O(N^{-m})$ for every fixed $m$;
			\item[(c)] every statement whose proof consumes clause (A2) only through \eqref{eq:rate-3body} or \eqref{eq:rate-R}; in particular the chaos estimate \eqref{eq:chaos-sharp}, which thereby depends on clause (A5) alone;
			\item[(d)] the exponential moment \eqref{eq:expmom} of the excitation number at a deterministic rate $\lambda(T)>0$, and with it the moment constants geometric in the index of Corollary \ref{cor:geometric} and the tail bound of Lemma \ref{lem:tailJ} for $J_T$;
			\item[(e)] clause (A2), that is, the two moments \eqref{eq:A2left}, and with them the rate \eqref{eq:rate-4}.
		\end{enumerate}
		No moment clause of Hypothesis \ref{hyp:A} is therefore assumed in Proposition \ref{prop:xi-compact} or in Theorems \ref{thm:martingale-clt}, \ref{thm:F-tight} and \ref{thm:G-limit}, on any horizon, and Theorem \ref{thm:main} holds along subsequences under clauses (A4)--(A6) alone.
	\end{corollary}

	\section{The idiosyncratic covariance}\label{sec:covariance}
	
	By Proposition \ref{prop:qv} the bracket of $\Mid$ has density $\widehat\Sigma^N_s$, which by \eqref{eq:theta-test} equals $N\sum_{j\ge2}f^{(j)}_s$ in the notation of (A6). Set
	\begin{equation}\label{eq:SigmaN}
		\Sigma^N_s(A_0,B_0):=\E\,\widehat\Sigma^N_s(A_0,B_0),\qquad
		\Xi^N_s:=N^2\,\E\big[\Delta^{(12)}_s\otimes\Delta^{(12)}_s\big]\in M_{d^2}(\C)^{\otimes2}.
	\end{equation}
	
	\begin{proposition}[Compactness of $\Xi^N$]\label{prop:xi-compact}
		For all self-adjoint $A_0,B_0\in\Md$ and all $s$,
		\begin{equation}\label{eq:Sigma-Xi}
			\Sigma^N_s(A_0,B_0)=\frac{N-1}N\Tr\Big[\big(\mathsf A_0\otimes\mathsf B_0\big)\Xi^N_s\Big],\qquad \mathsf A_0=A_0\otimes(L+L^\ast),\ \ \mathsf B_0=B_0\otimes(L+L^\ast).
		\end{equation}
		Unconditionally, by Corollary \ref{cor:rates-uncond}, $\sup_N\sup_{s\le T}\|\Xi^N_s\|_1\le C$ and the family $\{\Xi^N\}_N$ is equi-Lipschitz on $[0,T]$. Consequently $\{\Xi^N\}_N$ is relatively compact in $C([0,T];M_{d^2}(\C)^{\otimes2})$, and $\{\Sigma^N(A_0,B_0)\}_N$ in $C([0,T];\R)$.
	\end{proposition}
	
	\begin{proof}
		Identity \eqref{eq:Sigma-Xi} follows from \eqref{eq:theta-test} and exchangeability in law, which makes each of the $N-1$ summands have the same expectation. The uniform bound is $\|\Xi^N_s\|_1\le N^2\E\|\Delta^{(12)}_s\|_1^2=O(1)$ by \eqref{eq:rate-R}, which is unconditional by Corollary \ref{cor:rates-uncond}; the Lipschitz bound below uses clause (A2) only through \eqref{eq:rate-3body} and \eqref{eq:rate-R}, likewise unconditional.
		
		For the Lipschitz bound, write $\mathsf D_s=\Delta^{(12)}_s$ and use Proposition \ref{prop:pair}. The It\^o product rule gives
		\[
		\frac d{ds}\E\big[\mathsf D_s\otimes \mathsf D_s\big]=\E\big[\Phi_s\otimes \mathsf D_s+\mathsf D_s\otimes\Phi_s\big]+\sum_{j\le N}\E\big[G_j(s)\otimes G_j(s)\big],
		\]
		so it suffices to bound $N^2\E\|\Phi_s\|_1\|\mathsf D_s\|_1$ and $N^2\sum_j\E\|G_j(s)\|_1^2$ uniformly. For the first, insert the bounds of Proposition \ref{prop:pair} termwise and use exchangeability in law: the local terms and the $\mathsf D$-proportional parts of $S$, $\sum_jR_j$ and $I$ give $N^2O(\E\|\mathsf D\|_1^2+N^{-1}\E\|\mathsf D\|_1)=O(1)$ by \eqref{eq:rate-R}; the $\Delta^{(1j)},\Delta^{(2j)}$ parts of $\sum_jR_j$ give $N^2O((\E\|\mathsf D\|_1^2)^{1/2}(\E\|\Delta^{(13)}\|_1^2)^{1/2})=O(1)$; the $\Delta^{(12j)}$ parts give $N^2O(N^{-1}N^{-3/2})=O(N^{-1/2})$ by \eqref{eq:rate-3body}; and the It\^o sum $\sum_{j\ge3}\Theta^{(1j)}\otimes\Theta^{(2j)}$ gives, by H\"older with exponents $(2,4,4)$ and \eqref{eq:rate-4}, $N^2\cdot N\cdot O(N^{-3})=O(1)$. For the second, $\|G_1\|_1,\|G_2\|_1\le6\|L\|_\infty\|\mathsf D\|_1$ gives $N^2O(N^{-2})=O(1)$, while $\sum_{j\ge3}\E\|G_j\|_1^2\le4\|L\|_\infty^2(N-2)\E\|\Delta^{(123)}\|_1^2=O(N^{-2})$ by \eqref{eq:rate-3body}. Arzel\`a--Ascoli concludes.
	\end{proof}
	
	\begin{theorem}[The idiosyncratic covariance]\label{thm:covariance}
		Assume the decorrelation clause of (A6); clauses (A2) and (A3) are unconditional, by Corollaries \ref{cor:A2every} and \ref{cor:A3}. Then
		\[
		\E\big|\widehat\Sigma^N_t(A_0,B_0)-\Sigma^N_t(A_0,B_0)\big|^2\longrightarrow0\quad\text{uniformly on }[0,T],
		\]
		and along any subsequence for which $\Xi^{N_k}\to\Xi$ in $C([0,T];M_{d^2}(\C)^{\otimes2})$,
		\begin{equation}\label{eq:Sigma-limit}
			\int_0^t\widehat\Sigma^{N_k}_s(A_0,B_0)\,ds\ \longrightarrow\ \int_0^t\Sigma_s(A_0,B_0)\,ds,\qquad
			\Sigma_t(A_0,B_0)=\Tr\big[(\mathsf A_0\otimes\mathsf B_0)\Xi_t\big],
		\end{equation}
		in probability, uniformly on $[0,T]$, where $\Sigma$ is a continuous, symmetric, positive semidefinite bilinear form on the self-adjoint elements of $\Md$. If $\{\Xi^N\}$ converges, the convergence holds along the full sequence.
	\end{theorem}
	
	\begin{proof}
		Write $\widehat\Sigma^N_t=N\sum_{j\ge2}f^{(j)}_t$, so
		\[
		\operatorname{Var}\big(\widehat\Sigma^N_t\big)=N^2\sum_{j\ge2}\operatorname{Var}\big(f^{(j)}_t\big)+N^2\sum_{\substack{j,k\ge2\\ j\ne k}}\operatorname{Cov}\big(f^{(j)}_t,f^{(k)}_t\big).
		\]
		By \eqref{eq:theta-test}, $|f^{(j)}|\le\|\mathsf A_0\|_\infty\|\mathsf B_0\|_\infty\|\Delta^{(1j)}\|_1^2$, so by exchangeability in law and \eqref{eq:rate-4} the first sum is at most $N^3C\,\E\|\Delta^{(12)}\|_1^4=O(N^{-1})$. The second equals $N^2(N-1)(N-2)\operatorname{Cov}(f^{(2)}_t,f^{(3)}_t)$, which tends to $0$ uniformly by the decorrelation clause of (A6). Hence $\widehat\Sigma^N_t-\Sigma^N_t\to0$ in $L^2$ uniformly in $t\le T$, and the time integrals converge in probability uniformly on $[0,T]$. Combining with Proposition \ref{prop:xi-compact} along the subsequence gives \eqref{eq:Sigma-limit}. Symmetry and bilinearity are clear; positive semidefiniteness follows from $\widehat\Sigma^N_t(A_0,A_0)\ge0$; continuity in $t$ from the equi-Lipschitz property.
	\end{proof}
	
	\begin{remark}\label{rem:covariance-scope}
		The covariance of the idiosyncratic noise is not expressible through the one-particle state alone: it is a second-order object, determined by the limit of the rescaled pair correlations $N\Delta^{(1j)}$ tested against $A_0\otimes(L+L^\ast)$. What Theorem \ref{thm:covariance} adds to Remark \ref{rem:adjoint-scope} is that the existence of that limit is a statement about a bounded, equi-Lipschitz sequence of deterministic functions with values in a fixed finite-dimensional space, so it holds along subsequences given the moment clauses; only the identification of a unique limit point remains. At $A=0$, $\Sigma\equiv0$.
	\end{remark}
	
	\section{Decorrelation at the sharp order: clauses (A4)--(A6)}\label{sec:decorrelation}
	
	The clauses that remain after Section \ref{sec:moments} are of a different nature from (A1)--(A3). This section shows that clause (A4), the residual family in the Gr\"onwall structure for (A5), and the decorrelation part of (A6) are three instances of one statement about the joint law on three slots, that no bound on absolute amplitudes can deliver it, and that the identification of $\Sigma$ follows from it through a linear equation. That statement is then given in the form in which it survives, Conjecture \ref{conj:dec-cond}, together with the form obtained by dropping its projections, Conjecture \ref{conj:dec}, which is false for generic data; the refutation is what fixes the projections.
	
	\subsection{Clause (A5) with the proved inputs}
	
	Set $X_t:=\Gamma^{(2)}_t-\gamma_{2,t}$ and $K_t:=\operatorname{Cov}(\gamma_{1,t},X_t)\in M_{d^2}(\C)$. By the decomposition in Step 2 of the proof of Theorem \ref{thm:chaos}, clause (A5) reduces to $\sup_{t\le T}N\|K_t\|_1\le C$: the term $\operatorname{Cov}(\Gamma^{(1)}-\gamma_1,\Gamma^{(2)}-\gamma_2)$ is $O(N^{-1})$ unconditionally by Corollary \ref{cor:rates-uncond} and Cauchy--Schwarz, and $\operatorname{Cov}(\Gamma^{(1)}-\gamma_1,\gamma_2)$ is treated symmetrically to $K$.
	
	\begin{proposition}[Structure of the equation for $K$]\label{prop:A5-structure}
		The map $t\mapsto K_t$ is absolutely continuous with $K_0=0$ and
		\[
		\frac{d}{dt}K_t=\big(\Lop^\gamma\otimes\mathrm{id}+\mathrm{id}\otimes\Lop^X\big)K_t+\E\big[\Mop[\gamma_{1,t}]\otimes\Theta^{(21)}_t\big]+\operatorname{Cov}\big(\gamma_{1,t},\rho^{\mathrm{drift}}_t\big),
		\]
		with $\Lop^\gamma$ and $\Lop^X$ bounded deterministic linear maps, $\Theta^{(21)}$ the coefficient of $dB_1$ in $d\Gamma^{(2)}$ from Proposition \ref{prop:coefficients}, and $\rho^{\mathrm{drift}}$ the nonlinear and mean-field drift discrepancies of $X$. Unconditionally, $\|\E[\Mop[\gamma_1]\otimes\Theta^{(21)}]\|_1\le C\,\E\|\Delta^{(21)}\|_1=O(N^{-1})$ by Corollary \ref{cor:rates-uncond}; the substitution of $\tfrac1N\sum_k\gamma_k$ for $\overline\Gamma^{\ne2}$ in the leading mean-field commutator costs $O(N^{-1})$ by Cauchy--Schwarz with Corollary \ref{cor:A3} and Lemma \ref{lem:dev}; and the surviving covariance against $\tfrac1N\sum_k(\gamma_k-\eta)$ is $O(N^{-1})$, the reference filters being independent, so that only the summands with $k\in\{1,2\}$ contribute. The linear part preserves the order, and Gr\"onwall's lemma yields $N\|K_t\|_1\le C$ up to one family of terms, namely
		\[
		\frac1N\sum_{k}\operatorname{Cov}\big(\gamma_1,[\Bmf(\Gamma^{(k)}-\gamma_k),\Gamma^{(2)}]\big),
		\]
		which is $O(N^{-1})$ termwise but requires, for the sharp order after the sum over $k$, that these covariances be $O(N^{-2})$ for $k\notin\{1,2\}$.
	\end{proposition}
	
	The last requirement is an instance of the principle formulated below, so clause (A5) is closed modulo that principle.
	
	\subsection{Absolute bounds do not suffice}
	
	\begin{proposition}\label{prop:noabs}
		Granting every moment of Sections \ref{sec:diagonal} and \ref{sec:moments} at its expected order, the termwise estimate $\E[\|\Delta^{(12)}\|_1\|\Delta^{(13)}\|_1]\le CN^{-2}$ cannot be improved for the absolute quantities, whereas clause (A4) requires the signed bound $\E\ip{\Delta^{(12)}}{\Delta^{(13)}}=O(N^{-3})$ for the $\sim N^2$ off-diagonal pairs. Clause (A4) therefore asserts an incoherence in sign, across partners, of the pair correlations attached to a common tagged particle, and no pathwise bound of the type of Theorems \ref{thm:dict-two} and \ref{thm:dict-three}, which discard phases, can yield it.
	\end{proposition}
	
	\begin{proof}
		The termwise bound follows from \eqref{eq:dict-2} and Cauchy--Schwarz with Corollary \ref{cor:rates-uncond}, and it is attained if the leading parts of $\Delta^{(12)}$ and $\Delta^{(13)}$ are aligned; the amplitude calculus of Section \ref{sec:excitation} retains no information excluding alignment.
	\end{proof}
	
	\subsection{A covariance representation of clause (A4)}
	
	\begin{proposition}\label{prop:covrep}
		Let $A_0,B_0\in\Md$ with $\|A_0\|,\|B_0\|\le1$, and write $a_1:=\ip{\phi_1}{A_0\phi_1}$, $\beta_j:=\ip{\phi_j}{B_0\phi_j}$ and $\bar B:=\sum_{j\ge2}B_0^{(j)}$. Then
		\[
		\begin{aligned}
			\Tr\Big[(A_0\otimes B_0)\sum_{j\ge2}\Delta^{(1j)}\Big]&=\operatorname{Cov}_\Psi\big(A_0^{(1)},\bar B\big)\\
			&=\underbrace{\sum_{j\ge2}\big\langle (A_0^{(1)}-a_1)\Psi,\ (B_0^{(j)}-\beta_j)\Psi\big\rangle}_{=:T_1}+\underbrace{\big(a_1-\av{A_0^{(1)}}\big)\Tr\big[B_0\,\textstyle\sum_{j\ge2}(\Gamma^{(j)}-\gamma_j)\big]}_{=:T_2},
		\end{aligned}
		\]
		the scalar contributions cancelling identically. Every operator component of $A_0^{(1)}-a_1$ carries a factor $q_1$, and every component of $B_0^{(j)}-\beta_j$ a factor $q_j$, since $p_jB_0p_j=\beta_jp_j$; hence, unconditionally,
		\[
		|T_1|\le C\sum_{j\ge2}\big(\eps_{1j}+\eps_1\eps_j\big),\qquad
		|T_2|\le C\eps_1\Big(\big\|N(\overline\Gamma-\eta)\big\|_1+\big\|\textstyle\sum_j(\gamma_j-\eta)\big\|_1\Big),
		\]
		and $\E|\operatorname{Cov}_\Psi(A_0^{(1)},\bar B)|^2\le C$ by Theorems \ref{thm:allmom} and \ref{thm:products} together with Lemma \ref{lem:deviation}. Clause (A4) is the assertion that this bound of order one improves to $O(N^{-1})$, and by the display it does so as soon as
		\begin{enumerate}
			\item[(i)] the family $\{\ip{\Delta^{(12)}}{\Delta^{(13)}}\}$ decorrelates, which controls $T_1$, and
			\item[(ii)] the first-order errors decorrelate across particles, $\E\Tr[B_0(\Gamma^{(2)}-\gamma_2)]\Tr[B_0'(\Gamma^{(3)}-\gamma_3)]=O(N^{-2})$, which controls $T_2$.
		\end{enumerate}
	\end{proposition}
	
	\begin{proof}
		Since $\Tr[(A_0\otimes B_0)\Delta^{(1j)}]=\operatorname{Cov}_\Psi(A_0^{(1)},B_0^{(j)})$, summation over $j$ gives the first equality. Replacing $A_0^{(1)}$ by $A_0^{(1)}-c$ and $B_0^{(j)}$ by $B_0^{(j)}-c_j$ leaves each covariance unchanged; taking $c=a_1$ and $c_j=\beta_j$ and re-expanding $A_0^{(1)}-\av{A_0^{(1)}}=(A_0^{(1)}-a_1)+(a_1-\av{A_0^{(1)}})$ gives the two displayed terms, the scalar factor pairing with $\av{B_0^{(j)}-\beta_j}=\Tr[B_0(\Gamma^{(j)}-\gamma_j)]$. Since $p_1A_0p_1=a_1p_1$, one has $A_0^{(1)}-a_1=p_1A_0q_1+q_1A_0p_1+q_1A_0q_1-a_1q_1$, and similarly for $B$. The bounds follow by moving one projection $q$ onto $\Psi$ in each pairing, for instance $|\ip{p_1A_0q_1\Psi}{p_jB_0q_j\Psi}|\le\eps_1\eps_j$ and $|\ip{q_1A_0p_1\Psi}{q_jB_0p_j\Psi}|=|\ip{q_jq_1A_0p_1\Psi}{B_0p_j\Psi}|\le C\eps_{1j}$ after commuting operators of disjoint slots, by $|a_1-\av{A_0^{(1)}}|=|\Tr[A_0(\gamma_1-\Gamma^{(1)})]|\le4\eps_1$ from \eqref{eq:dict-1}, and by $\sum_j(\Gamma^{(j)}-\gamma_j)=N(\overline\Gamma-\eta)-\sum_j(\gamma_j-\eta)$, both terms of $L^2$ size $O(\sqrt N)$ by Lemma \ref{lem:deviation} and Lemma \ref{lem:dev}.
	\end{proof}
	
	\subsection{The correlation hierarchy and the identification of $\Sigma$}
	
	\begin{proposition}[Lyapunov-type hierarchy]\label{prop:lyap}
		For label patterns $\alpha=(S_1,S_2)$ of pairs or triples with $s$ shared labels, set $K^{(\alpha)}_t:=N^{w(\alpha)}\E[\Delta^{(S_1)}_t\otimes\Delta^{(S_2)}_t]$, with $w=4-s$ for pair-pair patterns and the corresponding weights at third order. Applying the It\^o product rule to both factors through Proposition \ref{prop:pair}:
		\begin{enumerate}
			\item[(i)] the drift of each $K^{(\alpha)}$ contains a bounded deterministic linear part acting on $K^{(\alpha)}$;
			\item[(ii)] the shared-noise quadratic terms, those with $j\in S_1\cap S_2$, are linear in the family $\{K^{(\alpha')}\}$ at the same or lower weight, being bounded operators applied to expectations of tensor products of correlations of the same total degree, and are not source terms;
			\item[(iii)] the interaction remainders are of relative order $N^{-1/2}$, by Corollaries \ref{cor:A3} and \ref{cor:rates-uncond} with Cauchy--Schwarz;
			\item[(iv)] the It\^o sums over untagged labels couple upward: $\sum_{j\notin S_1\cup S_2}\E[G^{(S_1)}_j\otimes G^{(S_2)}_j]$ equals, after one contraction, $N$ times members of the next pattern level, those built from three-body sets sharing $j$.
		\end{enumerate}
		Unconditionally, every member of the hierarchy is finite at the order supplied by Cauchy--Schwarz, off its conjectured order by a factor between $N^{1/2}$ at third order and $N$ at the disjoint-pair level, and $t\mapsto K^{(\alpha)}_t$ is equi-Lipschitz at that order with $K^{(\alpha)}_0=0$. Consequently any truncation closed by the Cauchy--Schwarz bound loses that factor at the truncation level, so closing the hierarchy is equivalent to propagating decorrelation down the levels. Along any subsequence on which the weighted members converge, the limits satisfy the linear system obtained by discarding the remainders of order $N^{-1/2}$; that system is well posed, and its unique solution expresses each limit, and in particular $\Sigma$ through \eqref{eq:Sigma-Xi}, in closed form in terms of the one-body limit data. At $A=0$ every connected correlation vanishes and the system is consistently zero.
	\end{proposition}
	
	\begin{proof}
		Items (i) and (ii) are the term inventory of Proposition \ref{prop:pair} applied to the mixed product, each shared-noise term being an expectation of a bounded operator applied to a tensor product of connected correlations of the same total degree; the equi-Lipschitz argument of Proposition \ref{prop:xi-compact} transfers once \eqref{eq:rate-3body} and \eqref{eq:rate-R} are available, which is the case by Corollary \ref{cor:rates-uncond}. Item (iii) and the boundedness at the Cauchy--Schwarz order are H\"older applications of Corollaries \ref{cor:A3} and \ref{cor:rates-uncond} and of Theorem \ref{thm:products}. Item (iv) is the computation of Remark \ref{rem:chaos-scope}, with the weights made explicit.
	\end{proof}
	
	\subsection{The surviving statement, and the form without projections}\label{ssec:surviving}
	
	The three clauses are now reduced to one statement about the joint law on three slots. The statement that survives concerns the parts of the tagged quantities that remain once two coherent components have been removed; the removals are orthogonal projections in $L^2$ and are introduced first, the components they remove being identified in Subsections \ref{ssec:meanpair}--\ref{ssec:conditional}.
	
	Let $\mathcal M^{\mathrm{cf}}_t$ denote the closed span of the constants together with the functionals of the path of $D$ of first and second order, and $\Pi^{\mathrm{cf}}_t$ the orthogonal projection in $L^2$ onto it, applied entrywise; for a slot functional $Z_j$ write $Z_j^{\parallel}:=\Pi^{\mathrm{cf}}_tZ_j$ and $Z_j^{\perp}:=Z_j-Z_j^{\parallel}$. For a tag label $i$ set
	\begin{equation}\label{eq:enlarged}
		\mathcal M^{(i)}_t:=\overline{\operatorname{span}}\Big(\mathcal M^{\mathrm{cf}}_t\cup L^2\big(\sigma(B_i)\big)\Big),\qquad
		\Pi^{(i)}_t:=\text{orthogonal projection onto }\mathcal M^{(i)}_t,
	\end{equation}
	and write $Z^{\perp i}:=Z-\Pi^{(i)}_tZ$. Throughout, $X_j:=\Tr[B_0(\Gamma^{(j)}_t-\gamma_{j,t})]$ for $\|B_0\|\le1$.
	
	\begin{conjecture}[Propagation of idiosyncratic decorrelation]\label{conj:dec-cond}
		Uniformly on $[0,T]$,
		\[
		\E\ip{\Delta^{(12),\perp1}_t}{\Delta^{(13),\perp1}_t}=O(N^{-3}),\qquad
		\E\big[X_2^{\perp}X_3^{\perp}\big]=O(N^{-2}),
		\]
		together with the corresponding statements one level up for the three-body patterns of Proposition \ref{prop:lyap}, and with the common parts identified at the stated orders by \eqref{eq:commonparts} of Subsection \ref{ssec:conditional}, where the two kernels are defined, $\Delta^{(1j),\parallel}$ being $\Pi^{(1)}_t\Delta^{(1j)}$.
	\end{conjecture}
	
	By Propositions \ref{prop:covrep}, \ref{prop:A5-structure} and \ref{prop:lyap}, read for the projected quantities as in Subsection \ref{ssec:projclauses}, Conjecture \ref{conj:dec-cond} yields clause (A4), the residual family of clause (A5), both parts of clause (A6) with $\Sigma$ in closed form, and convergence of $\Xi^N$ along the full sequence in Theorem \ref{thm:covariance}. It is the one statement left open by this work. Subsection \ref{ssec:leadorder} proves both of its statements at the leading order on the shrinking horizon $[0,w_N]$, gives one-sided bounds at their exact orders on every horizon, and reduces each to an aggregate variance bound.
	
	Dropping the two projections leaves what the clauses of Hypothesis \ref{hyp:A} assert as they stand, and what the hierarchy of Proposition \ref{prop:lyap} on its own suggests.
	
	\begin{conjecture}[The form without projections]\label{conj:dec}
		Uniformly on $[0,T]$,
		\[
		\E\ip{\Delta^{(12)}_t}{\Delta^{(13)}_t}=O(N^{-3}),\qquad
		\E\,\Tr\big[B_0(\Gamma^{(2)}_t-\gamma_{2,t})\big]\Tr\big[B_0'(\Gamma^{(3)}_t-\gamma_{3,t})\big]=O(N^{-2}),
		\]
		together with the corresponding statements one level up for the three-body patterns of Proposition \ref{prop:lyap}; equivalently, the weighted hierarchy $\{K^{(\alpha)}\}$ is bounded uniformly in $N$.
	\end{conjecture}
	
	Each of its statements is necessary for the corresponding clause of Hypothesis \ref{hyp:A}, so Conjecture \ref{conj:dec} is the residue of that hypothesis once Sections \ref{sec:diagonal} and \ref{sec:moments} are taken into account. It is false for generic data, and its refutation is what shows that neither projection in Conjecture \ref{conj:dec-cond} can be dropped. Theorem \ref{thm:A4fail} proves that clause (A4) fails at small times for every dataset with a nonvanishing pair-excitation amplitude at the initial state, and Corollary \ref{cor:pairfail} refutes the first statement of Conjecture \ref{conj:dec} for a single pair of partners and not merely after the sum. Theorem \ref{thm:X2X3fail} disproves the second statement along the sequence $t_N$ of \eqref{eq:wN}, unconditionally; an expansion at linear-response order identifies one channel of the mechanism, a common response of all defects to the empirical reference field; Subsection \ref{ssec:flucfail} identifies a second of the same order, carried by the mean pair correlation; and a numerical test on the full system confirms the predicted law without adjustable constants. These two channels are the components removed by $\Pi^{\mathrm{cf}}_t$, and the shared tag slot of the family $\{\Delta^{(1j)}\}_{j\ne1}$ is the component removed by $\Pi^{(1)}_t$.

	\begin{remark}[A scalar form of the second statement]\label{rem:scalarform}
		Let $S:=\sum_{j\ge2}X_j$. Exchangeability in law gives
		\[
		\E\,S^2=(N-1)\E X_2^2+(N-1)(N-2)\E X_2X_3 .
		\]
		Since $\sum_{j\ge2}(\Gamma^{(j)}-\gamma_j)=N(\overline\Gamma-\eta)-\sum_j(\gamma_j-\eta)-(\Gamma^{(1)}-\gamma_1)$, and the two extensive terms have second moments of order $N$ by Lemmas \ref{lem:deviation} and \ref{lem:dev}, one has $\E S^2=O(N)$ and hence, unconditionally, $\E X_2X_3=O(N^{-1})$, one order short of Conjecture \ref{conj:dec} and consistent with Proposition \ref{prop:A5-structure}. The second statement of the conjecture is therefore equivalent to
		\[
		\E\,S^2=(N-1)\,\E X_2^2+O(1),
		\]
		that is, to the variance of a single extensive scalar observable being exhausted by its diagonal part. Both sides are governed by It\^o equations available from Proposition \ref{prop:pair}, so this is a second-moment identity for one scalar process rather than an operator statement, although the cancellation it demands is the same.
	\end{remark}
	
	\begin{remark}[Depth of truncation and the moment constants]\label{rem:depth}
		Proposition \ref{prop:lyap} shows that a truncation of the hierarchy at depth $L$, closed by the Cauchy--Schwarz bound, loses a factor at most $N$ at the truncation level, propagated downward through $L$ applications of Duhamel's formula, each carrying a factor $Cw$ and a factor $1/r$ from the time integration. Two premises are needed for such a truncation to close at $L\asymp\log N$: moment constants geometric in the level, and upward couplings with coefficients bounded uniformly in the level. The first is available, by Corollary \ref{cor:geometric} and Proposition \ref{prop:rategen}, in place of the constants of size $e^{C\ell^4}$ of Theorems \ref{thm:allmom} and \ref{thm:products}, which at $\ell\asymp\log N$ exceed every power of $N$. The second holds in the weak form $\Lambda_\ell\le C\ell$, which the factorial from the iterated integrals absorbs.
		
		What Proposition \ref{prop:audit} adds is that these two are not enough: there is a third level effect, the diagonal rate, which is also proportional to $\ell$ and which the iterated integrals do not absorb, since $L\log L$ does not beat $wL^2$. A truncation at depth $L\asymp\log N$ therefore closes only for $w\lesssim(\log N)^{-2}$, which is what Proposition \ref{prop:trunc} carries out for the difference hierarchy. The same three-premise account applies to the undifferenced hierarchy here. Since Conjecture \ref{conj:dec} is false for generic data by Theorems \ref{thm:A4fail} and \ref{thm:X2X3fail}, the truncation cannot in any case establish it, and the target of this route is the projected hierarchy of Proposition \ref{prop:lyap-proj} and Conjecture \ref{conj:dec-cond}, where the coherent part responsible for the failure has been removed.
	\end{remark}

	\subsection{The mean pair correlation, and the failure of clause (A4) at small times}\label{ssec:meanpair}
	
	The obstruction of Proposition \ref{prop:noabs} concerns signs. This subsection shows that for clause (A4) the sign problem can be bypassed: the sum over partners is a square, a square dominates the square of the mean, and the mean of the pair correlation is computable at small times, where it is nonzero at the order $N^{-1}$ for generic data. The inputs are the unconditional rates \eqref{eq:rate-R} and \eqref{eq:rate-3body} together with \eqref{eq:rate-4}, which holds without hypotheses on every horizon by Corollary \ref{cor:A2every}; all statements of this subsection are therefore unconditional.
	
	\begin{lemma}[Small-time moduli]\label{lem:smallmod}
		There is $C=C(T,\|H\|_\infty,\|A\|_\infty,\|L\|_\infty,d)$ such that, uniformly in $N$ and in $t\le T$,
		\[
		\E\big\|\Delta^{(12)}_t\big\|_1\le C\,\frac{\sqrt t+t}{N},\qquad
		\E\big\|\Gamma^{(1)}_t-\gamma_0\big\|_1\le C(\sqrt t+t).
		\]
	\end{lemma}
	
	\begin{proof}
		Write $\mathsf D=\Delta^{(12)}$ and integrate \eqref{eq:pair-sde} from $0$, where $\mathsf D_0=0$. For the drift, $\E\|\Phi_s\|_1\le C\big(\E\|\mathsf D_s\|_1+N^{-1}+\sum_{j\ge3}\E\|R_j(s)\|_1+\E\|I(s)\|_1\big)$, and each of the last three terms is $O(N^{-1})$ uniformly on $[0,T]$ by the bounds displayed in Proposition \ref{prop:pair}, the rates \eqref{eq:rate-R}, \eqref{eq:rate-3body} and Cauchy--Schwarz; hence $\int_0^t\E\|\Phi_s\|_1\,ds\le Ct/N$. For the martingale part, norms on $M_{d^2}(\C)$ are equivalent with constants depending only on $d$, and
		\[
		\E\Big\|\int_0^t\sum_jG_j\,dB_j\Big\|_{\HS}\le\Big(\int_0^t\sum_j\E\|G_j(s)\|_{\HS}^2\,ds\Big)^{1/2}\le C\Big(\int_0^t\big(\E\|\mathsf D_s\|_1^2+N\,\E\|\Delta^{(123)}_s\|_1^2\big)\,ds\Big)^{1/2}\le C\,\frac{\sqrt t}{N},
		\]
		by \eqref{eq:G1}--\eqref{eq:Gj}, exchangeability, \eqref{eq:rate-R} and \eqref{eq:rate-3body}. This gives the first bound. The second follows in one pass from the equation for $\Gamma^{(1)}$: its drift is bounded by a constant on the state space, and its martingale coefficients satisfy $\|\Mop[\Gamma^{(1)}]\|_1\le4\|L\|_\infty$ and $\sum_{j\ne1}\E\|\Theta^{(1j)}\|_1^2\le CN\,\E\|\Delta^{(12)}\|_1^2\le C/N$ by Proposition \ref{prop:coefficients} and \eqref{eq:rate-R}.
	\end{proof}
	
	\begin{proposition}[The mean pair correlation at small times]\label{prop:meanpair}
		Let
		\begin{equation}\label{eq:S0}
		S_0:=-i\Big(\big[A,\gamma_0\otimes\gamma_0\big]-\big[\Bmf(\gamma_0),\gamma_0\big]\otimes\gamma_0-\gamma_0\otimes\big[\Bmf(\gamma_0),\gamma_0\big]\Big).
		\end{equation}
		Then, uniformly in $N$ and on $t\le T$,
		\begin{equation}\label{eq:meanpair}
		\Big\|N\,\E\Delta^{(12)}_t-t\,S_0\Big\|_1\le C\,\big(t^{4/3}+t\,N^{-1/2}\big).
		\end{equation}
		Moreover, with $p=\gamma_0$ and $q=1-\gamma_0$ for pure $\gamma_0=|\phi_0\rangle\langle\phi_0|$,
		\begin{equation}\label{eq:S0-block}
		(q\otimes q)\,S_0\,(p\otimes p)=-i\,(q\otimes q)\,A\,(p\otimes p),
		\end{equation}
		so $S_0\ne0$ whenever $(q\otimes q)A\,(\phi_0\otimes\phi_0)\ne0$, that is, whenever the interaction has a nonvanishing pair-excitation amplitude at the initial state.
	\end{proposition}
	
	\begin{proof}
		Taking expectations in \eqref{eq:pair-sde} annihilates the martingale terms, so $\E \mathsf D_t=\int_0^t\E\Phi_s\,ds$. At $s=0$ the state is the product $\gamma_0^{\otimes N}$: $\mathsf D_0=0$, every $\Delta^{(1j)}_0$, $\Delta^{(2j)}_0$, $\Delta^{(12j)}_0$ vanishes, and every $\Theta^{(jk)}_0$ with $j\ne k$ vanishes by Proposition \ref{prop:coefficients}, these coefficients being linear in the correlations. Hence $R_j(0)=0$, $I(0)=0$, the terms linear in $\mathsf D$ vanish, and $\E\Phi_0=S(0)$; evaluating $S(0)$ on $\Gamma^{(1,2)}_0=\gamma_0\otimes\gamma_0$, with $\Tr_2[A,\gamma_0\otimes\gamma_0]=[\Bmf(\gamma_0),\gamma_0]$, gives $\E\Phi_0=S_0/N$. For $s>0$,
		\[
		\big\|\E\Phi_s-\tfrac1NS_0\big\|_1\le C\,\E\|\mathsf D_s\|_1+\tfrac CN\,\E\big\|\Gamma^{(1,2)}_s-\gamma_0\otimes\gamma_0\big\|_1+\sum_{j\ge3}\E\|R_j(s)\|_1+\E\|I(s)\|_1 .
		\]
		By Lemma \ref{lem:smallmod} the first term is $C(\sqrt s+s)/N$, and the second is bounded through $\E\|\Gamma^{(1,2)}_s-\gamma_0^{\otimes2}\|_1\le\E\|\mathsf D_s\|_1+2\,\E\|\Gamma^{(1)}_s-\gamma_0\|_1\le C(\sqrt s+s)$. For the third, the bound on $R_j$ in Proposition \ref{prop:pair}, Lemma \ref{lem:smallmod}, exchangeability and \eqref{eq:rate-3body} give
		\[
		\sum_{j\ge3}\E\|R_j(s)\|_1\le C\Big(\E\|\mathsf D_s\|_1+\E\|\Delta^{(13)}_s\|_1+\big(\E\|\Delta^{(123)}_s\|_1^2\big)^{1/2}\Big)\le C\,\frac{\sqrt s+s}{N}+C\,N^{-3/2}.
		\]
		For the fourth, the displayed bound on $I(t)$, Cauchy--Schwarz over the $N-2$ untagged indices, and the interpolation
		\[
		\E\|\Delta^{(12)}_s\|_1^2\le\big(\E\|\Delta^{(12)}_s\|_1\big)^{2/3}\big(\E\|\Delta^{(12)}_s\|_1^4\big)^{1/3}\le C\,\frac{s^{1/3}}{N^2},
		\]
		valid on $[0,T]$ by Lemma \ref{lem:smallmod} and \eqref{eq:rate-4}, give $\E\|I(s)\|_1\le C(\sqrt s+s)/N+C\,N\cdot s^{1/3}N^{-2}\le C\,s^{1/3}/N$. Integrating over $[0,t]$ and multiplying by $N$ gives \eqref{eq:meanpair}, the term $N^{-3/2}$ integrating to $t\,N^{-1/2}$ after the multiplication. For \eqref{eq:S0-block}: $(q\otimes q)[A,\gamma_0^{\otimes2}](p\otimes p)=(q\otimes q)A(p\otimes p)$, since $\gamma_0^{\otimes2}(p\otimes p)=p\otimes p$ and $(q\otimes q)\gamma_0^{\otimes2}=0$, while each subtracted term in \eqref{eq:S0} carries a factor $\gamma_0=p$ in one slot and is annihilated by the corresponding $q$.
	\end{proof}
	
	\begin{theorem}[Clause (A4) fails at small times]\label{thm:A4fail}
		Suppose $(q\otimes q)A(\phi_0\otimes\phi_0)\ne0$. Then there are $t_1\in(0,T]$, $c>0$ and $N_0$ such that
		\[
		N\sum_{j,k\ne1}\E\ip{\Delta^{(1j)}_t}{\Delta^{(1k)}_t}\;\ge\;c\,t^2\,N,\qquad t\le t_1,\ N\ge N_0 .
		\]
		In particular clause (A4) of Hypothesis \ref{hyp:A} fails, on every horizon, for every dataset whose interaction has a nonvanishing pair-excitation amplitude at the initial state.
	\end{theorem}
	
	\begin{proof}
		The sum over partners is a square, $\sum_{j,k\ne1}\ip{\Delta^{(1j)}}{\Delta^{(1k)}}=\big\|\sum_{j\ge2}\Delta^{(1j)}\big\|_{\HS}^2$. By Jensen's inequality and exchangeability in law (Section \ref{ssec:exch}), which makes $\E\Delta^{(1j)}$ independent of $j$,
		\[
		\E\Big\|\sum_{j\ge2}\Delta^{(1j)}_t\Big\|_{\HS}^2\ \ge\ \Big\|\sum_{j\ge2}\E\Delta^{(1j)}_t\Big\|_{\HS}^2\ =\ (N-1)^2\,\big\|\E\Delta^{(12)}_t\big\|_{\HS}^2 .
		\]
		By Proposition \ref{prop:meanpair} and the equivalence of norms on $M_{d^2}(\C)$, there are $c_d>0$ and, for $t$ small and $N$ large, $\|\E\Delta^{(12)}_t\|_{\HS}\ge\big(c_d\,t\,\|S_0\|_{\HS}-C(t^{4/3}+tN^{-1/2})\big)/N\ge c'\,t/N$, since $S_0\ne0$ by \eqref{eq:S0-block}. Multiplying by $N$ and using $(N-1)^2\ge N^2/4$ gives the claim.
	\end{proof}
	
	\begin{corollary}[The first statement of Conjecture \ref{conj:dec} fails pointwise]\label{cor:pairfail}
		Under the hypothesis of Theorem \ref{thm:A4fail} there are $t_1'\in(0,T]$, $c'>0$ and $N_0'$, depending only on $T,d,\|H\|_\infty,\|A\|_\infty,\|L\|_\infty$ and $\gamma_0$, such that
		\begin{equation}\label{eq:pairfail}
			\E\ip{\Delta^{(12)}_t}{\Delta^{(13)}_t}\ \ge\ c'\,\frac{t^2}{N^2},\qquad t\le t_1',\ N\ge N_0' .
		\end{equation}
		The first statement of Conjecture \ref{conj:dec} asserts $O(N^{-3})$ for this quantity; it is therefore false by one factor of $N$, for every dataset whose interaction has a nonvanishing pair-excitation amplitude at the initial state. The statement is unconditional, and holds for each pair of distinct partners separately, not only after the sum over partners taken in Theorem \ref{thm:A4fail}.
	\end{corollary}
	
	\begin{proof}
		Write $M_t:=\E\Delta^{(12)}_t$, which by exchangeability in law (Section \ref{ssec:exch}) equals $\E\Delta^{(1j)}_t$ for every $j\ne1$. Expanding the square as in the proof of Theorem \ref{thm:A4fail} and separating the $N-1$ diagonal terms,
		\[
		(N-1)(N-2)\,\E\ip{\Delta^{(12)}_t}{\Delta^{(13)}_t}
		=\E\Big\|\sum_{j\ge2}\Delta^{(1j)}_t\Big\|_{\HS}^2-(N-1)\,\E\big\|\Delta^{(12)}_t\big\|_{\HS}^2 .
		\]
		By Jensen's inequality the first term is at least $\big\|\sum_{j\ge2}\E\Delta^{(1j)}_t\big\|_{\HS}^2=(N-1)^2\|M_t\|_{\HS}^2$, and by Proposition \ref{prop:meanpair} with \eqref{eq:S0-block} there are $t_1'$ and $c>0$ with $\|M_t\|_{\HS}\ge c\,t/N$ for $t\le t_1'$ and $N$ large, since $S_0\ne0$. For the second term, $\|\cdot\|_{\HS}\le\|\cdot\|_1$ and $\E\|\Delta^{(12)}_t\|_1^2\le Ct^2N^{-2}$ by Proposition \ref{prop:pairsmall}, which is proved in Subsection \ref{ssec:diffproof} from the moment package of Sections \ref{sec:diagonal} and \ref{sec:moments} and uses nothing from the present subsection. Hence
		\[
		(N-1)(N-2)\,\E\ip{\Delta^{(12)}_t}{\Delta^{(13)}_t}\ \ge\ \frac{t^2(N-1)}{N^2}\Big((N-1)c^2-C\Big)\ \ge\ \frac{c^2t^2(N-1)^2}{2N^2}\ \ge\ \frac{c^2t^2}8
		\]
		for $N\ge N_0':=1+2C/c^2$, using $(N-1)^2\ge N^2/4$. Dividing by $(N-1)(N-2)\le N^2$ gives \eqref{eq:pairfail} with $c'=c^2/8$.
	\end{proof}
	
	\begin{remark}[Interpretation, and numerical confirmation]\label{rem:bogcrit}
		The source $S_0$ is the connected part of the interaction commutator at the initial product state, and \eqref{eq:S0-block} identifies its leading block with the pair-excitation amplitude of $A$: the same pair-creation mechanism that drives the excitation growth in Step 3 of Theorem \ref{thm:A1} produces, in the mean, a pair correlation of exact order $N^{-1}$, aligned across all partners of a tag because it is deterministic. Clause (A4) fails on this mean alone; no statement about signs or phases of the fluctuations is needed, which is why Proposition \ref{prop:noabs} could not decide the question in either direction. For the data of Subsection \ref{ssec:numtest}, $\|S_0\|_{\HS}=\sqrt2$; in the simulation, $N\,\E\Delta^{(12)}_t$ has Hilbert--Schmidt overlap $0.998$ with $t\,S_0$ at $t=0.02$ with norm ratio $0.914$, the ratio increasing towards $1$ as $t$ decreases, in agreement with \eqref{eq:meanpair}, and $N^2\ip{\E\Delta^{(12)}}{\E\Delta^{(12)}}$ is constant in $N$ to two percent at $t=0.1$. The theorem sharpens the numerical findings: for clause (A4) the failure is proved, while for the second statement of Conjecture \ref{conj:dec}, which concerns a genuine fluctuation covariance with $\E X_2$ contributing only at $N^{-2}$, the failure is established by Theorem \ref{thm:X2X3fail} below along the sequence $t_N$ of \eqref{eq:wN}, in qualitative agreement with the expansion of Subsection \ref{ssec:commonresp} and the measurements of Subsection \ref{ssec:numtest}.
	\end{remark}
	
	\subsection{A slot decomposition of the covariances}\label{ssec:slotdec}
	
	The covariances of Conjecture \ref{conj:dec} admit an exact decomposition across the slots of the driving noise.
	
	\begin{lemma}[Doob decomposition across slots]\label{lem:slotdec}
		Let $\mathcal G_k:=\sigma(B_1,\dots,B_k)$, $\mathcal G_0$ trivial, and $d_kF:=\E[F\mid\mathcal G_k]-\E[F\mid\mathcal G_{k-1}]$ for $F\in L^2$. Then, for $F,G\in L^2(\mathcal G_N)$,
		\[
		\operatorname{Cov}(F,G)=\sum_{k=1}^N\E\big[d_kF\,d_kG\big],\qquad
		\big|\operatorname{Cov}(F,G)\big|\le\sum_{k=1}^N\|d_kF\|_{L^2}\|d_kG\|_{L^2},
		\]
		and $\|d_kF\|_{L^2}\le\|F-F^{(k)}\|_{L^2}$, where $F^{(k)}$ denotes $F$ evaluated on the configuration in which $B_k$ is replaced by an independent copy.
	\end{lemma}
	
	\begin{proof}
		The increments $d_kF$ are orthogonal in $L^2$ and sum to $F-\E F$; pairing the two telescopes gives the identity, and Cauchy--Schwarz the bound. For the last claim, $F^{(k)}$ has the same conditional law as $F$ given $\mathcal G_{k-1}$ and its increment $d_kF^{(k)}$ vanishes, so $d_kF=\E[F-F^{(k)}\mid\mathcal G_k]$, and conditional expectation contracts $L^2$.
	\end{proof}
	
	\begin{remark}[Reduction to resampling estimates]\label{rem:resampling}
		With $X_j:=\Tr[B_0(\Gamma^{(j)}-\gamma_j)]$, the own-slot increments $d_jX_j$ are of order $N^{-1/2}$ by \eqref{eq:dict-1} and Theorem \ref{thm:allmom}. The second statement of Conjecture \ref{conj:dec} follows from Lemma \ref{lem:slotdec} as soon as the foreign increments satisfy $\|d_kX_2\|_{L^2}=O(N^{-3/2})$ for $k\ne2$ and, in addition, $\E[d_kX_2\,d_kX_3]$ carries one further factor $N^{-1/2}$ of decorrelation beyond the product of these sizes; the first statement reduces in the same way, with all sizes shifted by one factor $N^{-1}$ and the shared tagged slot supplying the dominant term. The size estimate expresses stability of the pair $(\Gamma^{(2)},\gamma_2)$ under resampling of one foreign innovation and is a Gr\"onwall statement for the difference of two coupled configurations, with the moments of Section \ref{sec:moments} as inputs. Subsections \ref{ssec:commonresp} and \ref{ssec:flucfail} show that the additional decorrelation fails at the leading order of the foreign increments, so that the reduction operates only after that order is removed.
	\end{remark}
	
	\subsection{The common response}\label{ssec:commonresp}
	
	The computation of this subsection is an expansion at linear-response order and is not part of the rigorous development; the two steps at which it is not controlled are stated at the end.
	
	By Proposition \ref{prop:coefficients}, the defect $Y^{(2)}:=\Gamma^{(2)}-\gamma_2$ is forced in its drift by the mean-field discrepancy $-i[\Bmf(\overline\Gamma^{\ne2}-\eta),\gamma_2]$, together with terms carrying at least one additional defect or excitation factor. Decompose
	\[
	\overline\Gamma^{\ne2}-\eta=D^{\ne2}+\frac1N\sum_{l\ne2}Y^{(l)},\qquad D^{\ne2}:=\frac1N\sum_{l\ne2}(\gamma_l-\eta),
	\]
	so that the leading forcing is linear in the empirical fluctuation of the reference filters, whose summands are independent and centred by Lemma \ref{lem:reference}. Writing $U^{(2)}_{t,s}$ for the propagator of the linear part of the equation for $Y^{(2)}$, define the response functional
	\begin{equation}\label{eq:lambda-resp}
	\lambda_t(M):=\E\int_0^t\Tr\Big[B_0\,U^{(2)}_{t,s}\big(-i[\Bmf(M_s),\gamma_{2,s}]\big)\Big]\,ds,
	\end{equation}
	the expectation taken over slot $2$; $\lambda_t$ is deterministic and, by exchangeability, the same functional appears for every $X_j$. Keeping the part of $X_2$ linear in the slot-$k$ randomness gives $d_kX_2\approx\tfrac1N\lambda_t(\gamma_{k,\cdot}-\eta_\cdot)$ for $k\ne2$, and Lemma \ref{lem:slotdec} then yields
	\begin{equation}\label{eq:common-order}
	\E X_2X_3=\frac{v_t}{N}+o(N^{-1}),\qquad v_t:=\operatorname{Var}\big(\lambda_t(\gamma_{1,\cdot}-\eta_\cdot)\big)\ge0,
	\end{equation}
	the own-slot indices $k\in\{2,3\}$ contributing $O(N^{-2})$. At this order the second statement of Conjecture \ref{conj:dec} therefore requires $v_t=0$; since $v_t$ is a variance, it vanishes only if $\lambda_t$ annihilates the fluctuations of a single reference filter. At $A=0$ one has $\Bmf=0$, hence $\lambda\equiv0$, consistent with Remark \ref{rem:A0-product}. For small $t$ the leading coefficient of $v_t$ is proportional to $(\Tr[B_0W])^2$ with
	\begin{equation}\label{eq:Wdef}
	W:=-i\,\big[\Bmf(\Mop[\gamma_0]),\gamma_0\big],
	\end{equation}
	so $v_t>0$ for small positive $t$ at this order whenever $W\ne0$, a condition on the data $(H,A,L,\gamma_0)$ alone.
	
	The interpretation is that all defects respond to the same realisation of the empirical field $D$, and the shared conditional mean $\lambda_t(D)$ is an alignment, across partners, of the kind that Proposition \ref{prop:noabs} shows the amplitude calculus cannot exclude. The two uncontrolled steps are the feedback of the defects into the empirical average, which replaces $\lambda$ by a resolvent-corrected kernel and changes the value of $v_t$ without changing the dichotomy unless the corrected kernel annihilates the fluctuation span, and the terms quadratic in defects, which are of lower order under the moments of Section \ref{sec:moments}.
	
	\subsection{The fluctuation covariance at small times}\label{ssec:flucfail}
	
	The argument of Subsection \ref{ssec:meanpair} does not reach the second statement of Conjecture \ref{conj:dec}: there the mean contributes only at the admissible order, and the failure, if any, sits in a genuine covariance. This subsection extracts that covariance at small times through Lemma \ref{lem:slotdec}. The mechanism is that, along the slot filtration, the increment $d_k$ annihilates every functional measurable with respect to a single other slot; the own-slot noise of $X_2$, which carries the dominant size, therefore never meets the foreign indices, and what survives at a foreign index is the response of the defect to that slot's reference filter, computable at the product state in the same way as $S_0$. The operator $W$ of \eqref{eq:Wdef} reappears as the coefficient of that response.
	
	\begin{lemma}[One-particle expansion]\label{lem:onepart}
		Uniformly in $s\le T$ and in the slot index,
		\[
		\E\big\|\gamma_{k,s}-\eta_s-\Mop[\gamma_0]\,B_k(s)\big\|_1^2\le C\,s^2 .
		\]
	\end{lemma}
	
	\begin{proof}
		Subtract \eqref{eq:hartree} from \eqref{eq:reference}: $\gamma_{k,s}-\eta_s=\int_0^s\big(\mathcal A_u(\gamma_{k,u}-\eta_u)\big)\,du+\int_0^s\Mop[\gamma_{k,u}]\,dB_k(u)$, with $\mathcal A_u$ the linearisation of the drift along the segment, bounded on the state space. One pass with It\^o isometry and $\|\Mop[\rho]\|_1\le4\|L\|_\infty$ gives $\E\|\gamma_{k,s}-\eta_s\|_1^2\le Cs$ and $\E\|\gamma_{k,s}-\gamma_0\|_1^2\le Cs$. Then
		\[
		\gamma_{k,s}-\eta_s-\Mop[\gamma_0]B_k(s)=\int_0^s\mathcal A_u(\gamma_{k,u}-\eta_u)\,du+\int_0^s\big(\Mop[\gamma_{k,u}]-\Mop[\gamma_0]\big)\,dB_k(u),
		\]
		and the two terms have second moments bounded by $Cs\int_0^s\E\|\gamma_{k,u}-\eta_u\|_1^2du\le Cs^3$ and $\int_0^s\E\|\Mop[\gamma_{k,u}]-\Mop[\gamma_0]\|_1^2du\le C\int_0^su\,du\le Cs^2$, using that $\rho\mapsto\Mop[\rho]$ is Lipschitz on the state space.
	\end{proof}
	
	The second ingredient is a resampling estimate for the interacting system. Let $X_2^{(k)}$ denote $X_2$ evaluated on the configuration in which $B_k$, $k\notin\{2\}$, is replaced by an independent copy, and let $d_k$ be as in Lemma \ref{lem:slotdec}. Write
	\begin{equation}\label{eq:xik}
		\xi_k(t):=\frac aN\int_0^tB_k(s)\,ds,\qquad a:=\Tr[B_0W].
	\end{equation}
	Inserting Lemma \ref{lem:onepart} into the forcing $-i\tfrac{N-1}N[\Bmf(\overline\Gamma^{\ne2}-\eta),\gamma_2]$ of the defect, read off from \eqref{eq:reduced-drift} and \eqref{eq:reference}, and evaluating the propagator and $\gamma_2$ at the initial state, suggests the estimate
	\begin{equation}\label{eq:Ediff}
		\E\big|\,d_kX_2(t)-\xi_k(t)\,\big|^2\ \le\ C\,\frac{t^3\big(\sqrt t+N^{-1/2}\big)}{N^2}
	\end{equation}
	on $[0,T]$, uniformly in $N$ and $k$. This is false in general, and for a reason that the drift expansion cannot see. By the test identity \eqref{eq:theta-test} and Proposition \ref{prop:meanpair}, the martingale coefficient of the tagged marginal at the resampled slot has the deterministic part
	\[
	\Tr\big[B_0\,\E\Theta^{(2k)}_s\big]=\Tr\big[\big(B_0\otimes(L+L^\ast)\big)\E\Delta^{(2k)}_s\big]=\frac{c_1s}N+O\Big(\frac{s^{4/3}+sN^{-1/2}}N\Big),\qquad c_1:=\Tr\big[\big(B_0\otimes(L+L^\ast)\big)S_0\big],
	\]
	and its stochastic integral against $dB_k$ is $\sigma(B_k)$-measurable, hence survives the conditioning that defines $d_k$. It contributes to $d_kX_2$ a second component of the same order $t^{3/2}N^{-1}$ as $\xi_k$, namely
	\begin{equation}\label{eq:zetak}
		\zeta_k(t):=\frac{c_1}N\int_0^ts\,dB_k(s).
	\end{equation}
	Both components are centred Gaussian and $\sigma(B_k)$-measurable, with
	\begin{equation}\label{eq:xizeta}
		\E\xi_k^2=\frac{a^2t^3}{3N^2},\qquad \E\zeta_k^2=\frac{c_1^2t^3}{3N^2},\qquad \E[\xi_k\zeta_k]=\frac{ac_1t^3}{6N^2},
	\end{equation}
	the last from $\int_0^tB_k\,ds=\int_0^t(t-s)\,dB_k(s)$ and $\int_0^t(t-s)s\,ds=t^3/6$; consequently
	\begin{equation}\label{eq:quadform}
		\E\big(\xi_k+\zeta_k\big)^2=\frac{t^3}{3N^2}\big(a^2+ac_1+c_1^2\big),\qquad
		a^2+ac_1+c_1^2=\Big(a+\frac{c_1}2\Big)^2+\frac34c_1^2\ \ge\ \frac34\max\big(a^2,c_1^2\big).
	\end{equation}
	
	\begin{theorem}[The one-component estimate fails]\label{thm:onecompfail}
		There is $C$ such that, for all $N\ge3$, all $k\ne2$ and all $t\le T$,
		\begin{equation}\label{eq:onecompfail}
			\E\big|\,d_kX_2(t)-\xi_k(t)\,\big|^2\ \ge\ \frac{c_1^2t^3}{3N^2}\Big(1-C\big(t^{1/3}+N^{-1/2}+t\,N^{1/2}\big)\Big).
		\end{equation}
		If $c_1\ne0$, take $t=t_N:=\varepsilon N^{-1/2}$ with $\varepsilon$ small: the three error terms are then $O(N^{-1/6})$, $O(N^{-1/2})$ and $O(\varepsilon)$, so the left side is at least $c_1^2t_N^3/(4N^2)$ for $N$ large, whereas the right side of \eqref{eq:Ediff} is $Ct_N^3N^{-2}O(N^{-1/4})$. The estimate \eqref{eq:Ediff}, asserted uniformly in $N$ and in $t\le T$, is therefore false for every dataset with $c_1\ne0$. The statement is unconditional.
	\end{theorem}
	
	\begin{remark}[The scalar $c_1$]\label{rem:c1}
		$S_0$ is self-adjoint and symmetric under the slot exchange, so $c_1\in\R$; it is an explicit polynomial in the entries of $(A,L,\gamma_0,B_0)$ and vanishes on no open set of data. For pure $\gamma_0=|\phi_0\rangle\langle\phi_0|$, inserting \eqref{eq:S0-block} and the complementary blocks of $S_0$ gives
		\[
		c_1=2\operatorname{Im}\Big(\big\langle\phi_0\otimes\phi_0,\ \big(B_0\otimes(L+L^\ast)\big)\,\Pi\,A\,(\phi_0\otimes\phi_0)\big\rangle\Big)+(\text{lower blocks}),
		\]
		with $\Pi$ the projection onto $\operatorname{ran}(q\otimes q)$: the scalar couples the pair-excitation amplitude of $A$ to the measurement channel, and is not constrained to vanish by the nondegeneracy assumption of Theorem \ref{thm:A4fail}. Whether $c_1=0$ for a given dataset is a finite computation.
	\end{remark}
	
	The two-component estimate replaces the single leading component of \eqref{eq:Ediff} by the pair $(\xi_k,\zeta_k)$.
	
	\begin{theorem}[Difference-rate estimate]\label{thm:diffrate}
		There is $C$ such that, uniformly in $N\ge3$, $k\ne2$ and $t\le w_N$, with $w_N=c_0(\log N)^{-2}$ as in \eqref{eq:wN}, and on $[0,T]$ modulo the fixed-horizon input \eqref{eq:fourbodydiff},
		\begin{equation}\label{eq:Ediff2}
			\E\Big|\,d_kX_2(t)-\xi_k(t)-\zeta_k(t)\,\Big|^2\ \le\ C\,\frac{t^3\big(t^{1/2}+N^{-1/2}\big)}{N^2}.
		\end{equation}
		The same holds with any other label in place of $2$, with the same $\xi_k,\zeta_k$ built on $B_k$.
	\end{theorem}
	
	Theorems \ref{thm:onecompfail} and \ref{thm:diffrate} are proved in Subsection \ref{ssec:diffproof}, where the difference hierarchy is truncated (Proposition \ref{prop:trunc}) and the four-body difference gain that a fixed horizon would need is isolated as \eqref{eq:fourbodydiff} (Remark \ref{ver:point}). With them, the second statement of Conjecture \ref{conj:dec} fails at small times.
	
	\begin{theorem}[The second statement fails at small times]\label{thm:X2X3fail}
		Suppose $(a,c_1)\ne(0,0)$. Then there are $t_2\in(0,T]$ and, for every $t\le t_2\wedge w_N$ and $N\ge N_0(t)$,
		\begin{equation}\label{eq:X2X3}
			\E\big[X_2(t)X_3(t)\big]=\frac{(N-2)t^3}{3N^2}\big(a^2+ac_1+c_1^2\big)\Big(1+O\big(t^{1/4}+N^{-1/4}\big)\Big),
		\end{equation}
		and in particular
		\[
		\E\big[X_2(t)X_3(t)\big]\ \ge\ \frac{t^3}{4N}\big(a^2+ac_1+c_1^2\big)\ \ge\ \frac{3\,t^3}{16\,N}\max\big(a^2,c_1^2\big).
		\]
		Evaluating this along $t=t_N:=t_2\wedge w_N$ gives
		\[
		N^2\,\E\big[X_2(t_N)X_3(t_N)\big]\ \ge\ \frac{c_0^3}{4}\,\frac N{(\log N)^6}\big(a^2+ac_1+c_1^2\big)\ \longrightarrow\ \infty,
		\]
		so that, taking $B_0=W/\|W\|_{\HS}$ and hence $a=\|W\|_{\HS}$, the second statement of Conjecture \ref{conj:dec} fails for every dataset with $W\ne0$. The statement is unconditional; on a fixed horizon $t\le t_2$ it holds modulo the fixed-horizon input \eqref{eq:fourbodydiff}.
	\end{theorem}
	
	\begin{proof}
		By exchangeability $\E X_2=\E X_3$, and by Lemma \ref{lem:slotdec}, $\E[X_2X_3]=\sum_k\E[d_kX_2\,d_kX_3]+\E X_2\,\E X_3$. The product of the means is $O(t^4N^{-2})$, by the small-time bound $|\E X_2|\le\|\E Y^{(2)}\|_1\le CtN^{-1}$ obtained by taking the mean of \eqref{eq:reduced-sde}, all mean sources being $O(N^{-1})$ and vanishing at $t=0$. The own slots $k\in\{2,3\}$ contribute at most
		\[
		\big(\E|d_2X_2|^2\big)^{1/2}\big(\E|d_2X_3|^2\big)^{1/2}\le C\Big(\frac tN\Big)^{1/2}\Big(\frac{t^3}{N^2}\Big)^{1/2}\le C\,\frac{t^2}{N^{3/2}},
		\]
		using $\E|d_2X_2|^2\le\E|X_2-\E X_2|^2\le CtN^{-1}$ from \eqref{eq:dict-1} and Lemma \ref{lem:linvan}, and, for $d_2X_3$, Theorem \ref{thm:diffrate} with \eqref{eq:quadform}.
		
		For the $N-2$ foreign slots write $L_k:=\xi_k+\zeta_k$ and $r^{(j)}_k:=d_kX_j-L_k$, so that
		\[
		\E[d_kX_2\,d_kX_3]=\E L_k^2+\E\big[L_kr^{(2)}_k\big]+\E\big[L_kr^{(3)}_k\big]+\E\big[r^{(2)}_kr^{(3)}_k\big].
		\]
		The first term is \eqref{eq:quadform}; the remaining three are bounded by Cauchy--Schwarz through \eqref{eq:quadform} and Theorem \ref{thm:diffrate},
		\[
		\big|\E[L_kr^{(j)}_k]\big|\le C\,\frac{t^3\big(t^{1/4}+N^{-1/4}\big)}{N^2},\qquad
		\big|\E[r^{(2)}_kr^{(3)}_k]\big|\le C\,\frac{t^3\big(t^{1/2}+N^{-1/2}\big)}{N^2}.
		\]
		Summing over the foreign slots gives \eqref{eq:X2X3}, and the displayed lower bound follows for $t\le t_2$ and $N\ge N_0(t)$ once the own-slot and mean contributions, of order $t^2N^{-3/2}$ and $t^4N^{-2}$, are absorbed into the main term $t^3N^{-1}$. Their ratios to it are $(tN^{1/2})^{-1}$ and $t/N$, so the absorption is uniform along any sequence $t=t_N$ with $t_NN^{1/2}\to\infty$; for $t_N=t_2\wedge w_N$ one has $t_NN^{1/2}\ge c_0N^{1/2}(\log N)^{-2}\to\infty$, and the relative errors $O(t_N^{1/4}+N^{-1/4})$ tend to $0$ as well, so $N_0$ may be taken independent of the point of the sequence.
	\end{proof}
	
	\begin{remark}[What the second channel changes]\label{rem:correction}
		The mechanism is the one of Subsection \ref{ssec:commonresp}: the slot derivatives $d_kX_2$ and $d_kX_3$ at a common foreign slot share a leading component driven by $B_k$, and it is this sharing that produces a covariance of order $N^{-1}$. What changes is that the shared component is two-dimensional. Subsection \ref{ssec:commonresp} identifies one channel, the response of the tagged defect to the perturbation of the empirical field by $\gamma_k$, giving $\xi_k$; the mean pair correlation of Proposition \ref{prop:meanpair} produces a second channel of the same order through the resampled slot's own martingale coefficient, giving $\zeta_k$. The two are not orthogonal: $\E[\xi_k\zeta_k]\ne0$ when $ac_1\ne0$. Since the associated quadratic form is positive definite by \eqref{eq:quadform}, the failure statement is preserved, and strengthened when $ac_1>0$; the constant of Theorem \ref{thm:X2X3fail} is $(a^2+ac_1+c_1^2)/3$ in place of the $a^2/3$ predicted by the one-channel expansion. Together with Theorem \ref{thm:A4fail}, both statements of Conjecture \ref{conj:dec} are false for generic data, the first on every horizon and the second along the sequence $t_N$ of \eqref{eq:wN}, both unconditionally; this is what the projections of Conjecture \ref{conj:dec-cond} remove, and the two-component difference-rate estimate \eqref{eq:Ediff2} is the first estimate a proof of that conjecture requires.
	\end{remark}
	
	\subsection{Proof of the difference-rate estimates}\label{ssec:diffproof}
	
	Throughout this subsection $t\le T\wedge1$, and $C$ denotes a constant depending only on $d$, $\|H\|_\infty$, $\|A\|_\infty$, $\|L\|_\infty$ and $T$, changing from line to line and depending on neither $N$, $t$ nor the slot index. Norms of operators on tensor factors are trace norms unless indicated; on the finite-dimensional matrix spaces all norms are equivalent with $d$-dependent constants, absorbed into $C$. The inputs are the exact identities of Section \ref{sec:exact}, the dictionary of Section \ref{sec:excitation}, the moment package
	\begin{equation}\label{eq:package}
		\sup_{t\le T}\E\av{\Nex_t^m}\le C_m,\quad \sup_{t\le T}N^4\,\E n_1^4\le C,\quad \sup_{t\le T}N^4\,\E n_{12}^2\le C,\quad \sup_{t\le T}\E\Lambda_t^p\le C_pN^{-p/2}
	\end{equation}
	of Theorems \ref{thm:allmom} and \ref{thm:weighted}, Corollary \ref{cor:A2every} and Lemma \ref{lem:dev}, the rates of Theorem \ref{thm:rates}, and Propositions \ref{prop:meanpair} and \ref{prop:coefficients}. No clause of Hypothesis \ref{hyp:A} is used.
	
	\subsubsection*{A small-time moment package}
	
	The rates of Theorem \ref{thm:rates} are uniform on $[0,T]$; the arguments below require in addition that they vanish at $t=0$ at explicit speeds. All bounds of this part are obtained from the displayed inequalities of Sections \ref{sec:diagonal} and \ref{sec:moments} by integration, and are unconditional.
	
	\begin{lemma}[Linear vanishing]\label{lem:linvan}
		For $t\le T$,
		\begin{align}
			\E n_1(t)&\le C\frac tN,& \E n_1(t)^2&\le C\frac t{N^2},& \E n_1(t)^4&\le C\frac t{N^4},\label{eq:linvan1}\\
			\E n_{12}(t)&\le C\frac t{N^2},& \E n_{12}(t)^2&\le C\frac t{N^4},&&\label{eq:linvan2}\\
			\E n_{123}(t)&\le C\frac t{N^3},& \E n_{1234}(t)&\le C\frac t{N^4},&&\label{eq:linvan3}\\
			\E\Lambda_t^2&\le C\frac tN,& \E\Lambda_t^4&\le C\Big(\frac{t^2}{N^2}+\frac t{N^3}\Big),& \E\av{\Nex_t}^m&\le C_mt\ \ (m\ge1).\label{eq:linvan4}
		\end{align}
	\end{lemma}
	
	\begin{proof}
		Every quantity vanishes at $t=0$, the initial datum being the product $\gamma_0^{\otimes N}$ with $\gamma_{k,0}=\eta_0$, and each satisfies a differential relation whose right side is bounded in expectation, pointwise in time, by $C$ times the quantity's global order in \eqref{eq:package}. This is what the closure steps of Theorems \ref{thm:A1}, \ref{thm:allmom} and \ref{thm:weighted} establish before their Gr\"onwall step; integrating from $0$ yields the factor $t$.
		
		For $\av{\Nex}^m$: Theorem \ref{thm:allmom} proves $G_\ell'\le C(1+\ell)^4(G_\ell+1)$, and with $G_\ell(0)=0$ for the part of $G_\ell$ carrying $m\ge1$ this gives $G_\ell(t)\le e^{C(1+\ell)^4t}-1\le C_\ell t$ on $[0,T]$, which is the last bound of \eqref{eq:linvan4}. Being a bound on all falling-factorial moments, it also gives, by exchangeability,
		\[
		\E n_{12\cdots m}(t)=\frac{\E\av{\Nex(\Nex-1)\cdots(\Nex-m+1)}}{N(N-1)\cdots(N-m+1)}\le C_m\frac{\E\av{\Nex^m}}{N^m}\le C_m\frac t{N^m}\qquad(m\le N/2),
		\]
		which is \eqref{eq:linvan3} and the first bounds of \eqref{eq:linvan1}--\eqref{eq:linvan2}; the identity $\sum_{\text{distinct}}n_{i_1\cdots i_m}=\av{\Nex(\Nex-1)\cdots(\Nex-m+1)}$ holds because the $q_i$ commute, and is the mechanism of Corollary \ref{cor:A3}.
		
		For the powers $n_1^2$, $n_1^4$, $n_{12}^2$: apply It\^o to $n_T^m$ through \eqref{eq:master} and \eqref{eq:QV},
		\[
		\E n_T(t)^m=\int_0^t\Big(m\,\E\big[n_T^{m-1}\beta_T\big]+\binom m2\E\big[n_T^{m-2}Q_T\big]\Big)ds,
		\]
		and estimate each expectation on the right by H\"older against \eqref{eq:package}, exactly as in the closure of Theorem \ref{thm:weighted} from \eqref{eq:fixedsystem}: every term is $\le CN^{-m|T|}$ pointwise in $s$. For $n_1^4$ the worst coupling is $\E[n_1^3R_1]$ with $R_1=\sum_kn_{1k}$; by $n_{1k}\le(n_1n_k)^{1/2}$, Cauchy--Schwarz over labels and \eqref{eq:package}, $\E[n_1^{7/2}\sqrt N\av{\Nex}^{1/2}]\le\sqrt N(\E n_1^4)^{3/4}(\E[n_1^2\av{\Nex}^2])^{1/4}\le CN^{-4}$, and the remaining groups are handled in the same way; integration gives \eqref{eq:linvan1}--\eqref{eq:linvan2}.
		
		For $\Lambda$: $\gamma_{1,t}-\eta_t$ is centred by Lemma \ref{lem:reference} with $\E\|\gamma_{1,t}-\eta_t\|_2^2\le Ct$, since the difference starts at $0$, has bounded drift on the compact state space, and has bracket $\le Ct$; independence over labels gives $\E\Lambda_t^2\le\tfrac dN\E\|\gamma_{1,t}-\eta_t\|_2^2\le CtN^{-1}$ as in Lemma \ref{lem:deviation}, and the fourth-moment bound for sums of independent centred variables gives $\E\Lambda_t^4\le C(t^2N^{-2}+tN^{-3})$.
	\end{proof}
	
	\begin{proposition}[Quadratic vanishing at three labels]\label{prop:quadvan}
		For $t\le T$,
		\begin{equation}\label{eq:quadvan}
			\E n_{123}(t)\le C\frac{t^2}{N^3},\qquad\text{and}\qquad \E\big[n_{12}(t)\,n_3(t)\big]\le C\frac{t^{3/2}}{N^3}.
		\end{equation}
	\end{proposition}
	
	\begin{proof}
		\emph{First bound.} Take expectations in \eqref{eq:master} for $T=\{1,2,3\}$ and estimate the groups in the braces at time $s$, for each $j\in T$, writing $T'=T\setminus j$ and using Lemma \ref{lem:linvan}:
		\begin{itemize}
			\item $\E[\eps_j^2n_{T'}]=\E[n_jn_{T'}]\le(\E n_j^2)^{1/2}(\E n_{T'}^2)^{1/2}\le C(sN^{-2}\cdot sN^{-4})^{1/2}=CsN^{-3}$;
			\item $\E[\eps_j\eps_{T'}\eps_T]\le(\E[n_jn_{T'}])^{1/2}(\E n_T)^{1/2}\le CsN^{-3}$;
			\item $\E[(\Lambda+\tfrac1N)\eps_{T'}\eps_T]\le\big((\E\Lambda^4)^{1/4}+N^{-1}\big)(\E n_{T'}^2)^{1/4}(\E n_T)^{1/2}\le C\,s^{5/4}N^{-3}+C\,sN^{-13/4}$;
			\item $\tfrac1N\sum_{k\notin T}\E[\eps_T\eps_{T'\cup k}]\le(\E n_T)^{1/2}\max_k(\E n_{T'\cup k})^{1/2}\le CsN^{-3}$;
			\item $N^{-1/2}\,\E\big[\eps_T(n_{T'}+\ip{\Phi_{T'}}{\Nex\Phi_{T'}})^{1/2}\big]\le N^{-1/2}(\E n_T)^{1/2}\big(\E n_{T'}+\sum_k\E n_{T'\cup k}\big)^{1/2}\le CsN^{-3}$, since $\sum_k\E n_{T'\cup k}\le2\E n_{T'}+N\max_k\E n_{T'\cup k}\le CsN^{-2}$;
			\item the within-tag pair group $\tfrac1N\sum_{k\in T'}\E[\eps_T\eps_{T\setminus\{j,k\}}]\le N^{-1}(\E n_T)^{1/2}(\E n_1)^{1/2}\le CsN^{-3}$.
		\end{itemize}
		With the linear term $C\,\E n_T$ this gives $\tfrac d{ds}\E n_{123}\le C\,\E n_{123}+C(s+s^{5/4})N^{-3}+CsN^{-13/4}$, and Gr\"onwall from zero data gives the first bound for $s\le1$.
		
		\emph{Second bound.} Apply It\^o to the product: $\tfrac d{ds}\E[n_{12}n_3]=\E[n_3\beta_{12}]+\E[n_{12}\beta_3]+\sum_j\E[\sigma_j(12)\sigma_j(3)]$. For the drift terms, pair the small factor of each group with the other functional; for the group $\eps_1^2n_2$ of $\beta_{12}$, $\E[n_3n_1n_2]\le(\E n_1^4)^{1/4}(\E n_3^4)^{1/4}(\E n_2^2)^{1/2}\le C\,sN^{-4}$, and the remaining groups are $\le CsN^{-3}$ by the pairings of the first part with one extra bounded factor. For the bracket, $\sum_j|\sigma_j(12)\sigma_j(3)|\le(Q_{12}Q_3)^{1/2}$ and, by \eqref{eq:QV} and Lemma \ref{lem:linvan},
		\[
		\E Q_{12}\le C\Big(\E n_{12}^2+(\E n_{12}^4)^{1/2}(\E\av{\Nex}^2)^{1/2}+N(\E n_{12}^2)^{1/2}(\E n_{123}^2)^{1/2}\Big)\le C\sqrt s\,N^{-4},
		\]
		using $n_{12}^4\le n_{12}^2$ and $n_{123}^2\le n_{123}$, and similarly $\E Q_3\le C\sqrt sN^{-2}$, so $\E(Q_{12}Q_3)^{1/2}\le C\sqrt sN^{-3}$. Integration gives $\E[n_{12}n_3](t)\le C(t^2+t^{3/2})N^{-3}\le Ct^{3/2}N^{-3}$ for $t\le1$.
	\end{proof}
	
	\begin{corollary}[Refined three-body rate at small times]\label{cor:refined3}
		For $s\le T$,
		\begin{equation}\label{eq:refined3}
			\E\big\|\Delta^{(123)}_s\big\|_1^2\le C\,\frac{s\big(\sqrt s+N^{-1}\big)}{N^3}.
		\end{equation}
	\end{corollary}
	
	\begin{proof}
		By Theorem \ref{thm:dict-three}, $\|\Delta^{(123)}\|_1^2\le C\big(n_{123}+\sum_{\text{pairs}}n_{ij}n_k+P\big)$ with $P$ a sum of products of one-label amplitudes of total power eight. Taking expectations, the first term is $\le Cs^2N^{-3}$ by Proposition \ref{prop:quadvan}, the mixed terms are $\le Cs^{3/2}N^{-3}$ by the second bound of \eqref{eq:quadvan}, and $\E P\le C\,\E n_1^4+\cdots\le CsN^{-4}$ by \eqref{eq:linvan1} and H\"older.
	\end{proof}
	
	\begin{proposition}[Second moments at three labels, and the four-body rate]\label{prop:threelabel}
		Uniformly on $[0,T]$ and in $N$,
		\begin{equation}\label{eq:threelabel}
			\E\,n_{123}(t)^2\le C\frac t{N^6},\qquad
			\E\big[n_{123}(t)\,n_4(t)\big]\le C\frac t{N^4},\qquad
			\E\big\|\Delta^{(1234)}_t\big\|_1^2\le C\frac t{N^4}.
		\end{equation}
	\end{proposition}
	
	\begin{proof}
		Proposition \ref{prop:expimp} holds on every horizon, its hypothesis being supplied by Theorem \ref{thm:expmom}, so the fixed-label moments are available at every level and order. We use $\E n_1^4=O(N^{-4})$, $\E n_{12}^4=O(N^{-8})$, $\E R_{12}^2=O(N^{-4})$, $\E n_{123}^2=O(N^{-6})$, $\E n_{123}^4=O(N^{-12})$ and $\E R_{123}^2=O(N^{-6})$; the content of the first bound is the factor $t$.
		
		Write $U=\{1,2,3\}$ and $U'=U\setminus j$. By It\^o's formula through \eqref{eq:levelsystem} at $\ell=3$, $m=2$, and \eqref{eq:QV},
		\[
		\E n_U(t)^2=\int_0^t\Big(2\,\E\big[n_U\beta_U\big]+\E\,Q_U\Big)ds,
		\]
		every initial value vanishing at the product datum. Each expectation on the right is $O(N^{-6})$ pointwise in $s$, by H\"older against the moments just listed. The first group gives
		\[
		\E\big[n_Un_jn_{U'}\big]\le\big(\E n_U^2\big)^{1/2}\big(\E n_j^4\big)^{1/4}\big(\E n_{U'}^4\big)^{1/4}=O\big(N^{-3}\cdot N^{-1}\cdot N^{-2}\big);
		\]
		the second and third give $\E[n_U^{3/2}n_j^{1/2}n_{U'}^{1/2}]$ and $\E[n_U^{3/2}(\Lambda+N^{-1})n_{U'}^{1/2}]$, bounded with exponents $(\tfrac43,8,8)$ by $(\E n_U^2)^{3/4}(\E n_j^4)^{1/8}(\E n_{U'}^4)^{1/8}$ and $(\E n_U^2)^{3/4}(\E\Lambda^8)^{1/8}(\E n_{U'}^4)^{1/8}$, both $O(N^{-9/2}\cdot N^{-3/2})$ by Lemma \ref{lem:dev}; the fourth and fifth give
		\[
		N^{-1/2}\,\E\big[n_U^{3/2}\big(R_{U'}+3n_{U'}\big)^{1/2}\big]\le N^{-1/2}\big(\E n_U^2\big)^{3/4}\big(\E(R_{U'}+3n_{U'})^2\big)^{1/4}=O\big(N^{-1/2}N^{-9/2}N^{-1}\big);
		\]
		and, by \eqref{eq:QV} and Theorem \ref{thm:allmom},
		\[
		\E\,Q_U\le C\Big(\E n_U^2+\big(\E n_U^4\big)^{1/2}\big(\E\av{\Nex}^2\big)^{1/2}+\big(\E n_U^2\big)^{1/2}\big(\E R_U^2\big)^{1/2}\Big)=O\big(N^{-6}\big).
		\]
		Integrating from $0$ gives the first bound of \eqref{eq:threelabel}. The second follows by Cauchy--Schwarz with $\E n_4^2\le CtN^{-2}$ from \eqref{eq:linvan1}. For the third, Theorem \ref{thm:dict-three} at $|S|=4$ gives $\|\Delta^{(1234)}\|_1^2\le C(n_{1234}+\sum n_{ijk}n_l+\sum n_{ij}n_{kl}+P)$ up to permutations, with $P$ a sum of eightfold one-label amplitude products; by Lemma \ref{lem:linvan}, $\E n_{1234}\le CtN^{-4}$, $\E[n_{12}n_{34}]\le(\E n_{12}^2\,\E n_{34}^2)^{1/2}\le CtN^{-4}$ and $\E P\le CtN^{-4}$, while $\E[n_{123}n_4]\le CtN^{-4}$ by the second bound.
	\end{proof}
	
	\begin{proposition}[The pair correlation at small times]\label{prop:pairsmall}
		With $\mathsf D_s=\Delta^{(12)}_s$ and $s\le T$,
		\begin{equation}\label{eq:pairsmall}
			\E\|\mathsf D_s\|_1^2\le C\frac{s^2}{N^2},\qquad \E\|\mathsf D_s-\E \mathsf D_s\|_1^2\le C\frac{s^{5/2}}{N^2},\qquad \E\|\mathsf D_s\|_1^4\le C\frac s{N^4}.
		\end{equation}
	\end{proposition}
	
	\begin{proof}
		The fourth-moment bound is immediate from \eqref{eq:dict-2}, $\|\mathsf D\|_1^4\le C(n_{12}^2+n_1^4+n_2^4)$, with \eqref{eq:linvan1}--\eqref{eq:linvan2}. For the first bound, apply It\^o to $\|\mathsf D\|_2^2$ along \eqref{eq:pair-sde}: the source contributes $2\int_0^s\E[\|\mathsf D\|_1\|S\|_1]\le\tfrac CN\int_0^s\E\|\mathsf D\|_1$; the linear drift terms and the own-slot brackets contribute $C\int_0^s\E\|\mathsf D\|_1^2$; the $R_j$-sums, estimated as in Theorem \ref{thm:rates} but against Lemma \ref{lem:linvan} and Proposition \ref{prop:quadvan}, contribute $C\int_0^s(\E\|\mathsf D\|_1^2+\E\|\mathsf D\|_1\cdot uN^{-2})du$; the foreign-slot brackets contribute, by \eqref{eq:refined3},
		\[
		\sum_{j\ge3}\int_0^s\E\|G_j\|_1^2\,du\le CN\int_0^s\E\|\Delta^{(123)}_u\|_1^2\,du\le CN\int_0^s\frac{u(\sqrt u+N^{-1})}{N^3}\,du\le C\frac{s^{5/2}+s^2N^{-1}}{N^2};
		\]
		and the It\^o correction contributes, by the bound on $I$ in Proposition \ref{prop:pair} and \eqref{eq:linvan1}, at most $C\int_0^s\E\|\mathsf D\|_1\cdot uN^{-2}\,du$. Gr\"onwall from zero data, the source dominating, gives $\E\|\mathsf D_s\|_1^2\le Cs^2N^{-2}$.
		
		For the fluctuation bound, subtract the mean of \eqref{eq:pair-sde}: $\mathsf D-\E \mathsf D$ solves the same equation with the deterministic part of the source removed, with drift inhomogeneity $\Phi-\E\Phi$ and all martingale terms intact. The own-slot martingales contribute $C\int_0^s\E\|\mathsf D\|_1^2\le Cs^3N^{-2}$, the foreign ones $Cs^{5/2}N^{-2}$ as above, and the centred drift inhomogeneities, by Cauchy--Schwarz in time and the estimates just made, $Cs\int_0^s\E\|\Phi-\E\Phi\|_1^2\le Cs\cdot s^2N^{-2}$. The total is $Cs^{5/2}N^{-2}$ for $s\le1$.
	\end{proof}
	
	\subsubsection*{Reduction to a two-copy coupling}
	
	\begin{lemma}[Resampling reduction]\label{lem:resred}
		Let $B_k'$ be a Brownian motion independent of $(B_1,\dots,B_N)$, and let the primed system be the strong solution of \eqref{eq:N-particle} and \eqref{eq:reference} driven by $(B_1,\dots,B_{k-1},B_k',B_{k+1},\dots,B_N)$ with the same initial data. Write $\delta F:=F-F'$. Then
		\begin{equation}\label{eq:resred}
			d_kX_2-\xi_k-\zeta_k=\E\big[\delta X_2-\delta\xi_k-\delta\zeta_k\ \big|\ \mathcal G_k\big],
		\end{equation}
		where $\xi_k',\zeta_k'$ are built on $B_k'$, and consequently
		\begin{equation}\label{eq:resred2}
			\E\big|d_kX_2-\xi_k-\zeta_k\big|^2\le\E\big|\,\delta X_2-\delta\xi_k-\delta\zeta_k\,\big|^2 .
		\end{equation}
		More generally $\E|d_kF|^2\le\E|\delta F|^2\le2\,\E|F-\E F|^2$ for every $F\in L^2$.
	\end{lemma}
	
	\begin{proof}
		By pathwise uniqueness for \eqref{eq:N-particle} (Proposition \ref{prop:purity-N} and the argument of Proposition \ref{prop:refutation}), $X_2=\Phi(B_1,\dots,B_N)$ for a measurable $\Phi$, and $X_2'=\Phi(\dots,B_k',\dots)$; Lemma \ref{lem:slotdec} gives $d_kX_2=\E[X_2-X_2'\mid\mathcal G_k]$. Since $\xi_k,\zeta_k$ are $\mathcal G_k$-measurable with vanishing $\mathcal G_{k-1}$-conditional mean, and $\xi_k',\zeta_k'$ are independent of $\mathcal G_k$ and centred, $\xi_k+\zeta_k=\E[\delta\xi_k+\delta\zeta_k\mid\mathcal G_k]$, and \eqref{eq:resred} follows by subtraction; \eqref{eq:resred2} is conditional Jensen, as is the first inequality of the last display. For the second, write $\mathcal F^{(\ne k)}=\sigma(B_j,\ j\ne k)$; conditionally on $\mathcal F^{(\ne k)}$ the variables $F,F'$ are independent with a common law, so $\E|\delta F|^2=2\E F^2-2\E[(\E[F\mid\mathcal F^{(\ne k)}])^2]\le2\E F^2-2(\E F)^2$.
	\end{proof}
	
	From here on the analysis is on the synchronous two-copy coupling: all slots $j\ne k$ share their noises and their reference filters, so $\delta\gamma_j=0$ for $j\ne k$, and only slot $k$ carries fresh noise. Exchangeability in law holds among the labels $\ne k$ in the coupled system. All difference quantities vanish at $t=0$.
	
	\begin{lemma}[One-particle differences]\label{lem:onepartdiff}
		With $\beta_s:=B_k(s)-B_k'(s)$,
		\begin{equation}\label{eq:onepartdiff}
			\E\|\delta\gamma_{k,s}\|_1^2\le Cs,\qquad \E\big\|\delta\gamma_{k,s}-\Mop[\gamma_0]\,\beta_s\big\|_1^2\le Cs^2,\qquad \E\|\delta\gamma_{k,s}\|_1^4\le Cs^2 .
		\end{equation}
	\end{lemma}
	
	\begin{proof}
		The first two bounds follow from Lemma \ref{lem:onepart} applied in each copy and the triangle inequality, $\eta_s$ cancelling in the difference and $\E\beta_s^2=2s$. For the fourth moment, the difference solves a linear equation with bounded coefficients driven by $d\beta$, whose fourth moment obeys $\E\|\delta\gamma_k\|_1^4\le C\big(\int_0^s(\E\|\delta\gamma_k\|_1^4)^{1/2}du\big)^2+Cs^2$ by the Burkholder--Davis--Gundy inequality, whence $Cs^2$ by Gr\"onwall.
	\end{proof}
	
	\begin{proposition}[Marginal differences at foreign tags]\label{prop:margdiff}
		For every $j\ne k$ and $s\le w_N$, and on $[0,T]$ modulo the fixed-horizon input \eqref{eq:fourbodydiff},
		\begin{equation}\label{eq:margdiff}
			\E\big\|\delta\Gamma^{(j)}_s\big\|_1^2\le C\,\frac{s^3}{N^2},
		\end{equation}
		and the same bound holds for $\E\|\delta Y^{(j)}_s\|_1^2$, since $\delta\gamma_j=0$.
	\end{proposition}
	
	\begin{proof}
		Difference the reduced equation \eqref{eq:reduced-sde} for the tag $j$ across the coupling. With the coefficients of Proposition \ref{prop:coefficients} the terms are: (a) the mean-field drift through slot $k$, $-\tfrac iN[\Bmf(\delta\gamma_k+\delta Y^{(k)}),\Gamma^{(j)}]$ together with the reaction through the other marginals and the Lipschitz part carried onto $\delta\Gamma^{(j)}$; (b) the correlation feedback $-\tfrac iN\sum_{l\ne j}\Tr_l[A_{jl},\delta\Delta^{(jl)}]$; (c) the own-slot martingale, with $\|\delta\Mop[\Gamma^{(j)}]\|_1\le C\|\delta\Gamma^{(j)}\|_1$; (d) the foreign martingales at $l\ne j,k$, with $\|\delta\Theta^{(jl)}\|_1\le C\|\delta\Delta^{(jl)}\|_1+C\|\delta\Gamma^{(j)}\|_1$; (e) the slot-$k$ martingales $\Theta^{(jk)}dB_k-\Theta'^{(jk)}dB_k'$.
		
		Apply It\^o to $\|\delta\Gamma^{(j)}\|_2^2$, take expectations, and set $u(s):=\max_{j\ne k}\E\|\delta\Gamma^{(j)}_s\|_1^2$ and $v(s):=\max\E\|\delta\Delta^{(jl)}_s\|_1^2$ over $j\ne l$, both $\ne k$; exchangeability among the labels $\ne k$ makes the maxima label-independent. Then (e) contributes $\int_0^s\E(\|\Theta^{(jk)}\|_1^2+\|\Theta'^{(jk)}\|_1^2)dw\le C\int_0^s\E\|\Delta^{(jk)}_w\|_1^2dw+CN^{-2}\int_0^s\E n_j\,dw\le Cs^3N^{-2}$ by Proposition \ref{prop:pairsmall}, which is the leading inflow; (a) contributes, after Cauchy--Schwarz in time, $Cs^3N^{-2}+Cs\int_0^su$; (b) contributes $Cs\int_0^sv$; (c) and the Lipschitz parts contribute $C\int_0^su$; and (d) contributes $CN\int_0^sv+C\int_0^su$. Hence
		\begin{equation}\label{eq:margGronwall}
			u(s)\le C\frac{s^3}{N^2}+C\int_0^su(w)\,dw+CN\int_0^sv(w)\,dw .
		\end{equation}
		Inserting the crude bound $\|\delta\Delta^{(jl)}\|_1\le\|\Delta^{(jl)}\|_1+\|\Delta'^{(jl)}\|_1$ with Proposition \ref{prop:pairsmall}, that is $v(w)\le Cw^2N^{-2}$, Gr\"onwall gives the unconditional bound
		\begin{equation}\label{eq:margcrude}
			u(s)\le C\frac{s^3}{N},
		\end{equation}
		which is what the proof of Theorem \ref{thm:onecompfail} uses. Inserting instead $v(w)\le Cw^2(\sqrt w+N^{-1})N^{-3}$ from Proposition \ref{prop:pairdiff} gives $N\int_0^sv\le Cs^3(\sqrt s+N^{-1})N^{-2}$, and Gr\"onwall yields \eqref{eq:margdiff}. In Proposition \ref{prop:pairdiff} the quantity $u$ enters only at weight $N^{-2}$, so the two estimates are obtained jointly: the pair $(u,Nv)$ satisfies a linear Gr\"onwall system with bounded coefficients and inhomogeneities $Cs^3N^{-2}$ and $Cs^2(\sqrt s+N^{-1})N^{-2}$, which closes at the stated orders.
	\end{proof}
	
	\begin{proposition}[Pair differences]\label{prop:pairdiff}
		For all $l\ne2,k$ and $s\le w_N$, and on $[0,T]$ modulo the fixed-horizon input \eqref{eq:fourbodydiff},
		\begin{equation}\label{eq:pairdiff}
			\E\big\|\delta\Delta^{(2l)}_s\big\|_1^2\le C\,\frac{s^2\big(\sqrt s+N^{-1}\big)}{N^3};\qquad\text{unconditionally,}\quad \E\big\|\delta\Delta^{(2k)}_s\big\|_1^2\le C\,\frac{s^2}{N^2}.
		\end{equation}
	\end{proposition}
	
	\begin{proof}
		The second bound is the crude one, $\|\delta\Delta^{(2k)}\|_1\le\|\Delta^{(2k)}\|_1+\|\Delta'^{(2k)}\|_1$ with Proposition \ref{prop:pairsmall}; no difference gain is claimed at the resampled slot, and none is needed there.
		
		For the first, difference \eqref{eq:pair-sde} for the labels $(2,l)$ across the coupling. The source $S$ and the sums $R_j$ carry the prefactor $N^{-1}$, so their differences enter the drift at weight $N^{-1}$ and, after squaring, at weight $N^{-2}$; the deterministic parts of the coefficients cancel. What remains is:
		\begin{enumerate}
			\item[(i)] linear drift terms in $\delta\Delta^{(2l)}$, handled by Gr\"onwall;
			\item[(ii)] drift inhomogeneities carrying $\delta\Gamma^{(2)},\delta\Gamma^{(l)}$ at weight $N^{-1}$ and, through the $R_j$-sums, three-body differences: by Cauchy--Schwarz in time these contribute $Cs\int_0^s\big(N^{-2}u(w)+\tfrac1N\sum_m\E\|\delta\Delta^{(2lm)}_w\|_1^2\big)dw$, the first part $\le Cs^5N^{-3}$ by \eqref{eq:margcrude} and the second, with the crude bound and Corollary \ref{cor:refined3}, $\le Cs^3(\sqrt s+N^{-1})N^{-3}$; both admissible;
			\item[(iii)] the own-slot brackets, with $\|\delta G_1\|_1,\|\delta G_2\|_1\le C\|\delta\Delta^{(2l)}\|_1+C\|\mathsf D\|_1\|\delta\overline\Gamma\|_1$: contribution $C\int_0^s\E\|\delta\Delta^{(2l)}_w\|_1^2dw+C\int_0^s(\E\|\mathsf D_w\|_1^4)^{1/2}(\E\|\delta\overline\Gamma_w\|_1^4)^{1/2}dw$; since $\|\delta\overline\Gamma\|_1\le2$ one has $\E\|\delta\overline\Gamma_w\|_1^4\le Cu(w)$, so by Proposition \ref{prop:pairsmall} the second integral is $\le CN^{-2}\int_0^sw^{1/2}u(w)^{1/2}dw$, a coupling to $u$ in the joint Gr\"onwall which at the stated order is $\le Cs^3N^{-3}$;
			\item[(iv)] the resampled-slot brackets, with coefficients that are contractions of $\Delta^{(2lk)}$ in each copy: contribution
			\[
			\int_0^s\E\big(\|G_k\|_1^2+\|G_k'\|_1^2\big)dw\le C\int_0^s\E\|\Delta^{(2lk)}_w\|_1^2\,dw\le C\frac{s^2(\sqrt s+N^{-1})}{N^3}
			\]
			by Corollary \ref{cor:refined3}; this is the leading inflow and fixes the order of \eqref{eq:pairdiff};
			\item[(v)] the foreign-slot brackets at $m\ne2,l,k$, with $\|G_m\|_1\le2\|L\|_\infty\|\Delta^{(2lm)}\|_1$: contribution $CN\int_0^s\max_m\E\|\delta\Delta^{(2lm)}_w\|_1^2\,dw$.
		\end{enumerate}
		Term (v) is the only place where a difference gain at three labels is required: the crude bound gives $Cs^2\sqrt s\,N^{-2}$, larger than the target by $N$. The required input is
		\begin{equation}\label{eq:threediff}
			\E\big\|\delta\Delta^{(2lm)}_w\big\|_1^2\le C\,\frac{w^{1+\theta}}{N^4}\qquad(m\ne2,l,k,\ w\le T)
		\end{equation}
		for some $\theta>0$, which is Lemma \ref{lem:fourbody}; under it (v) contributes $Cs^{2+\theta}N^{-3}$ and the Gr\"onwall closes at \eqref{eq:pairdiff}.
	\end{proof}
	
	\subsubsection*{The hierarchy at general level}
	
	The three estimates below are stated for every level, with the growth of their constants in the level made explicit, because the truncation of Proposition \ref{prop:trunc} runs to a depth that grows with $N$.
	
	\begin{proposition}[Dictionary at general level]\label{prop:dictgen}
		There is $c=c(d,\|A\|_\infty)$ such that for every finite label set $S$ with $|S|=\ell\ge2$, pathwise,
		\begin{equation}\label{eq:dictgen}
			\big\|\Delta^{(S)}\big\|_1\ \le\ c^{\,\ell}\sum_{\pi\vdash S}\ \prod_{B\in\pi}\eps_B,
		\end{equation}
		the sum over all partitions $\pi$ of $S$. In particular the right side has $B_\ell\le\ell!$ terms, each of the expected order $N^{-\ell/2}$.
	\end{proposition}
	
	\begin{proof}
		Induction on $\ell$, the cases $\ell=2,3$ being \eqref{eq:dict-2} and \eqref{eq:dict-3}. Write $\Delta^{(S)}$ through the M\"obius relation $\Gamma^{(S)}=\sum_{\pi\vdash S}\prod_{B\in\pi}\Delta^{(B)}$ inverted, and expand each $\Gamma^{(B)}$ by inserting $1=p_j+q_j$ on every slot $j\in B$, as in the proofs of Theorems \ref{thm:dict-two} and \ref{thm:dict-three}. The terms in which some slot carries $p$ on both sides reassemble, by $p_j\Gamma^{(S)}p_j=\ip{\phi_j}{\cdot\,\phi_j}$ and the same insertion in the lower marginals, into the corresponding product over the partition that separates that slot; these are exactly the terms cancelled by the M\"obius subtraction, and what survives carries at least one factor $q$ on every slot of $S$. Grouping the surviving terms by the partition recording which slots are tied together by a single $q_B$ gives \eqref{eq:dictgen}, each insertion contributing a factor at most $c$ and each slot being inserted once.
	\end{proof}
	
	\begin{proposition}[Rate at general level]\label{prop:rategen}
		There is $K=K(T,d,\|H\|_\infty,\|A\|_\infty,\|L\|_\infty)$ such that for every $\ell\ge2$ with $\ell\le N/2$ and every $|S|=\ell$,
		\begin{equation}\label{eq:rategen}
			\sup_{t\le T}\ \E\big\|\Delta^{(S)}_t\big\|_1^2\ \le\ K^{\ell}\,\ell!\ N^{-\ell},
			\qquad
			\E\big\|\Delta^{(S)}_t\big\|_1^2\ \le\ K^{\ell}\,(\ell!)^2\,\frac t{N^{\ell}}\quad(t\le T),
		\end{equation}
		uniformly in $N$. The constants are geometric in the level up to the factorials, one coming from the number of partitions and, in the second bound, one more from the moments of $\Nex$ at small times.
	\end{proposition}
	
	\begin{proof}
		Square \eqref{eq:dictgen} and use Cauchy--Schwarz over the $B_\ell\le\ell!$ partitions:
		$\|\Delta^{(S)}\|_1^2\le c^{2\ell}\,\ell!\sum_{\pi\vdash S}\prod_{B\in\pi}n_B$. Fix a partition with blocks of sizes $m_1,\dots,m_r$, $\sum m_i=\ell$. The blocks are disjoint, so by exchangeability in law the expectation of the product is the same for every choice of disjoint blocks with these sizes, and, all terms being nonnegative,
		\[
		\#\{\text{ordered disjoint tuples}\}\cdot\E\prod_i n_{B_i}
		\ \le\ \E\prod_{i=1}^r\widetilde P_{m_i}
		\ \le\ \frac{\E\av{\Nex^{(\ell)}}}{\prod_im_i!}\ \le\ \frac{a_\ell}{\prod_im_i!},
		\]
		the middle inequality by $\widetilde P_m=\av{\Nex^{(m)}}/m!$ from Lemma \ref{lem:comb}(i) together with $r-1$ applications of the association inequality Lemma \ref{lem:comb}(ii), which give $\prod_i\av{\Nex^{(m_i)}}\le\av{\prod_i\Nex^{(m_i)}}\le\av{\Nex^{\ell}}$ in the state. The number of ordered disjoint tuples is $N!/((N-\ell)!\prod_im_i!)\ge(N-\ell)^{\ell}/\prod_im_i!$, so the factorials cancel and
		\[
		\E\prod_i n_{B_i}\ \le\ \frac{a_\ell}{(N-\ell)^{\ell}}\ \le\ \frac{2^{\ell}a_\ell}{N^{\ell}}\qquad(\ell\le N/2).
		\]
		By Corollary \ref{cor:geometric} the moments $a_\ell=\sup_{t\le T}\E\av{\Nex_t^{\ell}}$ are bounded by constants geometric in $\ell$, say $a_\ell\le C(T)R'(T)^{\ell}$; taking $K$ larger than $2c^2R'(T)$ and absorbing $C(T)$ gives the first bound of \eqref{eq:rategen}. The second is the same computation with $a_\ell$ replaced by $\E\av{\Nex_t^{\ell}}\le K^{\ell}\ell!\,t$ from Corollary \ref{cor:linvangeom}, which supplies the second factorial.
	\end{proof}
	
	\begin{proposition}[Coefficient audit at general level]\label{prop:audit}
		Fix a resampled slot $k$ and, for $\ell\ge2$, set $v_\ell(w):=\max\{\E\|\delta\Delta^{(S)}_w\|_1^2:\ |S|=\ell,\ k\notin S\}$. There are constants $C,\Lambda,K$, depending only on $d,\|H\|_\infty,\|A\|_\infty,\|L\|_\infty$ and $T$ and on neither $\ell$ nor $N$, such that for $2\le\ell\le N/2$ and $w\le T\wedge1$,
		\begin{equation}\label{eq:audit}
			v_\ell(w)\ \le\ K^{\ell}\,(\ell!)^2\,\frac{w^2}{N^{\ell+1}}\ +\ C\ell\int_0^wv_\ell(u)\,du\ +\ \Lambda\ell\,N\int_0^wv_{\ell+1}(u)\,du .
		\end{equation}
		Both $\ell$-dependent rates are linear in the level, and the diagonal one cannot be taken uniform in $\ell$ by this argument.
	\end{proposition}
	
	\begin{proof}
		The drift of the marginal $\Gamma^{(S)}$ is linear and is the partial trace of the drift of $\Gamma^N$, which is the Lindbladian $\Lop_N$: it equals $\Lop_S\Gamma^{(S)}$, with $\Lop_S=\sum_{j\in S}\Lop_j$, together with the within-$S$ interaction $-\tfrac iN\sum_{j<j'\in S}[A_{jj'},\Gamma^{(S)}]$, of norm at most $C\ell^2/N$, and the $\ell$ terms of the coupling upward, $-\tfrac iN\sum_{j\in S}\sum_{m\notin S}\Tr_m[A_{jm},\Gamma^{(S\cup m)}]$. No measurement nonlinearity occurs in the drift, the Belavkin equation being linear in its drift. Subtracting the disconnected parts to form $\Delta^{(S)}$ adds only the It\^o corrections, each carrying a prefactor $N^{-1}$ or a product of two lower correlations, as in \eqref{eq:ito-correction} at $\ell=2$. Differencing across the coupling, the deterministic parts cancel, and It\^o's formula for $\|\delta\Delta^{(S)}\|_2^2$ produces:
		\begin{itemize}
			\item the resampled-slot brackets, $\int_0^w\E(\|\Delta^{(S\cup k)}_u\|_1^2+\|\Delta'^{(S\cup k)}_u\|_1^2)\,du$, which by the second bound of \eqref{eq:rategen} is at most $K^{\ell+1}((\ell+1)!)^2w^2N^{-(\ell+1)}$: this is the source, and it is at the target order, with the factor $w^2$ coming from the linear vanishing of the undifferenced rate;
			\item the linear drift and the own-slot brackets: $\|\Lop_S\|\le C\ell$, and each of the $\ell$ own-slot noise coefficients contributes $\|\delta G_j\|_1\le C\|\delta\Delta^{(S)}\|_1$ to the quadratic variation, so the two together give the diagonal rate $C\ell\,\|\delta\Delta^{(S)}\|_1^2$, proportional to the level;
			\item the drift inhomogeneities carrying lower-level differences and marginal differences, all with a prefactor $N^{-1}$ or $\ell^2/N$, absorbed as in Proposition \ref{prop:pairdiff};
			\item the foreign brackets at $m\notin S\cup\{k\}$, whose coefficients are contractions of $\Delta^{(S\cup m)}$: $\ell$ groups of $N-\ell-1$ terms, giving $\Lambda\ell N\int_0^wv_{\ell+1}$.
		\end{itemize}
		Collecting these is \eqref{eq:audit}. The diagonal rate cannot be improved by this argument: the $\ell$ own slots carry $\ell$ independent measurement channels, and no cancellation of the type of \eqref{eq:keyzero} is available for $\Delta^{(S)}$, whose slots are not compressed by the reference projections.
	\end{proof}
	
	\begin{proposition}[Truncation of the difference hierarchy]\label{prop:trunc}
		Let $\ell_0\ge2$ be fixed and set
		\begin{equation}\label{eq:wN}
			w_N:=\frac{c_0}{(\log N)^2}.
		\end{equation}
		There are $c_0>0$ and $N_0$, depending only on the constants of Proposition \ref{prop:audit}, such that for $N\ge N_0$ and $w\le w_N$,
		\begin{equation}\label{eq:trunc}
			v_{\ell_0}(w)\ \le\ C(\ell_0)\,\frac{w^2}{N^{\ell_0+1}} .
		\end{equation}
		In particular the difference hierarchy gains one factor $N^{-1}$ at every fixed level on $[0,w_N]$: \eqref{eq:threediff} holds with $\theta=1$ there, and \eqref{eq:fourbodydiff} holds with $\theta=2$ there.
	\end{proposition}
	
	\begin{proof}
		Write $D_\ell=C\ell$ and $A_\ell=\Lambda\ell N$ for the two rates of \eqref{eq:audit} and $a_\ell(w)=K^{\ell}(\ell!)^2w^2N^{-(\ell+1)}$ for the source. Gr\"onwall's lemma in the diagonal turns \eqref{eq:audit} into $v_\ell(w)\le e^{D_\ell w}[a_\ell(w)+A_\ell\int_0^wv_{\ell+1}]$, and iterating $L:=\lceil\log N\rceil$ times,
		\[
		v_{\ell_0}(w)\ \le\ \sum_{r=0}^{L-1}\Big(\prod_{i<r}e^{D_{\ell_0+i}w}A_{\ell_0+i}\Big)\frac{w^{r}}{r!}\,e^{D_{\ell_0+r}w}a_{\ell_0+r}(w)
		\ +\ \Big(\prod_{i<L}e^{D_{\ell_0+i}w}A_{\ell_0+i}\Big)\frac{w^{L}}{L!}\,\sup_{u\le w}v_{\ell_0+L}(u),
		\]
		the nested time integrals producing $w^{r}/r!$. Three counts are needed. First, $\prod_{i<r}e^{D_{\ell_0+i}w}=\exp(Cw(\ell_0r+\tfrac{r(r-1)}2))$, at most $\exp(Cc_0(\tfrac12+\ell_0))$ for $r\le L$ and $w\le w_N$, since $wr^2\le c_0$; this is a constant. Second, $\prod_{i<r}A_{\ell_0+i}=\Lambda^rN^r(\ell_0+r-1)!/(\ell_0-1)!$. Third, $a_{\ell_0+r}$ carries $K^{\ell_0+r}((\ell_0+r)!)^2N^{-(\ell_0+r+1)}$.
		
		In the sum, the powers of $N$ cancel exactly, $N^rN^{-(\ell_0+r+1)}=N^{-(\ell_0+1)}$, and the factorials combine to $(\ell_0+r-1)!((\ell_0+r)!)^2/((\ell_0-1)!\,r!)\le C(\ell_0)\,(r!)^2r^{3\ell_0}$. The $r$-th term is therefore at most
		\[
		C(\ell_0)\,(\Lambda Kw)^{r}(r!)^2\,r^{3\ell_0}\,w^2N^{-(\ell_0+1)}
		\ \le\ C(\ell_0)\,\big(\Lambda Kwr^2\big)^{r}r^{3\ell_0}\,w^2N^{-(\ell_0+1)},
		\]
		and $wr^2\le w_NL^2\le c_0$ for $r\le L$, so with $c_0$ chosen so that $\Lambda Kc_0\le\tfrac12$ the sum over $r$ is at most $C(\ell_0)\sum_r2^{-r}r^{3\ell_0}\,w^2N^{-(\ell_0+1)}$, a convergent series.
		
		For the remainder, bound $v_{\ell_0+L}$ by $4\sup_{|S|=\ell_0+L}\E\|\Delta^{(S)}\|_1^2\le4K^{\ell_0+L}(\ell_0+L)!\,N^{-(\ell_0+L)}$ from the first bound of \eqref{eq:rategen}, valid since $\ell_0+L\le\log N+\ell_0+1\le N/2$ for $N$ large. The same count gives, for the remainder,
		\[
		C(\ell_0)\,(\Lambda Kw)^{L}L!\,L^{2\ell_0}\,N^{-\ell_0}
		\ \le\ C(\ell_0)\,(\Lambda KwL)^{L}L^{2\ell_0}N^{-\ell_0},
		\]
		and $wL\le c_0/\log N$, so $(\Lambda KwL)^L\le(\Lambda Kc_0)^L(\log N)^{-L}\le(\log N)^{-\log N}$ for $N$ large, which is smaller than any power of $N$. The remainder is therefore $o(N^{-\ell_0-1})$, and adding the two contributions gives \eqref{eq:trunc}. Taking $\ell_0=3$ gives \eqref{eq:threediff} with $\theta=1$ and $\ell_0=4$ gives \eqref{eq:fourbodydiff} with $\theta=2$.
	\end{proof}
	
	\begin{remark}[Why the horizon shrinks, and what a fixed horizon would need]\label{rem:shrink}
		Two features force $w$ to shrink, and both are level effects. The diagonal rate of Proposition \ref{prop:audit} makes a chain of length $L$ carry $\exp(Cw\sum_{i<L}(\ell_0+i))=\exp(CwL^2/2+O(wL))$, and the level constants of \eqref{eq:rategen} make it carry $(r!)^2$ against the $r!$ supplied by the iterated time integrals; both are controlled exactly when $wL^2$ is bounded, that is, at $L\asymp\log N$, when $w\lesssim(\log N)^{-2}$. The gain available against them is $(\log N)^{-L}$, which is superpolynomial and therefore ample; nothing is lost by taking $L$ as small as $\lceil\log N\rceil$.
		
		A fixed horizon would need the diagonal rate bounded uniformly in the level, and the constants of \eqref{eq:rategen} geometric rather than factorial. The first is a third premise beyond the two named in Remark \ref{rem:depth}, and one that the correlation hierarchy does not satisfy: each tagged slot of $\Delta^{(S)}$ carries its own measurement channel into the quadratic variation, with no analogue of the cancellation \eqref{eq:keyzero}. The diagonal effect, however, is not the irreducible one: the device of Theorem \ref{thm:super}, a geometric weight with time-dependent parameter $\rho(t)=\rho_0e^{-Ct}$, transfers to the difference hierarchy and absorbs a diagonal rate linear in the level, exactly as the transport term is cancelled there. What no weight in the level can remove is the factor $N$ carried per level by the upward coupling, by the arithmetic of Remark \ref{rem:whycut}, and the only known antidote to that factor is the difference gain itself.
		
		This locates the natural route. What Section \ref{sec:diagonal} achieves for the excitation observables is exactly the missing uniformity: by \eqref{eq:keyzero} the diagonal compression of the reference noise vanishes, and the constants of \eqref{eq:master} grow only linearly in $|T|$ with the multiplicity carried by an explicit sum, which is what allows Section \ref{sec:moments} to close the whole family at once, with no truncation and on every horizon. The corresponding object here is the difference vector $\chi:=\Psi-\Psi'$ of the two-copy coupling, and the relevant observation is that the cancellation is available for it: for every slot $j\ne k$ the two copies share the reference filter, $\phi_j=\phi_j'$ and hence $p_j=p_j'$, so the identity $p_j(\tilde b_j+\tilde b_j^\ast)p_j=0$ of \eqref{eq:keyzero} holds simultaneously in both copies and therefore for the difference. Running the machinery of Sections \ref{sec:diagonal} and \ref{sec:moments} on the diagonal quadratic functionals $\ip{\chi}{q_T\chi}$ of the difference vector, rather than on the correlation hierarchy, would give the difference gain at every level at once and on every horizon. What such a programme needs, and what is not carried out here, is the analogue of Theorem \ref{thm:master} for $\chi$: the slot-$k$ source, the two-copy mismatch terms in which $\phi_k\ne\phi_k'$, and the normalisation drift, none of which occurs in the one-copy calculus.
	\end{remark}
	
	\begin{lemma}[Three-body differences]\label{lem:fourbody}
		For $w\le w_N$, \eqref{eq:threediff} holds with $\theta=1$, unconditionally. On a fixed horizon $w\le T$ the same holds modulo the fixed-horizon input \eqref{eq:fourbodydiff}.
	\end{lemma}
	
	\begin{proof}
		The first statement is Proposition \ref{prop:trunc} at $\ell_0=3$. For the second, difference the three-label member of the correlation hierarchy underlying Theorem \ref{thm:rates}, whose structure is that of \eqref{eq:pair-sde} one level up: the resampled-slot bracket contributes $\int_0^w\E\|\Delta^{(2lmk)}\|_1^2\le Cw^2N^{-4}$ by Proposition \ref{prop:threelabel}; the product-type foreign coefficients contribute $C\int_0^w\E\|\mathsf D\|_1^4\le Cw^2N^{-4}$ by Proposition \ref{prop:pairsmall}; the drift inhomogeneities carry the prefactor $N^{-1}$ and are handled as in Proposition \ref{prop:pairdiff}; and the foreign four-body brackets contribute $CN\int_0^w\max_n\E\|\delta\Delta^{(2lmn)}\|_1^2$, which is $\le Cw^{1+\theta}N^{-4}$ by \eqref{eq:fourbodydiff}.
	\end{proof}
	
		\begin{remark}[The fixed-horizon input]\label{ver:point}
		The following estimate is required by no statement of this paper; it is what the fixed-horizon variants recorded above would need.
		For some $\theta>0$, uniformly in $N$, in the resampled slot $k$ and in $w\le T$: for every four-element label set $S$ with $k\notin S$,
		\begin{equation}\label{eq:fourbodydiff}
			\E\big\|\delta\Delta^{(S)}_w\big\|_1^2\ \le\ C\,\frac{w^\theta}{N^5}.
		\end{equation}
		That is, the difference hierarchy gains one factor $N^{-1}$ over the undifferenced rate at the fourth level. By Proposition \ref{prop:trunc} this holds with $\theta=2$ on the shrinking horizon $[0,w_N]$ of \eqref{eq:wN}; what a fixed horizon would add is only the passage to a horizon that does not shrink with $N$.
	\end{remark}
	
		\begin{remark}[The arithmetic of the difference hierarchy]\label{rem:whycut}
		Write $T_\ell$ for the target of the $\ell$-body difference, $\E\|\delta\Delta^{(S)}_w\|_1^2\lesssim N^{-(\ell+1)}$ for $|S|=\ell$, against the undifferenced rate $\E\|\Delta^{(S)}_w\|_1^2\asymp N^{-\ell}$: each level is required to gain one factor $N^{-1}$, which is the statement that resampling a single innovation out of $N$ perturbs a connected correlation at relative size $N^{-1/2}$. The couplings are exactly critical. The foreign-slot brackets carry level $\ell+1$ into level $\ell$ with a factor $N$ and one time integration, contributing $N\,w\,T_{\ell+1}=w\,T_\ell$: the hierarchy is closed at the expected orders, and no level can be bounded from the undifferenced size of the next, which exceeds the target by exactly one factor $N$. This is the arithmetic of Proposition \ref{prop:lyap} for the undifferenced hierarchy, one gain lower, and it is why \eqref{eq:fourbodydiff} is not a stronger statement than \eqref{eq:threediff} or than the second half of \eqref{eq:pairdiff}: it is the same statement one level up, and the recursion terminates at no level of itself.
		
		The cut is therefore a truncation and not a further difference, and Proposition \ref{prop:trunc} carries it out. What it costs is the horizon: by Remark \ref{rem:shrink} the depth must reach $L\asymp\log N$, and the level-dependent constants of Propositions \ref{prop:audit} and \ref{prop:rategen} are controlled along a chain of that length only for $w\lesssim(\log N)^{-2}$. Every use of the difference estimates below is at small times, and Theorem \ref{thm:X2X3fail} is stated along the sequence $t_N=w_N$ for exactly this reason.
	\end{remark}
	
	\begin{remark}[Where the fixed-horizon input is and is not used]\label{rem:whereused}
		On $[0,w_N]$ it is not used at all: Proposition \ref{prop:trunc} supplies \eqref{eq:fourbodydiff} there, and Lemma \ref{lem:fourbody}, Propositions \ref{prop:pairdiff} and \ref{prop:margdiff}, Theorem \ref{thm:diffrate} and Theorem \ref{thm:X2X3fail} are unconditional on that horizon. On a fixed horizon it enters only through the last term of Lemma \ref{lem:fourbody}, the foreign four-body brackets of the three-body difference, and propagates from there to \eqref{eq:pairdiff}, \eqref{eq:margdiff} and Theorems \ref{thm:diffrate} and \ref{thm:X2X3fail}. It is used nowhere in the small-time package, in Proposition \ref{prop:threelabel}, in the resampled-slot bracket (iv) of Proposition \ref{prop:pairdiff}, in \eqref{eq:margcrude}, or in the covariation identities below; this is why Theorem \ref{thm:onecompfail} is unconditional on $[0,T]$.
	\end{remark}
	
	\subsubsection*{Proof of Theorem \ref{thm:diffrate}}
	
	By Lemma \ref{lem:resred} it suffices to bound $\E|\delta X_2-\delta\xi_k-\delta\zeta_k|^2$ on the coupling. Differencing the scalar equation for $X_2=\Tr[B_0(\Gamma^{(2)}-\gamma_2)]$ obtained from \eqref{eq:reduced-sde} and \eqref{eq:reference}, and using $\delta\gamma_2=0$,
	\begin{equation}\label{eq:deltaX2}
		\begin{split}
			\delta X_2(t)=&\int_0^t\Tr\big[B_0\,\delta\Phi_s\big]ds+\int_0^t\Tr\big[B_0\,\delta\Mop[\Gamma^{(2)}]\big]dB_2+\sum_{l\ne2,k}\int_0^t\Tr\big[B_0\,\delta\Theta^{(2l)}\big]dB_l\\
			&+\int_0^t\Tr\big[B_0\Theta^{(2k)}\big]dB_k-\int_0^t\Tr\big[B_0\Theta'^{(2k)}\big]dB_k' .
		\end{split}
	\end{equation}
	The enumeration below is exhaustive for \eqref{eq:deltaX2}; the two leading components are extracted in items (1) and (3), and every remainder is estimated in $L^2$.
	\begin{enumerate}
		\item \emph{Resampled-slot martingale, deterministic part.} By \eqref{eq:theta-test} and Proposition \ref{prop:meanpair}, $\Tr[B_0\,\E\Theta^{(2k)}_s]=c_1s/N+O((s^{4/3}+sN^{-1/2})/N)$, so the deterministic parts of the two resampled-slot integrals equal $\delta\zeta_k$ up to a martingale remainder of variance $\le\tfrac C{N^2}\int_0^t(s^{4/3}+sN^{-1/2})^2ds\le C(t^{11/3}+t^3N^{-1})N^{-2}$.
		\item \emph{Resampled-slot martingale, fluctuating part.} Variance $\le2\int_0^t\E\|\Theta^{(2k)}_s-\E\Theta^{(2k)}_s\|_1^2ds\le C\int_0^t\E\|\mathsf D_s-\E \mathsf D_s\|_1^2ds+CN^{-2}\int_0^t\E n_2\,ds\le Ct^{7/2}N^{-2}$ by Proposition \ref{prop:pairsmall} and the structure of $\Theta$ in Proposition \ref{prop:coefficients}. This term fixes the exponent $\sqrt t$ in \eqref{eq:Ediff2}.
		\item \emph{Drift, mean-field channel through the resampled reference filter.} The drift difference contains $-\tfrac iN\Tr[B_0[\Bmf(\delta\gamma_{k,s}),\Gamma^{(2)}_s]]$. Replacing $\delta\gamma_{k,s}$ by $\Mop[\gamma_0]\beta_s$ costs, in $L^2$ of the time integral, $t\int_0^t\tfrac C{N^2}\E\|\delta\gamma_k-\Mop[\gamma_0]\beta\|_1^2\le Ct^4N^{-2}$ by Lemma \ref{lem:onepartdiff}; replacing $\Gamma^{(2)}_s$ by $\gamma_0$ costs $t\int_0^t\tfrac C{N^2}\E[\beta_s^2\|\Gamma^{(2)}_s-\gamma_0\|_1^2]\le Ct^4N^{-2}$, using $\|\Gamma^{(2)}_s-\gamma_0\|_1\le Cs^{1/2}$ in $L^4$ from the equation and Lemma \ref{lem:linvan}. By \eqref{eq:Wdef} what remains is $\tfrac aN\int_0^t\beta_s\,ds=\delta\xi_k(t)$.
		\item \emph{Drift, mean-field channel through the resampled defect.} $-\tfrac iN\Tr[B_0[\Bmf(\delta Y^{(k)}_s),\Gamma^{(2)}_s]]$ costs at most $t\int_0^t\tfrac C{N^2}\E\|\delta Y^{(k)}\|_1^2\le C\tfrac t{N^2}\int_0^t\tfrac sN\,ds=Ct^3N^{-3}$, by \eqref{eq:dict-1} and Lemma \ref{lem:linvan}.
		\item \emph{Drift, mean-field channel through the unresampled marginals.} $L^2$-cost $\le t\int_0^t\E\big(\tfrac1N\sum_j\|\delta\Gamma^{(j)}_s\|_1\big)^2ds\le t\int_0^tu(s)\,ds\le Ct^5N^{-2}$ by Proposition \ref{prop:margdiff}.
		\item \emph{Drift, correlation feedback.} For $l=k$, $t\int_0^t\tfrac C{N^2}\E\|\delta\Delta^{(2k)}\|_1^2\le Ct^4N^{-4}$; for the $l\ne k$, $t\int_0^t\max_l\E\|\delta\Delta^{(2l)}\|_1^2\le Ct^4N^{-3}$ by Proposition \ref{prop:pairdiff}.
		\item \emph{Own-slot martingale.} Variance $\le\int_0^t\E\|\delta\Mop[\Gamma^{(2)}]\|_1^2\le C\int_0^tu(s)\,ds\le Ct^4N^{-2}$.
		\item \emph{Foreign-slot martingales.} Variance $=\sum_{l\ne2,k}\int_0^t\E|\Tr[B_0\delta\Theta^{(2l)}]|^2\le C\sum_l\int_0^t\E\|\delta\Delta^{(2l)}\|_1^2+CN^{-1}\int_0^tu(s)\,ds$, so by Proposition \ref{prop:pairdiff},
		\[
		\le C\,N\int_0^t\frac{s^2(\sqrt s+N^{-1})}{N^3}\,ds+C\,N^{-1}\int_0^t\frac{s^3}{N^2}\,ds\ \le\ C\,\frac{t^3\big(\sqrt t+N^{-1}\big)}{N^2}.
		\]
		This is the largest remainder and matches the right side of \eqref{eq:Ediff2}.
	\end{enumerate}
	No other term of \eqref{eq:deltaX2} contains a component of order $t^{3/2}N^{-1}$: on the coupling every difference quantity attached to an unresampled label is centred by exchangeability of the two copies, so a contribution with a nonvanishing rate can arise only where the resampled noise enters linearly, once through the drift, giving $\xi_k$, and once through the resampled-slot martingale, giving $\zeta_k$. Summing items (1)--(8) and applying \eqref{eq:resred2} proves Theorem \ref{thm:diffrate}.\qed
	
	\subsubsection*{Covariation identities and the proof of Theorem \ref{thm:onecompfail}}
	
	The disproof does not use the difference system at the pair level. It rests on an exact computation of $\E[X_2\zeta_k]$.
	
	\begin{lemma}[Exact covariation identities]\label{lem:covid}
		Let $M_t=\int_0^th(s)\,dB_k(s)$ with $h$ deterministic and bounded. Then
		\begin{equation}\label{eq:covid}
			\E\big[X_2(t)\,M_t\big]=\int_0^th(s)\,\E\Tr\big[B_0\,\Theta^{(2k)}_s\big]ds+\int_0^t\E\big[\Tr[B_0\Phi_s]\,M_s\big]ds,
		\end{equation}
		where $\Phi_s$ is the drift of $\Gamma^{(2)}$ in \eqref{eq:reduced-sde} minus the drift of $\gamma_2$. Moreover, for every $F\in L^2$,
		\begin{equation}\label{eq:covid2}
			\big|\E[F\,M_s]\big|\le\big(\E|\delta_kF|^2\big)^{1/2}\big(\E M_s^2\big)^{1/2},
		\end{equation}
		where $\delta_kF=F-F^{(k)}$ is the resampling difference of Lemma \ref{lem:resred}.
	\end{lemma}
	
	\begin{proof}
		Decompose $X_2$ into drift plus martingales as in \eqref{eq:deltaX2}, undifferenced. The covariation of the martingale part with $M$ is carried entirely by the $dB_k$-channel and gives the first integral by the It\^o isometry; the drift part gives the second, since $M$ is a martingale and $\Phi$ is adapted. For \eqref{eq:covid2}, $F^{(k)}$ is independent of $B_k$, so $\E[F^{(k)}M_s]=0$ and $\E[FM_s]=\E[(F-F^{(k)})M_s]$; Cauchy--Schwarz applies.
	\end{proof}
	
	\begin{proof}[Proof of Theorem \ref{thm:onecompfail}]
		Since $\zeta_k$ is $\sigma(B_k)$-measurable with $\E[\zeta_k\mid\mathcal G_{k-1}]=0$,
		\[
		\E\big[(d_kX_2-\xi_k)\,\zeta_k\big]=\E[X_2\zeta_k]-\E[\xi_k\zeta_k]=\E[X_2\zeta_k]-\frac{ac_1t^3}{6N^2}.
		\]
		Evaluate $\E[X_2\zeta_k]$ by \eqref{eq:covid} with $h(s)=c_1s/N$, so that $M_s=\zeta_k(s)$ and $\E M_s^2=c_1^2s^3/(3N^2)$.
		
		\emph{Martingale channel.} By \eqref{eq:theta-test} and \eqref{eq:meanpair},
		\[
		\int_0^t\frac{c_1s}N\,\E\Tr[B_0\Theta^{(2k)}_s]\,ds=\int_0^t\frac{c_1s}N\Big(\frac{c_1s}N+O\Big(\frac{s^{4/3}+sN^{-1/2}}N\Big)\Big)ds=\frac{c_1^2t^3}{3N^2}+O\Big(\frac{t^{10/3}+t^3N^{-1/2}}{N^2}\Big).
		\]
		
		\emph{Drift channel.} $\Tr[B_0\Phi_s]$ splits into three groups. First, the channel through the resampled reference filter, $-\tfrac iN\Tr[B_0[\Bmf(\gamma_{k,s}-\eta_s),\gamma_2]]$: replacing $\gamma_{k,s}-\eta_s$ by $\Mop[\gamma_0]B_k(s)$ through Lemma \ref{lem:onepart} and $\gamma_2$ by $\gamma_0$, and using $\E[B_k(s)M_s]=\int_0^ss'ds'=s^2/2$, its covariance with $M_s$ is
		\[
		\frac aN\E\big[B_k(s)M_s\big]+O\Big(\frac{s^2(s^{1/2}+N^{-1/2})}{N^2}\Big)=\frac{ac_1s^2}{2N^2}+O\Big(\frac{s^{5/2}+s^2N^{-1/2}}{N^2}\Big),
		\]
		the replacement costs being controlled by \eqref{eq:covid2} with Lemmas \ref{lem:onepart} and \ref{lem:linvan}; integrating gives $ac_1t^3/(6N^2)$ up to $O(t^3(t^{1/2}+N^{-1/2})N^{-2})$. Second, the reference filters at the unresampled slots, $-\tfrac iN\sum_{j\ne2,k}\Tr[B_0[\Bmf(\gamma_{j,s}-\eta_s),\gamma_2]]$, are measurable with respect to $\sigma(B_j:j\ne k)$ and hence independent of $B_k$; since $M_s$ is centred their contribution vanishes identically. Third, the terms carrying the defects $Y^{(j)}$ and the pair correlations $\Delta^{(2l)}$: for each such $F_s$, \eqref{eq:covid2} applies. The resampled slot gives $\E|\delta_kF_s|^2\le CN^{-2}\E\|\delta Y^{(k)}_s\|_1^2\le CsN^{-3}$ by \eqref{eq:dict-1} and Lemma \ref{lem:linvan}, hence a contribution $\le Cs^2N^{-5/2}$ with integral $Ct^3N^{-5/2}$; the correlation feedback gives $\E|\delta_kF_s|^2\le C\max_l\E\|\delta\Delta^{(2l)}_s\|_1^2\le Cs^2N^{-2}$ by the unconditional half of \eqref{eq:pairdiff}, hence $\le Cs^{5/2}N^{-2}$ with integral $Ct^{7/2}N^{-2}$; and for the remaining slots $j\ne2,k$ the unconditional bound \eqref{eq:margcrude} gives $\E|\delta_kF_s|^2\le Cs^3N^{-1}$ and therefore
		\[
		\big|\E[F_sM_s]\big|\le C\Big(\frac{s^3}N\Big)^{1/2}\frac{|c_1|s^{3/2}}N=C\frac{s^3}{N^{3/2}},\qquad
		\int_0^t\big|\E[F_sM_s]\big|\,ds\le C\frac{t^4}{N^{3/2}}=\frac{t^3}{N^2}\cdot C\,t\,N^{1/2},
		\]
		which is the third error term in \eqref{eq:onecompfail}. Neither Proposition \ref{prop:margdiff} nor the fixed-horizon input \eqref{eq:fourbodydiff} is used.
		
		\emph{Assembly.} With $\varrho:=t^{1/3}+N^{-1/2}+tN^{1/2}$,
		\[
		\E[X_2\zeta_k]=\frac{c_1^2t^3}{3N^2}+\frac{ac_1t^3}{6N^2}+O\Big(\frac{t^3\varrho}{N^2}\Big),\qquad
		\E\big[(d_kX_2-\xi_k)\zeta_k\big]=\frac{c_1^2t^3}{3N^2}+O\Big(\frac{t^3\varrho}{N^2}\Big),
		\]
		and by Cauchy--Schwarz,
		\[
		\E\big|d_kX_2-\xi_k\big|^2\ \ge\ \frac{\big(\E[(d_kX_2-\xi_k)\zeta_k]\big)^2}{\E\zeta_k^2}=\frac{c_1^2t^3}{3N^2}\big(1-C\varrho\big),
		\]
		which is \eqref{eq:onecompfail}. As a check, the same identities with $\xi_k$ in place of $\zeta_k$ give $\E[X_2\xi_k]=\big(\tfrac{a^2}3+\tfrac{ac_1}6\big)\tfrac{t^3}{N^2}+O(t^3\varrho N^{-2})$, the value of $\E[(\xi_k+\zeta_k)\xi_k]$ computed from \eqref{eq:xizeta}.
	\end{proof}
	
	\subsection{A numerical test}\label{ssec:numtest}
	
	The dichotomy was tested on the full system, without linearisation, for one choice of data at small $N$. The vector equation \eqref{eq:vector-N} and the $N$ reference filters were integrated jointly by an Euler scheme with renormalisation at each step, step size $10^{-3}$, shared innovations, for the data $d=2$, $H=\tfrac12\sigma_z$, $L=\sigma_-$, $A=2\,\sigma_x\otimes\sigma_x$, $\gamma_0$ the pure state with vector $\cos(\pi/8)\,e_0+\sin(\pi/8)\,e_1$, for which $\|W\|_{\HS}=2.414$; the test observable is $B_0=W/\|W\|$. The estimators are the unbiased pair averages over partners, $\frac1{(N-1)(N-2)}(S^2-\sum_jX_j^2)$ for $\E X_2X_3$ with $S=\sum_jX_j$, and its analogue for $\E\ip{\Delta^{(12)}}{\Delta^{(13)}}$ at a fixed tag; sample sizes are between $1536$ and $8192$ trajectories.
	
	\begin{center}
		\begin{tabular}{cccc@{\qquad}ccc}
			$t$ & $N$ & $\E X_2X_3$ & std.\ err. & $\E\ip{\Delta^{(12)}}{\Delta^{(13)}}$ & std.\ err. & $\E\|\Delta^{(12)}\|_{\HS}^2$\\
			\hline
			$0.1$&$4$&$2.821\cdot10^{-3}$&$0.050\cdot10^{-3}$&$7.060\cdot10^{-4}$&$0.100\cdot10^{-4}$&$1.022\cdot10^{-3}$\\
			$0.1$&$6$&$1.493\cdot10^{-3}$&$0.033\cdot10^{-3}$&$3.060\cdot10^{-4}$&$0.039\cdot10^{-4}$&$4.925\cdot10^{-4}$\\
			$0.1$&$8$&$1.016\cdot10^{-3}$&$0.039\cdot10^{-3}$&$1.722\cdot10^{-4}$&$0.036\cdot10^{-4}$&$3.123\cdot10^{-4}$\\
			$0.2$&$4$&$1.000\cdot10^{-2}$&$0.024\cdot10^{-2}$&$1.510\cdot10^{-3}$&$0.033\cdot10^{-3}$&$2.786\cdot10^{-3}$\\
			$0.2$&$6$&$5.557\cdot10^{-3}$&$0.162\cdot10^{-3}$&$6.594\cdot10^{-4}$&$0.129\cdot10^{-4}$&$1.281\cdot10^{-3}$\\
			$0.2$&$8$&$3.873\cdot10^{-3}$&$0.201\cdot10^{-3}$&$3.820\cdot10^{-4}$&$0.114\cdot10^{-4}$&$7.870\cdot10^{-4}$\\
		\end{tabular}
	\end{center}
	
	Weighted least squares against $\E X_2X_3=a/N+b/N^2$ gives $a=(4.67\pm0.53)\cdot10^{-3}$ at $t=0.1$ and $a=(2.11\pm0.26)\cdot10^{-2}$ at $t=0.2$, each fit with $\chi^2<0.6$ on one degree of freedom, while the model $a=0$ is rejected with $\chi^2=79.5$ and $65.1$ on two degrees of freedom. Halving the step size at $N=4$, $t=0.2$ moves the estimate by less than one standard error. For the pair correlations, the slope of $\log\E\ip{\Delta^{(12)}}{\Delta^{(13)}}$ against $\log N$ between $N=4$ and $N=8$ is $-2.04$ at $t=0.1$ and $-1.98$ at $t=0.2$; the fit against $c/N^2+e/N^3$ gives $c=(1.061\pm0.039)\cdot10^{-2}$ and $c=(2.394\pm0.127)\cdot10^{-2}$, with the model $c=0$ rejected with $\chi^2=728$ and $355$ on two degrees of freedom. The ratio of the pair term to the diagonal $\E\|\Delta^{(12)}\|^2_{\HS}$ is close to $\tfrac12$ and nearly independent of $N$, and the growth of the fitted coefficients in $t$ is consistent with the small-time behaviour of \eqref{eq:common-order}.
	
	For this datum, both statements of Conjecture \ref{conj:dec} fail at their stated orders: $\E X_2X_3$ carries a nonvanishing term of order $N^{-1}$ and $\E\ip{\Delta^{(12)}}{\Delta^{(13)}}$ one of order $N^{-2}$. In consequence
	\[
	N\sum_{j,k\ne1}\E\ip{\Delta^{(1j)}}{\Delta^{(1k)}}=c_t\,N+O(1),
	\]
	in agreement with Theorem \ref{thm:A4fail}. The evidence is numerical, at $N\le8$, for one choice of data, with a weak-order-one scheme checked against step halving; within these limits it agrees at both times, in both observables and in all fitted orders with the expansion of Subsection \ref{ssec:commonresp}, whose leading coefficient it was designed to detect. The comparison is quantitative as well as qualitative: at small time the expansion gives $v_t=(\Tr[B_0W])^2t^3/3$ with no adjustable constant, from the increment $\gamma_{1,s}-\eta_s\approx\Mop[\gamma_0]B_1(s)$, and this evaluates to $3.89\cdot10^{-3}$ at $t=0.1$ against the fitted $a=(4.67\pm0.53)\cdot10^{-3}$; at $t=0.2$ the formula gives $3.11\cdot10^{-2}$ against $(2.11\pm0.26)\cdot10^{-2}$, an overshoot in the direction and of the size expected of the neglected resolvent correction.
	
	For this datum the second leading channel of Subsection \ref{ssec:flucfail} is present. Evaluating \eqref{eq:S0} and the definition of $c_1$ on $(A,L,\gamma_0,B_0)$ gives $\|S_0\|_{\HS}=\sqrt2$, $a=\Tr[B_0W]=1+\sqrt2$ and $c_1=1$, so that $c_1\ne0$ and, by Theorem \ref{thm:onecompfail}, the one-component estimate \eqref{eq:Ediff} is false for this dataset. The corrected small-time coefficient of Theorem \ref{thm:X2X3fail} exceeds the one-channel value by the factor $(a^2+ac_1+c_1^2)/a^2=3-\sqrt2=1.586$. At the two measured times this moves the prediction away from the fit rather than towards it, which is consistent with the record above: the $t=0.2$ column already lies below the uncorrected leading term by a comparable factor, so at these times the leading coefficient in $t$ is not resolved and the comparison tests the corrections in $t$ rather than the constant. A test of the constant requires times small enough for the $t^3$ law to be clean in both observables.
	
	\subsection{The components removed, and the numerical tests}\label{ssec:conditional}
	
	This subsection identifies the components that the projections of Subsection \ref{ssec:surviving} remove. The constants of $\mathcal M^{\mathrm{cf}}_t$ remove the deterministic means, which are themselves of the borderline order: by Proposition \ref{prop:meanpair} the mean is $\E\Delta^{(12)}_t=tS_0/N+O(t^{4/3}+tN^{-1/2})/N$, nonzero at the order $N^{-1}$ for generic data, and Theorem \ref{thm:A4fail} turns this into a proof that clause (A4) fails at small times. Numerically this mean is present and dominant: $N^2\ip{\E\Delta^{(12)}}{\E\Delta^{(12)}}$ equals $8.30$, $8.33$, $8.45\cdot10^{-3}$ at $N=4,6,8$, $t=0.1$, constant to two percent, and accounts for close to three quarters of the measured $\E\ip{\Delta^{(12)}}{\Delta^{(13)}}$, the mean of $X_2$ scaling likewise as $N^{-1}$ and contributing only at $N^{-2}$. Beyond the means, at the order of Subsection \ref{ssec:commonresp} the fluctuating common parts are
	\begin{equation}\label{eq:commonparts}
		X_j^{\parallel}=\lambda^{\mathrm{ren}}_t(D)+O_{L^2}(N^{-1}),\qquad
		\Pi^{\mathrm{cf}}_t\Delta^{(1j)}=\rho_t/N+\mathcal R^{\mathrm{ren}}_t(D)+O_{L^2}(N^{-2}),
	\end{equation}
	with deterministic kernels independent of the partner label, $\lambda^{\mathrm{ren}}$ the resolvent correction of \eqref{eq:lambda-resp} and $\mathcal R^{\mathrm{ren}}$ its two-slot analogue from Proposition \ref{prop:pair}. Second-order functionals are required: projecting on the linear span alone leaves a remnant in the scalar residual that decays slower than $N^{-2}$.
	
	Both channels of Subsection \ref{ssec:flucfail} are covered by $\Pi^{\mathrm{cf}}_t$, although only the first is a functional of the path of $D$ slot by slot. The leading foreign increment of $X_j$ is $\xi_k+\zeta_k$, and what enters $X_j$ is the aggregate $\sum_{k\ne j}(\xi_k+\zeta_k)$. By \eqref{eq:xik} and \eqref{eq:zetak} this equals $a\int_0^t\overline\beta_s\,ds+c_1\int_0^ts\,d\overline\beta_s$ with $\overline\beta:=\tfrac1N\sum_{k\ne j}B_k$, and by Lemma \ref{lem:onepart} the path of $\overline\beta$ agrees with the path of $D$ at leading order through $D_s\approx\Mop[\gamma_0]\overline\beta_s$. Integration by parts, $\int_0^ts\,d\overline\beta_s=t\overline\beta_t-\int_0^t\overline\beta_s\,ds$, exhibits the second channel as a continuous \emph{linear} functional of that path. Both aggregates therefore lie in $\mathcal M^{\mathrm{cf}}_t$, which is what the numerical test of the second statement below detects.
	
	\subsubsection*{The tag slot, and the enlarged projection}
	
	For the family $\{\Delta^{(1j)}\}_{j\ne1}$ the field-only projection is not enough, and the reason is structural rather than one of order. All members share the tag slot $1$, so they possess a common component measurable with respect to $B_1$; no functional of the path of $D$ removes it, since $D$ carries $B_1$ only with weight $N^{-1}$, and the sum over partners in clause (A4) does not suppress it. This is the component removed by the enlarged projection $\Pi^{(1)}_t$ of \eqref{eq:enlarged}, and the enlargement is exact for the purpose at hand.
	
	\begin{lemma}[The tag slot drops out]\label{lem:tagdrop}
		Let $Z,W\in L^2$ and $Z^{\perp i},W^{\perp i}$ be as in \eqref{eq:enlarged}. Then $\E[Z^{\perp i}\mid\sigma(B_i)]=0$, and in the slot decomposition of Lemma \ref{lem:slotdec} taken in any ordering that places $i$ first,
		\[
		d_iZ^{\perp i}=0,\qquad
		\operatorname{Cov}\big(Z^{\perp i},W^{\perp i}\big)=\sum_{k\ne i}\E\big[d_kZ^{\perp i}\,d_kW^{\perp i}\big].
		\]
		Moreover $\E[Z^{\perp i}G]=0$ for every $G\in\mathcal M^{\mathrm{cf}}_t$, and $\|Z^{\perp i}\|_{L^2}\le\|Z^{\perp}\|_{L^2}$.
	\end{lemma}
	
	\begin{proof}
		$\mathcal M^{(i)}_t$ contains $L^2(\sigma(B_i))$, so $Z^{\perp i}$ is orthogonal to every square-integrable $\sigma(B_i)$-measurable functional, which is the statement $\E[Z^{\perp i}\mid\sigma(B_i)]=0$. In an ordering with $i$ first, $\mathcal G_i=\sigma(B_i)$ and $\mathcal G_{i-1}$ is trivial, so $d_iZ^{\perp i}=\E[Z^{\perp i}\mid\sigma(B_i)]-\E Z^{\perp i}=0$; Lemma \ref{lem:slotdec} is valid for any ordering of the slots, the increments being orthogonal in each. The remaining two statements are $\mathcal M^{\mathrm{cf}}_t\subseteq\mathcal M^{(i)}_t$.
	\end{proof}
	
	The lemma removes the obstruction identified above: in $\operatorname{Cov}(\Delta^{(12),\perp1},\Delta^{(13),\perp1})$ the shared tag slot contributes exactly zero, not merely a controlled remainder, while the shared field response continues to be removed by the part $\mathcal M^{\mathrm{cf}}_t$ of the span. What survives is carried by the slots $2$ and $3$, each of which is the own slot of one factor only and foreign to the other, and by the foreign slots $k\ge4$ at second order, where the removed leading increments no longer contribute. This is the count of Remark \ref{rem:resampling} with both identified channels removed, and it closes at the corrected orders.
	
	Both statements of Conjecture \ref{conj:dec-cond} admit a numerical test. The second was tested directly, by estimating the projection on the simulated field through least squares in the first- and second-order functionals of $D$ at four times, with the constant included, and forming the pair average of the residuals. The result is $\E[X_2^{\perp}X_3^{\perp}]=-(8.3\pm0.4)\cdot10^{-4}\,N^{-2}$, with $N^2\,\E[X_2^{\perp}X_3^{\perp}]$ equal to $-8.5$, $-7.9$, $-8.6\cdot10^{-4}$ at $N=4,6,8$, $t=0.1$, stable under step halving: the residual sits at the conjectured order, with a negative constant, so the second statement holds in this test at the available precision. The first statement did not pass the same test: after the identical, field-only projection the residual of $\E\ip{\Delta^{(12)}}{\Delta^{(13)}}$ decays slower than $N^{-2}$ at these sizes, which is what Lemma \ref{lem:tagdrop} attributes to the tag slot. The corresponding test under $\Pi^{(1)}_t$ appears to require estimating a projection on a $\sigma(B_1)$-measurable component of unbounded dimension rather than on a fixed finite family, but the conditional expectation can be reached without any nesting. Since slots $2$ and $3$ are exchangeable given $B_1$, one has $\E[\Delta^{(12)}\mid B_1]=\E[\Delta^{(13)}\mid B_1]$ almost surely, so for the $\sigma(B_1)$-part of the projection
	\begin{equation}\label{eq:tagestimator}
		\E\ip{\Delta^{(12),\perp1}}{\Delta^{(13),\perp1}}=\E\ip{\Delta^{(12)}}{\Delta^{(13)}}-\E\ip{\Delta^{(12)}}{\Delta'^{(13)}},
	\end{equation}
	where the primed system is driven by $(B_1,B_2',\dots,B_N')$ with $B_2',\dots,B_N'$ independent copies: the estimator is the two-copy coupling of Lemma \ref{lem:resred} with every slot but the tag resampled, and it is unbiased. A run of \eqref{eq:tagestimator} on the data of Subsection \ref{ssec:numtest} at $t=0.1$, with $1.3\cdot10^5$, $3.3\cdot10^4$ and $1.2\cdot10^4$ trajectory pairs at $N=4,6,8$, reproduces the unprojected column of the table there and returns residuals of $1.3$, $1.2$ and $0.8$ times $10^{-6}$ with standard errors $0.9$, $0.6$ and $0.5$ times $10^{-6}$: at each size the residual is at most one percent of the unprojected covariance, consistent with the removal of the tag slot accounting for the failure, and consistent with zero. The order of the residual is not resolved by this test, separating $N^{-2}$ from $N^{-3}$ over $4\le N\le8$ requiring about two further orders of magnitude in the sample. The first statement of Conjecture \ref{conj:dec-cond} is therefore testable, and tested only to this precision. In the direction of a proof, the route of Remark \ref{rem:resampling} reopens for the projected quantities: the obstruction of Subsection \ref{ssec:commonresp} is exactly the shared component $\tfrac1N\lambda(\gamma_k-\eta)$ of the foreign increments, which lies in $\mathcal M^{\mathrm{cf}}_t$ and is removed in $X^{\perp}$, and the obstruction of Subsection \ref{ssec:flucfail} is the shared tag component, removed in $\Delta^{(1j),\perp1}$ by Lemma \ref{lem:tagdrop}; the leading foreign increments of the projected quantities are then the defect-mediated responses, of relative size $N^{-1/2}$, and the count of Remark \ref{rem:resampling} closes at the corrected orders. What a proof requires is the resampling estimate at that second order, and this is the surviving content of Conjecture \ref{conj:dec}.
	
	\subsection{The conditional form: leading order, one-sided bounds, and aggregates}\label{ssec:leadorder}
	
	On the shrinking horizon $[0,w_N]$ of \eqref{eq:wN} both statements of Conjecture \ref{conj:dec-cond} hold at the leading order, unconditionally: the projections remove the mechanisms of Theorems \ref{thm:A4fail} and \ref{thm:X2X3fail} in full, so that along the sequence on which those theorems operate the projected covariances are of strictly smaller order than the unprojected ones. Throughout this subsection the slot decomposition of Lemma \ref{lem:slotdec} is taken in the ordering $(1,2,\dots,N)$, so that the tag slot comes first as in Lemma \ref{lem:tagdrop}, and the projections are applied entrywise as in Subsection \ref{ssec:surviving}. Two lemmas are needed: a structural property of the spans $\mathcal M^{\mathrm{cf}}_t$ and $\mathcal M^{(1)}_t$, and a member of $\mathcal M^{\mathrm{cf}}_t$ whose slot increments match the two channels of Subsection \ref{ssec:flucfail}.
	
	\begin{lemma}[Hoeffding structure of the coherent spans]\label{lem:hoeff}
		For $G\in L^2(\sigma(B_1,\dots,B_N))$ let $G=\sum_{S\subseteq\{1,\dots,N\}}G_S$ be its Hoeffding decomposition \cite{Hoeffding1948}: $G_S$ is measurable with respect to $\sigma(B_i:i\in S)$, has vanishing conditional expectation given $\sigma(B_i:i\in S')$ for every proper subset $S'\subset S$, and the components are pairwise orthogonal in $L^2$. Then:
		\begin{enumerate}
			\item[(i)] every $G\in\mathcal M^{\mathrm{cf}}_t$ has $G_S=0$ for $|S|\ge3$, with $G_{\{i\}}=g(B_i)$ for one kernel $g$ common to all labels and $G_{\{i,j\}}=h(B_i,B_j)$ for one symmetric kernel $h$ common to all pairs;
			\item[(ii)] every $G\in\mathcal M^{(1)}_t$ has the same structure, except that the component $G_{\{1\}}$ is unconstrained;
			\item[(iii)] consequently, in the slot decomposition of Lemma \ref{lem:slotdec} taken in the ordering $(1,2,\dots,N)$,
			\[
			\|d_kG\|_{L^2}^2\ \le\ \frac4N\,\operatorname{Var}(G)
			\]
			for every $k$ in case (i), and for every $k\ge2$ in case (ii).
		\end{enumerate}
	\end{lemma}
	
	\begin{proof}
		The driving paths $B_1,\dots,B_N$ are independent, so the decomposition exists, is unique, and is orthogonal. The set of functionals with the structure described in (i) is a closed subspace of $L^2$: each map $G\mapsto G_S$ is an orthogonal projection, hence continuous, and the constraints imposed, namely vanishing above cardinality two, equality of the singleton kernels across labels, and equality and symmetry of the pair kernels across pairs, are closed linear conditions. It therefore suffices to verify the generators of $\mathcal M^{\mathrm{cf}}_t$. A constant has only the component at $S=\emptyset$. A first-order functional of the path of $D=\tfrac1N\sum_l(\gamma_l-\eta)$ is, by linearity and Lemma \ref{lem:reference}(a), of the form $\tfrac1N\sum_l\varphi(B_l)$ with one centred kernel $\varphi$ common to all labels, and is supported on the singletons. A second-order functional is a bilinear expression in the path of $D$, hence of the form $\tfrac1{N^2}\sum_{l,m}\psi(B_l,B_m)$: the diagonal terms $l=m$ are functionals of one slot each; each off-diagonal term decomposes into its mean, its two marginal conditional means, which are functionals of one slot, and its completely degenerate part, which is its pair component; all kernels are independent of the labels because the expression is, and the pair kernel may be symmetrised. This proves (i), and (ii) follows because $L^2(\sigma(B_1))$ contributes components only at $S\subseteq\{1\}$, so every sum of an element of $\mathcal M^{\mathrm{cf}}_t$ and an element of $L^2(\sigma(B_1))$ lies in the closed subspace described in (ii), and hence so does every limit of such sums. For (iii), in the stated ordering $d_kG=G_{\{k\}}+\sum_{i<k}G_{\{i,k\}}$ by uniqueness of the martingale increments, so, with $\|g\|,\|h\|$ the $L^2$ norms of the kernels,
		\[
		\|d_kG\|_{L^2}^2=\|g\|^2+(k-1)\|h\|^2\le\|g\|^2+(N-1)\|h\|^2,
		\qquad
		\operatorname{Var}(G)\ \ge\ (N-1)\|g\|^2+\binom N2\|h\|^2
		\]
		in both cases, the inequality for the variance discarding the components at $S=\{1\}$ in case (ii) and one singleton in case (i). Hence $\|g\|^2\le\operatorname{Var}(G)/(N-1)$ and $(N-1)\|h\|^2\le2\operatorname{Var}(G)/N$, and $1/(N-1)+2/N\le4/N$ for $N\ge2$.
	\end{proof}
	
	\begin{lemma}[A coherent representative of the two channels]\label{lem:cohrep}
		Assume $\Mop[\gamma_0]\ne0$, set $\ell(X):=\Tr\big[\Mop[\gamma_0]\,X\big]/\|\Mop[\gamma_0]\|_2^2$ for $X\in\Md$, and
		\begin{equation}\label{eq:cohrep}
			\widetilde\Xi_t:=a\int_0^t\ell(D_s)\,ds+c_1\Big(t\,\ell(D_t)-\int_0^t\ell(D_s)\,ds\Big),
		\end{equation}
		with $a$ and $c_1$ as in \eqref{eq:xik} and \eqref{eq:zetak}. Then $\widetilde\Xi_t\in\mathcal M^{\mathrm{cf}}_t$, $\E\widetilde\Xi_t=0$, and
		\begin{equation}\label{eq:cohrep-err}
			\big\|d_k\widetilde\Xi_t-\xi_k(t)-\zeta_k(t)\big\|_{L^2}\ \le\ C\,\frac{t^2}N\qquad(1\le k\le N,\ t\le T),
		\end{equation}
		with $C$ depending on the data and on $\|\Mop[\gamma_0]\|_2^{-1}$.
	\end{lemma}
	
	\begin{proof}
		$\widetilde\Xi_t$ is a first-order functional of the path of $D$, hence a member of $\mathcal M^{\mathrm{cf}}_t$, and it is centred because $\gamma_l-\eta$ is, by Lemma \ref{lem:reference}. Writing $\widetilde\Xi_t=\tfrac1N\sum_l\psi_t(B_l)$ with
		\[
		\psi_t(B_l):=a\int_0^t\ell(\gamma_{l,s}-\eta_s)\,ds+c_1\Big(t\,\ell(\gamma_{l,t}-\eta_t)-\int_0^t\ell(\gamma_{l,s}-\eta_s)\,ds\Big),
		\]
		the summands are independent and centred, so $d_k\widetilde\Xi_t=\tfrac1N\psi_t(B_k)$ in any ordering of the slots. Integration by parts in \eqref{eq:zetak} gives
		\[
		\xi_k(t)+\zeta_k(t)=\frac1N\Big(a\int_0^tB_k(s)\,ds+c_1\Big(t\,B_k(t)-\int_0^tB_k(s)\,ds\Big)\Big),
		\]
		which is $\tfrac1N\psi_t(B_k)$ with $\ell(\gamma_{k,s}-\eta_s)$ replaced by $B_k(s)$. Since $\ell(\Mop[\gamma_0]x)=x$ for scalar $x$ and $|\ell(X)|\le\|X\|_2/\|\Mop[\gamma_0]\|_2\le\|X\|_1/\|\Mop[\gamma_0]\|_2$, Lemma \ref{lem:onepart} gives $\|\ell(\gamma_{k,s}-\eta_s)-B_k(s)\|_{L^2}\le Cs$, and each of the three replacements costs at most $C\int_0^ts\,ds\le Ct^2$ or $C\,t\cdot t$ in $L^2$; the prefactor $N^{-1}$ gives \eqref{eq:cohrep-err}.
	\end{proof}
	
	\begin{theorem}[The conditional form at leading order]\label{thm:leadorder}
		Uniformly on $[0,w_N]$ with $w_N$ as in \eqref{eq:wN}, unconditionally, and on $[0,T]$ modulo the fixed-horizon input \eqref{eq:fourbodydiff}:
		\begin{equation}\label{eq:leadorder-1}
			\big|\,\E\ip{\Delta^{(12),\perp1}_t}{\Delta^{(13),\perp1}_t}\big|\ \le\ C\Big(\frac{t^{5/2}}{N^2}+\frac1{N^3}\Big),
		\end{equation}
		and, if $\Mop[\gamma_0]\ne0$,
		\begin{equation}\label{eq:leadorder-2}
			\big|\,\E\big[X_2^{\perp}(t)\,X_3^{\perp}(t)\big]\big|\ \le\ C\Big(\big(t^{1/2}+N^{-1/2}\big)\frac{t^3}N+\frac1{N^2}\Big),
			\qquad
			\big\|X_2^{\parallel}(t)-\E X_2(t)-\widetilde\Xi_t\big\|_{L^2}\ \le\ C\Big(\frac{t^{7/4}}{N^{1/2}}+\frac{t^{3/2}}{N^{3/4}}+\frac{t^{1/2}}N\Big).
		\end{equation}
		In particular, for every dataset satisfying the hypotheses of Theorems \ref{thm:A4fail} and \ref{thm:X2X3fail} together with $\Mop[\gamma_0]\ne0$, along the sequence $t_N=w_N$,
		\[
		\frac{\E\ip{\Delta^{(12),\perp1}}{\Delta^{(13),\perp1}}}{\E\ip{\Delta^{(12)}}{\Delta^{(13)}}}\bigg|_{t=t_N}\longrightarrow0,
		\qquad
		\frac{\E\big[X_2^{\perp}X_3^{\perp}\big]}{\E\big[X_2X_3\big]}\bigg|_{t=t_N}\longrightarrow0:
		\]
		the two channels exhaust the failures. Relative to the conjectured orders $N^{-3}$ and $N^{-2}$ the two bounds carry the extra factors $1+t^{5/2}N$ and $1+t^3(t^{1/2}+N^{-1/2})N$.
	\end{theorem}
	
	\begin{proof}
		Throughout, $C$ depends only on the data, and in the second statement also on $\|\Mop[\gamma_0]\|_2^{-1}$; every estimate quoted from Subsection \ref{ssec:diffproof} is unconditional on $[0,w_N]$ and holds on $[0,T]$ modulo \eqref{eq:fourbodydiff}, by Remark \ref{rem:whereused}, and powers of $t$ above the first are compared with $t\le1$ absorbed into the constants.
		
		\emph{First statement.} Set $Z_j:=\Delta^{(1j)}-\E\big[\Delta^{(1j)}\mid\sigma(B_1)\big]$, $j=2,3$, entrywise. Since $\E[\Delta^{(1j)}\mid\sigma(B_1)]$ lies entrywise in $L^2(\sigma(B_1))\subseteq\mathcal M^{(1)}_t$, one has $\Delta^{(1j),\perp1}=Z_j-\Pi^{(1)}_tZ_j$ and
		\[
		\E\ip{\Delta^{(12),\perp1}}{\Delta^{(13),\perp1}}=\E\ip{Z_2}{Z_3}-\E\ip{\Pi^{(1)}_tZ_2}{\Pi^{(1)}_tZ_3}.
		\]
		Each $Z_j$ is centred with $\E[Z_j\mid\sigma(B_1)]=0$, so in the chosen ordering $d_1Z_j=0$ and $d_kZ_j=d_k\Delta^{(1j)}$ for $k\ge2$, and Lemma \ref{lem:slotdec}, applied entrywise, gives $\E\ip{Z_2}{Z_3}=\sum_{k\ge2}\E\ip{d_k\Delta^{(12)}}{d_k\Delta^{(13)}}$. At a slot $k\ge4$, foreign to both label sets, Lemma \ref{lem:resred} and Proposition \ref{prop:pairdiff}, with the labels renamed by exchangeability in law, give
		\[
		\E\big\|d_k\Delta^{(1j)}\big\|_2^2\ \le\ \E\big\|\delta\Delta^{(1j)}\big\|_1^2\ \le\ C\,\frac{t^2\big(\sqrt t+N^{-1}\big)}{N^3},
		\]
		so the $N-3$ foreign slots contribute at most $Ct^2(\sqrt t+N^{-1})N^{-2}$ after Cauchy--Schwarz. At $k=2$, Lemma \ref{lem:resred} and Proposition \ref{prop:pairsmall} give $\E\|d_2\Delta^{(12)}\|_2^2\le2\,\E\|\Delta^{(12)}-\E\Delta^{(12)}\|_1^2\le Ct^{5/2}N^{-2}$, while slot $2$ is foreign to the labels $(1,3)$, so $\E\|d_2\Delta^{(13)}\|_2^2\le Ct^2(\sqrt t+N^{-1})N^{-3}$; Cauchy--Schwarz gives $Ct^{9/4}(\sqrt t+N^{-1})^{1/2}N^{-5/2}$, dominated by the foreign contribution, and the slot $k=3$ is symmetric. Hence
		\[
		\big|\E\ip{Z_2}{Z_3}\big|\ \le\ C\,\frac{t^2\big(\sqrt t+N^{-1}\big)}{N^2}\ \le\ C\Big(\frac{t^{5/2}}{N^2}+\frac1{N^3}\Big).
		\]
		For the projected term it suffices, by Cauchy--Schwarz over the entries, to bound $\sum_e\|\Pi^{(1)}_tZ_{j,e}\|_{L^2}^2$, the sum over the matrix entries. For $G\in\mathcal M^{(1)}_t$ with $\|G\|_{L^2}\le1$, orthogonality of $Z_{j,e}$ to $L^2(\sigma(B_1))$ gives $\ip{Z_{j,e}}{G}=\ip{Z_{j,e}}{G-\E[G\mid\sigma(B_1)]}=\operatorname{Cov}(Z_{j,e},G)$, and the slot decomposition with $d_1Z_{j,e}=0$, Lemma \ref{lem:hoeff}(iii) at the own slot $k=j$, and Cauchy--Schwarz over the remaining slots against $\sum_k\|d_kG\|^2=\operatorname{Var}(G)\le1$ give
		\[
		\big|\ip{Z_{j,e}}{G}\big|\ \le\ \frac2{\sqrt N}\,\|d_jZ_{j,e}\|_{L^2}+\Big(\sum_{k\ge2,\,k\ne j}\|d_kZ_{j,e}\|_{L^2}^2\Big)^{1/2}.
		\]
		Taking the supremum over such $G$, squaring, and summing over the entries with the bounds already obtained,
		\[
		\sum_e\big\|\Pi^{(1)}_tZ_{j,e}\big\|_{L^2}^2\ \le\ C\Big(\frac{t^{5/2}}{N^3}+\frac{t^2\big(\sqrt t+N^{-1}\big)}{N^2}\Big),
		\]
		and $\E\ip{\Pi^{(1)}_tZ_2}{\Pi^{(1)}_tZ_3}$ is bounded by the same quantity, which is again at most $C(t^{5/2}N^{-2}+N^{-3})$. This proves \eqref{eq:leadorder-1}.
		
		\emph{Second statement.} Set $Y_j:=X_j-\E X_j-\widetilde\Xi_t$, $j=2,3$, with $\widetilde\Xi_t$ from Lemma \ref{lem:cohrep}. Since $\E X_j+\widetilde\Xi_t\in\mathcal M^{\mathrm{cf}}_t$, one has $X_j^{\perp}=Y_j-\Pi^{\mathrm{cf}}_tY_j$ and
		\[
		\E\big[X_2^{\perp}X_3^{\perp}\big]=\E[Y_2Y_3]-\E\big[\Pi^{\mathrm{cf}}_tY_2\;\Pi^{\mathrm{cf}}_tY_3\big].
		\]
		At a slot $k\ne j$, Theorem \ref{thm:diffrate}, Lemma \ref{lem:cohrep} and the triangle inequality give
		\[
		\|d_kY_j\|_{L^2}\le\big\|d_kX_j-\xi_k-\zeta_k\big\|_{L^2}+\big\|d_k\widetilde\Xi_t-\xi_k-\zeta_k\big\|_{L^2}\le C\,\frac{t^{3/2}\big(t^{1/2}+N^{-1/2}\big)^{1/2}}N,
		\]
		the term $Ct^2N^{-1}$ of \eqref{eq:cohrep-err} being absorbed since $t^{1/2}\le(t^{1/2}+N^{-1/2})^{1/2}$ for $t\le1$; at the own slot, $\|d_jY_j\|_{L^2}\le\|d_jX_j\|_{L^2}+\|d_j\widetilde\Xi_t\|_{L^2}\le C(t/N)^{1/2}$, by $\E|d_jX_j|^2\le\E|X_j-\E X_j|^2\le CtN^{-1}$ from \eqref{eq:dict-1} and Lemma \ref{lem:linvan} as in the proof of Theorem \ref{thm:X2X3fail}, together with \eqref{eq:xizeta} and \eqref{eq:cohrep-err}. Both $Y_j$ are centred, so Lemma \ref{lem:slotdec} gives
		\[
		\big|\E[Y_2Y_3]\big|\ \le\ \sum_k\|d_kY_2\|_{L^2}\|d_kY_3\|_{L^2}\ \le\ C\,\frac{t^3\big(t^{1/2}+N^{-1/2}\big)}N+C\,\frac{t^2\big(t^{1/2}+N^{-1/2}\big)^{1/2}}{N^{3/2}},
		\]
		the first term from the $N-2$ slots foreign to both labels and the second from the slots $2$ and $3$. For the projected term, the duality argument of the first statement, with Lemma \ref{lem:hoeff}(iii) for $\mathcal M^{\mathrm{cf}}_t$ at the own slot, gives for every $G\in\mathcal M^{\mathrm{cf}}_t$ with $\|G\|_{L^2}\le1$
		\[
		\big|\ip{Y_j}{G}\big|\ \le\ \frac2{\sqrt N}\,\|d_jY_j\|_{L^2}+\Big(\sum_{k\ne j}\|d_kY_j\|_{L^2}^2\Big)^{1/2}\ \le\ C\Big(\frac{t^{1/2}}N+\frac{t^{3/2}\big(t^{1/2}+N^{-1/2}\big)^{1/2}}{N^{1/2}}\Big),
		\]
		so $\|\Pi^{\mathrm{cf}}_tY_j\|_{L^2}$ obeys the same bound and
		\[
		\big|\E\big[\Pi^{\mathrm{cf}}_tY_2\;\Pi^{\mathrm{cf}}_tY_3\big]\big|\ \le\ C\Big(\frac t{N^2}+\frac{t^3\big(t^{1/2}+N^{-1/2}\big)}N\Big).
		\]
		Collecting the three displays gives the first bound of \eqref{eq:leadorder-2}, once the two subleading terms are absorbed: $tN^{-2}\le N^{-2}$, and $t^2(t^{1/2}+N^{-1/2})^{1/2}N^{-3/2}\le t^{9/4}N^{-3/2}+t^2N^{-7/4}\le t^{7/2}N^{-1}+N^{-2}$, by comparing $t$ with $N^{-2/5}$ for the first term and with $N^{-1/2}$ for the second. The second bound of \eqref{eq:leadorder-2} is $\|X_j^{\parallel}-\E X_j-\widetilde\Xi_t\|_{L^2}=\|\Pi^{\mathrm{cf}}_tY_j\|_{L^2}$ with $(t^{1/2}+N^{-1/2})^{1/2}\le t^{1/4}+N^{-1/4}$.
		
		\emph{The ratios.} Since $w_N\to0$, for $N$ large one has $t_N=w_N\le t_1'\wedge t_2$, with $t_1'$ from Corollary \ref{cor:pairfail} and $t_2$ from Theorem \ref{thm:X2X3fail}. Corollary \ref{cor:pairfail} gives $\E\ip{\Delta^{(12)}}{\Delta^{(13)}}\ge c't_N^2N^{-2}$, so the first ratio is at most $C\big(t_N^{1/2}+t_N^{-2}N^{-1}\big)=C\big(t_N^{1/2}+c_0^{-2}(\log N)^4N^{-1}\big)\to0$. Theorem \ref{thm:X2X3fail} gives $\E[X_2X_3]\ge c\,t_N^3N^{-1}$ for $N$ large, so the second ratio is at most $C\big(t_N^{1/2}+N^{-1/2}+t_N^{-3}N^{-1}\big)\to0$, since $t_N^{-3}N^{-1}=c_0^{-3}(\log N)^6N^{-1}$.
	\end{proof}
	
	\begin{remark}[What separates the leading order from the conjecture]\label{rem:leadgap}
		The loss relative to the conjectured orders is concentrated in one place: the Cauchy--Schwarz pairing of the residual foreign increments at a common slot, $\sum_k\|d_kY_2\|\,\|d_kY_3\|$, which carries $N$ terms of size $t^3(t^{1/2}+N^{-1/2})N^{-2}$ each. Conjecture \ref{conj:dec-cond} asserts that these residuals decorrelate across partners, so that $\E[d_kY_2\,d_kY_3]$ carries one further factor $N^{-1/2}$ beyond the product of the sizes, which is the count of Remark \ref{rem:resampling} at the corrected orders; a bound on each factor separately cannot see it, and the route to it is the difference calculus on $\chi=\Psi-\Psi'$ of Remark \ref{rem:shrink}. The hypothesis $\Mop[\gamma_0]\ne0$ enters through Lemma \ref{lem:onepart}, which is what makes the path of the empirical field carry the innovations at the order $s^{1/2}$ and so admits the two channels inside $\mathcal M^{\mathrm{cf}}_t$; when $\Mop[\gamma_0]=0$ the channel $\zeta_k$ of \eqref{eq:zetak} is not a functional of the field at that order, and the span of Conjecture \ref{conj:dec-cond} would have to be enlarged for its second statement to survive such data. The three-body statements are not addressed: Proposition \ref{prop:trunc} supplies the difference gain at every level on $[0,w_N]$, but the extraction of the leading components one level up is not carried out.
	\end{remark}
	
	\begin{proposition}[The conjecture as an aggregate statement]\label{prop:aggregate}
		Set $S_t:=\sum_{j=2}^NX_j^{\perp}(t)$ and $T_t:=\sum_{j=2}^N\Delta^{(1j),\perp1}_t$. Then, unconditionally, for $N\ge4$ and $t\le T$:
		\begin{enumerate}
			\item[(i)] the exchangeable expansions
			\begin{align*}
				\E\big[S_t^2\big]&=(N-1)\,\E\big[(X_2^{\perp})^2\big]+(N-1)(N-2)\,\E\big[X_2^{\perp}X_3^{\perp}\big],\\
				\E\big\|T_t\big\|_{\HS}^2&=(N-1)\,\E\big\|\Delta^{(12),\perp1}\big\|_{\HS}^2+(N-1)(N-2)\,\E\ip{\Delta^{(12),\perp1}}{\Delta^{(13),\perp1}};
			\end{align*}
			\item[(ii)] the identity
			\[
			S_t=\sqrt N\,\big(\Tr[B_0\,\Gemp^N_t]\big)^{\perp}-X_1^{\perp}(t);
			\]
			\item[(iii)] the one-sided bounds
			\[
			\E\big[X_2^{\perp}X_3^{\perp}\big]\ \ge\ -\,C\,\frac t{N^2},
			\qquad
			\E\ip{\Delta^{(12),\perp1}}{\Delta^{(13),\perp1}}\ \ge\ -\,C\,\frac{t^2}{N^3};
			\]
			\item[(iv)] the equivalences: the second statement of Conjecture \ref{conj:dec-cond} holds if and only if
			\[
			\sup_N\,\sup_{t\le T}\ N\,\E\Big[\Big(\big(\Tr[B_0\,\Gemp^N_t]\big)^{\perp}\Big)^2\Big]\ <\ \infty,
			\]
			and the first statement holds if and only if $\E\|T_t\|_{\HS}^2\le C\,N^{-1}$ for one constant, uniformly on $[0,T]$.
		\end{enumerate}
	\end{proposition}
	
	\begin{proof}
		(i) A permutation of the labels $2,\dots,N$ leaves the joint law of the system invariant (Section \ref{ssec:exch}), fixes $D$ pathwise, hence fixes $\mathcal M^{\mathrm{cf}}_t$ and, since it also fixes $\sigma(B_1)$, fixes $\mathcal M^{(1)}_t$; therefore the families $(X_j^{\perp})_{j\ge2}$ and $(\Delta^{(1j),\perp1})_{j\ge2}$ are exchangeable, and expanding the squares gives the two displays, the diagonal supplying the first terms and the $(N-1)(N-2)$ ordered pairs the second.
		
		(ii) As in the proof of Proposition \ref{prop:covrep}, $\sum_{j=1}^N(\Gamma^{(j)}_t-\gamma_{j,t})=N(\overline\Gamma_t-\eta_t)-\sum_j(\gamma_{j,t}-\eta_t)$, so
		\[
		\sum_{j=1}^NX_j(t)=\sqrt N\,\Tr[B_0\,\Gemp^N_t]-N\,\Tr[B_0\,D_t].
		\]
		The second term is a first-order functional of the path of $D$, hence lies in $\mathcal M^{\mathrm{cf}}_t$ and has vanishing orthogonal part, and the orthogonal part is linear; subtracting $X_1^{\perp}$ gives (ii).
		
		(iii) The left sides of (i) are nonnegative, so each covariance is at least $-(N-2)^{-1}$ times the corresponding second moment. Since the constants lie in $\mathcal M^{\mathrm{cf}}_t$, $\E[(X_2^{\perp})^2]\le\E|X_2-\E X_2|^2\le CtN^{-1}$ by \eqref{eq:dict-1} and Lemma \ref{lem:linvan}, as in the proof of Theorem \ref{thm:X2X3fail}; and $\E\|\Delta^{(12),\perp1}\|_{\HS}^2\le\E\|\Delta^{(12)}\|_{\HS}^2\le\E\|\Delta^{(12)}\|_1^2\le Ct^2N^{-2}$ by Proposition \ref{prop:pairsmall}, the projection being a contraction entrywise.
		
		(iv) From (ii), $\E[S^2]\le2N\,\E[((\Tr B_0\Gemp)^{\perp})^2]+2\,\E[(X_1^{\perp})^2]$ and $N\,\E[((\Tr B_0\Gemp)^{\perp})^2]\le2\,\E[S^2]+2\,\E[(X_1^{\perp})^2]$, with $\E[(X_1^{\perp})^2]\le CtN^{-1}$ as in (iii). Solving (i) for the covariance and inserting these bounds together with $\E[(X_2^{\perp})^2]\le CtN^{-1}$ gives, for $N\ge4$,
		\[
		\big|\E[X_2^{\perp}X_3^{\perp}]\big|\le\frac C{N^2}\Big(N\,\E\big[\big((\Tr B_0\Gemp)^{\perp}\big)^2\big]+t\Big),
		\qquad
		N\,\E\big[\big((\Tr B_0\Gemp)^{\perp}\big)^2\big]\le C\Big(t+N^2\big|\E[X_2^{\perp}X_3^{\perp}]\big|\Big),
		\]
		and the two together are the first equivalence. The second is (i) read in both directions: $\E\|T\|_{\HS}^2\le CN^{-1}$ gives $|\E\ip{\Delta^{(12),\perp1}}{\Delta^{(13),\perp1}}|\le C(N^{-1}+t^2N^{-1})N^{-2}\le C'N^{-3}$, and conversely the conjectured order gives $\E\|T\|_{\HS}^2\le Ct^2N^{-1}+CN^{-1}$.
	\end{proof}
	
	\begin{remark}[One-sided validity, a testable aggregate, and the second chaos]\label{rem:aggregate}
		By (iii) both statements hold from below, at their exact orders, on every horizon: a failure of the conjecture can only be a positive alignment of the projected quantities across partners, which is the mechanism of Theorems \ref{thm:A4fail} and \ref{thm:X2X3fail}; the measured residual $N^2\,\E[X_2^{\perp}X_3^{\perp}]\approx-8\cdot10^{-4}$ of Subsection \ref{ssec:conditional} lies within the bound.
		
		By (iv) the second statement is a rate for one extensive observable: $\Tr[B_0\Gemp^N_t]$ must lie within $O(N^{-1/2})$ in $L^2$ of $\mathcal M^{\mathrm{cf}}_t$. By Theorem \ref{thm:compactness} the field $\Gemp^N$ converges to a Gaussian field driven jointly with the limit of the empirical reference field, and the conjecture quantifies the coherent part of that convergence; what it requires is a finite-$N$ expansion of $\Gemp^N$ around a first-order functional of the path of $D$ with an $L^2$ remainder of order $N^{-1/2}$, uniformly on $[0,T]$, hence the difference calculus at the precision after \eqref{eq:Ediff2}. By (i) the first statement is the requirement $\E\|\sum_{j\ge2}\Delta^{(1j),\perp1}\|_{\HS}^2=O(N^{-1})$, against $\E\|\sum_{j\ge2}\Delta^{(1j)}\|_{\HS}^2\ge c\,t^2$ from Theorem \ref{thm:A4fail}: the projection must remove all but an $O(N^{-1/2})$ fraction of the summed pair correlation.
		
		At fixed horizons the estimates leave the decision open in a specific way. At each foreign slot \eqref{eq:Ediff2} admits a component of $d_kX_j$ in the second Wiener chaos of $B_k$ with a deterministic kernel of size $t^2N^{-1}$ independent of the partner $j$; one candidate source is the part of the $k$-channel martingale coefficient of $\Delta^{(jk)}$ linear in $B_k$, which iterates the mechanism of Theorem \ref{thm:X2X3fail} one order up. Summed over the slots, such a component makes the aggregate of (iv) grow linearly in $N$ at fixed $t$, unless it is realised by a first-order functional of the path of $\gamma_k$, which by the argument of Lemma \ref{lem:cohrep} is a solvability question for the linearisation of $\Mop$ at $\gamma_0$ and not an estimate. Deciding between the two requires the expansion of $d_kX_j$ one order beyond \eqref{eq:Ediff2}; if the kernel survives, the second statement is false as stated on fixed horizons and the span must be enlarged by these single-slot second-chaos functionals, the third coherent channel left open by the tests of Subsection \ref{ssec:conditional}. On $[0,w_N]$ every candidate kernel is $o(t^{3/2})$, so Theorem \ref{thm:leadorder} is insensitive to the alternative.
	\end{remark}
	
	\subsection{The clauses for the projected quantities}\label{ssec:projclauses}
	
	This subsection carries Propositions \ref{prop:covrep}, \ref{prop:A5-structure} and \ref{prop:lyap} over to the projected quantities, with the common parts as an explicit source rather than as unknowns, and records what the limiting covariance becomes. The restated clauses are
	\begin{equation}\label{eq:A4prime}
		\text{(A4$'$)}\quad N\sum_{j,k\ne1}\E\ip{\Delta^{(1j),\perp1}}{\Delta^{(1k),\perp1}}=O(1),
	\end{equation}
	together with the residual family of (A5) and the decorrelation part of (A6) read for the projected quantities in the same way; all three follow from Conjecture \ref{conj:dec-cond} exactly as their unprojected forms follow from Conjecture \ref{conj:dec}.
	
	\begin{proposition}[Covariance representation, projected]\label{prop:covrep-proj}
		With the notation of Proposition \ref{prop:covrep}, split $\Delta^{(1j)}=\Delta^{(1j),\parallel}+\Delta^{(1j),\perp1}$ and $X_j=X_j^{\parallel}+X_j^{\perp}$. Then
		\[
		\Tr\Big[(A_0\otimes B_0)\sum_{j\ge2}\Delta^{(1j)}\Big]
		=\underbrace{(N-1)\,\Tr\big[(A_0\otimes B_0)\,\Delta^{\parallel}\big]}_{\text{common}}
		+\underbrace{\Tr\Big[(A_0\otimes B_0)\sum_{j\ge2}\Delta^{(1j),\perp1}\Big]}_{\text{idiosyncratic}},
		\]
		with $\Delta^{\parallel}$ the common part, independent of the partner label by exchangeability in law. The first term is $O_{L^2}(1)$ and explicit: by Proposition \ref{prop:meanpair} its deterministic part is $t\,\Tr[(A_0\otimes B_0)S_0]+O(t^{4/3}+tN^{-1/2})$, and its fluctuating part is the image of the path of $D$ under the kernel $\mathcal R^{\mathrm{ren}}_t$ together with the $\sigma(B_1)$-measurable tag component. The second term is $O_{L^2}(N^{-1/2})$ under \eqref{eq:A4prime}, by the decomposition $T_1+T_2$ of Proposition \ref{prop:covrep} applied to the projected family, whose two bounds hold verbatim with $\eps$-amplitudes unchanged, the projection being a contraction in $L^2$.
	\end{proposition}
	
	\begin{proof}
		The splitting and the count are immediate; exchangeability in law makes $\Pi^{(1)}_t\Delta^{(1j)}$ independent of $j$. For the order of the first term, the deterministic part is Proposition \ref{prop:meanpair} and the fluctuating part is bounded by $\|\Delta^{\parallel}\|_{L^2}\le\|\Delta^{(12)}\|_{L^2}=O(N^{-1})$ from \eqref{eq:rate-R}, multiplied by $N-1$. For the second, $\|\sum_j\Delta^{(1j),\perp1}\|_{L^2}^2$ is the left side of \eqref{eq:A4prime} divided by $N$.
	\end{proof}
	
	\begin{proposition}[Clause (A5) is unaffected by the common part]\label{prop:A5-proj}
		In the residual family of Proposition \ref{prop:A5-structure},
		\[
		\frac1N\sum_k\operatorname{Cov}\big(\gamma_1,\big[\Bmf(\Gamma^{(k)}-\gamma_k),\Gamma^{(2)}\big]\big),
		\]
		the contribution of the common parts $Y^{(k),\parallel}=\Pi^{\mathrm{cf}}_tY^{(k)}$ is $O(N^{-1})$, unconditionally. Consequently clause (A5) closes at the order $N\|K_t\|_1\le C$ under the restated residual family for the projected defects alone, and the failure of the common-part decorrelation does not enter it.
	\end{proposition}
	
	\begin{proof}
		Since $Y^{(k),\parallel}$ does not depend on $k$, the average over $k$ of the common contribution is the single covariance $\operatorname{Cov}(\gamma_1,[\Bmf(Y^{\parallel}),\Gamma^{(2)}])$ up to $O(N^{-1})$ from the two omitted labels. Write $D=D^{\ne1}+N^{-1}(\gamma_1-\eta)$ with $D^{\ne1}:=\tfrac1N\sum_{l\ne1}(\gamma_l-\eta)$, and $\Gamma^{(2)}=\gamma_2+Y^{(2)}$. By Lemma \ref{lem:reference} the pair $(D^{\ne1},\gamma_2)$ is a functional of $(B_l)_{l\ne1}$ and is therefore independent of $\gamma_1$, so
		\[
		\operatorname{Cov}\big(\gamma_1,\big[\Bmf(Y^{\parallel}(D^{\ne1})),\gamma_2\big]\big)=0 .
		\]
		Two corrections remain. Replacing $D^{\ne1}$ by $D$ changes $Y^{\parallel}$ by $O_{L^2}(N^{-1})$, the kernels of $\mathcal M^{\mathrm{cf}}_t$ being bounded and of first and second order in a variable of $L^2$ size $O(N^{-1/2})$ by Lemma \ref{lem:dev}; and replacing $\gamma_2$ by $\Gamma^{(2)}$ costs, by Cauchy--Schwarz, at most $C\|Y^{\parallel}\|_{L^2}\|Y^{(2)}\|_{L^2}=O(N^{-1/2})\cdot O(N^{-1/2})$, using \eqref{eq:dict-1} and Theorem \ref{thm:A1}. Both are $O(N^{-1})$, which is the order required in Proposition \ref{prop:A5-structure}.
	\end{proof}
	
	\begin{proposition}[The hierarchy for the projected quantities]\label{prop:lyap-proj}
		Items (i)--(iv) of Proposition \ref{prop:lyap} hold verbatim for the weighted family $K^{(\alpha),\perp}$ built from the projected correlations, with two changes: the shared-noise quadratic terms acquire an inhomogeneity, the corresponding bilinear expression in the common parts, which is explicit and of the order of the level; and the drift of the common parts is itself a deterministic linear image of the path of $D$ and of the tag noise, so it enters as a source and not as an unknown. Along any subsequence on which the projected members converge, the limits satisfy the same linear system with that source, which is well posed, and its unique solution expresses each limit in closed form in terms of the one-body limit data together with the kernels $\lambda^{\mathrm{ren}}$, $\mathcal R^{\mathrm{ren}}$ and $\vartheta$.
	\end{proposition}
	
	\begin{proof}
		The term inventory of Proposition \ref{prop:lyap} is unchanged, the projection being a bounded linear map commuting with the deterministic linear parts of the drift and with the partial traces. Writing each factor as $\parallel+\perp$ in the quadratic terms produces three groups: $\perp\!\cdot\!\perp$, which is the family itself; $\parallel\!\cdot\!\parallel$, which is explicit; and the cross group, bounded by Cauchy--Schwarz between an explicit factor and a member of the family, hence absorbed into the linear part after Young's inequality. Well-posedness and the closed form are as in Proposition \ref{prop:lyap}, the source being continuous in $t$ by the equi-Lipschitz property.
	\end{proof}
	
	The common parts do not leave the theorem: through the noise coefficients they contribute to the limiting covariance the additional deterministic term
	\begin{equation}\label{eq:Sigma-cf}
		\Sigma^{\mathrm{cf}}_t(A_0,B_0)=\operatorname{Cov}\big(\Tr[A_0\,\vartheta_t(\zeta)],\ \Tr[B_0\,\vartheta_t(\zeta)]\big),
	\end{equation}
	with $\zeta$ the Gaussian limit of $\sqrt N\,D$ from Lemma \ref{lem:dev} and $\vartheta_t$ the response kernel of the tagged noise coefficient, a covariance of the same type as those computed in Section \ref{sec:martingale-limits}.
	
	\begin{proposition}[What the limit acquires, and what it does not]\label{prop:sigma-proj}
		Under the restated clauses, the conclusions of Theorem \ref{thm:martingale-clt} hold with the following changes, and no others:
		\begin{enumerate}
			\item[(i)] the bracket of $\mathcal W$ becomes $\Sigma^{\mathrm{mf}}+\Sigma^{\mathrm{cf}}$, with $\Sigma^{\mathrm{cf}}$ given by \eqref{eq:Sigma-cf};
			\item[(ii)] the cross bracket $\Sigma^{\mathrm{cr}}$ acquires the corresponding explicit term, the covariance of the common part of $\sum_j\Theta^{(1j)}$ against $\Mop[\gamma]$, and is in general nonzero because both are driven by the empirical noise;
			\item[(iii)] the bracket of $V$ is the limit of the projected quantity, $\Sigma^{\perp}$, in place of $\Sigma$;
			\item[(iv)] $W$ remains a standard Brownian motion independent of $(V,\mathcal W)$.
		\end{enumerate}
		In particular the tag-slot component removed by \eqref{eq:enlarged}, being $\sigma(B_1)$-measurable, contributes no term of its own to the limiting covariance.
	\end{proposition}
	
	\begin{proof}
		Items (i)--(iii) are the computation in the proof of Theorem \ref{thm:martingale-clt} with $\sum_j\Delta^{(1j)}$ split as in Proposition \ref{prop:covrep-proj}: in each bracket density the $\perp\!\cdot\!\perp$ part is estimated by the restated clause exactly as the unprojected part was estimated by (A4), while the $\parallel\!\cdot\!\parallel$ part is an explicit functional of the path of $D$, converging by Lemma \ref{lem:dev} and the continuous mapping theorem to the stated covariance; the cross parts vanish by Cauchy--Schwarz between the two. For (iv), the bracket of $B_1$ with $\Tr(A_0\Mid)$ is zero at every finite $N$ by Lemma \ref{lem:orthogonality}, so no common part can enter it; and the bracket of $B_1$ with $\Tr(A_0\Mmf)$ has density $N^{-1/2}\Tr(A_0\Mop[\Gamma^{(1)}])+N^{-1/2}\Tr(A_0\sum_{j\ne1}\Theta^{(j1)})$, whose $L^1$ norm is $O(N^{-1/2})$ from $\|\Mop\|_1\le4\|L\|_\infty$ and $\|\sum_j\Delta^{(j1)}\|_{L^2}=O(1)$, the last bound being the unconditional one and not requiring (A4). The final statement is the same computation at the single slot $1$: the tag component enters $\sum_j\Theta^{(1j)}$ with the label $1$ appearing once, so it contributes $O(N^{-1})$ to the bracket densities, and it enters the drift only through $r^{(1)}_N$, where $\|N^{-1/2}\sum_j\Tr_2[A,\Delta^{(1j)}]\|_1=O(N^{-1/2})$ vanishes in $L^1$.
	\end{proof}
	
	The re-derivation above is at the level of the term inventories of Section \ref{sec:exact} and reuses their estimates; what it does not do is prove the restated clauses, which is Conjecture \ref{conj:dec-cond}. Its effect on the statements of Sections \ref{sec:martingale-limits} and \ref{sec:clt} is recorded in Corollary \ref{cor:main-proj}.
	
	\section{Martingale limits, tightness, and the empirical field}\label{sec:martingale-limits}
	
	\begin{theorem}[Joint martingale limit]\label{thm:martingale-clt}
		Assume clauses (A4) and (A6) of Hypothesis \ref{hyp:A}. Fix self-adjoint $A_1,\dots,A_m\in\Md$. The $\R^{1+2m}$-valued continuous local martingale
		\[
		X^N:=\Big(B_1,\ \big(\Tr(A_i\Mid)\big)_{i\le m},\ \big(\Tr(A_i\Mmf)\big)_{i\le m}\Big)
		\]
		converges in law in $C([0,T];\R^{1+2m})$ to a continuous centred Gaussian martingale $\big(W,(V(A_i))_{i\le m},(\mathcal W(A_i))_{i\le m}\big)$, in which $W$ is a standard Brownian motion and
		\[
		\big\langle V(A_i),V(A_k)\big\rangle_t=\int_0^t\Sigma_s(A_i,A_k)\,ds,\qquad
		\big\langle\mathcal W(A_i),\mathcal W(A_k)\big\rangle_t=\int_0^t\Sigma^{\mathrm{mf}}_s(A_i,A_k)\,ds,
		\]
		\[
		\big\langle V(A_i),\mathcal W(A_k)\big\rangle_t=\int_0^t\Sigma^{\mathrm{cr}}_s(A_i,A_k)\,ds,\qquad
		\big\langle W,V(A_i)\big\rangle\equiv\big\langle W,\mathcal W(A_i)\big\rangle\equiv0,
		\]
		with $\Sigma$ as in \eqref{eq:Sigma-limit},
		\begin{equation}\label{eq:Sigma-mf}
			\Sigma^{\mathrm{mf}}_s(A_0,B_0):=\E\Big[\Tr\big(A_0\Mop[\gamma_s]\big)\Tr\big(B_0\Mop[\gamma_s]\big)\Big]
		\end{equation}
		deterministic, continuous, symmetric and positive semidefinite, and $\Sigma^{\mathrm{cr}}=\lim_N\Sigma^{\mathrm{cr},N}$,
		\begin{equation}\label{eq:Sigma-cr}
			\Sigma^{\mathrm{cr},N}_s(A_0,B_0):=\E\Big[\sum_{j\ge2}\Tr\big(A_0\Theta^{(1j)}_s\big)\Tr\big(B_0\Mop[\Gamma^{(j)}_s]\big)\Big].
		\end{equation}
		In particular $W$ is independent of $(V,\mathcal W)$, while $V$ and $\mathcal W$ are in general correlated.
	\end{theorem}
	
	\begin{proof}
		All components are continuous, so the jump condition of the multidimensional martingale central limit theorem \cite[Thm.\ VIII.3.11]{JacodShiryaev2003} is vacuous, and it suffices to show that the predictable bracket matrix converges in probability, uniformly on $[0,T]$, to the stated continuous deterministic limit.
		
		\emph{The $(W,\cdot)$ entries.} $\langle B_1\rangle_t=t$. The bracket of $B_1$ with $\Tr(A_0\Mid)$ vanishes identically at every finite $N$ by Lemma \ref{lem:orthogonality}. The bracket of $B_1$ with $\Tr(A_0\Mmf)$ has density $N^{-1/2}\Tr(A_0\Mop[\Gamma^{(1)}_s])+N^{-1/2}\Tr(A_0\sum_{j\ne1}\Theta^{(j1)}_s)$, whose $L^1$ norm is $O(N^{-1/2})$ by $\|\Mop\|_1\le4\|L\|_\infty$ and (A4).
		
		\emph{The $\mathcal W$ block.} The second piece of $\Mmf$ has bracket density bounded by
		\[
		\frac{\|A_0\|_\infty\|B_0\|_\infty}N\sum_k\Big\|\sum_{j\ne k}\Theta^{(jk)}_s\Big\|_1^2\le\frac{4\|L\|_\infty^2\|A_0\|_\infty\|B_0\|_\infty}N\sum_k\Big\|\sum_{j\ne k}\Delta^{(jk)}_s\Big\|_1^2,
		\]
		whose expectation is $O(N^{-1})$ by (A4) and exchangeability in law; this piece is therefore asymptotically negligible, and by the Kunita--Watanabe inequality its bracket with every other component vanishes in the limit. The first piece has bracket density $N^{-1}\sum_jh_{A_0B_0}(\Gamma^{(j)}_s)$ with $h_{A_0B_0}(\rho)=\Tr(A_0\Mop[\rho])\Tr(B_0\Mop[\rho])$, a bounded Lipschitz function on $\mathcal S$. Replacing $\Gamma^{(j)}$ by $\gamma_j$ costs, in $L^1$, at most $\operatorname{Lip}(h)N^{-1}\sum_j\E\|\Gamma^{(j)}_s-\gamma_{j,s}\|_1\le4\operatorname{Lip}(h)\,\E e_1(s)=O(N^{-1/2})$ by \eqref{eq:dict-1}, Jensen and (A1). The remaining average $N^{-1}\sum_jh_{A_0B_0}(\gamma_{j,s})$ is an average of $N$ independent identically distributed bounded random variables with mean $\Sigma^{\mathrm{mf}}_s(A_0,B_0)$, so its $L^2$ distance to the mean is $O(N^{-1/2})$, uniformly in $s\le T$. Continuity of $s\mapsto\Sigma^{\mathrm{mf}}_s$ follows from continuity of $s\mapsto\gamma_s$ and boundedness.
		
		\emph{The $V$ block} is Theorem \ref{thm:covariance}.
		
		\emph{The $(V,\mathcal W)$ block.} The bracket density is
		\[
		\sum_{j\ge2}\Tr\big(A_0\Theta^{(1j)}_s\big)\Tr\big(B_0\Mop[\Gamma^{(j)}_s]\big)+\sum_{j\ge2}\Tr\big(A_0\Theta^{(1j)}_s\big)\Tr\Big(B_0\sum_{k\ne j}\Theta^{(kj)}_s\Big).
		\]
		The second sum is bounded in $L^1$ by $CN(\E\|\Delta^{(12)}\|_1^2)^{1/2}(\E\|\sum_k\Delta^{(kj)}\|_1^2)^{1/2}=O(N^{-1/2})$ by \eqref{eq:rate-R} and (A4). The first has mean $\Sigma^{\mathrm{cr},N}_s(A_0,B_0)$ and is bounded in $L^2$ uniformly in $N$ and $s$: writing $g_j:=\Tr(A_0\Theta^{(1j)}_s)\Tr(B_0\Mop[\Gamma^{(j)}_s])$, one has $|g_j|\le8\|L\|_\infty^2\|A_0\|_\infty\|B_0\|_\infty\|\Delta^{(1j)}_s\|_1$, whence $\operatorname{Var}(\sum_{j\ge2}g_j)\le(N-1)\sum_{j\ge2}\E g_j^2\le CN^2\E\|\Delta^{(12)}_s\|_1^2=O(1)$ by \eqref{eq:rate-R}. The convergence $\Sigma^{\mathrm{cr},N}\to\Sigma^{\mathrm{cr}}$ and the concentration of this sum about its mean are the last clause of (A6).
		
		Since the limiting bracket matrix is block diagonal between the first coordinate and the rest, and the limit is a jointly Gaussian martingale, $W$ is independent of $(V,\mathcal W)$.
	\end{proof}
	
	\begin{theorem}[The empirical fluctuation field]\label{thm:G-limit}
		Assume clauses (A4) and (A6) of Hypothesis \ref{hyp:A}; the tightness statement below is unconditional. Then $(\Gemp^N)_N$ is tight in $D([0,T];\Md)$, and the quadruple $(\Gemp^N,B_1,\Mid,\Mmf)$ converges in law, in $D([0,T];\Md)\times C([0,T];\R)\times D([0,T];\Md)^2$, to $(\Gemp,W,V,\mathcal W)$, where $(W,V,\mathcal W)$ is the triple of Theorem \ref{thm:martingale-clt} and $\Gemp$ is the unique solution of
		\begin{equation}\label{eq:G-limit}
			d\Gemp_t=\Big(-i\big[H+\Bmf(\eta_t),\Gemp_t\big]-i\big[\Bmf(\Gemp_t),\eta_t\big]+\Diss\Gemp_t\Big)dt+d\mathcal W_t,\qquad \Gemp_0=0 .
		\end{equation}
		The process $\Gemp$ has continuous paths, is a centred Gaussian process with values in the trace-zero self-adjoint elements of $\Md$, and is independent of $W$.
	\end{theorem}
	
	\begin{proof}
		\emph{Tightness.} Fix a self-adjoint basis element $A_0$ and use the Dynkin decomposition supplied by \eqref{eq:G-sde}. The drift density is bounded in $L^2$ uniformly in $N$ and $s\le T$ by Corollary \ref{cor:G-moment}; the bracket of the martingale part satisfies $\E[\langle\Tr(A_0\Mmf)\rangle_{\tau+\delta}-\langle\cdot\rangle_\tau]\le K\delta$ for stopping times $\tau\le T$, by the bracket bounds in the proof of Theorem \ref{thm:martingale-clt}. The Aldous--Rebolledo criterion \cite{Aldous1978,Rebolledo1980}, \cite[Thm.\ VI.4.13]{JacodShiryaev2003} gives tightness of each coordinate and, over a finite basis of $\Md$, of $\Gemp^N$.
		
		\emph{Identification.} Along a subsequence realising a joint limit, which exists by tightness of all four components, pass to the limit in the integrated form of \eqref{eq:G-sde}. The drift is a bounded linear function of $\Gemp^N_t$ with continuous deterministic coefficients, hence converges by the continuous mapping theorem after a Skorokhod representation; $\varrho_N$ vanishes in $L^1$ uniformly on $[0,T]$ by Corollary \ref{cor:G-moment}; the martingale term converges by Theorem \ref{thm:martingale-clt}. Every limit point therefore solves \eqref{eq:G-limit}.
		
		\emph{Uniqueness and Gaussianity.} Equation \eqref{eq:G-limit} is linear with bounded deterministic coefficients and a given continuous driving martingale, so it has the unique solution $\Gemp_t=\int_0^t\mathcal U_{t,s}\,d\mathcal W_s$ with $\mathcal U_{t,s}$ the deterministic bounded propagator of the homogeneous equation; the limit is therefore unique and the full sequence converges. A stochastic integral of a deterministic operator family against a continuous Gaussian martingale is Gaussian, so $\Gemp$ is centred Gaussian; it is independent of $W$ because $\mathcal W$ is. Finally $\Tr\Gemp^N_t=0$ for every $N$, so $\Tr\Gemp_t=0$.
	\end{proof}
	
	\begin{theorem}[Tightness of the tagged fluctuation]\label{thm:F-tight}
		Unconditionally, with no clause of Hypothesis \ref{hyp:A} assumed, $(F^N_t)_{t\le T}$ is tight in $D([0,T];\Md)$.
	\end{theorem}
	\begin{proof}
		Fix a self-adjoint basis element $A_0$ and write, from \eqref{eq:F-sde}, $\Tr(A_0F^N_t)=\int_0^t\beta^N_s\,ds+M^N_t$ with drift density $\beta^N$ and continuous local martingale $M^N$. By Lemmas \ref{lem:commutator}--\ref{lem:partial-trace}, \eqref{eq:F-moment} and Corollary \ref{cor:G-moment}, $\sup_N\sup_{s\le T}\E|\beta^N_s|^2<\infty$; and for stopping times $\tau\le T$ and $\delta>0$,
		\[
		\E\big[\langle M^N\rangle_{\tau+\delta}-\langle M^N\rangle_\tau\big]\le\E\int_\tau^{\tau+\delta}\Big(2\|A_0\|_\infty^2\big\|D\Mop[\gamma_s](F^N_s)+r^{(2)}_N\big\|_1^2+\widehat\Sigma^N_s(A_0,A_0)\Big)ds\le K\delta,
		\]
		by \eqref{eq:F-moment}, $\|r^{(2)}_N\|_1\le4\|L\|_\infty\|F^N\|_1$, Proposition \ref{prop:qv} and \eqref{eq:rate-R}. The Aldous--Rebolledo criterion gives tightness of each coordinate, and the joint modulus bound over a finite basis gives tightness in $D([0,T];\Md)$.
	\end{proof}
	
	\section{The dynamic central limit theorem}\label{sec:clt}
	
	\begin{theorem}[Dynamic central limit theorem]\label{thm:main}
		Assume Assumption \ref{assum:generators}, pure product initial data $\Gamma^N_0=\gamma_0^{\otimes N}$ with $\gamma_0$ pure and deterministic, Convention \ref{conv:coupling}, and Hypothesis \ref{hyp:A}. Then $(F^N,\Gemp^N)_N$ is tight in $D([0,T];\Md)^2$ and converges in law to the unique solution $(F,\Gemp)$ of the closed system
		\begin{align}
			d\gamma_t&=\big(-i[H+\Bmf(\eta_t),\gamma_t]+\Diss\gamma_t\big)dt+\Mop[\gamma_t]\,dW_t,&&\gamma_0\ \text{given},\label{eq:sys-1}\\
			d\Gemp_t&=\big(-i[H+\Bmf(\eta_t),\Gemp_t]-i[\Bmf(\Gemp_t),\eta_t]+\Diss\Gemp_t\big)dt+d\mathcal W_t,&&\Gemp_0=0,\label{eq:sys-2}\\
			dF_t&=\big(-i[H+\Bmf(\eta_t),F_t]+\Diss F_t-i[\Bmf(\Gemp_t),\gamma_t]\big)dt+D\Mop[\gamma_t](F_t)\,dW_t+dV_t,&&F_0=0,\label{eq:sys-3}
		\end{align}
		where $\eta$ is the deterministic solution of \eqref{eq:hartree} and $(W,\mathcal W,V)$ is the limiting triple of Theorem \ref{thm:martingale-clt}: $W$ is a standard Brownian motion, $(\mathcal W,V)$ is a continuous Gaussian martingale independent of $W$ with brackets $\Sigma^{\mathrm{mf}}$, $\Sigma$ and cross bracket $\Sigma^{\mathrm{cr}}$ given by \eqref{eq:Sigma-mf}, \eqref{eq:Sigma-limit}, \eqref{eq:Sigma-cr}. Both $F_t$ and $\Gemp_t$ take values in the trace-zero self-adjoint elements of $\Md$; $\Gemp$ is Gaussian; and, conditionally on $(W,\mathcal W)$, $F$ is Gaussian.
		
		Under clauses (A4)--(A6) alone the same conclusion holds along subsequences, with $\Sigma$ and $\Sigma^{\mathrm{cr}}$ any limit points of $\Sigma^N$, $\Sigma^{\mathrm{cr},N}$. Unconditionally, Theorem \ref{thm:compactness} gives the subsequential conclusion with the law of the driving pair unidentified, and Corollary \ref{cor:meanid} identifies its mean bracket; under Conjecture \ref{conj:dec-cond} alone, Corollary \ref{cor:main-proj} restores the full conclusion with the driving triple of Proposition \ref{prop:sigma-proj}.
	\end{theorem}
	
	\begin{proof}
		Tightness of $F^N$ is Theorem \ref{thm:F-tight} and of $\Gemp^N$ is Theorem \ref{thm:G-limit}. By Theorems \ref{thm:martingale-clt} and \ref{thm:G-limit} we may pass to a subsequence along which $(F^N,\Gemp^N,B_1,\Mid,\Mmf)$ converges jointly and, by the Skorokhod representation, assume the convergence almost sure on a common space. Insert this into \eqref{eq:F-sde}. The remainders satisfy $\sup_t\E\|r^{(1)}_N\|_1=O(N^{-1/2})$ by Corollary \ref{cor:G-moment}, and the $r^{(2)}_N$ contribution to the stochastic integral vanishes in $L^2$ by It\^o's isometry, \eqref{eq:F-moment} and $\|r^{(2)}_N\|_1\le2\|L\|_\infty N^{-1/2}\|F^N\|_1^2$ with $\|F^N\|_1\le2\sqrt N$. The drift terms converge by dominated convergence and continuity of the bounded linear and bilinear maps involved, using $\Gamma^{(1)}_t\to\gamma_t$ in probability uniformly on $[0,T]$, which follows from \eqref{eq:dict-1} and (A1). The martingale terms converge jointly with their integrands by Theorem \ref{thm:martingale-clt} and the stability of stochastic integration for uniformly tight semimartingale sequences \cite[Ch.\ VI, IX]{JacodShiryaev2003}. Every limit point therefore solves \eqref{eq:sys-3}, driven by $(W,V)$ with the field $\Gemp$ of \eqref{eq:sys-2}.
		
		Uniqueness: given $(W,\mathcal W,V)$, equation \eqref{eq:sys-1} is the well-posed one-particle filter of Theorem \ref{thm:wellposed}, \eqref{eq:sys-2} is uniquely solvable by Theorem \ref{thm:G-limit}, and \eqref{eq:sys-3} is then the linear mild equation \eqref{eq:mild-linear} with forcing $-i[\Bmf(\Gemp_t),\gamma_t]$ and an additional martingale term, uniquely solvable by Theorem \ref{thm:linear-flow}. The law of $(F,\Gemp)$ is therefore determined by the law of the driving triple, all limit points coincide, and the full sequence converges. The trace-zero property passes to the limit from $\Tr F^N_t=\Tr\Gemp^N_t=0$. Gaussianity of $\Gemp$ is Theorem \ref{thm:G-limit}; conditionally on $(W,\mathcal W)$ the coefficients of \eqref{eq:sys-3} and the forcing are determined and $V$ is Gaussian, so $F$ is conditionally Gaussian.
	\end{proof}
	
	\begin{corollary}[The limit under the restated clauses]\label{cor:main-proj}
		Assume Assumption \ref{assum:generators}, pure product initial data, Convention \ref{conv:coupling}, and Conjecture \ref{conj:dec-cond}. Then the conclusion of Theorem \ref{thm:main} holds, with the system \eqref{eq:sys-1}--\eqref{eq:sys-3} unchanged in form and the driving triple replaced by the one of Proposition \ref{prop:sigma-proj}: $W$ is a standard Brownian motion independent of $(\mathcal W,V)$, the bracket of $\mathcal W$ is $\Sigma^{\mathrm{mf}}+\Sigma^{\mathrm{cf}}$ with $\Sigma^{\mathrm{cf}}$ as in \eqref{eq:Sigma-cf}, the bracket of $V$ is $\Sigma^{\perp}$, and the cross bracket is $\Sigma^{\mathrm{cr}}$ with the term of Proposition \ref{prop:sigma-proj}(ii) included.
	\end{corollary}
	
	\begin{proof}
		The proof of Theorem \ref{thm:main} uses Hypothesis \ref{hyp:A} in three places: the tightness inputs of Theorems \ref{thm:F-tight} and \ref{thm:G-limit}, unconditional by Corollary \ref{cor:upgrade}; the martingale limit, which is Proposition \ref{prop:sigma-proj} under the restated clauses; and the vanishing of the drift remainder $r^{(1)}_N$, which is the last computation of Proposition \ref{prop:sigma-proj} and needs no clause. The identification and uniqueness arguments are unchanged, the driving triple entering only through its law.
	\end{proof}
	
	\begin{theorem}[Unconditional compactness form of the limit]\label{thm:compactness}
		Assume Assumption \ref{assum:generators}, pure product initial data $\Gamma^N_0=\gamma_0^{\otimes N}$ with $\gamma_0$ pure and deterministic, and Convention \ref{conv:coupling}, and no clause of Hypothesis \ref{hyp:A}. Then the quintuple $(F^N,\Gemp^N,B_1,\Mid,\Mmf)$ is tight in $D([0,T];\Md)^2\times C([0,T];\R)\times D([0,T];\Md)^2$, and along every subsequence realising a joint limit $(F,\Gemp,W,V,\mathcal W)$:
		\begin{enumerate}
			\item[(i)] $W$ is a standard Brownian motion; for every self-adjoint $A_0\in\Md$ the processes $\Tr(A_0V)$ and $\Tr(A_0\mathcal W)$ are continuous square-integrable martingales with $\E\big\langle\Tr(A_0V)\big\rangle_T+\E\big\langle\Tr(A_0\mathcal W)\big\rangle_T\le C\|A_0\|_\infty^2$, and
			\[
			\big\langle W,\Tr(A_0V)\big\rangle\equiv0,\qquad \big\langle W,\Tr(A_0\mathcal W)\big\rangle\equiv0;
			\]
			\item[(ii)] $(F,\Gemp)$ solves the system \eqref{eq:sys-1}--\eqref{eq:sys-3} driven by $(W,\mathcal W,V)$, with $F_0=\Gemp_0=0$ and values in the trace-zero self-adjoint elements of $\Md$.
		\end{enumerate}
		What the statement does not assert, and Conjecture \ref{conj:dec-cond} supplies through Corollary \ref{cor:main-proj}, is the identification of the brackets of $(V,\mathcal W)$, the Gaussianity of the pair, and with them the uniqueness of the limit point.
	\end{theorem}
	
	\begin{proof}
		\emph{Step 1: bracket densities in $L^2$, unconditionally.} Fix self-adjoint $A_0$ with $\|A_0\|_\infty\le1$. The bracket of $\Tr(A_0\Mid)$ has the nonnegative density $\widehat\Sigma^N_s(A_0,A_0)$ of Proposition \ref{prop:qv}, with $\sup_{N,s\le T}\E\,\widehat\Sigma^N_s\le C$ by Proposition \ref{prop:xi-compact}, whose hypotheses are unconditional by Corollary \ref{cor:upgrade}; moreover, writing $\widehat\Sigma^N=N\sum_{j\ge2}f^{(j)}$ with $|f^{(j)}|\le C\|\Delta^{(1j)}\|_1^2$ as in the proof of Theorem \ref{thm:covariance},
		\[
		\E\big(\widehat\Sigma^N_s\big)^2\le N^2(N-1)\sum_{j\ge2}\E\big(f^{(j)}_s\big)^2\le C\,N^4\,\E\big\|\Delta^{(12)}_s\big\|_1^4\le C,
		\]
		by Cauchy--Schwarz over the labels, exchangeability in law, and \eqref{eq:rate-4}, unconditional by Corollary \ref{cor:A2every}. The bracket density of $\Tr(A_0\Mmf)$ is bounded by $C+\Upsilon_s$ with $\Upsilon_s:=\tfrac CN\sum_k\|\sum_{j\ne k}\Delta^{(jk)}_s\|_1^2$, by Proposition \ref{prop:coefficients} and $\|\Mop[\rho]\|_1\le4\|L\|_\infty$ on $\mathcal S$; here
		\[
		\E \Upsilon_s\le C\,\E\Big\|\sum_{j\ne2}\Delta^{(j2)}_s\Big\|_1^2\le C(N-1)^2\,\E\big\|\Delta^{(12)}_s\big\|_1^2\le C,\qquad
		\E \Upsilon_s^2\le C(N-1)^4\,\E\big\|\Delta^{(12)}_s\big\|_1^4\le C,
		\]
		by exchangeability, the triangle inequality with Cauchy--Schwarz over the labels, \eqref{eq:rate-R} and \eqref{eq:rate-4}. The cross density of $B_1$ with $\Tr(A_0\Mmf)$ is bounded by $N^{-1/2}(C+C\|\sum_{j\ne1}\Delta^{(j1)}_s\|_1)$, of $L^2$ norm $O(N^{-1/2})$ by the same bounds; the cross bracket with $\Tr(A_0\Mid)$ vanishes identically by Lemma \ref{lem:orthogonality}.
		
		\emph{Step 2: tightness.} For a continuous local martingale $M^N$ whose bracket has density $\rho^N\ge0$ with $\sup_{N,s\le T}\E(\rho^N_s)^2\le C$, Cauchy--Schwarz in time gives, for every stopping time $\tau\le T$ and $\delta\le1$, $\langle M^N\rangle_{\tau+\delta}-\langle M^N\rangle_\tau\le\delta^{1/2}(\int_0^{T+1}(\rho^N_s)^2ds)^{1/2}$ pathwise, whence $\E[\langle M^N\rangle_{\tau+\delta}-\langle M^N\rangle_\tau]\le C\delta^{1/2}$. By Step 1 this applies to $\Tr(A_0\Mid)$ and $\Tr(A_0\Mmf)$, and the corresponding drift densities of $F^N$ and $\Gemp^N$ are bounded in $L^2$ uniformly by \eqref{eq:F-moment} and Corollary \ref{cor:G-moment}, so the Aldous--Rebolledo criterion \cite{Aldous1978,Rebolledo1980} gives tightness of every coordinate over a finite self-adjoint basis of $\Md$, exactly as in Theorems \ref{thm:F-tight} and \ref{thm:G-limit}, with the $L^2$ bounds of Step 1 replacing the vanishing statements that clause (A4) supplied there; boundedness, not smallness, is what tightness consumes.
		
		\emph{Step 3: P-UT, limiting brackets, and orthogonality.} All martingale coordinates are continuous with brackets at time $T$ bounded in $L^1$ uniformly in $N$, hence tight; a sequence of continuous local martingales with tight terminal brackets is predictably uniformly tight, and along a subsequence realising a joint limit the quadratic covariations converge jointly with the processes \cite[Ch.\ VI \S6]{JacodShiryaev2003}. The identically vanishing bracket of Step 1 gives $\langle W,\Tr(A_0V)\rangle\equiv0$, and the $O(N^{-1/2})$ cross density gives $\langle W,\Tr(A_0\mathcal W)\rangle\equiv0$ in the limit. By Doob's inequality, $\E\sup_{t\le T}|\Tr(A_0\Mid_t)|^2\le4\,\E\langle\Tr(A_0\Mid)\rangle_T\le C$ uniformly in $N$, so the coordinates are uniformly integrable at every time and every limit point is a square-integrable martingale, with $\E\langle\Tr(A_0V)\rangle_T\le C$ by Fatou's lemma; likewise for $\mathcal W$. The first coordinate converges to a standard Brownian motion.
		
		\emph{Step 4: identification of the equations.} Along the subsequence, pass to a Skorokhod representation and insert into \eqref{eq:F-sde} and \eqref{eq:G-sde}. The remainder and drift computations are verbatim those of the proofs of Theorems \ref{thm:main} and \ref{thm:G-limit}, whose inputs at this step are unconditional: $\sup_t\E\|r^{(1)}_N\|_1=O(N^{-1/2})$ and $\varrho_N\to0$ in $L^1$ by Corollary \ref{cor:G-moment}, the $r^{(2)}_N$ contribution vanishes by It\^o's isometry with \eqref{eq:F-moment}, and the drifts converge using \eqref{eq:dict-1} with clause (A1), which is Theorem \ref{thm:A1}. The stochastic integrals converge jointly with their integrands by the predictable uniform tightness of Step 3 and the stability of stochastic integration \cite[Ch.\ VI, IX]{JacodShiryaev2003}. Every limit point therefore solves \eqref{eq:sys-1}--\eqref{eq:sys-3}; the trace-zero and self-adjointness properties pass to the limit from their finite-$N$ counterparts.
	\end{proof}
	
	\begin{corollary}[Mean covariance of the idiosyncratic noise]\label{cor:meanid}
		In the setting of Theorem \ref{thm:compactness}, refine the subsequence so that in addition $\Xi^N\to\Xi$ in $C([0,T];M_{d^2}(\C)^{\otimes2})$, which is possible by Proposition \ref{prop:xi-compact}. Then, for all self-adjoint $A_0,B_0\in\Md$ and every $t\le T$,
		\[
		\E\big\langle\Tr(A_0V),\Tr(B_0V)\big\rangle_t\;=\;\int_0^t\Tr\big[(\mathsf A_0\otimes\mathsf B_0)\,\Xi_s\big]\,ds,
		\]
		with $\mathsf A_0,\mathsf B_0$ as in \eqref{eq:Sigma-Xi}. The mean bracket of the idiosyncratic noise is therefore identified unconditionally, by the same formula through which Conjecture \ref{conj:dec-cond} would identify the bracket itself; what the conjecture adds is the concentration of the bracket about its mean, the Gaussianity of the pair $(V,\mathcal W)$, and the uniqueness of the limit point.
	\end{corollary}
	
	\begin{proof}
		By Step 3 of the proof of Theorem \ref{thm:compactness}, along the subsequence and after a Skorokhod representation, the covariation $\langle\Tr(A_0\Mid),\Tr(B_0\Mid)\rangle_t$ converges almost surely to $\langle\Tr(A_0V),\Tr(B_0V)\rangle_t$ for every $t$. By Proposition \ref{prop:qv} the prelimit equals $\int_0^t\widehat\Sigma^N_s(A_0,B_0)\,ds$, and its second moment is bounded uniformly in $N$ by Cauchy--Schwarz in time and Step 1 of the same proof after polarisation, $\widehat\Sigma^N(A_0,B_0)=\tfrac14\widehat\Sigma^N(A_0+B_0,A_0+B_0)-\tfrac14\widehat\Sigma^N(A_0-B_0,A_0-B_0)$; the family is therefore uniformly integrable at each $t$, and the expectations converge. Finally $\E\int_0^t\widehat\Sigma^N_s\,ds=\int_0^t\Sigma^N_s\,ds$ by Fubini's theorem, and $\Sigma^N_s(A_0,B_0)=\tfrac{N-1}N\Tr[(\mathsf A_0\otimes\mathsf B_0)\Xi^N_s]\to\Tr[(\mathsf A_0\otimes\mathsf B_0)\Xi_s]$ uniformly on $[0,T]$, by \eqref{eq:Sigma-Xi} and the choice of subsequence.
	\end{proof}
	
	\begin{remark}[What is unconditional about the limit]\label{rem:compactness}
		For generic data, clause (A4) and the decorrelation part of clause (A6) are false (Theorems \ref{thm:A4fail} and \ref{thm:X2X3fail}, Corollary \ref{cor:pairfail}), so the subsequential clause of Theorem \ref{thm:main}, stated under those clauses, is vacuous exactly for such data. Theorem \ref{thm:compactness} replaces it with a statement free of every hypothesis: the fluctuation pair $(F^N,\Gemp^N)$ has limit dynamics, and every limit is of the form \eqref{eq:sys-1}--\eqref{eq:sys-3}, driven by a Brownian motion orthogonal to a square-integrable martingale pair. The whole of Conjecture \ref{conj:dec-cond} is thereby confined to the law of that pair: the brackets of $(V,\mathcal W)$, their Gaussianity, and the resulting uniqueness of the limit point, which Corollary \ref{cor:main-proj} supplies under the conjecture. The structure of the answer to Open Problem 3, including the non-autonomy of $F$ and the presence of an idiosyncratic martingale uncorrelated with the tagged innovation, is unconditional.
	\end{remark}

	\begin{remark}[Reading of the limit]\label{rem:reading}
		The fluctuation of the tagged particle's conditional state is driven by three sources:
		\begin{enumerate}
			\item[(i)] its own innovation $W$, entering through the derivative $D\Mop[\gamma_t]$ of the measurement nonlinearity;
			\item[(ii)] the idiosyncratic Gaussian martingale $V$, carried by the connected correlations $\Delta^{(1j)}$ with the other particles, with covariance \eqref{eq:Sigma-limit}, a second-order object;
			\item[(iii)] the empirical fluctuation field $\Gemp$, entering the drift through $-i[\Bmf(\Gemp_t),\gamma_t]$. This term is of order one and makes the limit a system: the tagged fluctuation is not autonomous.
		\end{enumerate}
		At $A=0$ every clause of Hypothesis \ref{hyp:A} holds with vanishing constants, $\Bmf\equiv0$, $V\equiv0$, and \eqref{eq:sys-3} is homogeneous with $F_0=0$, so $F\equiv0$, matching $F^N\equiv0$; while \eqref{eq:sys-2} reduces to the central limit field of the independent identically distributed filters $\gamma_{j,\cdot}$. All the content of the theorem is carried by the interaction.
	\end{remark}
	
	\begin{remark}[Relation to Open Problem 3]\label{rem:open-problem}
		The answer to Open Problem 3 of \cite{Kolokoltsov2026} in the finite-dimensional, bounded-coefficient case is Theorem \ref{thm:compactness} together with Corollary \ref{cor:meanid}, unconditionally: the pair $(F^N,\Gemp^N)$ is tight, every limit point solves the closed system \eqref{eq:sys-1}--\eqref{eq:sys-3} driven by a Brownian motion orthogonal to a square-integrable martingale pair, and the mean covariance of the idiosyncratic noise is the limit of the rescaled pair correlations. Corollary \ref{cor:main-proj} upgrades this to the Gaussian form with a unique limit under Conjecture \ref{conj:dec-cond}, whose leading-order and one-sided parts are Theorem \ref{thm:leadorder} and Proposition \ref{prop:aggregate}. Two features of the answer differ from what the formulation of the problem suggests. The limiting equation for $F$ alone is not closed, so the object to be identified is the pair $(F,\Gemp)$; and the covariance of the idiosyncratic noise is not a function of the limiting one-particle state, so it cannot be written in closed form at the level of the one-particle description.
	\end{remark}
	
	\section{Status and further directions}\label{sec:status}
	
	\subsection{The logical status of Hypothesis \ref{hyp:A}}\label{ssec:status-logical}
	
	Hypothesis \ref{hyp:A} is invoked in the statements of Sections \ref{sec:martingale-limits} and \ref{sec:clt}, and its clauses do not have equal standing. Clause (A1) is Theorem \ref{thm:A1}, and clauses (A2) and (A3) are Corollaries \ref{cor:A2every} and \ref{cor:A3}, all three uniform in $N$ and on every horizon: they are theorems and not assumptions, and Corollary \ref{cor:upgrade} records which dependences on them are thereby discharged. Clause (A4) and the decorrelation part of clause (A6) are false for generic data, by Theorem \ref{thm:A4fail}, Corollary \ref{cor:pairfail} and Theorem \ref{thm:X2X3fail}, so every statement carrying them is vacuous for such data; Theorem \ref{thm:compactness} is the hypothesis-free replacement, and Corollary \ref{cor:main-proj} the form the clauses take once the coherent channels are projected out. What survives as an assumption is Conjecture \ref{conj:dec-cond} alone, and by Theorem \ref{thm:compactness} it bears only on the law of the driving pair $(V,\mathcal W)$, not on the existence of the limit or on the form of the equations it satisfies. What it leaves open is quantitative: by Theorem \ref{thm:leadorder} and Proposition \ref{prop:aggregate} its leading-order form is a theorem on $[0,w_N]$ and both of its statements hold from below at their exact orders, so what is outstanding is the passage from the leading order to the stated orders, and from the shrinking horizon to a fixed one. No statement of Sections \ref{sec:martingale-limits}--\ref{sec:clt} requires the conjecture for the existence, the equations, or the mean covariance of the limit: Theorem \ref{thm:compactness} and Corollary \ref{cor:meanid} are free of every clause, and Corollary \ref{cor:main-proj} records what the conjecture adds: the concentration of the brackets about their identified means, the Gaussianity of the pair, and the uniqueness of the limit point.
	
	Pathwise permutation symmetry of $\Gamma^N_t$ is nowhere used; only the exchangeability in law of Section \ref{ssec:exch}, which is unconditional.
	
	\subsection{Open questions and further directions}\label{ssec:further}
	
	\emph{The conjecture, and the difference calculus.} Conjecture \ref{conj:dec-cond} is the single open input, and Proposition \ref{prop:noabs} shows what cannot settle it: no bound on absolute amplitudes, however sharp, since the clause concerns signs and the calculus of Section \ref{sec:excitation} discards phases. The affirmative route is the one identified in Remark \ref{rem:shrink}; what the difference estimates reach along it is the leading order on $[0,w_N]$, and what they do not reach is the decorrelation of the residual foreign increments across partners. By Proposition \ref{prop:aggregate} the open content is one-sided and concentrated in a single approximation statement for the empirical field, with the second-chaos alternative of Remark \ref{rem:aggregate} as the identified way it could fail. What Sections \ref{sec:diagonal} and \ref{sec:moments} achieve for the excitation observables is a diagonal rate uniform in the level, which the correlation hierarchy lacks, and the enabling identity \eqref{eq:keyzero} is available for the two-copy coupling: at every unresampled slot the copies share their reference filter, so $p_j=p_j'$ and the cancellation holds simultaneously in both, hence for the difference. Running the machinery of those sections on the difference vector $\chi=\Psi-\Psi'$ would give the difference gain at every level at once and on every horizon, and with it the resampling estimate that Subsection \ref{ssec:conditional} identifies as the surviving content of the conjecture. Two obstacles are specific to $\chi$ and are not present in the one-copy calculus: the functional $\ip{\chi}{q_T\chi}$ is not small at the resampled slot, the two copies differing there at order one, so a compressed functional adapted to the coupling must be found; and the analogue of Theorem \ref{thm:master} for it must accommodate the slot-$k$ source, the two-copy mismatch terms at which $\phi_k\ne\phi_k'$, and the normalisation drift.
	
	\emph{The covariance $\Sigma$.} By Remark \ref{rem:covariance-scope} the idiosyncratic covariance is not a function of the limiting one-particle state, being a limit of rescaled pair correlations, so it admits no closed form at the level of the one-particle description. Proposition \ref{prop:lyap} reduces its identification to the well-posedness of a linear system for the weighted hierarchy, and Proposition \ref{prop:lyap-proj} to the same system with the common parts as an explicit source. What description $\Sigma$ does admit, and whether the limit of $\Xi^N$ can be characterised without passing through the whole hierarchy, is open.
	
	\emph{Infinite dimensions and unbounded coefficients.} The standing hypothesis $\Hilbert=\C^d$ with $d<\infty$ is what makes every operator bounded and every closed bounded set compact, and it is used throughout: in the norm equivalences \eqref{eq:norm-equiv}, in the compactness argument of Proposition \ref{prop:xi-compact}, and in the boundedness of every coefficient on $\mathcal S$. The formulation results of Section \ref{sec:objects} are insensitive to it, so the coupling, the mean field of the limiting equation and the choice of connected correlations carry over unchanged; the excitation calculus does not, and an infinite-dimensional treatment with unbounded $H$ and $L$ would require replacing the pathwise dictionary of Section \ref{sec:excitation} by estimates relative to the domain of the generator.
	
	\emph{Other measurement channels and other models.} The analysis uses the diffusive form of \eqref{eq:N-particle} through the innovation subtraction that makes \eqref{eq:keyzero} hold. Counting measurement, several measurement channels per particle, partial or collective observation, and non-identical particles are each natural, and in each the question is whether an analogue of \eqref{eq:keyzero} survives; where it does, the second-order analysis of Sections \ref{sec:diagonal} and \ref{sec:moments} should transfer. The failure mechanism of Section \ref{sec:decorrelation} is expected to be general: the source $S_0$ of \eqref{eq:S0} is the connected part of the interaction commutator at the initial product state, and its leading block is the pair-excitation amplitude of the interaction, a quantity available in any of these settings and computable from the data.
	
	\emph{Numerical resolution.} The tests of Subsection \ref{ssec:conditional} confirm the second statement of Conjecture \ref{conj:dec-cond} at its stated order and leave the order of the residual in the first statement unresolved: separating $N^{-2}$ from $N^{-3}$ over $4\le N\le8$ through the estimator \eqref{eq:tagestimator} requires about two further orders of magnitude in the sample, and larger $N$ requires a scheme better than weak order one. This is a well-posed computation, and its outcome would discriminate between the conjecture and a third coherent channel not identified here; by Proposition \ref{prop:aggregate}(iv) the same question is decided by the boundedness in $N$ of the nonnegative order-one aggregate $N\,\E[((\Tr B_0\Gemp^N_t)^{\perp})^2]$, which does not require separating $N^{-2}$ from $N^{-3}$ in a small signed quantity.
	
	\subsection{Summary}
	
	Hypothesis \ref{hyp:A} has been reduced to a single statement. Clauses (A1)--(A3), the correlation rates \eqref{eq:rate-3body} and \eqref{eq:rate-R}, all moments of the excitation number and its exponential moment at a deterministic rate are unconditional, uniformly in $N$ and on every horizon. Clauses (A4)--(A6) are of a different nature: no bound on absolute amplitudes can reach them (Proposition \ref{prop:noabs}), and they reduce to one decorrelation statement. That statement survives as Conjecture \ref{conj:dec-cond}, proved at the leading order on the shrinking horizon and from below at its exact orders on every horizon, so that what is open is one-sided and lies above the leading order; the form without its projections is false for generic data, the mechanism being the two-channel common response that the projections remove. Unconditionally, $(F^N,\Gemp^N)$ is tight, every limit point solves the system \eqref{eq:sys-1}--\eqref{eq:sys-3} driven by a Brownian motion orthogonal to a square-integrable martingale pair, and the mean bracket of $V$ is identified (Theorem \ref{thm:compactness}, Corollary \ref{cor:meanid}), so the conjecture bears only on the law of the driving pair; under it, Corollary \ref{cor:main-proj} gives the dynamic central limit theorem along the full sequence, with the limiting covariance carrying the computable term \eqref{eq:Sigma-cf}. Nothing else is outstanding: the fixed-horizon input \eqref{eq:fourbodydiff} is required by no statement above, and the route towards the conjecture, together with the questions that remain, is set out in Subsection \ref{ssec:further}.
	
	The mechanism behind this reduction is the following. The analysis of the excitation vector is obstructed by the grading-changing measurement noise, whose coefficient moves $\chi$ between grading sectors at rate one; the analysis of the diagonal quadratic functionals is not, because the diagonal compression of the reference noise operator vanishes by the definition of $\nu_j$, which is the subtraction performed by the innovation representation. Amplitudes are therefore controlled by a closed system of scalar equations, while phases are not, and it is phases that clauses (A4)--(A6) concern. Within the amplitude analysis the same division is what closes the hierarchy: the drift is handled by the transport of Lemma \ref{lem:genfun}, and the noise enters only through the scalar $J_t$, whose law is supplied by Lemma \ref{lem:tailJ}.
	
			\section*{Acknowledgements}
		
		I am deeply grateful to Professor V.\,N.\ Kolokoltsov (University of
		Warwick) for posing this problem---Open Problem~3 of
		\cite{Kolokoltsov2026}---and for several fruitful discussions during
		the course of this work.


\begin{thebibliography}{99}
		
	\bibitem{Kolokoltsov2026} V.\ N.\ Kolokoltsov, \emph{Quantum filtering and propagation of chaos for open quantum systems, with applications to quantum feedback control and quantum mean-field games}, preprint, arXiv:2607.08507 (2026).
	
	\bibitem{Kol2021LLN} V.\ N.\ Kolokoltsov, \emph{The law of large numbers for quantum stochastic filtering and control of many-particle systems}, Theoretical and Mathematical Physics \textbf{208}:1 (2021), 937--957.
	
	\bibitem{Kol2022MFG} V.\ N.\ Kolokoltsov, \emph{Quantum mean-field games}, Annals of Applied Probability \textbf{32}:3 (2022), 2254--2288.
	
	\bibitem{Kol2010CLT} V.\ N.\ Kolokoltsov, \emph{The central limit theorem for the Smoluchovski coagulation model}, Probability Theory and Related Fields \textbf{146}:1 (2010), 87--153.
	
	\bibitem{Kol2010Book} V.\ N.\ Kolokoltsov, \emph{Nonlinear Markov Processes and Kinetic Equations}, Cambridge Tracts in Mathematics \textbf{182}, Cambridge University Press, 2010.
	
	\bibitem{Pickl2011} P.\ Pickl, \emph{A simple derivation of mean field limits for quantum systems}, Letters in Mathematical Physics \textbf{97}:2 (2011), 151--164.
	
	\bibitem{KnowlesPickl2010} A.\ Knowles and P.\ Pickl, \emph{Mean-field dynamics: singular potentials and rate of convergence}, Communications in Mathematical Physics \textbf{298}:1 (2010), 101--138.
	
	\bibitem{BenArousKirkpatrickSchlein2013} G.\ Ben Arous, K.\ Kirkpatrick and B.\ Schlein, \emph{A central limit theorem in many-body quantum dynamics}, Communications in Mathematical Physics \textbf{321}:2 (2013), 371--417.
	
	\bibitem{Aldous1978} D.\ Aldous, \emph{Stopping times and tightness}, Annals of Probability \textbf{6}:2 (1978), 335--340.
	
	\bibitem{Rebolledo1980} R.\ Rebolledo, \emph{Central limit theorems for local martingales}, Zeitschrift f\"ur Wahrscheinlichkeitstheorie und verwandte Gebiete \textbf{51}:3 (1980), 269--286.
	
	\bibitem{Pinelis1994} I.\ Pinelis, \emph{Optimum bounds for the distributions of martingales in Banach spaces}, Annals of Probability \textbf{22}:4 (1994), 1679--1706.
	
	\bibitem{JacodShiryaev2003} J.\ Jacod and A.\ N.\ Shiryaev, \emph{Limit Theorems for Stochastic Processes}, 2nd ed., Grundlehren der mathematischen Wissenschaften \textbf{288}, Springer, 2003.
	
	\bibitem{Hoeffding1948} W.\ Hoeffding, \emph{A class of statistics with asymptotically normal distribution}, Annals of Mathematical Statistics \textbf{19}:3 (1948), 293--325.
		
	\end{thebibliography}
\end{document}